\documentclass[11pt,leqno]{amsart}
\usepackage{amsmath,amsfonts,amsthm,amscd,amssymb,mathrsfs,graphicx,xcolor,bm}
\usepackage{epstopdf}
\numberwithin{equation}{section}
\usepackage{fullpage}

\newcommand\s{\sigma}

\renewcommand{\L}{\Lambda}

\newcommand{\Hs}{\hat{\s}}
\newcommand{\tl}{\tilde{\l}}
\newcommand{\ls}{{\l}_*}

\renewcommand\d{\partial}
\renewcommand\a{\alpha}
\renewcommand\b{\beta}

\def\g{\gamma}

\def\ess{{\rm ess}}

\def\l{\lambda}
\def\eps{\varepsilon }
\def\e{\varepsilon}

\renewcommand\d{\partial}
\renewcommand\a{\alpha}

\renewcommand\b{\beta}

\newcommand\R{\mathbb R}
\newcommand\C{\mathbb C}

\def\g{\gamma}

\def\ess{{\rm ess}}
\def\eps{\varepsilon}
\def\e{\varepsilon}

\def\l{\lambda}

\newcommand\br{\begin{remark}}
	\newcommand\er{\end{remark}}
\newcommand\bp{\begin{pmatrix}}
	\newcommand\ep{\end{pmatrix}}
\newcommand{\be}{\begin{equation}}
	\newcommand{\ee}{\end{equation}}
\newcommand\ba{\begin{equation}\begin{aligned}}
		\newcommand\ea{\end{aligned}\end{equation}}

\newcommand{\bap}{\begin{app}}
	\newcommand{\eap}{\end{app}}
\newcommand{\begs}{\begin{exams}}
	\newcommand{\eegs}{\end{exams}}
\newcommand{\beg}{\begin{example}}
	\newcommand{\eeg}{\end{exaplem}}
\newcommand{\bpr}{\begin{proposition}}
	\newcommand{\epr}{\end{proposition}}
\newcommand{\bt}{\begin{theorem}}
	\newcommand{\et}{\end{theorem}}
\newcommand{\bc}{\begin{corollary}}
	\newcommand{\ec}{\end{corollary}}
\newcommand{\bl}{\begin{lemma}}
	\newcommand{\el}{\end{lemma}}
\newcommand{\bd}{\begin{definition}}
	\newcommand{\ed}{\end{definition}}
\newcommand{\brs}{\begin{remarks}}
	\newcommand{\ers}{\end{remarks}}

\newtheorem{hypothesis}{Hypothesis}

\newtheorem{question}{Question}

\newcommand{\RR}{{\mathbb R}}

\newcommand{\NN}{{\mathbb N}}
\newcommand{\ZZ}{{\mathbb Z}}
\newcommand{\TT}{{\mathbb T}}
\newcommand{\CC}{{\mathbb C}}

\newcommand{\sC}{\mathscr{C}}
\newcommand{\sN}{\mathscr{N}}
\newcommand{\sQ}{\mathscr{Q}}
\newcommand{\const}{\text{\rm constant}}

\newcommand{\sgn}{\text{\rm sgn}}
\newtheorem{theorem}{Theorem}[section]
\newtheorem{proposition}[theorem]{Proposition}
\newtheorem{corollary}[theorem]{Corollary}
\newtheorem{lemma}[theorem]{Lemma}

\theoremstyle{remark}
\newtheorem{remark}[theorem]{Remark}
\newtheorem{remarks}[theorem]{Remarks}
\theoremstyle{definition}
\newtheorem{definition}[theorem]{Definition}

\newtheorem{example}[theorem]{Example}

\newtheorem{obs}[theorem]{Observation}

\newcommand\cA{{\mathcal { A}}}
\newcommand\cB{{\mathcal  B}}

\newcommand\cU{{\mathcal  U}}
\newcommand\cH{{\mathcal  H}}

\newcommand\cV{{\mathcal  V}}

\newcommand\cC{{\mathcal  C}}
\newcommand\cI{{\mathcal  I}}
\newcommand\cR{{\mathcal  R}}
\newcommand\cG{{\mathcal  G}}
\newcommand\cK{{\mathcal  K}}
\newcommand\cL{{\mathcal  L}}
\newcommand\cN{{\mathcal  N}}
\newcommand\cE{{\mathcal  E}}
\newcommand\cF{{\mathcal  F}}
\newcommand\cP{{\mathcal  P}}
\newcommand\cQ{{\mathcal Q}}
\newcommand\cO{{\mathcal O}}
\newcommand\cX{{\mathcal X}}

\newcommand\cM{{\mathcal M}}
\newcommand\cT{{\mathcal T}}
\newcommand\cS{{\mathcal S}}

\newcommand{\spec}{\operatorname{spec}}

\newcommand{\rank}{\text{\rm{rank}}}

\newcommand{\beq}{\begin{equation}}
	\newcommand{\eeq}{\end{equation}}

\newcommand{\Hx}{\hat{x}}
\newcommand{\Ht}{\hat{t}}
\newcommand{\Tt}{\tilde{t}}

\usepackage{color,soul}

\title{Convective Turing bifurcation with conservation laws}
\author{Aric Wheeler}
\address{Indiana University, Bloomington, IN 47405}
\email{awheele@iu.edu }
\thanks{Research of A.W. was partially supported
	under NSF grant no. DMS-1700279.}
\author{Kevin Zumbrun}
\address{Indiana University, Bloomington, IN 47405}
\email{kzumbrun@indiana.edu} 
\thanks{Research of K.Z. was partially supported
	under NSF grants no. DMS-0300487 and DMS-0801745.}
\begin{document}
	\begin{abstract}
		Generalizing results of \cite{MC,S} and \cite{HSZ} for certain model reaction-diffusion and
		reaction-convection-diffusion equations,
		we derive and rigorously justify weakly nonlinear amplitude equations
		governing general Turing bifurcation in the presence of conservation laws.
		In the nonconvective, reaction-diffusion case, this is seen similarly as in \cite{MC,S}
		to be a real Ginsburg-Landau equation 
		coupled with a diffusion equation in
		a large-scale mean-mode vector comprising variables associated with conservation laws.
		In the general, convective case, by contrast, the amplitude equations as noted in \cite{HSZ}
		consist of a complex Ginsburg-Landau equation 
		coupled with a singular convection-diffusion equation featuring rapidly-propagating
		modes with speed $\sim 1/\eps$ where $\eps$ measures amplitude of the wave as a disturbance from 
		a background steady state.
		Different from the partially coupled case considered in \cite{HSZ}
		in the context of B\'enard-Marangoni convection/inclined flow,
		the Ginzburg Landau and mean-mode equations are here fully coupled, leading to substantial
		new difficulties in the analysis. 
		Applications are to biological morphogenesis, in particular vasculogenesis,
		as described by the Murray-Oster and other mechanochemical/hydrodynamical models.
	\end{abstract}
	
	\maketitle
	
	\tableofcontents
	
	\section{Introduction}
	In this paper, motivated by the ``initiation problem,'' or formation of structural patterns from a 
	uniform homogeneous state, in vasculogenesis and related
	mechanochemical/hydrodynamical pattern formation phenomena \cite{WZ1},
	we formally derive and rigorously validate Eckhaus-type amplitude equations \cite{E}
	governing convective Turing bifurcation in the presence of conservation laws.
	
	\subsection{Convective Turing bifurcation}\label{s:conT}
	In a general Turing bifurcation scenario \cite{T}, a family of constant solutions $u^\mu(x)\equiv \const$
	of evolution equations $u_t=\cN^\mu (u)$ indexed by a bifurcation parameter $\mu$ undergoes a change in linear
	stability as $\mu$ crosses a certain bifurcation point, without loss of generality $\mu=0$ in generic 
	fashion: namely, through the passage across the imaginary axis 
	of a single pair of complex conjugate eigenvalues $ (\lambda_*, \bar \lambda_*)(k,\mu) \in \spec S(ik,\mu) $
	at Fourier frequency $k_*\neq 0$ of the associated linear constant-coefficient operator 
	$L^\mu=d\cN^\mu|_{u^\mu}=S(\partial_x, \mu)$:
	\be\label{eigs}
	\lambda_*(k_*,0)= i\omega_*, \quad \omega_*\in \R 
	\ee
	where $\Re \sgn \lambda_*(k_*,\mu)=\sgn \mu$, with $\Re \spec S(ik,\mu)<0$ for $\mu<0$ and all $k\in \R$.
	
	By $SO(2)$ invariance, corresponding to translation invariance in $x$, of the underlying equation, the
	destabilizing eigenvalue(s) necessarily form a conjugate pair, even if the passage is through the origin.
	In the classic case originally treated by Turing \cite{T}, the equations are of reaction diffusion type, depending
	on even derivatives or even powers of derivatives of $u$, and the critical (destabilizing) eigenvalues
	are assumed to pass through the origin, or $\omega_*=0$ in \eqref{eigs}.
	In this case, there is an additional reflection symmetry in $x$, making it an $O(2)$ bifurcation.
	In either case, there occurs for any fixed $k$ near $k_*$ an equivariant bifurcation as $\mu$ increases
	near $\mu=0$ to nontrivial periodic traveling waves with spatial period $k$; 
	This may be seen, for example, by center manifold reduction to an $SO(2)$ symmetric planar ODE \cite{CK,CaK}.
	In the classical, $O(2)$ case, the waves are stationary; in the general, $SO(2)$ case, dubbed {\it convective}
	in \cite{WZ1}, they are generically of nonzero speed.
	See also the earlier \cite{M3} for a discussion of existence from a spatial dynamics point of view.
	
	\subsubsection{Amplitude equations and Eckhaus stability}\label{s:Eck}
	A profound generalization of the above ODE-based observations, widely used in applications, 
	is the ``weakly unstable'' approximation of Eckhaus \cite{E,SS,NW,En}
	consisting of formal asymptotic solutions 
	\ba\label{form}
	U^{\e}(x,t)&=\frac{1}{2}\e A(\Hx,\Ht)e^{i\xi}r + \cO(\e^2) +c.c.,
	\quad \xi=k_* x + \omega_* t,\\
	\mu&=\eps^2, \quad \Hx=\e(x+ \Im\d_k\ls(k_*,0)t), \quad  \Ht=\e^2t,
	\ea
	with amplitude $A\in \CC$, modulating the neutral linear solutions $e^{i\xi)}r + c.c.$ at $\mu=0$,
	satisfying an {\it amplitude equation} given by the complex Ginzburg-Landau equation (cGL):
	\begin{equation}\label{eq:cGL}
		A_{\Ht}=-\frac{1}{2}\d_k^2\ls(k_*,0)A_{\Hx\Hx}+\d_\mu\ls(k_*,0)A+\gamma|A|^2A,
	\end{equation}
	where the Landau constant $\gamma\in \CC$ is determined by the form of $\cN$
	and spectral structure of $S(k_*,0)$.
	In the $O(2)$-symmetric case, \eqref{eq:cGL} reduces to the real Ginzburg-Landau equation (rGL);
	$\ls, \gamma\in \R$.
	
	Approximation \eqref{form}-\eqref{eq:cGL}, {\it capturing also PDE behavior},
	models arbitrary small-amplitude solutions of $u_t=\cN^\mu(u)$ for $\mu>0$, 
	including the bifurcating periodic patterns of Turing and many others as well.
	See, for example, \cite{AK,vSH} for descriptions of a rich variety of solutions supported by (cGL).
	In particular, it reduces the question of stability of bifurcating periodic patterns to the corresponding question for
	periodic solutions $A(\Hx,\Ht)=\alpha e^{i(\kappa \Hx+ \omega \Ht)}$ of \eqref{eq:cGL}, where
	$\alpha \in \CC$ is constant.
	Making the change of coordinates $\tilde A=A e^{-i(\kappa \Hx+ \omega \Ht)}$ factoring out the exponential converts
	periodic solutions to constant solutions $\tilde A=\alpha\equiv \const$, while maintaining the translation-invariance
	of \eqref{eq:cGL}.
	Thus, the linearized equations, in transformed coordinates, are constant-coefficient, allowing an
	explicit determination of spectra leading to the formal {\it Eckhaus stability criterion} \cite{AK,E,En,WZ2}:
	\ba\label{Eck}
	&\kappa^2< \kappa_S^2\\
	&\: = 2 \frac{ \Im(\partial_k \ls(k_*,0) )\Im \gamma \Re (\partial_\mu \ls(k_*,0) ) \Re \gamma 
		+\Re(\partial_k \ls(k_*,0)) \Re (\partial_\mu \ls(k_*,0)) \Re \gamma^2}
	{\Re \partial_k \ls(k_*,0)(2\Im \gamma^2\Re(\partial_k \ls(k_*,0)) +\Im(\partial_k \ls(k_*,0)) 
		\Im(\gamma) \Re(\gamma)+ 3\Re(\partial_k \ls(k_*,0)) \Re( \gamma^2))}.
	\ea
	
	\subsubsection{Rigorous justification}\label{s:Rigor}
	Rigorous validity for time $t\lesssim \eps^{-2}$ of approximation \eqref{form}-\eqref{eq:cGL}
	has been established for general solutions of real and complex Ginzburg-Landau equations in, e.g., 
	\cite{CE,S3,KT} and \cite{vH} for various classes of initial data on \eqref{eq:cGL}.
	See also \cite{M3} for a general overview, and discussion from the spatial dynamics point of view.
	As regards time-asymptotic stability of periodic Turing patterns, rigorous validity of the Eckhaus
	stability criterion, along with predictions to lowest order of the spectra of bifurcating waves,
	has been shown for specific reaction diffusion models in \cite{M1,M2,S1,S2,SZJV}, using Lyapunov-Schmidt
	reduction arguments pioneered by Mielke and Schneider.
	These results have recently been extended in \cite{WZ1,WZ2} to the case of general (convective) Turing
	bifurcations for a large class of quasilinear evolution equations, including higher-order parabolic 
	and dissipative odd-order systems, along with certain nonlocal systems as well,
	effectively completing existence/stability theory for standard Turing patterns.
	
	\subsection{Bifurcation with conservation laws}\label{s:stdT}
	In modern biomorphology, the reaction diffusion case treated for simplicity in \cite{T} has
	given way to mechanochemical/hydrodynamical models 
	\be\label{eq1}
	\d_t w+\d_x f^\mu(w)=g^\mu(w)+ \d_x(b^\mu(w)\d_x w), \quad w\in \R^n
	\ee
	incorporating also convection;
	see, for example, the Murray-Oster model \cite{MO} and descendents \cite{Ma,Mai}
	in the case of vasculogenesis. See \cite{WZ1,SBP} for further discussion.
	More important for the present discussion, these models have the property that {\it $g^\mu$ is of partial rank
		$n-r$}, with $r>0$.
	Moreover, they have the relaxation-type structure that there exist vectors $\ell_j\in \R^{1\times n}$ such that
	\be\label{cons}
	\ell_j g^\mu\equiv 0, \qquad j=1,\dots, r, 
	\ee
	independent of $\mu$, as a consequence of which \eqref{eq1} possesses $r$ {\it conservation laws} 
	$\int \ell_j w \equiv  \const.$
	
	As discussed in \cite{MC,HSZ}, similar conservation law structure may be found in binary mixtures, 
	liquid crystals, and surface-tension driven convection.
	For all such models, it follows from \eqref{cons} that the symbols $S(ik,\mu)$ of linearized operators $L^\mu$ about 
	constant solutions $u^\mu$ have for all $\mu$ an $r$-fold kernel at $k=0$.
	Hence, the classic Turing bifurcation of Section \ref{s:conT} is replaced by a degenerate scenario involving
	$r$ neutral linear modes in addition to the destabilizing critical modes \eqref{eigs}.
	
	\subsubsection{The model of Matthews-Cox}\label{s:MC}
	As noted by Matthews and Cox \cite{MC}, for models with conservation laws, 
	the onset of pattern formation is no longer described by the standard Ginzburg-Landau approximation 
	\eqref{form}-\eqref{eq:cGL}.
	In the $O(2)$ reflection-symmetric case with a single conservation law, they deduce by symmetry 
	considerations that amplitude equations, if they exist, should be of form  
	\ba\label{eq:MC}
	A_{\Ht}&=a A_{\Hx \Hx}+b A+ c|A|^2A + d AB,\\
	B_{\Ht} &=  e B_{\Hx \Hx} + f (|A|^2)_{\Hx\Hx},
	\ea
	$A\in \CC$, $B\in \R$, where $a,b,c,d,e,f\in \R$,
	governing the evolution of the amplitudes $A$ of ``pattern-modes'' $e^{\pm i\xi}$ as in \eqref{form},
	and $B$ of the neutral ``mean-mode'' associated with the kernel of $S^0(0,0)$,
	i.e., a real Ginzburg-Landau equation in $A$, weakly coupled with nonlinear diffusion diffusion in $B$.
	
	For a model scalar reaction diffusion equation consisting of a differentiated Swift-Hohenberg equation,
	they verify that Ansatz \eqref{eq:MC} may indeed be derived by formal matched asymptotic expansion.
	More recently, for the same model equation, Sukhtayev \cite{S} has rigorously verified this expansion by
	Lyapunov-Schmidt reduction following Mielke--Schneider \cite{M1,M2,S1,S2,SZJV}.
	
	\subsubsection{The model of H\"acker-Schneider-Zimmermann}\label{s:HSZ}
	Based on the example of \eqref{eq:MC}, one might conjecture in the convective ($SO(2)$ but not $O(2)$ symmetric)
	case amplitude equations of form
	\ba\label{eq:conjecture}
	A_{\Ht}&=a A_{\Hx \Hx}+b A+ c|A|^2A + d AB,\\
	B_{\Ht} &=  e B_{\Hx \Hx} + f B_{\Hx} ,
	\ea
	$A\in \CC$, $B\in \R^r$, where now $a,b,c$ are complex scalars, $d\in \C^{1\times r}$, and
	$e,f\in \R^{r\times r}$: that is, a {\it complex Ginzburg-Landau equation} 
	weakly coupled to a {\it real vectorial convection-diffusion equation},
	with $f\equiv 0$ and $a,b,c,d$ real in the nonconvective $O(2)$-symmetric reaction diffusion case.
	(Here, new term $f B_{\Hx}$ is needed to match the spectral expansion of neutral
	modes to first order in $k$.)
	
	However, for the general, convective case, this turns out to be far from the case!
	Indeed, as we shall see, 
	under the scaling of \cite{MC}, with $\mu=\eps^2$, a formal multi-scale expansion leads, rather, 
	to a {\it singular-convection} version
	\ba\label{eq:conamp}
	A_{\Ht}&=a A_{\Hx \Hx}+b A+ c|A|^2A + d AB,\\
	B_{\Ht} &=  e B_{\Hx \Hx} + \eps^{-1}( f B + h|A|^2)_{\Hx} + \d_{\Hx}\Re(gA\overline{A}_{\Hx} ),
	\ea
	of \eqref{eq:conjecture}, with an additional fast time-scale corresponding to rapid convection 
	in mean modes $B$. 
	
	This singular scaling of mean modes has been observed previously; see \cite{HSZ} and references therein.
	In particular, there were derived by H\"acker, Schneider, and Zimmermann
	in the context of B\'enard-Marangoni convection and
	inclined-plane flow, mode-filtered equations \cite[Eqs. (22)-(23) and (33)-(34]{HSZ} that  
	effectively correspond to \eqref{eq:conamp} with vanishing coupling coefficient $d=0$ and simultaneously
	diagonal matrices $f$ and $e$.
	
	The resulting triangular structure, however, allowed them to treat $A$- and $B$-equations separately,
	hence, noting that (i) the $B$-equation enters at a higher-order power in $\eps$, and (ii) that the
	linear constant-coefficient generator 
	$ e^{ e \partial_x^2 + \eps^{-1} f \partial_x  } $ for the $B$-equation by Fourier transform is
	contractive on $L^2$, to show validity to order $\eps^2$ on a time-interval of order $1/ \eps^2$
	of the approximation given by the $A$-equation alone, for localized, or $W^{k,p}$ solutions, $p<\infty$.
	That is, the complex Ginzburg-Landau equation $A_{\Ht}=a A_{\Hx \Hx}+b A+ c|A|^2A $
	is treated in this analysis as an (nonsingular!) amplitude equation, with the $B$-equation more or less
	subsumed in the underlying dynamics.
	
	Existence of modulating traveling fronts has been treated in a similar setting in \cite{H}; however, to
	our knowledge, there has been up to now no treatment of periodic patterns and their stability/behavior
	in this case.
	And, indeed, for what we will call the {\it model of H\"acker-Schneider-Zimmermann}, i.e., 
	\eqref{eq:conamp} with $d=0$ and $e$, $f$ diagonal, 
	the spectra of the linearized equations about an exponential solution in $A$
	with $B$ held constant- corresponding
	to a periodic solution of the underlying PDE- decouples by triangular structure $d=0$ into that for the
	complex Ginzburg-Landau $A$-equation, and that for the constant-coefficient $B$-equation with $A$ set to
	zero, which can be read from the Fourier symbol as 
	$ \lambda(k)\in \sigma (-\eps^{-1}fk -ek^2)$, $k\in \R$,
	hence automatically stable by simultaneous diagonality of $e$ and $f$.
	
	Thus, for the models considered in \cite{HSZ}, though they capture the important new phenomenon of singular scaling
	in the convective case, the main issues considered in this paper do not arise.
	
	\subsection{Main results}\label{s:main}
	The goals of this paper are two: first, to extend the results of \cite{MC,S} 
	to (nonconvective) Turing bifurcations for general systems of reaction-diffusion
	equations, with arbitrary number of conservation laws $r$, and, second and more important, 
	to derive an appropriate analog \eqref{eq:conamp}, generalizing results of \cite{HSZ}, 
	for the general case of convective Turing bifurcations with conservation laws,
	suitable for the treatment of vasculogenesis and other biomorphology models involving mechanical effects.
	
	Our first main result is to show for the $O(2)$-symmetric reaction diffusion case that this is indeed correct,
	first by formal multi-scale expansion, and then by rigorous Lyapunov-Schmidt reduction in the style
	of \cite{M1,M2,S1,S2,SZJV,WZ2,S}; see Theorems \ref{thm:O2multilin} and \ref{thm:LSReduction} respectively.
	Our second main result, and the main conclusion of the paper, is that in the general ($SO(2)$ but not $O(2)$)
	convective case, behavior is governed rather by \eqref{eq:conamp}, both by formal multi-scale
	expansion, and by rigorous Lyapunov-Schmidt reduction;
	see Section \ref{sec:MSE} for the multiscale expansion argument and 
	Theorem \ref{thm:LSReduction} for Lyapunov-Schmidt reduction.
	
	The analysis for both sets of results is carried out essentially at the same time, 
	with the first treated as a special case of the second; see Sections \ref{sec:MSE} and \ref{sec:ExistenceLS}.
	This is carried out in an extremely general setting, including nonlocal equations;
	in particular, we expect that our results apply (with suitable modification to account for
	incomplete parabolicity) to the biomorphology models \cite{MO,Ma,Mai,SBP} cited in Section \ref{s:stdT}.
	Our derivation identifies at the same time a simple {\it compatibility condition} in terms of the spectral
	structure of the Fourier symbol $S(0,0)$ with respect to neutral modes, automatically satisfied in the
	$O(2)$-symmetric case, that is necessary and sufficient for vanishing of the singular term
	\be\label{singterm}
	\eps^{-1}( f B + h|A|^2)_{\Hx}
	\ee
	in the equations, that is, vanishing of the coefficients $f$ and $h$, in which case \eqref{eq:conamp}
	reduces to something closely resembling the form \eqref{eq:MC} proposed in \cite{MC}, with a simple coupled diffusion equation in $B$ and the real Ginzburg-Landau replaced with a complex Ginzburg-Landau.
	
	\medskip
	
	The rapid convection of \eqref{eq:conamp}, similarly as in other dispersive phenomena such as
	the low Mach number limit, etc., has an averaging effect but is somewhat neutral as regards 
	well-posedness.  This is easy to see at the linear level.
	We establish nonlinear consistency of amplitude equations \eqref{eq:conamp}
	by a proof of well-posedness, i.e., bounded-time existence, in Section \ref{sub:wellposed}.
	
	\medskip
	
	We note further that periodic solutions of \eqref{eq:conamp} are, by inspection, of form
	$$
	\hbox{\rm $A(\Hx,\Ht)=\alpha e^{i\xi} + c.c.$, \quad $B(\Hx, \Ht)\equiv \beta$,\quad
		with $\alpha, \beta \equiv \const$},
	$$
	that is, standard (cGL) solutions in $A$ adjoined with a constant solution in $B$.
	By the same coordinate transformation $\tilde A(\Hx,\Ht):= e^{-i\xi }A(\Hx,\Ht)$ used in the stability
	analysis of (cGL) waves, we may thus transform the linear stability problem for \eqref{eq:conamp}
	to a constant-coefficient system in $(\tilde A, B)$, for which spectra may be determined
	via the associated linear dispersion relation, to give in principle an Ekhaus-type stability criterion.
	We explore these directions, and the effects of singularity $1/\eps$, in Section \ref{sec:MSEstability}.
	An interesting finding is that, though neutral for well-posedness of the amplitude equations, 
	{\it $1/\eps$ terms can be crucial for determining stability} of periodic solutions thereof, see Section \ref{sub:dispersion} for an example of this phenomenon in a special case.
	
	\medskip
	
	To give a bit more detail, the linear constant-coefficient system determining linearized stability
	for the amplitude equations \eqref{eq:conamp} is effectively a two-parameter matrix perturbation problem
	in $\eps\to 0$ and $\sigma \in \R$, $\sigma$ denoting Fourier frequency, 
	with particular emphasis on the low-frequency, $\sigma\to 0$ regime.
	We restrict ourselves here to the fixed-$\eps$, $\sigma \to 0$ regime, reducing to a one-parameter problem
	treatable by standard matrix perturbation techniques \cite{K}, in order to obtain simple
	{\it necessary conditions} for stability in terms of the second-order Taylor expansion about $\sigma=0$,
	similarly as in the $O(2)$-symmetric case.
	
	However, even this is far from simple. For, one immediately sees that neutral modes perturb from what
	is generically a $2\times 2$ nontrivial Jordan block, so that one would intuitively expect a square-root
	singularity at $\sigma=0$ and a nonanalytic Puiseux rather than a Taylor expansion. 
	What saves the day is a fundamental observation of \cite{JZ2,BJZ2,JNRZ} connecting general modulation
	equations with exact spectral perturbation expansions in the presence of conservation laws: namely,
	that the same conservation structure generically leading to a nontrivial Jordan block imposes vanishing
	to first order of perturbation terms in the opposite corner from the nonvanishing entry of the Jordan block,
	recovering analyticity of the perturbation expansion rather than the usual Puiseux structure.
	
	This allows us to extract {\it $(m+ 1)$} Eckhaus-type stability conditions, where $m$ is the nummber of conservation laws, as compared to the single Eckhaus condition in the standard case without conservation laws, 
	in terms of the sign of the real part
	of the second-order Taylor expansions of the $m+1$ neutral  eigenvalues bifurcating from zero at $\sigma=0$.
	Moreover, the same matrix eigenstructure allows us to compute the difference between the eigenvalues
	of the formal approximation problem and the exact spectral expansion problem obtained by Lyapunov-Schmidt
	reduction, which we show to be a higher-order perturbation of the formal matrix eigenvalue problem.
	This validates the conclusions of the formal stability analysis also for the full PDE problem, on an appropriately
	small regime in the Bloch-Floquet exponent, corresponding to the Fourier frequency $\sigma$ for the formal problem,
	thereby giving necessary stability conditions for the full problem as well.

	\medskip
	
	We emphasize that, different from the $O(2)$ case, in which we obtain both necessary and sufficient stability conditions, we obtain in the $SO(2)$ case only {\it necessary conditions} for stability, or, equivalently,
	{\it sufficient conditions} for instability. Nonetheless, these are easily evaluated, and
	give nontrivial information that appears of practical use in applications, and we expect them to have
	importance in biomorphology similar to that of the classical Eckhaus condition in general pattern formation.
	We note in this regard that the classical conditions were used profitably for several decades in applications before
	being rigorously validated in \cite{M1,M2,S1,S2,SZJV}.
	
	\subsubsection{Reduced equations}\label{s:reduced}
	A natural alternative approach is, rather than seeking systems for which singular term \eqref{singterm} vanishes,
	to seek {\it solutions} for which it vanishes, i.e., a sort of equilibrium reduction similar to Darcy's law
	in fluid dynamics.  With this assumption, yielding constraint $( f B + h|A|^2)_{\Hx}=0$, we obtain
	\be\label{eq:B}
	B = B_0(\Ht) - (h/f)|A|^2, 
	\ee
	where $B_0(\Ht)$ is a constant of integration independent of $\Hx$, collapsing
	\eqref{eq:conamp} to a single parametrically forced Ginzburg-Landau equation 
	\be\label{eq:red}
	A_{\Ht}=a A_{\Hx \Hx}+(b+dB_0(\Ht)) A+ \tilde c|A|^2A, \qquad \tilde c:= c- dh/f,
	\ee
	with modified nonlinear coefficient $\tilde c$, with parameter $B_0$ in principle determined by compatibility
	conditions for solution of the equations for higher-order terms corresponding to $A$ and $B$ modes.
	
	It is not clear to us what is the most general class of data for which this procedure can be carried out.
	However, we show in Section \ref{sub:darcy} that for the class relevant to our studies here, of $L^2$ perturbations
	of periodic solutions of \eqref{eq:conamp}, i.e. of $A$ periodic, and $|A|,B$ constant, with $\Hx$ set
	on the Torus or the whole line, the forcing parameter $B_0(\Ht)$ may be uniquely determined, along
	with arbitrarily higher-order expansions in $A$ and $B$ modes, with $B_0(\Ht)\equiv \const$ for periodic solutions.
	We refer to the resulting equation \eqref{eq:red} as the {\it reduced}, or {\it Darcy} model.
	
	Solutions of the reduced model may be thought of as an equilibrium manifold for the full equations \eqref{eq:conamp}.
	Provided that the background periodic wave is well-posed as a solution of \eqref{eq:conamp}, i.e., boundedly stable up to time $O(1)$, this might be expected
	to represent the $O(1)$ time behavior in \eqref{eq:conamp}, after an initial time-layer consisting of averaging
	through rapid convection. More, assuming bounded time $O(1)$ stability of the periodic solutions of the full nonlinear PDE that these asymptotic expansions
	approximate, one may expect that the reduced asymptotic expansion well-approximates exact solutions, as shown in
	a number of related settings in nonlinear geometric optics \cite{DJMR,MS}.
	
	\br\label{simprmk}
	Interestingly, the behavior of mode $A$ prescribed by the reduced equation \eqref{eq:red} in the convective $SO(2)$
	case is somewhat simpler than that prescribed by the Matthews-Cox system \eqref{eq:MC} in the $O(2)$ reaction diffusion
	case, reflecting the simplifying effects of rapid averaging.
	This is similar to the observations of \cite{HSZ}; indeed, the Darcy model may be recognized as
	the natural generalization of the model of
	\cite{HSZ} to the fully coupled case, {\it valid when
		the full system \eqref{eq:conamp} is bounded linearly stable}, i.e., well-posed in the context of bounded-time,
	small-$\eps$ convergence as considered in \cite{HSZ}.
	\er
	
	\subsection{Discussion and open problems}
	Singular $O(\eps^{-1})$ convection and associated phenomena
	occurring for convective bifurcations appears of substantial interest for applications.
	We note the related feature of nonzero propagation speed observed for convection in binary
	mixture \cite{SZ,LBK} in the scenario of ``near-Hopf'' Turing bifurcations, i.e., under the scaling $k_*=O(\eps)$,
	for which the singularity $\eps^{-1}$ in \eqref{eq:conamp} cancels, giving finite speed.
	From this point of view, the singular amplitude equations \eqref{eq:conamp} could be viewed as describing the
	generic case, away from Hopf bifurcation.
	
	A second observation of interest seems to be the explicit identification of ``relaxation-type' structure
	\eqref{cons}, which did not arise in the previous, scalar analysis of \cite{MC}.
	This suggests interesting possible contact with bifurcation and pattern formation in
	relaxation models\cite{JK,L,JNRYZ,BJNRZ,MZ,YZ} and systems of conservation laws \cite{BJZ1,BMZ}.
	In particular, it should be possible by linear stability analysis of \eqref{eq:conamp} to answer the
	open question posed in \cite{BJZ1} of understanding numerically-observed stability boundaries for
	Turing bifurcation in systems of conservation laws.
	
	We emphasize again that, whereas our results in the $O(2)$ symmetric case give necessary and sufficient
	conditions for stability, our results for $SO(2)$ symmetric systems give only {\it necessary} conditions
	for stability. We conjecture based on numerical observations and general optimism 
	that our necessary conditions are also sufficient. To prove or disprove this conjecture is
	by our reckoning {\it the} fundamental open problem in the theory, 
	hinging on the analysis of a complicated two-parameter matrix perturbation problem.
	This is an object of our current investigation.
	
	In this paper, we have rigorously validated the singular expansion \eqref{eq:conamp} as a predictor of existence
	and time-asymptotic stability of periodic convective Turing patterns with conservation laws.
	The complementary problem of established bounded-time validity of \eqref{eq:conamp} as in \cite{CE,S3,KT,vH} 
	is an extremely interesting, though apparently challenging open problem.
	An interesting somewhat related problem is to study other small-amplitude solutions 
	such as front- or pulse-type traveling waves.
	We note that the latter involve nontrivial profiles in mean modes $B$,
	hence, different from the periodic case, the associated existence problem does not
	reduce to a corresponding (cGL) problem with $B\equiv \const$.
	
	When stability holds, we have provided also the nonsingular reduced, Darcy model \eqref{eq:B}-\eqref{eq:red}
	as a description of behavior,
	presumably ``averaged'' behavior on $O(1)$ time scale
	after an initial $O(\eps)$ initial-layer in time.
	For data consisting of $L^2$ perturbations of exact
	periodic waves, we have shown that this reduced model is well-posed to all orders up to time $O(1)$.
	A very interesting open problem is, assuming stability of the underlying periodic wave,
	to rigorously validate the resulting approximation in a similar class of data, 
	both with respect to the singular model \eqref{eq:conamp} and the original PDE.
	We expect that this can be treated by the methods described in \cite{HSZ}.
	
	An interesting related question, following the philosophy of Liu \cite{L} and Mielke and Schneider 
	(\cite{S1,S2,S3} and references therein), is whether, under appropriate stability assumptions,
	solutions of the reduced (Darcy) model represent also {\it time-asymptotic} as well as {small-$\eps$} behavior.
	However, the answer to this question appears to be negative; for, the time-asymptotic behavior, as established for
	arbitrary-amplitude waves in \cite{JNRZ} is based on the full $(m+1)$-dimensional Whitham modulation equation,
	which in the context of amplitude equations, corresponds to the $(m+1)$-dimensional reduced system consisting
	of bifurcating neutral modes.
	Yet, the asymptotics of the Darcy model depend only on a single-dimensional model corresponding to the reduced
	system consisting of its single neutral mode, a Burgers-type equation governing the derivative of phase
	modulation.
	The $(m+1)$-dimensional Whitham equations likewise include an equation for the derivative of the phase modulation,
	but it is generically coupled to the remaining $m$ equations, so that these two descriptions of asymptotic
	behavior cannot agree.
	Thus, it appears that this is an interesting case in which $\eps\to 0$ and $t\to \infty$ asymptotics do not in general agree.
	A more subtle question is whether, in the case that the background wave is 
	time-asymtotically stable, the Darcy model might serve as both both short- and long-time limit,
	or at least, due to $e^{-ct/\eps}$ vs. $e^{-ct}$ decay rates, might represent a sort of slow
	stable manifold.
	
	We emphasize again the difference between the present analysis and the foundational work of Mielke-Schneider, 
	represented by the presence of an intermediate time scale, 
	and the resulting importance of the singular model \eqref{eq:conamp}.
	Namely, whereas in the latter case existence, stability, and behavior are all accurately predicted by a
	single, nonsingular amplitude equation, the complex Ginzburg-Landau equation, in the present setting
	existence and behavior are well-predicted by the nonsingular reduced equation \eqref{eq:B}-\eqref{eq:red}, 
	a modified Ginzburg-Landau equation, but stability depends on the more complicated singular model \eqref{eq:conamp}.
	
	Finally, we mention the originally-motivating question of applications to biomorphology,
	specifically, determination of bifurcations and their stability for specific models,
	as a potentially fruitful direction for further study.
	The above-observed connection with relaxation models and systems of conservation laws suggests
	moreover that techniques as in \cite{W,FST,JNRZ}, developed for the latter areas,
	may be helpful in understanding ``emergent'' behavior in vasculogenesis, i.e.,
	behavior of ``fully-developed'' large-amplitude patterns of convective Turing patterns with conservation laws.
	
	Here again, we note the importance of the full, $(2+r)$-variable singular model \eqref{eq:conamp}, 
	this time as a bridge between the $2$-variable Darcy model \eqref{eq:red} relevant to small, $\eps\ll 1$ 
	amplitude behavior, and the $(r+1)$-variable models of \cite{JNRZ} governing fully-developed $\eps \sim 1$ amplitude, patterns, where $r$ is the number of conservation laws (dimension of $B$).
	
	Likewise, it is an important open problem to extend our results to the case of incomplete parabolicity
	$\det B=0$ occurring for a number of physical models, in particular the motivating examples cited from 
	biomorphology.
	We expect this should be treatable by tools developed in, e.g., \cite{MZ,Z}.
	
	\medskip
	\textbf{Acknowledgements:}
	Thanks to Jared Bronski for suggesting the interesting Remark \ref{rem:bronski}, 
	to Miguel Rodrigues for suggesting Remark \ref{rem:spacetime},  and to Ryan Goh for pointing out the
	recent related paper \cite{H} on existence of front type solutions in SO2 systems with conservation laws.
	Thanks to Olivier Lafitte for suggesting the current expanded exposition of the Darcy model, and for noting its connection with \cite{DJMR}, and to Benjamin Melinand for interesting related conversations.
	Thanks to all five for their interest in the paper and for several helpful discussions.

	\medskip

	{\bf Note:} We have learned just before posting our paper of the preprint \cite{GHLP}
	treating related questions of formal asymptotics for nonlocal systems with 
	conservation laws, motivated by similar biological questions.
	In our notation, their analysis concerns a 2-variable, 2-conservation law model in
	the O(2) symmetric case, hence is a nonlocal analog of the study of \cite{MC}.
	We note that our analysis applies to their case provided that the symmetric smoothing kernel $K$
	in their model is sufficiently smooth in Fourier space, or, equivalently, sufficiently
	localized in $x$.
	Interestingly, they carry out a full linear stability analysis for this case, which we do not.
	
	\bigskip

\section{Preliminaries}\label{sec:preliminaries}
	\begin{hypothesis}\label{hyp:Lin}
	Let $S(k,\mu)\in C^\infty(\RR^2;M_n(\CC))$ be the symbol of a differential operator, and write $S(k,\mu)=\sum_{j=0}^m (ik)^j\cL_j(\mu)$. Assume that $S(k,\mu)$ satisfies the following list of properties:
	\begin{enumerate}
		\item For all $(k,\mu)\in\RR^2$ $S(k,\mu)=\overline{S(-k,\mu)}$.
		\item There exists $r\in\NN$, $r\geq 1$ such that $r=\dim\ker(\cL_0(\mu))$ for all $\mu\in\RR$. Moreover 0 is a semi-simple eigenvalue of $S(0,\mu)$ for all $\mu$ and the projection onto $\ker(\cL_0(\mu))$, $\Pi_0$ is constant with respect to $\mu$.
		\item For all $\mu<0$ and all $k>0$ $\s(S(k,\mu))\subset\{z:\Re z<0 \}$ and $\s(\cL_0(\mu))\subset\{z:\Re z\leq 0 \}$.
		\item At $\mu=0$ there is a unique $k_*>0$ so that $\s(S(k_*,0))\cap\{z: \Re z=0 \}$ is nonempty, moreover there is exactly one neutral eigenvalue and it's simple.
		\item Let $\tl(k,\mu)$ be the unique eigenvalue satisfying $\Re\tl(k_*,0)=0$. Then $\Re \tl_k(k_*,0)=0$, $\Re \tl_{kk}(k_*,0)<0$, $\Re \tl_\mu(k_*,0)>0$.
		\item There exists $\kappa_0>0$ sufficiently small, $\mu_0>0$, $\Lambda_0>0$ so that for $K:=\{0,\pm k_* \}+\{x: |x|\leq\kappa_0 \}$ we have for all $|\mu|\leq \mu_0$, $k\not\in K$ $\s(S(k,\mu))\subset\{z: \Re z\leq-\Lambda_0\}$ and $\s(S(0,\mu))\subset \{z: \Re z \leq -\Lambda_0\}\cup \{0\}$.
		\item For $K$, $\mu_0$ as above, there are constants $c_1,c_2>0$ so that for all $|\mu|\leq\mu_0$ and $k\not\in K$ we have $c_1(1+|k|^2)^\frac{m}{2}\leq \s_{min}(S(k,\mu))\leq \s_{max}(S(k,\mu))\leq c_2(1+|k|^2)^\frac{m}{2}$ where $\s_{min}(A)$ and $\s_{max}(A)$ refer to the smallest and largest singular values of $A$ respectively.
		\item $-i\Pi_0S_k(0,\mu)\Pi_0$ is hyperbolic and $\Pi_0 S_{kk}(0,0)\Pi_0-2\Pi_0 S_k(0,0)N_0S_k(0,0)\Pi_0$ is positive.
	\end{enumerate}
\end{hypothesis}
As a consequence of the first and seventh conditions, the operator $L(k,\mu)$ defined by
\begin{equation}
L(k,\mu)U(\xi):=\sum_{\eta\in\ZZ} S(k\eta,\mu)\hat{U}(\eta)e^{i\eta\xi},
\end{equation}
is continuous from $H^m_{per}([0,2\pi];\RR^n)\to L^2_{per}([0,2\pi],\RR^n)$ and translation invariant. 
From proposition 2.2 of \cite{WZ1}, we know that hypotheses 6 and 7 both follow from some mild assumptions on $\cL_m$ and $\cL_{m-1}$ and hypothesis 3. The last hypothesis in \ref{hyp:Lin} is a technical assumption related to the structure of the coefficients of equation for $B$ in \eqref{eq:conjecture}. We remark that the last hypothesis is almost automatic for $r=1$, i.e. a single conservation law, the only way it can fail is for the quantity $\Pi_0S_{kk}(0,\mu)\Pi_0$ to vanish.\\

We have the following analog of Proposition 4.2 of \cite{WZ1}.
\begin{proposition}\label{prop:Tbounded}
	Assume Hypothesis \ref{hyp:Lin}. Let $P$ be the projection onto the neutral eigenspaces of $L(k_*,0)$, that is $P$ is the projection given by
	\be
		PU(\xi):=\Pi_0 \hat{U}(0)+\Pi_1 \hat{U}(1)e^{i\xi}+c.c.,
	\ee
	where $\hat{U}$ is the Fourier transform of $U$,
	and write $k=k_*+\kappa$. Then
	\be
		T(k,\mu):=((I-P)L(k,\mu)(I-P))^{-1}:(I-P)L^2_{per}(\RR;\RR^n)\to (I-P)H^m_{per}(\RR;\RR^n),
	\ee
	is a bounded operator for $\kappa$ and $\mu$ sufficiently small, with bounds independent of $\kappa$ and $\mu$. More generally, one has for all $s\in\RR$ that $T:(I-P)H^s_{per}(\RR;\RR^n)\to (I-P)H^{s+m}_{per}(\RR;\RR^n)$ in a bounded manner, with bounds only depending on $s,\kappa$, and $\mu$.
\end{proposition}
The only difference in the proof of this proposition relative to the case with no conservation laws treated in \cite{WZ1} is that the symbol of $T(k,\mu)$ at frequency zero is now $((I-\Pi_0)S(0,\mu)(I-\Pi_0))^{-1}$. This is bounded uniformly in $\mu$ for $\mu$ small by continuity and clearly uniformly bounded in $\kappa$ as it is independent of $\kappa$.\\

We will consider general local quasilinear nonlinearities in the system 
\be
	U_t=L(\mu)U+\cN(U;\mu),
\ee where $U$ is periodic of some period. We will assume that $\cN$ has the following properties.
\begin{hypothesis}\label{hyp:Nonlin}
	$\cN:H^m_{per}(\RR;\RR^n)\rightarrow L^2_{per}(\RR;\RR^n)$ satisfies
	\begin{enumerate}
		\item $\tau_y\cN(u,\mu)=\cN(\tau_yu,\mu)$ for all $y\in\RR$, $\tau_hf(x):=f(x-h)$ a translation, all $\mu\in\RR$ and $u\in H^m_{per}(\RR;\RR^n)$.
		\item Writing $\cN(u,\mu)=\cN_1(u,\d_x u,\d_x^2 u,...,\d_x^{m-1}u)\d_x^m u+\cN_2(u,\d_xu,...,\d_x^{m-1}u)$, we define the auxiliary nonlinearity $\cN(U,k,\mu)$ for $2\pi$-periodic functions $U$ as follows: $\cN(U,k,\mu)=\cN_1(U,k\d_\xi U,...,k^{m-1}\d_\xi^{m-1}U)(k\d_\xi)^m U+\cN_2(U,...,k^{m-1}\d_\xi^{m-1}U)$. Then we assume that $\cN(0,k,\mu)=D_u\cN(0,k,\mu)=0$ holds for all $k,\mu$.
		\item $\Pi_0\widehat{\cN(u,k,\mu)}(0)=0$ for all $u,k,\mu$.
	\end{enumerate}
\end{hypothesis}
	Note the auxiliary nonlinearity $\cN(U,k,\mu)$ is defined on $H^m_{per}([0,2\pi];\RR^n)\times(0,\infty)\times\RR$ as opposed to $H^m_{per}(\RR;\RR^n)$. This is advantageous because $H^m_{per}([0,2\pi];\RR^n)\times(0,\infty)\times\RR$ is an open subset of a Banach space, but $H^m_{per}(\RR;\RR^n)$ is not an open subset of the Fr\'echet space $H^m_{loc}(\RR;\RR^n)$. Hence there are no technical difficulties in defining smoothness for the auxiliary nonlinearity. Conveniently, it also separates the dependence of $\cN$ on the waveform $U$ and the period $2\pi/k$.
\begin{remark}
	By changing variables appropriately, we may without loss of generality assume that $\Pi_0$ is the orthogonal projection onto the span of the first $r$ elements of the standard basis of $\RR^n$. Under this assumption, note that the fourth hypothesis of \ref{hyp:Nonlin} implies that the first $r$ equations are conservation laws. That said, the form of $\Pi_0$ is not particularly important for our arguments, so we will leave it as a general projection.
\end{remark}
We have the following key spectral identity.
\begin{proposition}\label{prop:spectralid}
	Let $M(x)=\sum_{j=0}^{m}x^jM_j$ be a matrix function where each $M_j\in M_n(\CC)$. Suppose that at $x=0$, there is exactly one eigenvalue equal to 0 and that it is simple. Let $\l(x)$ be that eigenvalue and define left and right eigenvectors $\ell(x)$, $r(x)$ satisfying the normalization condition $\ell'(x)r(x)=\ell(x)r'(x)=0$ for each $x$. Define a projection $\Pi:=r(0)\ell(0)$. Then we have the following formula for $\l''(0)$
	\begin{equation}\label{eq:spectralid}
	\l''(0)=2\left(\ell(0)M_2r(0)-\ell(0)M_1(I_n-\Pi)N(I_n-\Pi)M_1r(0) \right),
	\end{equation}
	where $N=\left((I_n-\Pi)M_0(I_n-\Pi)\right)^{-1}$.
\end{proposition}
We will also make extensive use of the following result giving the structure of translation invariant multilinear operators.
\begin{proposition}\label{prop:multilin}
	Let $M:\cP(\TT)^k\to \cM(\TT)$ be multilinear where $\cP(\TT)$ is the space of trigonometric polynomials and $\cM(\TT)$ is the space of Borel measurable functions on the torus. Suppose $M$ is translation invariant in the sense that for all translations $\tau_hf(x)=f(x-h)$ we have
	\begin{equation}
	\tau_hM(p_1,...,p_k)=M(\tau_hp_1,...,\tau_hp_k).
	\end{equation}
	Then there exists $\s:\ZZ^k\to\CC$ such that
	\begin{equation}
	M(e(l_1x),...,e(l_kx))=\s(l_1,...,l_k)e((l_1+...+l_k)x),
	\end{equation}
	where $e(lx)=e^{2\pi ilx}$.
\end{proposition}

\section{Multiscale expansion: existence}\label{sec:MSE}
In this section, we will seek to construct an approximate solution of
\begin{equation}\label{eq:evoleqn}
\frac{\d u}{\d t}=L(\mu)u+\cN(u,\mu),
\end{equation}
where $L(\mu)$ is the linearization about the constant state $u_*=0$.\\

Here, we adopt a modified form the of the Ansatz in \cite{WZ1} given by
\ba\label{eq:ansatz}
U^\e(\xi,\Hx,\Ht)&=\frac{1}{2}A(\Hx,\Ht)e^{i\xi}r+c.c.+\e^2(B(\Hx,\Ht)+\Psi_0(\Hx,\Ht))\\
&\quad+\e^2\frac{1}{2}\Psi_1(\Hx,\Ht)e^{i\xi}+c.c.+\e^2\frac{1}{2}\Psi_2(\Hx,\Ht)e^{2i\xi}+c.c.+\e^3\frac{1}{2}\Psi_3(\Hx,\Ht)e^{i\xi}+c.c.\\
&\quad+\e^3\Psi_4(\Hx,\Ht)+\e^4\Psi_5(\Hx,\Ht),
\ea
where $\xi=k_*(x-d_*t)$, $\Hx=\e(x-(d_*+\delta)t)$, and $\Ht=\e^2t$ for $d_*,\delta$ real numbers to be determined. We further assume that $\Pi_0B=B$ and $\Pi_0\Psi_0=0$.

\subsection{An example model}\label{sec:MSE_Example}
Before we tackle the general case, we illustrate the singular phenomenon addressed in Section \ref{s:main} with an example model. We take the linear operator to be
\begin{equation}\label{eq:examplelinop}
	L(\mu)=\begin{bmatrix}
	 2 & 1\\
	 1 & 2
	\end{bmatrix}\d_x^2+\begin{bmatrix}
	 0 & 0 \\
	 c_1+\mu & c_2
	\end{bmatrix}\d_x+\begin{bmatrix}
		0 & 0\\
		0 & -1
	\end{bmatrix},
\end{equation}
where $c_1$, $c_2$ are real numbers chosen to ensure that $L(\mu)$ satisfies the Turing hypotheses \ref{hyp:Lin}. 
\begin{lemma}\label{lem:examplecoefficients}
The choices $c_1\approx 10.5558$, $c_2\approx 1.2247$, and $k_*\approx 0.7598$ provides an example of a linear system with a conservation law undergoing a Turing bifurcation.
\end{lemma}

\begin{proof}[Sketch]
Let $S(k,c)$ be the Fourier symbol of the linear operator in \eqref{eq:examplelinop} at $\mu=0$ given by
\be
	S(k,c)=-k^2\begin{bmatrix}
		2 & 1\\ 1 & 2
	\end{bmatrix}+ik\begin{bmatrix}
	0 & 0\\
	c_1 & c_2
\end{bmatrix}+\begin{bmatrix}
0 & 0 \\
0 & -1
\end{bmatrix},
\ee
for $c=(c_1,c_2)$. Our first step is to show that we can choose $c$ so that $S(k,c)$ has a unique $k_*>0$ with the property that $S(k_*,c)$ has an eigenvalue with real part zero. We let $T(k,c)$ denote the trace of $S(k,c)$ and $D(k,c)$ the determinant of $S(k,c)$. We can express $T(k,c)$ and $D(k,c)$ as
\ba
	T(k,c)&=-4k^2+ikc_2-1,\\
	D(k,c)&=3k^4+2k^2+i(c_1-2c_2)k^3.
\ea
The characteristic polynomial $p(\l)$ of $S(k,c)$ can be expressed as
\be
	p(\l)=\l^2-T(k,c)\l+D(k,c).
\ee
We now assume that we have an eigenvalue of the form $\l=i\s$ for some real $\s$, and split the characteristic polynomial into its real and imaginary parts to get the polynomial system
\ba
	\Re(p(\l))=0:\quad -\s^2+kc_2\s+3k^4+2k^2&=0,\\
	\Im(p(\l))=0:\quad (4k^2+1)\s+(c_1-2c_2)k^3&=0.
\ea
The equation for the imaginary part determines $\s$ in terms of $k$ and $c$ by
\be
	\s=-\frac{(c_1-2c_2)k^3}{4k^2+1},
\ee
which upon substitution into the equation for the real part gives
\be
-\frac{(c_1-2c_2)^2}{(4k^2+1)^2}k^6-\frac{c_2(c_1-2c_2)}{4k^2+1}k^4+3k^4+2k^2=0.
\ee
We may divide by $k^2$ to remove the trivial solution coming from the conservation law, and multiplying by $(4k^2+1)^2$ leads to
\be
-(c_1-2c_2)^2k^4-c_2(c_1-2c_2)k^2(4k^2+1)+3k^2(4k^2+1)^2+2(4k^2+1)^2=0,
\ee
which is a degree 3 polynomial in $k^2$. Collecting like terms, we obtain
\be\label{eq:char_k_eqn}
	48k^6+(56-(c_1-2c_2)^2-4c_2(c_1-2c_2) )k^4+(19+c_2(2c_2-c_1) )k^2+2=0.
\ee
We now choose $(c_1,c_2)$ so that the discriminant, $\Delta(c)$, of \eqref{eq:char_k_eqn} vanishes, which tells us we have a repeated root which is necessarily real as our original polynomial is a cubic in $k^2$ with all real coefficients. Formally factoring \eqref{eq:char_k_eqn} we see that if $r_j$ denote the roots then
\begin{equation*}
	48(k^2-r_1)(k^2-r_2)(k^2-r_3)=48k^6+(56-(c_1-2c_2)^2-4c_2(c_1-2c_2) )k^4+(19+c_2(2c_2-c_1) )k^2+2.
\end{equation*}
This implies that the product of the roots must be negative, hence we either have one negative root and two positive roots or three negative roots. We then choose $(c_1,c_2)$ so that the repeated root, which we denote $k_*$, is positive. The method giving the sharpest constraints on $(c_1,c_2)$ in order for the repeated root to be positive is to use take advantage of the particularly simple formulas for the roots of a cubic when there is a repeated root. Let $c_*$ denote the values of $(c_1,c_2)$ in the statement of the lemma.\\

We note that the pure imaginary eigenvalue of $S(k_*,c_*)$ is simple because the trace satisfies $\Re T(k_*,c_*)<0$, hence the other eigenvalue must have negative real part. Letting $\tl(k_*,0)$ denote the pure imaginary eigenvalue of $S(k_*,c_*)$, we see that $\tl(k_*,0)$ extends to a smooth eigenvalue $\tl(k,\mu)$ of $S(k,(c_1^*,c_2^*)+\mu(1,0) )$ in a neighborhood of $(k_*,0)$. A routine calculation shows that $\Re\tl_\mu(k_*,0)>0$ for this choice of $c_*$. A lengthier calculation reveals that $\Delta(\mu):=\Delta(c_*+\mu(1,0))$ has a simple root at $\mu=0$ with $\Delta_\mu(0)>0$. We caution that the only relation between $\Delta_\mu(0)$ and $\Re\tl_\mu(k_*,0)$ is that they generically have the same sign, the actual values are unrelated. In this instance, $\Delta_\mu(0)\approx1.18*10^6$ and $\Re\tl_\mu(k_*,0)\approx0.0661$.
\end{proof}

\begin{remark}
	The argument suggests that our choice of constants in Lemma \ref{lem:examplecoefficients} is part of a one-parameter family of examples, but we will not pursue this here.
\end{remark}

Define the Fourier symbol $S(k,\mu)$ associated to \eqref{eq:examplelinop} to be
\be\label{eq:examplesymbol}
	S(k,\mu)=-k^2\begin{bmatrix}
		2 & 1\\
		1 & 2
	\end{bmatrix}+ik\begin{bmatrix}
		0 & 0 \\
		c_1+\mu & c_2
	\end{bmatrix}+\begin{bmatrix}
		0 & 0\\
		0 & -1
	\end{bmatrix},
\ee
where $c_1,c_2$ are as in Lemma \ref{lem:examplecoefficients}. For this choice of coefficients, one finds that the relevant quantities for the dispersion relation are
\ba\label{eq:exampledispersion}
&\tl(k_*,0)\approx 1.0746i, \quad  \tl_\mu(k_*,0)\approx 0.0661-0.0710i, \\
&\tl_k(k_*,0)\approx -2.2689i, \quad \tl_{kk}(k_*,0)\approx-1.3292-1.4717i.
\ea
One choice of corresponding eigenvectors $r$ and $\ell$ is to take
\begin{equation}
	r\approx\bp-0.0935-0.087i \\ 0.3489\ep, \quad \ell\approx\bp -3.7321+3.4731i & 1\ep.
\end{equation}

We will take the nonlinearity to be 
\be\label{eq:examplenonlin}
	\cN(U)=\frac{1}{2}\bp \d_x (u)^2\\
		v^2\ep.+\frac{c}{6}\bp \d_x(u)^3\\v^3\ep
\ee
for some $c\in\RR$ chosen so that the corresponding constant $\g$ in the cGL equation has the correct sign on its real part and where we've written $U=(u,v)$.
\begin{remark}
	All of the interesting behavior in mode 0 is captured by the simpler nonlinearity
	\begin{equation*}
		\cN'(U)=\frac{1}{2}\bp \d_x (u)^2\\v^2\ep.
	\end{equation*}
	However, for this nonlinearity the Turing bifurcation turns out to be subcritical.
\end{remark}
Replacing $\d_x$ in \eqref{eq:examplelinop} with $k_*\d_\xi+\e\d_{\Hx}$ and applying the resulting operator to \eqref{eq:ansatz}, we get
\ba\label{eq:exampleLAnsatz}
	L(\mu)U^\e&=\frac{1}{2}\e AS(k_*,0)re^{i\xi}+c.c.+\frac{1}{2}\e^2\left[-iS_k(k_*,0)A_{\Hx}r+S(k_*,0)\Psi_1\right]e^{i\xi}+c.c.\\
	&\quad+ \e^2S(0,0)(B+\Psi_0)+\frac{1}{2}\e^2S(2k_*,0)\Psi_2e^{2i\xi}+c.c.\\
	&\quad+\frac{1}{2}\e^3\left[-\frac{1}{2}S_{kk}(k_*,0)A_{\Hx\Hx}r+S_\mu(k_*,0)Ar-iS_k(k_*,0)\Psi_{1,\Hx}+S(k_*,0)\Psi_3\right]e^{i\xi}+c.c.\\
	&\quad+\e^3(S(0,0)\Psi_4-iS_k(0,0)(B+\Psi_0)_{\Hx} )\\
	&\quad+\e^4(S(0,0)\Psi_5-iS_k(0,0)\Psi_{4,\Hx}-\frac{1}{2}S_{kk}(0,0)(B+\Psi_0)_{\Hx\Hx}+S_\mu(0,0)(B+\Psi_0))+other,
\ea
where other refers to terms that will not be needed.\\

Let $\cQ:\RR^2\times\RR^2\to\RR^2$ be the bilinear form
\be
	\cQ\left(\bp x\\y\ep,\bp u\\w\ep \right)=\frac{1}{2}\bp xu\\ yw\ep,
\ee
and let $M(T)$ be the operator
\be
	M(T)=\bp T & 0\\
	0 & Id\ep,
\ee
where $T$ is any bounded operator $T:H^2_{per}(\RR;\RR^2)\to L^2_{per}(\RR;\RR^n)$.\\

One then has the identity $\cN(U)=M(\d_x)\cQ(U,U)$. With this notation in hand, we apply the same procedure used to compute $L(\mu)U^\e$ to compute $\cN(U^\e)$.
\ba\label{eq:exampleNAnsatz}
	\cN(U^\e)&=\frac{1}{4}\e^2A^2M(k_*\d_\xi)(\cQ(r,r)e^{2i\xi})+c.c.+\frac{1}{2}\e^2|A|^2M(k_*\d_\xi)(\cQ(r,\bar{r}))\\
	&\quad+\e^3\left[AM(k_*\d_\xi)(\cQ(r,B+\Psi_0)e^{i\xi})+\frac{1}{2}\bar{A}M(k_*\d_\xi)(\cQ(\bar{r},\Psi_2)e^{i\xi})\right]+c.c.\\
	&\quad+\e^3\left[\frac{1}{2}AM(k_*\d_\xi)\cQ(r,\bar{\Psi}_1)+\frac{1}{2}\bar{A}M(k_*\d_\xi)\cQ(\bar{r},\Psi_1)+\frac{1}{2}(\Pi_0\d_{\Hx})\cQ(Ar,\bar{A}\bar{r})\right]\\
	&\quad+\e^4\left[\frac{1}{2}AM(k_*\d_\xi)\cQ(r,\bar{\Psi}_3)+\frac{1}{2}\bar{A}M(k_*\d_\xi)\cQ(\bar{r},\Psi_3)+\frac{1}{2}M(k_*\d_\xi)\cQ(\Psi_2,\bar{\Psi}_2)\right.\\
	&\quad\left.+M(k_*\d_\xi)\cQ(B+\Psi_0,B+\Psi_0)+\frac{1}{2}(\Pi_0\d_{\Hx})\cQ(Ar,\bar{\Psi}_1)+\frac{1}{2}(\Pi_0\d_{\Hx})\cQ(\bar{A}\bar{r},\Psi_1)\right]+trilinear+other,
\ea
where other refers to terms that are higher order or have the wrong Fourier mode or both and trilinear refers to the terms involving the trilinear term in the nonlinearity.\\

We now plug in the expansions for $L(\mu)U^\e$ and $\cN(U^\e)$ into \eqref{eq:evoleqn} and match coefficients of $\e^ne^{i\eta\xi}$. First, we get at order $\e$ and fast Fourier mode 1
\begin{equation}\label{eq:exampleeps1mode1}
	A(S(k_*,0)+id_*k_* )r=0.
\end{equation}
For this equation to have a nontrivial solution for $A$, we necessarily have that $k_*d_*=-\Im\tl(k_*,0)$. Moving onto order $\e^2$, we start with the coefficient of $\e^2e^{0}$ given by
\begin{equation}\label{eq:exampleeps2mode0}
	S(0,0)(B+\Psi_0)+\frac{1}{2}|A|^2M(k_*\d_\xi)(\cQ(r,\bar{r}))=0.
\end{equation}
This equation is solvable for $\Psi_0$ if and only if $\Pi_0$\eqref{eq:exampleeps2mode0}=0. But this is clear by definition of $M(k_*\d_\xi)$ because it is acting on a constant vector and hence has first coefficient identically 0. Solving this equation for $\Psi_0$, we get
\begin{equation}\label{eq:examplePsi0}
	\Psi_0=-|A|^2\frac{1}{2}N_0\cQ(r,\bar{r}),
\end{equation}
where we define $N_0$ to be the matrix
\begin{equation}\label{eq:N0}
	N_0:=-((I-\Pi_0)S(0,0)(I-\Pi_0))^{-1}.
\end{equation}
To finish the mode we need to show that $\cQ(r,\bar{r})\in\RR^2$, but this is clear because writing $r=(r_1,r_2)$ we see
\begin{equation}
	\cQ(r,\bar{r})=\frac{1}{2}\bp r_1\bar{r}_1\\ r_2\bar{r}_2\ep.
\end{equation}
For the sake of completeness, in this example we can explicitly write $\Psi_0$ as 
\begin{equation*}
	\Psi_0=\frac{1}{4}|A|^2\bp 0 \\ |r_2|^2\ep.
\end{equation*}
Next, we look at the coefficient of $\e^2e^{i\xi}$ where we find
\begin{equation}\label{eq:exampleeps2mode1}
	-iS_k(k_*,0)A_{\Hx}r+(d_*+\delta)A_{\Hx}r+(S(k_*,0)+id_*k_*)\Psi_1=0.
\end{equation}
In order to be able to solve for $\Psi_1$ we need $\ell$\eqref{eq:exampleeps2mode1}=0. Expanding this out, we get
\begin{equation}\label{eq:delta}
	(\Im \tl_k(k_*,0)+d_*+\delta)A_{\Hx}=0.
\end{equation}
For this equation, we choose $\delta$ so that $A_{\Hx}$ is unconstrained.
\begin{remark}
	As expected, $d_*$ and $\delta$ coincide with the corresponding constants for the case where there are no conservation laws.
\end{remark}
Solving \eqref{eq:exampleeps2mode1} for $(I-\Pi_1)\Psi_1$ gives
\begin{equation}\label{eq:examplePsi1}
	(I-\Pi_1)\Psi_1=iA_{\Hx}N_1S_k(k_*,0)r,
\end{equation}
where $N_1$ is defined to be the matrix
\begin{equation}\label{eq:N1}
	N_1:=((I-\Pi_1)(S(k_*,0)+id_*k_*)(I-\Pi_1))^{-1}\approx\frac{1}{-3.0394+2.0052i}(I-\Pi_1).
\end{equation}
The last mode at $\e^2$ is $\e^2e^{2i\xi}$, given by
\begin{equation}
	\frac{1}{2}(S(2k_*,0)+2ik_*d_*)\Psi_2+\frac{1}{4}A^2M(k_*\d_\xi)\cQ(r,r)=0.
\end{equation}
As $S(2k_*,0)$ has eigenvalues with all negative real part, this is uniquely solvable for $\Psi_2$ giving the expression
\begin{equation}\label{eq:examplePsi2}
	\Psi_2=-\frac{1}{2}A^2N_2\begin{bmatrix}
	 2ik_* & 0\\
	 0 & 1 
	\end{bmatrix}\cQ(r,r),
\end{equation}
where $N_2$ is the matrix
\begin{equation}
	N_2:=(S(2k_*,0)+2ik_*d_*)^{-1}.
\end{equation}

We now come to order $\e^3$, where we will first look at the coefficient of $\e^3e^{i\xi}$. Here we find
\ba\label{eq:exampleeps3mode1}
	\frac{1}{2}A_{\Ht}r&=\frac{1}{2}(S(k_*,0)+id_*k_*)\Psi_3+\frac{1}{2}(-iS_k(k_*,0)+(d_*+\delta))\Psi_{1,\Hx}-\frac{1}{4}S_{kk}(k_*,0)A_{\Hx\Hx}r\\
	&\quad+\frac{1}{2}S_\mu(k_*,0)Ar+A\begin{bmatrix}
		ik_* & 0\\
		0 & 1
	\end{bmatrix}\cQ(r,B+\Psi_0)+\frac{1}{2}\bar{A}\begin{bmatrix}
	ik_* & 0\\
	0 & 1
\end{bmatrix}\cQ(\bar{r},\Psi_2)+trilinear.
\ea
For this equation to be solvable for $(I-\Pi_1)\Psi_3$ it is necessary and sufficient for $\ell$\eqref{eq:exampleeps3mode1} to vanish identically. Expanding $\ell$\eqref{eq:exampleeps3mode1} we find the following equation
\ba\label{eq:examplecgl1}
	A_{\Ht}&=\ell(-iS_k(k_*,0)+(d_*+\delta))\Psi_{1,\Hx}-\frac{1}{2}\ell S_{kk}(k_*,0)A_{\Hx\Hx}r\\
	&\quad+\tl_\mu(k_*,0)A+2A \ell\begin{bmatrix}
		ik_* & 0\\
		0 & 1
	\end{bmatrix}\cQ(r,B+\Psi_0)+\bar{A}\ell\begin{bmatrix}
	ik_* & 0\\
	0 & 1
\end{bmatrix}\cQ(\bar{r},\Psi_2).
\ea
For the quadratic terms depending only on $A$, we plug in \eqref{eq:examplePsi0} and \eqref{eq:examplePsi2} to get
\ba
&2A\begin{bmatrix}
ik_* & 0\\
0 & 1
\end{bmatrix}\cQ(r,\Psi_0)+\bar{A}\begin{bmatrix}
ik_* & 0\\
0 & 1
\end{bmatrix}\cQ(\bar{r},\Psi_2)\\
&=|A|^2A\begin{bmatrix}
ik_* & 0\\
0 & 1
\end{bmatrix}\cQ(r,-N_0\cQ(r,\bar{r}))+|A|^2A\begin{bmatrix}
ik_* & 0\\
0 & 1
\end{bmatrix}\cQ(\bar{r},-\frac{1}{2}N_2\begin{bmatrix}
2ik_* & 0\\
0 & 1 
\end{bmatrix}\cQ(r,r))\\
&=|A|^2AV,
\ea
for some known $V\in\CC^2$. We then let $\g:=\ell V$, which allows us to write the nonlinear terms of \eqref{eq:examplecgl1} as $\g|A|^2A$. We let $V_1$ be the vector defined by
\begin{equation}
	V_1\cdot W=\ell\begin{bmatrix}
	ik_* & 0\\
	0 & 1
	\end{bmatrix}\cQ(r,W),
\end{equation}
for all vectors $W\in\CC^2$, where $\cdot$ denotes the inner product on $\CC^2$ given by $(z_1,z_2)\cdot(w_1,w_2)=\bar{z_1}w_1+\bar{z_2}w_2$.\\

Returning to the linear part of \eqref{eq:examplecgl1}, we split $\Psi_1$ as $\Psi_1=\cA r+(I-\Pi_1)\Psi_1$. Doing this gives
\ba\label{eq:examplecgllin}
	&\ell(-iS_k(k_*,0)+(d_*+\delta))\Psi_{1,\Hx}-\frac{1}{2}\ell S_{kk}(k_*,0)A_{\Hx\Hx}r\\
	&\quad=(\Im \tl_k(k_*,0)+d_*+\delta)\cA_{\Hx}-i\ell S_k(k_*,0)(I-\Pi_1)\Psi_{1,\Hx}\frac{1}{2}\ell S_{kk}(k_*,0)A_{\Hx\Hx}r.
\ea
Note that by design, $\Im\tl_k(k_*,0)+d_*+\delta=0$. For the other $\Psi_1$ term, we recall \eqref{eq:examplePsi1} and plug it into \eqref{eq:examplecgllin}. This gives the following
\ba
	&(\Im \tl_k(k_*,0)+d_*+\delta)\cA_{\Hx}-i\ell S_k(k_*,0)(I-\Pi_1)\Psi_{1,\Hx}\frac{1}{2}\ell S_{kk}(k_*,0)A_{\Hx\Hx}r
	\\&\quad=A_{\Hx\Hx}\ell S_k(k_*,0)N_1S_k(k_*,0)r-\frac{1}{2}A_{\Hx\Hx}\ell S_{kk}(k_*,0)r.
\ea
Applying the key spectral identity in Proposition \eqref{prop:spectralid}, we get at last the desired (cGL) given by
\begin{equation}\label{eq:exampleCGL}
	A_{\Ht}=-\frac{1}{2}\tl_{kk}(k_*,0)A_{\Hx\Hx}+\tl_\mu(k_*,0)A+\g|A|^2A+A V_1\cdot B.
\end{equation}

We now turn to the other important coefficient at order $\e^3$, $\e^3e^{0}$, where we get
\ba\label{eq:exampleeps3mode0}
	&S(0,0)\Psi_4+(-iS_k(0,0)+d_*+\delta)(B+\Psi_0)_{\Hx}+\frac{1}{2}(\Pi_0\d_{\Hx})\cQ(Ar,\bar{A}\bar{r})\\
	&\quad+\frac{1}{2}AM(k_*\d_\xi)\cQ(r,\bar{\Psi}_1)+\frac{1}{2}\bar{A}M(k_*\d_\xi)\cQ(\bar{r},\Psi_1)=0.
\ea
As before, it is necessary and sufficient for $\Pi_0$\eqref{eq:exampleeps3mode0} to vanish in order to be able to solve for $(I-\Pi_0)\Psi_4$. However, writing out what this means we find the following
\begin{equation}\label{eq:examplePi0}
	\Pi_0(\begin{bmatrix} 0 & 0\\
	c_1 & c_2 \end{bmatrix}+d_*+\delta)B_{\Hx}+\Pi_0\begin{bmatrix} 0 & 0\\
	c_1 & c_2 \end{bmatrix}\Psi_{0,\Hx}+\frac{1}{2}\Pi_0\d_{\Hx}|A|^2\cQ(r,\bar{r})=0.
\end{equation}
One can easily check from the definition of $\Pi_0$ that
\begin{equation*}
	\Pi_0\begin{bmatrix}
	0 & 0\\
	c_1 & c_2
	\end{bmatrix}=0.
\end{equation*}
But $d_*+\delta\not=0$ by \eqref{eq:delta} and \eqref{eq:exampledispersion}. Moreover, it is relatively straightforward to verify that $\Pi_0r\not=0$. So what we find is that \eqref{eq:examplePi0} reduces to
\begin{equation}
	-\Im\tl_k(k_*,0)B_{\Hx}+\frac{1}{4}(|A|^2)_{\Hx}\bp |r_1|^2\\0\ep=0.
\end{equation}

To rectify the situation, we will bring in some of the terms from the coefficient of $\e^4e^{0}$. In this case, we have
\ba\label{eq:exampleeps4mode0}
	&(B+\Psi_0)_{\Ht}+(S(0,0)\Psi_5-iS_k(0,0)\Psi_{4,\Hx}-\frac{1}{2}S_{kk}(0,0)(B+\Psi_0)_{\Hx\Hx}+S_\mu(0,0)(B+\Psi_0))\\
	&\quad+\frac{1}{2}AM(k_*\d_\xi)\cQ(r,\bar{\Psi}_3)+\frac{1}{2}\bar{A}M(k_*\d_\xi)\cQ(\bar{r},\Psi_3)+\frac{1}{2}M(k_*\d_\xi)\cQ(\Psi_2,\bar{\Psi}_2)\\
	&\quad+M(k_*\d_\xi)\cQ(B+\Psi_0,B+\Psi_0)+\frac{1}{2}(\Pi_0\d_{\Hx})\cQ(Ar,\bar{\Psi}_1)+\frac{1}{2}(\Pi_0\d_{\Hx})\cQ(\bar{A}\bar{r},\Psi_1)+CM(k_*\d_\xi)\cC(Ar,\bar{Ar},B+\Psi_0)=0,
\ea
where $C$ is a known constant, and $\cC(U,V,W)$ is the trilinear form $\cC(U,V,W)=(u_1v_1w_1,u_2v_2w_2)$. Of particular importance are the terms depending only on $A$ and $B$ that survive in $\Pi_0$\eqref{eq:exampleeps4mode0}. Those terms are given by
\ba\label{eq:exampleeps4mode0AB}
	&B_{\Ht}-i\Pi_0S_k(0,0)(I-\Pi_0)\Psi_{4,\Hx}-\frac{1}{2}\Pi_0S_{kk}(0,0)(B+\Psi_0)_{\Hx\Hx}+\Pi_0S_\mu(0,0)(B+\Psi_0)\\
	&\quad+\frac{1}{2}(\Pi_0\d_{\Hx})\cQ(Ar,\overline{(I-\Pi_1)\Psi}_1)+\frac{1}{2}(\Pi_0\d_{\Hx})\cQ(\bar{A}\bar{r},(I-\Pi_1)\Psi_1),
\ea
where as before, we have that $\Pi_0M(k_*\d_\xi)$ vanishes. There are other terms in \eqref{eq:exampleeps4mode0} depending on the higher amplitudes $\cA:=\ell\Psi_1$ and $\cB:=\Pi_0\Psi_4$ so this ``splitting'' procedure does not produce a solution that is valid to $\cO(\e^4)$ without ``splitting'' $\cO(\e^5)$ as well. Recalling the expansion in \eqref{eq:examplePsi1}, we get that
\ba\label{eq:exampleeps4AB2}
& \frac{1}{2}(\Pi_0\d_{\Hx})\cQ(Ar,\overline{(I-\Pi_1)\Psi}_1)+\frac{1}{2}(\Pi_0\d_{\Hx})\cQ(\bar{A}\bar{r},(I-\Pi_1)\Psi_1)\\
&\quad=\frac{1}{2}(\Pi_0\d_{\Hx})\cQ(Ar,\overline{iA_{\Hx}N_1S_k(k_*,0)r})+\frac{1}{2}(\Pi_0\d_{\Hx})\cQ(\bar{A}\bar{r},iA_{\Hx}N_1S_k(k_*,0)r).
\ea
By the identity $\overline{\cQ(x,y)}=\cQ(\overline{x},\overline{y})$, we note that the expression in \eqref{eq:exampleeps4AB2} is real valued. We let $v:=\Pi_0\cQ(r,\overline{iN_1S_k(k_*,0)r})$. Because numerical evidence strongly suggests that $v$ has nonzero imaginary part, the best we can say about \eqref{eq:exampleeps4AB2} is that
\begin{equation}
	\frac{1}{2}(\Pi_0\d_{\Hx})\cQ(Ar,\overline{iA_{\Hx}N_1S_k(k_*,0)r})+\frac{1}{2}(\Pi_0\d_{\Hx})\cQ(\bar{A}\bar{r},iA_{\Hx}N_1S_k(k_*,0)r)=\d_{\Hx}\Re(A\bar{A}_{\Hx}v).
\end{equation}
Note that $\Pi_0 S_\mu(0,0)=0$ because $\Pi_0S(0,\mu)\equiv 0$ for all $\mu$ by hypothesis.
\begin{remark}
	One can also consider situations where $\Pi_0S(0,0)=0$, $S(0,\mu)$ always has a nontrivial kernel, but $\Pi_0 S_\mu(0,0)\not=0$. Physically, this corresponds to systems where the conserved direction depends on $\mu$. Another possible way for $\Pi_0S_\mu(0,0)\not=0$ to happen is a Turing-like bifurcation where the symbol degenerates at multiple frequencies simultaneously.
\end{remark}
With all of this information in hand, we return to \eqref{eq:exampleeps4mode0AB} and simplify to get
\begin{equation}\label{eq:example4mode0AB4}
	B_{\Ht}-\frac{1}{2}S_{kk}(0,0)(B+\Psi_0)_{\Hx\Hx}+\d_{\Hx}\Re(A\bar{A}_{\Hx}v).
\end{equation}
One can easily check from \eqref{eq:examplesymbol} that
\begin{equation}
	\Pi_0 S_{kk}(0,0)\Pi_0=\begin{bmatrix}
		-4 & 0 \\
		0 & 0
	\end{bmatrix} \quad \Pi_0 S_{kk}(0,0)(I-\Pi_0)=\begin{bmatrix}
		0  & -2\\
		0 & 0
	\end{bmatrix}.
\end{equation}
Hence our singular equation for $B$ is given by
\begin{equation}\label{eq:examplesingularBeqn}
	B_{\Ht}=\e^{-1}\left(-\Im\tl_k(k_*,0)B_{\Hx}+\frac{1}{4}(|A|^2)_{\Hx}\bp |r_1|^2\\0\ep\right)+2B_{\Hx\Hx}+\left(\frac{1}{4}|A|^2\bp |r_2|^2 \\ 0\ep\right)_{\Hx\Hx}+\d_{\Hx}\Re(A\bar{A}_{\Hx}v).
\end{equation}

This is not the only way to fix the problem with \eqref{eq:exampleeps3mode0}, one could alternatively introduce an intermediate time scale $\Tt=\e t$ and allow $B$ to depend on $\Tt$. If one does this, then \eqref{eq:exampleeps3mode0} becomes
\begin{equation}\label{eq:exampleadvec}
	B_{\Tt}=-\Im\tl_k(k_*,0)B_{\Hx}+\frac{1}{4}(|A|^2)_{\Hx}\bp |r_1|^2\\0\ep.
\end{equation}
The advantage of this approach is that now the equation for $B$ is non-singular. The disadvantage is that is not clear that the approximate solution produced by coupling \eqref{eq:exampleadvec} and \eqref{eq:exampleCGL} is actually valid out to $\Ht\sim 1$, or equivalently the original fast time $t$ out to $\cO(\e^{-2})$. Another way to view the nonsingular system is that it describes the boundary layer of the singular system obtained by coupling \eqref{eq:examplesingularBeqn} and \eqref{eq:exampleCGL}.
\subsection{The general case}\label{sec:MSE_general}
Now, we consider $L(\mu)=\sum_{j=0}^m\cL_j(\mu)\d_x^j$ and $\cN(U,\mu)$ a general differential operator satisfying Hypotheses \ref{hyp:Lin} and a general quasilinear nonlinearity satisfying Hypotheses \ref{hyp:Nonlin} respectively. We let $S(k,\mu)$ be the Fourier symbol associated to $L(\mu)$ given by
\begin{equation}
	S(k,\mu)=\sum_{j=0}^m\cL_j(\mu)(ik)^j.
\end{equation} 
For a bounded operator $T$, define $M(T)$ to be the operator
\begin{equation}
	M(T)=T\Pi_0+(I-\Pi_0).
\end{equation}
Further assume that $\cN(U,\mu)$ admits the representation
\begin{equation}\label{eq:MdxcN}
	\cN(U,\mu)=M(\d_x)\tilde{\cN}(U,\mu),
\end{equation}
for some local quasilinear nonlinearity $\tilde{\cN}(U,\mu)$ satisfying $\tilde{\cN}(0,\mu)=0$ for all $\mu$.\\

To begin, we note that \eqref{eq:exampleLAnsatz} carries over to the general case with minimal modifications. For the nonlinearity $\cN$ we use the splitting above, Taylor expand with respect to $(U,\mu)$ and replace the $\d_x$ in $M(\d_x)$ with $k_*\d_\xi+\e\d_{\Hx}$ to get
\ba\label{eq:generalTaylor}
	\cN(U^\e,\mu)&=\frac{1}{2}M(k_*\d_\xi)\tilde{\cN}(0,0)(U^\e,U^\e)+\frac{1}{2}\e \d_{\Hx}\Pi_0 D_U^2\tilde{\cN}(0,0)(U^\e,U^\e)\\
	&\quad+\frac{1}{2} \mu M(k_*\d_\xi)\d_\mu D_U^2\tilde{\cN}(0,0)(U^\e,U^\e)+\frac{1}{2}\e\mu \d_{\Hx}\Pi_0\d_\mu D_U^2\tilde{\cN}(0,0)(U^\e,U^\e)\\
	&\quad+\frac{1}{6}M(k_*\d_\xi) D_U^3\tilde{\cN}(0,0)(U^\e,U^\e,U^\e)+\frac{1}{6}\e \d_{\Hx}\Pi_0D_U^3\tilde{\cN}(0,0)(U^\e,U^\e,U^\e)+h.o.t.\ .
\ea
We proceed to construct in the same fashion as in Section \ref{sec:MSE_Example} an approximate solution of the form \eqref{eq:ansatz} of the equation.
\begin{equation}
	\frac{\d U^\e}{\d t}=L(\mu)U^\e+\cN(U^\e,\mu).
\end{equation}
Note that the coefficients of $\e e^{i\xi}$ given by \eqref{eq:exampleeps1mode1} and $\e^2e^{i\xi}$ given by \eqref{eq:exampleeps2mode1} are unchanged because they purely came from the linear operator $L(\mu)$. By Proposition \ref{prop:multilin}, there exists multipliers $\cQ$ and $\cC$ so that
\begin{equation}
	D_U^2\tilde{\cN}(0,0)(U,V)(\xi)=\sum_{\eta_1,\eta_2\in\ZZ}\cQ(\eta_1,\eta_2)(\hat{U}(\eta_1),\hat{V}(\eta_2))e^{i(\eta_1+\eta_2)\xi},
\end{equation}
and
\begin{equation}
	D_U^3\tilde{\cN}(0,0)(U,V,W)(\xi)=\sum_{\eta_1,\eta_2,\eta_3\in\ZZ}\cC(\eta_1,\eta_2,\eta_3)(\hat{U}(\eta_1),\hat{V}(\eta_2),\hat{W}(\eta_3))e^{i(\eta_1+\eta_2+\eta_3)\xi},
\end{equation}
for all $2\pi$ periodic functions $U,V,W$.\\

We turn to order $\e^2$ and look at the coefficients $\e^2e^{2i\xi}$ and $\e^2e^{0i\xi}$, which are given by
\begin{equation}\label{eq:eps2mode2}
	\frac{1}{2}(S(2k_*,0)+2ik_*d_*)\Psi_2+\frac{1}{8}A^2M(2ik_*)\cQ(1,1)(r,r)=0,
\end{equation}
and
\begin{equation}\label{eq:eps2mode0}
	S(0,0)(B+\Psi_0)+\frac{1}{4}|A|^2M(0)(\cQ(1,-1)(r,\bar{r}))=0,
\end{equation}
respectively, where $M(c)$ for a complex number $c$ is defined to be $M(T_c)$ where $T_c$ is the operator defined by multiplying by $c$. Note that \eqref{eq:eps2mode2} is uniquely solvable for $\Psi_2$ because $S(2k_*,0)+2ik_*d_*$ is invertible by our spectral stability assumption and we get the following formula for $\Psi_2$
\begin{equation}\label{eq:generalPsi2}
	\Psi_2=-\frac{1}{4}A^2(2(2k_*,0)+2ik_*d_*)^{-1}M(2ik_*)\cQ(1,1)(r,r).
\end{equation}
For \eqref{eq:eps2mode0}, we note by \eqref{eq:MdxcN} that $M(0)=(I-\Pi_0)$ and so it is uniquely solvable for $\Psi_0$ satisfying $(I-\Pi_0)\Psi_0=\Psi_0$ and that $\Psi_0$ is given by
\begin{equation}\label{eq:generalPsi0}
	\Psi_0=-\frac{1}{4}|A|^2N_0(I-\Pi_0)\cQ(1,-1)(r,\bar{r}),
\end{equation}
where $N_0$ is the analog of \eqref{eq:N0}.
\begin{lemma}\label{lem:genPsi0real}
	$\Psi_0$ as in \eqref{eq:generalPsi0} is real-valued.
\end{lemma}
\begin{proof}
	It suffices to note that $N_0$ is the inverse of a real matrix, hence is a real matrix, and that 
	\begin{equation*}
		2|A|^2(I-\Pi_0)\cQ(1,-1)(r,\bar{r}),
	\end{equation*}
	is the Fourier mean of the real valued function $(I-\Pi_0)D_U^2\cN(0,0)(U_A,U_A)$ where $U_A$ is
	\begin{equation*}
		U_A=Ae^{i\xi}r+c.c.
	\end{equation*}
\end{proof}
With order $\e^2$ taken care, we turn to order $\e^3$ where we have two equations of interest. The more interesting one is $\e^3e^{0i\xi}$ where we get
\ba\label{eq:eps3mode0}
	&S(0,0)\Psi_4-iS_k(0,0)(B+\Psi_0)_{\Hx}+(d_*+\delta)(B+\Psi_0)_{\Hx}\\
	&\quad+\frac{1}{8}M(0)\Big(\cQ(1,-1)(Ar,\overline{\Psi}_1)+\cQ(1,-1)(\Psi_1,\overline{Ar})\Big)\\
	&\quad+\frac{1}{4}M(0)\Big(\cQ_\nu(1,-1)(A_{\Hx}r,\overline{Ar})+\cQ_\nu(1,-1)(Ar,\overline{A_{\Hx}r})\Big)+\frac{1}{4} \d_{\Hx}\Pi_0\cQ(1,-1)(Ar,\overline{Ar})=0,
\ea
where $\cQ_\nu$ is defined by
\begin{equation}\label{eq:defcQnu}
	\cQ_\nu(\eta_1,\eta_2):=\frac{\d}{\d (i\eta_1)}\cQ(\eta_1,\eta_2)=-i\frac{\d}{\d\eta_1}\cQ(\eta_1,\eta_2).
\end{equation}
Informally, $\cQ_\nu$ is the multiplier associated to ``$\frac{\d}{\d \d_x}D_U^2\tilde{\cN}(0,\mu)$''. We will see later in Corollary \ref{cor:cQnureal} that $\Pi_0\cQ_\nu$ as defined above is in fact a real bilinear form. For now, it will suffice to know that $\cQ_\nu(\eta_1,\eta_2)=\overline{\cQ_\nu(-\eta_1,-\eta_2)}$ follows by a short calculation involving the chain rule and the reality condition for $\cQ$. Hence $\cQ_\nu$ is associated to a translation-invariant bilinear form which maps real functions to real functions.  In order to be able to solve \eqref{eq:eps3mode0}, it is necessary for $\Pi_0$\eqref{eq:eps3mode0} to vanish. As before $\Pi_0M(0)=0$, so what we're left with is
\ba\label{eq:Pi0eps3mode0}
	-i\Pi_0 S_k(0,0)(B+\Psi_0)_{\Hx}+(d_*+\delta)B_{\Hx}+\frac{1}{4}\d_{\Hx}|A|^2\Pi_0\cQ(1,-1)(r,\bar{r})=0.
\ea
From our example in Section \ref{sec:MSE_Example}, we know that the left hand side of \eqref{eq:Pi0eps3mode0} doesn't always automatically vanish. It was shown in \cite{S} that it does vanish identically for a model $O(2)$-invariant Swift-Hohenberg equation. We now show that \eqref{eq:Pi0eps3mode0} vanishes for all $O(2)$ invariant systems.
\subsubsection{The $O(2)$-invariant case}
\begin{theorem}\label{thm:O2compat}
	If $L(\mu)$ and $\cN(U,\mu)$ are $O(2)$-invariant, then $-i\Pi_0S_k(0,0)\Pi_0+(d_*+\delta)\Pi_0$ and $i\Pi_0S_k(0,0)N_0(I-\Pi_0)\cQ(1,-1)(r,\bar{r})+ \Pi_0\cQ(1,-1)(r,\bar{r})$ both vanish.
\end{theorem}
\begin{proof}
		We start by showing that $S_k(0,0)$ and $d_*+\delta$ vanish separately. If $L(\mu)=\sum_{j=0}^m\cL_j(\mu)\d_x^j$, then a short computation shows that $S_k(0,0)=i\cL_1(0)$. Hence if $L(\mu)$ commutes with reflections, then it only has even derivatives and hence $\cL_1(\mu)\equiv 0$ for all $\mu$ establishing the first claim. For $d_*+\delta$, we recall \eqref{eq:delta} which identifies $d_*+\delta$ as $-\Im\tl_k(k_*,0)$. But $O(2)$-invariance of $L$ forces $\tl(k,\mu)$ to be real in a neighborhood of $(k_*,0)$, proving the second claim.\\
		
		We're left with showing that $\Pi_0\cQ(1,-1)(r,\bar{r})$ vanishes. First we note that $\tl(k,\mu)$ real near $(k_*,0)$ implies that $r\in\RR^n$ is a real vector because $r$ is an eigenvector of a real matrix associated to a real eigenvalue. We let $\cR U(x)=U(-x)$ be the reflection map. Applying $\cR$ to $\Pi_0\cN$, we see that
		\begin{equation}
			\cR\Pi_0\cN(U,\mu)=\cR \d_x \Pi_0 \tilde{\cN}(U,\mu)=-\d_x\cR\Pi_0\tilde{\cN}(U,\mu).
		\end{equation}
		$O(2)$-invariance of $\cN$ and quasilinearity of $\tilde{\cN}$ then imply that
		\begin{equation}
			\cR\Pi_0\tilde{\cN}(U,\mu)=-\Pi_0\tilde{\cN}(\cR U,\mu).
		\end{equation}
		To see this, note that $\cR\Pi_0\tilde{\cN}(U,\mu)$ and $-\Pi_0\tilde{\cN}(\cR U,\mu)$ have the same derivative and thus they differ by a constant depending on $(U,\mu)$. Moreover, the value of either of them at a point $x$ only depends on $U$ and its derivatives at the point $-x$ by quasilinearity and locality; hence their difference is also quasilinear and local. It remains to show that a local quasilinear map $\cN$ which takes constant values for all $U$ is constant independently of $U$. This follows from locality because if $U_1$, $U_2$ are any two maps and $\psi$ a $C^\infty_c$ bump function taking value 1 on a small enough ball with support inside a sufficiently small ball then $\cN(\psi U_1+(1-\psi)U_2)$ takes value $\cN(U_1)$ on the set where $\psi=1$ by locality; but outside the support of $\psi$ it takes value $\cN(U_2)$. Testing $U=0$ shows that the universal constant as in the previous sentence is 0.\\
		
		This anti-commutation with $\cR$ is stable under differentiation in the sense that
		\begin{equation}
			\cR\Pi_0 D_U^2\tilde{\cN}(U,\mu)(V,W)=-\Pi_0D_U^2\tilde{\cN}(\cR U,\mu)(\cR V,\cR W).
		\end{equation}
		In particular, we see that
		\begin{equation}
			\cR\Pi_0 D_U^2\tilde{\cN}(0,0)(V,W)=-\Pi_0D_U^2\tilde{\cN}(0,0)(\cR V,\cR W).
		\end{equation}
		Specializing to $V=W=U_A$, where $U_A=\frac{1}{2}Ae^{i\xi}r+c.c.$, we see that $\cR U_A=U_{\bar{A}}$ because
		\begin{equation}
			\cR U_A(\xi)=\frac{1}{2}Ae^{-i\xi}r+\frac{1}{2}\bar{A}e^{i\xi}\bar{r}=\frac{1}{2}(Ae^{-i\xi}+\bar{A}e^{i\xi})r=U_{\bar{A}}(\xi).
		\end{equation}
		To finish the proof, it is a straightforward calculation to check that the Fourier means of $\Pi_0D_U^2\tilde{\cN}(0,0)(U_A,U_A)$ and $\Pi_0 D_U^2\tilde{\cN}(0,0)(U_{\bar{A}},U_{\bar{A}})$ are both given by $\frac{1}{2}|A|^2\Pi_0\cQ(1,-1)(r,r)$, and the mean vanishes identically as it also equal to minus itself.
\end{proof}
\begin{remark}
	In the nonlocal case, Theorem \ref{thm:O2compat} is quite a bit more delicate. In order to get a splitting $\sN(U,\mu)=M(\d_x)\tilde{\sN}(U,\mu)$ where $M(\d_x)$ is as before, one needs to make an assumption of the form
	\begin{equation*}
		\limsup_{k\to 0}|k|^{-1}|\Pi_0\widehat{\sN(U,\mu)}(k)|<\infty,
	\end{equation*}
	for all $U,\mu$. Under this stronger form of mean 0, we can make sense of $\Pi_0\tilde{\sN}(U,\mu):=\text{``}\d_x^{-1}\text{''}\Pi_0\sN(U,\mu)$ as a function whose Fourier transform is given by
	\begin{equation*}
		\Pi_0\widehat{\tilde{\sN}(U,\mu)}(k)=\frac{1}{ik}\Pi_0\widehat{\sN(U,\mu)}(k),
	\end{equation*}
	and extended continuously to $k=0$ whenever possible. Hence there is a $\tilde{\sN}$ satisfying $M(\d_x)\tilde{\sN}=\sN$. Crucially, this construction of $\tilde{\sN}$ inherits the translation invariance of $\sN$. This is of course not the only such choice of $\tilde{\sN}$ that will work, however, we restrict attention to the family of $\tilde{\sN}$ that differ from the example $\tilde{\sN}$ above by a $(U,\mu)$-dependent constant which commutes with translations in the sense that $c(U,\mu)=c(\tau_h U,\mu)$ for all translations $\tau_h$. To recover some form of uniqueness, note that a translation-invariant bilinear form $\sQ$ with the property that $\sQ(U,V)$ is constant for all $U,V$ has Fourier support on the line $\{(-\nu,\nu):\nu\in\RR\}$ and so there is at most one $\tilde{\sN}$ for which the bilinear form associated to $D_U^2\tilde{\sN}(0,\mu)$ is a smooth function on $\RR^2$. Asides from some possible smoothness issues, we can choose this representative $\tilde{\sN}$ and check that $\cR\tilde{\sN}(U,\mu)=-\tilde{\sN}(\cR U,\mu)$ holds for all $U,\mu$ because they both have smooth bilinear multipliers; and from here the proof works the same.
	\end{remark}
	As a secondary remark, it is possible for the coefficients of $B_{\Hx}$ and $|A|^2_{\Hx}$ in \eqref{eq:Pi0eps3mode0} to vanish separately in the general $SO(2)$ case, but that is very unlikely to happen. Roughly speaking, the coefficient of $B_{\Hx}$ vanishes when $B$ is traveling at the same speed as $A$. The coefficient of $|A|^2_{\Hx}$ is given by
	\begin{equation*}
	i\Pi_0S_k(0,0)N_0(I-\Pi_0)\cQ(1,-1)(r,\bar{r})+ \Pi_0\cQ(1,-1)(r,\bar{r}).
	\end{equation*}
	Hence the coefficient of $|A|^2_{\Hx}$ is related to how the nonlinearity $\cN$ ``mixes'' the non-conserved and conserved directions with each other and so it vanishes when $\cN$ doesn't ``mix'' them in a suitable sense.\\
	
Continuing with the $O(2)$-invariant case, we look the coefficient of $\e^3e^{i\xi}$. It is a straightforward modification of the argument in Section \ref{sec:MSE_Example} and \cite{WZ1} to show that the solvability condition analogous to \eqref{eq:exampleeps3mode1} gives rise to a (cGL) equation
\begin{equation}
	A_{\Ht}=-\frac{1}{2}\tl_{kk}(k_*,0)A_{\Hx\Hx}+\tl_\mu(k_*,0)A+\g|A|^2A+AV_1\cdot B,
\end{equation}
where $\g\in\CC$ and $V_1\in\CC^n$ are known.\\

Finally, we come to the order $\e^4$. In particular, we focus on the coefficient of $\e^4e^{0i\xi}$ where we get
\ba\label{eq:eps4mode0}
	(B+\Psi_0)_{\Ht}&=S(0,0)\Psi_5-iS_k(0,0)\Psi_{4,\Hx}-\frac{1}{2}S_{kk}(0,0)(B+\Psi_0)_{\Hx\Hx}\\
	&\quad+(d_*+\delta)(\Psi_4)_{\Hx}+S_\mu(0,0)(B+\Psi_0)\\
	&\quad+\frac{1}{4}M(0)\cQ(1,-1)(Ar,\bar{\Psi}_3)+\frac{1}{4}M(0)\cQ(-1,1)(\bar{A}\bar{r},\Psi_3)+\frac{1}{4}M(0)\cQ(2,-2)(\Psi_2,\bar{\Psi}_2)\\
	&\quad+\frac{1}{2}M(0)\cQ(0,0)(B+\Psi_0,B+\Psi_0)+\frac{1}{4}(\Pi_0\d_{\Hx})\cQ(1,-1)(Ar,\bar{\Psi}_1)+\frac{1}{4}(\Pi_0\d_{\Hx})\cQ(-1,1)(\bar{A}\bar{r},\Psi_1)\\
	&\quad+\frac{1}{4}M(0)\cC(1,-1,0)(Ar,\overline{Ar},B+\Psi_0)+\frac{1}{16}M(0)\cC(1,1,-2)(Ar,Ar,\overline{\Psi}_2)\\
	&\quad+\frac{1}{4}\d_{\Hx}\Pi_0\cQ_\nu(1,-1)(A_{\Hx}r,\overline{Ar})+\frac{1}{4}\d_{\Hx}\Pi_0\cQ_\nu(1,-1)(Ar,\overline{A_{\Hx}r})\\
	&\quad+\frac{1}{16}M(0)\cC(-1,-1,2)(\overline{Ar},\overline{Ar},\Psi_2)+\frac{1}{4} M(0)\cQ_\mu(1,-1)(Ar,\overline{Ar}),
\ea
where $\cQ_\mu$ is the multiplier associated to $\d_\mu D_U^2\tilde{\cN}(0,0)$. In order to be able to solve for $(I-\Pi_0)\Psi_5$, it necessary and sufficient for $\Pi_0$\eqref{eq:eps4mode0} to vanish. Applying $\Pi_0$ and recalling that $\Pi_0 M(0)=0$, we get
\ba\label{eq:Pi0eps4mode0}
	B_{\Ht}&=-i\Pi_0 S_k(0,0)\Psi_{4,\Hx}-\frac{1}{2}S_{kk}(0,0)(B+\Psi_0)_{\Hx\Hx}\\
	&\quad +(d_*+\delta)\Pi_0\Psi_{4,\Hx}+\Pi_0 S_\mu(0,0)(B+\Psi_0)\\
	&\quad +\frac{1}{4}\Pi_0\d_{\Hx}\cQ(1,-1)(Ar,\overline{\Psi_1})+\frac{1}{4}\Pi_0\d_{\Hx}\cQ(1,-1)(\overline{Ar},\Psi_1)\\
	&\quad +\frac{1}{4}\d_{\Hx}\Pi_0\cQ_\nu(1,-1)(A_{\Hx}r,\overline{Ar})+\frac{1}{4}\d_{\Hx}\Pi_0\cQ_\nu(1,-1)(Ar,\overline{A_{\Hx}r}).
\ea
We let $\cA:=\ell \Psi_1$ and $\cB:=\Pi_0\Psi_4$ be as of yet undetermined amplitudes. The terms in \eqref{eq:Pi0eps4mode0} that depend on $\cA$, $\cB$ are
\ba\label{eq:cAcBterms}
	&-i\Pi_0 S_k(0,0)\cB_{\Hx}+(d_*+\delta)\cB_{\Hx}\\
	&\quad+\d_{\Hx}(A\bar{\cA}+\bar{A}\cA )\Big( \frac{1}{4}\Pi_0\cQ(1,-1)(r,\bar{r})+\frac{1}{8}i\Pi_0S_k(0,0)(N_0\cQ(1,-1)(r,\bar{r}))\Big).
\ea
We see that \eqref{eq:cAcBterms} is the linearization of \eqref{eq:eps3mode0} about the ``solution'' $(A,B)$, which can be seen by plugging in the formula for $\Psi_0$ into \eqref{eq:eps3mode0}. However, we have the following important corollary of Theorem \ref{thm:O2compat} in the $O(2)$ invariant case.
\begin{corollary}
	If $L(\mu)$ and $\cN(U,\mu)$ are $O(2)$ invariant, then \eqref{eq:Pi0eps4mode0} is independent of the choice of $\cA$, $\cB$.
\end{corollary}
\begin{proof}
	The desired statement is equivalent to \eqref{eq:cAcBterms} vanishing identically, which is an immediate consequence of $S_k(0,0)$, $d_*+\delta$ and $\Pi_0\cQ(1,-1)(r,\bar{r})$ all being 0.
\end{proof}
Focusing on the $O(2)$ invariant case for now, we can simplify \eqref{eq:Pi0eps4mode0} using Theorem \ref{thm:O2compat} and the assumption $\Pi_0S(0,\mu)=0$ for all $\mu$ to get
\ba\label{eq:Pi0eps4mode0-1}
	B_{\Ht}&=-\frac{1}{2}S_{kk}(0,0)(B+\Psi_0)_{\Hx\Hx}\\
	&\quad+\frac{1}{4}\d_{\Hx}\Pi_0\Big(\cQ(1,-1)(A,\overline{(I-\Pi_1)\Psi_1})+\cQ(-1,1)(\overline{A}r,(I-\Pi_1)\Psi_1)\Big)\\
	&\quad+\frac{1}{4}\d_{\Hx}\Pi_0\cQ_\nu(1,-1)(A_{\Hx}r,\overline{Ar})+\frac{1}{4}\d_{\Hx}\Pi_0\cQ_\nu(1,-1)(Ar,\overline{A_{\Hx}r}).
\ea
We recall the formula for $(I-\Pi_1)\Psi_1$ in \eqref{eq:examplePsi1}
\begin{equation*}
	(I-\Pi_1)\Psi_1=iA_{\Hx}[(I-\Pi_1)S(k_*,0)(I-\Pi_1)]^{-1}S_k(k_*,0)r=iA_{\Hx}N_1S_k(k_*,0)r.
\end{equation*}
Note that $iS_k(k_*,0)$ is $i$ times a real matrix by $O(2)$ invariance of $L(\mu)$, so that $(I-\Pi_1)\Psi_1=iA_{\Hx}\tilde{r}$ for some real $\tilde{r}$. So we see that the second line of \eqref{eq:Pi0eps4mode0-1} is given by
\ba\label{eq:Pi0eps4mode0-2}
	&\frac{1}{4}\d_{\Hx}\Pi_0\Big(\cQ(1,-1)(A,\overline{(I-\Pi_1)\Psi_1})+\cQ(-1,1)(\overline{A}r,(I-\Pi_1)\Psi_1)\Big)\\
	&\quad=\frac{1}{4}\d_{\Hx}\Pi_0\Big(A\bar{A}_{\Hx}\cQ(1,-1)(r,-i\tilde{r})+\bar{A}A_{\Hx}\cQ(-1,1)(r,i\tilde{r})\Big).
\ea
We note that because $D_U^2\tilde{\cN}(0,0)$ sends pairs of real functions to a real function, we can apply Fourier inversion to see that
\begin{equation}
	\overline{\cQ(\eta_1,\eta_2)}=\cQ(-\eta_1,-\eta_2).
\end{equation}
In addition, since $\Pi_0D_U^2\tilde{\cN}(0,0)$ also anti-commutes with reflection we can apply Fourier inversion again to see that
\begin{equation}
	-\Pi_0\cQ(\eta_1,\eta_2)=\Pi_0\cQ(-\eta_1,-\eta_2).
\end{equation}
Equivalently,
\begin{equation}
	\Pi_0\cQ(\eta_1,\eta_2)=-\overline{\Pi_0\cQ(\eta_1,\eta_2)}.
\end{equation}
Since $\Pi_0\cQ(\eta_1,\eta_2)$ is a bilinear form which is equal to minus its own conjugate, it is of the form $i\Pi_0\cQ_{\RR}(\eta_1,\eta_2)$ where $\cQ_{\RR}(\eta_1,\eta_2)$ is a real bilinear form, which smoothly depends on $\eta_1,\eta_2$. This implies that
\begin{equation}
	\Pi_0\cQ(1,-1)(r,-i\tilde{r})=i\Pi_0\cQ_{\RR}(1,-1)(r,-i\tilde{r})=\Pi_0\cQ_{\RR}(1,-1)(r,\tilde{r})\in\RR^n.
\end{equation}
\begin{corollary}\label{cor:cQnureal}
	$\Pi_0\cQ_\nu(\eta_1,\eta_2)=-i\frac{\d}{\d\eta_1}\Pi_0\cQ(\eta_1,\eta_2)$ is a real bilinear form.
\end{corollary}
\begin{proof}
	This follows directly from $\Pi_0\cQ(\eta_1,\eta_2)=i\Pi_0\cQ_{\RR}(\eta_1,\eta_2)$.
\end{proof}
What we've shown is that \eqref{eq:Pi0eps4mode0-2} is given by
\begin{equation}
	\frac{1}{4}\d_{\Hx}\Pi_0\Big(A\bar{A}_{\Hx}\cQ(1,-1)(r,-i\tilde{r})+\bar{A}A_{\Hx}\cQ(-1,1)(r,i\tilde{r})\Big)=\frac{1}{4}\d_{\Hx}(A\bar{A}_{\Hx}+\bar{A}A_{\Hx})\Pi_0\cQ_{\RR}(1,-1)(r,\tilde{r}).
\end{equation}
To finish, note that by the product rule $A\bar{A}_{\Hx}+\bar{A}A_{\Hx}=\d_{\Hx}|A|^2$.\\

In summary, we've shown that if $L(\mu)$ and $\cN(U,\mu)$ are $O(2)$ invariant, then the solvability condition of $\e^4e^{0i\xi}$ is of the form predicted in \cite{MC}; namely
\begin{equation}\label{eq:O2Beqn}
	B_{\Ht}=-\frac{1}{2}\Pi_0S_{kk}(0,0)B_{\Hx\Hx}+|A|^2_{\Hx\Hx}V_0,
\end{equation}
where $V_0\in\RR^n$ is a known vector.
\begin{remark}
	There is an alternative argument avoiding the use of the symmetry arguments used above. This can be done by writing $\Pi_0D_U^2\tilde{\cN}(U,\mu)$ as a sum of bilinear forms of the form
	\begin{equation*}
		\Pi_0\cQ(\d_x^IU,\d_x^JU).
	\end{equation*}
	Where $I,J$ are integers. Then the anti-commutation with reflection implies that $I+J$ is odd, which implies that Fourier multiplier is $\Pi_0\cQ(\eta_1,\eta_2)=i^{I+J}\eta_1^I\eta_2^J\Pi_0\cQ=\pm i \eta_1^I\eta_2^J\Pi_0\cQ$. In this form it is clear why $\Pi_0\cQ(1,-1)(r,-i\tilde{r})$ is real whenever $r,\tilde{r}$ are real vectors.
\end{remark}
We have the following main result in the $O(2)$ symmetric case.
\begin{theorem}\label{thm:O2multilin}
	If $L(\mu)$, $\cN(U,\mu)$ are $O(2)$-invariant with $L(\mu)$ a differential operator and $\cN(U,\mu)=M(\d_x)\tilde{\cN}(U,\mu)$ for some quasilinear $\tilde{\cN}(U,\mu)$ satisfying $\tilde{\cN}(0,\mu)=0$, then approximate solutions of the form \eqref{eq:ansatz} satisfy
	\ba\label{eq:O2amplitude}
	A_{\Ht}&=-\frac{1}{2}\tl_{kk}(k_*,0)A_{\Hx\Hx}+\tl_\mu(k_*,0)A+\g|A|^2A+A V_1\cdot B,\\
	B_{\Ht}&=-\frac{1}{2}\Pi_0S_{kk}(k_*,0)B_{\Hx\Hx}+(|A|^2)_{\Hx\Hx}V_0,
	\ea
	where $V_1$ is defined by
	\begin{equation}\label{eq:defV1}
		V_1\cdot W=\ell M(ik_*)\cQ(1,0)(r,W)
	\end{equation}
	$\g$ is defined by
	\ba\label{eq:gammadef}
	\g&:=\ell M(ik_*)\Big(-\frac{1}{4}\cQ(k_*,0;0)(r,N_0\Re\cQ(k_*,-k_*;0)(r,\bar{r}))\\
	&\quad-\frac{1}{4}\cQ(-k_*,2k_*;0)(\bar{r},S_2\cQ(k_*,k_*;0)(r,r))+\frac{1}{16}\cC(k_*,k_*,-k_*;0)(r,r,\bar{r})\Big),
	\ea
	and $V_0\in\Pi_0\RR^n$ is defined by
	\be
		V_0:=\frac{1}{8}\Pi_0S_{kk}(0,0)N_0(I-\Pi_0)\cQ(1,-1)(r,\bar{r})+\frac{1}{4}\Pi_0\cQ(1,-1)(r,iN_1S_k(k_*,0)r)+\frac{1}{4}\Pi_0\cQ_\nu(1,-1)(r,\bar{r}).
	\ee
\end{theorem}
Hence the amplitude system in the $O(2)$-case is a natural generalization of the system predicted by Matthews-Cox in \cite{MC} and confirmed for a model Swift-Hohenberg by \cite{S}. Moreover, the amplitude system in Theorem \ref{thm:O2multilin} always has real coefficients. The reality of the coefficients of $A$ in the Ginzburg-Landau equation follows in a similar manner to the case with no conservation laws, $V_1$ is real because $M(ik_*)\cQ(1,0)$ is the multiplier of $D_U^2\cN(0;0)$ at the pair of frequencies $(1,0)$ and since $\cN$ commutes with reflections it follows that $D_U^2\cN(0,\mu)$ also commutes with reflections and thus has a real-valued multiplier. That $V_0$ is real is unfortunately somewhat scattered in the above exposition, for convenience Lemma \ref{lem:genPsi0real} shows reality of the first term in $V_0$; the second term is discussed after \eqref{eq:Pi0eps4mode0-2}; and the third term is handled by Corollary \ref{cor:cQnureal}.
\subsubsection{The compatible $SO(2)$-case}

	One can consider the ``compatible'' $SO(2)$-invariant case in which the conclusion of Theorem \ref{thm:O2compat} holds, i.e. the coefficients of $B$ and $|A|^2$ in the solvability condition \eqref{eq:eps3mode0} vanish identically. In such cases, one expects the equation for $B$ to be modified into an equation of the form
	\begin{equation}
		B_{\Ht}=D B_{\Hx\Hx}+(|A|^2)_{\Hx\Hx}W_0+\d_{\Hx}\Re(A\bar{A}_{\Hx}W_1),
	\end{equation}
	for a known real matrix $D$ and known vectors $W_0\in\Pi_0\RR^n$ and $W_1\in\Pi_0\CC^n$. In general, this is because for the nonlinear terms one must expend a power of $\e$ to replace $M(k_*\d_\xi)$ with $\e\d_{\Hx}\Pi_0$. So for the multilinear forms, we need three more powers of $\e$ coming from $U^\e$. For the trilinear term, this means we must always choose $A$ or $\bar{A}$, but then the corresponding function is of the form $|A|^2Ae^{i\xi}$ or $A^3e^{3i\xi}$ or conjugates thereof; in particular there is no constant term. Hence only quadratic terms can contribute and so one input must be $A$ or $\bar{A}$; which means the other term must either be $A_{\Hx}$ or its conjugate or $\cA$ or its conjugate; and compatibility eliminates the second possibility.\\
	
	For a large class of examples of compatible $SO(2)$-systems, we let $L(k,\mu)$ and $\cN$ be $O(2)$-invariant undergoing a Turing bifurcation, $\delta$ any real number of order $\sim\e$ and $f$ any quadratic order function of $U$. Then the following system undergoes a compatible $SO(2)$-invariant
	\be\label{eq:exampleSO2}
		\frac{\d U}{\d t}=L(\mu)U+\cN(U)+\delta\d_x f(U).
	\ee
	As we will see in Section \ref{sec:ExistenceLS}, one can construct unique traveling wave solutions $\tilde{u}_{\e,\kappa,\vec{\b}}^\delta$ to \eqref{eq:exampleSO2}. One may then ask whether they have a limit as $\delta\to0$, and if so, whether or not the limit is the unique traveling wave solution to $U_t=L(\mu)U+\cN(U)$. We will not address this question here however.\\

\subsubsection{The general $SO(2)$-case}
As we saw in the example model, the solvability condition \eqref{eq:eps3mode0} can be nontrivial without reflection symmetry. We also saw that in order for $B$ to remain a dynamical variable, as opposed to being determined by $A$, we needed to make the amplitude system singular. Here, we note that this is the general state of affairs for $SO(2)$, but not $O(2)$, invariant systems; that is either one determines $B$ in terms of $A$ by using \eqref{eq:eps3mode0} and a suitable condition related to \eqref{eq:eps4mode0} or one ends up with an amplitude system for $A$ and $B$ which is singular. We will discuss these ideas more in Section \ref{sec:OtherModels}.\\
	
\noindent \textbf{Summary}:\\

To briefly recapitulate the argument in this section, one takes the Ansatz in \eqref{eq:ansatz} and formally plugs it into the equation \eqref{eq:evoleqn}. By Taylor expanding the Ansatz and $L(\mu)$ and $\cN(u,\mu)$ with respect to $\e$, one obtains a system of equations up to a certain order in $\e$ that the coefficients of the Ansatz need to satisfy in order to be an approximate solution. Further splitting each order in $\e$ into the fast Fourier modes $e^{i\eta\xi}$, one can then solve for the parameters $d_*$ and $\delta$ as well as the amplitudes $A$ and $B$ by various compatibility conditions needed to carry out higher order expansions. To be specific, one solves for $d_*$ at order $\e$ fast mode 1, $\delta$ at order $\e^2$ fast mode 1, $A$ at order $\e^3$ fast mode 1, and $B$ at order $\e^3$ or $\e^4$ depending on whether or not the system is compatible and fast mode 0.

\section{Singular models}\label{sec:OtherModels}
Our main goal in this section is to discuss the different models we considered in Section \ref{sec:MSE} in more detail.\\

We may define two new variables as follows
\ba
\tilde{A}&:=A+\e\cA,\\
\tilde{B}&:=B+\e\cB,
\ea
where $\cA$ and $\cB$ are the first correctors in the Ansatz \eqref{eq:ansatz} to $A$ and $B$ respectively. We seek a solution valid to $\cO(\e^4)$, so we recall the equation for $A$, $B$, and for $\cA$ we claim that it has the following form
\ba
A_{\Ht}&=-\frac{1}{2}\tl_{kk}(k_*,0)A_{\Hx\Hx}+\tl_\mu(k_*,0)A+\g|A|^2A+AV_1\cdot B,\\
B_{\Ht}&=\e^{-1}\Big((-i\Pi_0 S_k(0,0)+d_*+\delta)B_{\Hx}+\d_{\Hx}|A|^2V_0 \Big)+(-i\Pi_0 S_k(0,0)+d_*+\delta)\cB_{\Hx}+\d_{\Hx}(A\bar{\cA}+\bar{A}\cA )V_0\\
&\quad+D B_{\Hx\Hx}+(|A|^2)_{\Hx\Hx}W_0+\d_{\Hx}\Re(A\bar{A}_{\Hx}W_1),\\
\cA_{\Ht}&=-\frac{1}{2}\tl_{kk}(k_*,0)\cA_{\Hx\Hx}+\tl_\mu(k_*,0)\cA+\g(2|A|^2\cA+A^2\overline{\cA})+\cA V_1\cdot B+AV_1\cdot\cB+F(A,B).
\ea
Proving the claimed form for the equation for $\cA$ is a straightforward modification of the proof of Theorem 5.4 in \cite{WZ1} and will thus be omitted. Note that $F$ is an in principle known smooth function of $A,B$ and their first three derivatives with respect to $\Hx$ which is a sum of local translation invariant multilinear forms. In particular, the third and first derivatives of $A$ are the only ones that appear in $F$ and $F$ linearly depends on $\d_{\Hx}^3A$. Adding $\e$ times the equation for $\cA$ to the equation for $A$ gives the following equation
\ba\label{eq:tildeAeqn1}
\tilde{A}_{\Ht}&=-\frac{1}{2}\tl_{kk}(k_*,0)\tilde{A}_{\Hx\Hx}+\tl_\mu(k_*,0)\tilde{A}+\g(|A|^2A+\e(2|A|^2\cA+A^2\overline{\cA}))\\
&\quad+AV_1\cdot B+\e\cA V_1\cdot B+\e AV_1\cdot\cB+ \e F(A,B).
\ea
Comparing with the following equation,
\be\label{eq:tildeAeqn2}
\tilde{A}_{\Ht}=-\frac{1}{2}\tl_{kk}(k_*,0)\tilde{A}_{\Hx\Hx}+\tl_\mu(k_*,0)\tilde{A}+\g|\tilde{A}|^2\tilde{A}+\tilde{A}V_1\cdot\tilde{B}+\e F(\tilde{A},\tilde{B}),
\ee
we see that \eqref{eq:tildeAeqn1} and \eqref{eq:tildeAeqn2} differ by $\cO(\e^2)$, which is harmless. We remark that one can get \eqref{eq:tildeAeqn2} by pulling the missing terms from higher order amplitude equations. We can perform a similar trick to the $B$ equation, to get the following system in $(\tilde{A},\tilde{B})$
\ba\label{eq:tildevarsystem}
\tilde{A}_{\Ht}&=-\frac{1}{2}\tl_{kk}(k_*,0)\tilde{A}_{\Hx\Hx}+\tl_\mu(k_*,0)\tilde{A}+\g|\tilde{A}|^2\tilde{A}+\tilde{A}V_1\cdot \tilde{B}+\e F(\tilde{A},\tilde{B}),\\
\tilde{B}_{\Ht}&=\e^{-1}\Big((-i\Pi_0 S_k(0,0)+d_*+\delta)\tilde{B}_{\Hx}+\d_{\Hx}|\tilde{A}|^2V_0 \Big)+D \tilde{B}_{\Hx\Hx}+(|\tilde{A}|^2)_{\Hx\Hx}W_0+\d_{\Hx}\Re(\tilde{A}\bar{\tilde{A}}_{\Hx}W_1).
\ea
We note that this system is invariant under the transformations $\tilde{A}\to\tilde{A}e^{i\chi}$ and $\tilde{B}\to\tilde{B}$ for all real $\chi$, which is a secondary manifestation of translation invariance from the original equation. In particular, this symmetry is an expression of translation invariance in $\xi$. Another important remark is that \eqref{eq:tildevarsystem} has somewhat different well-posedness theory as the original singular system in $(A,B)$ as it is still singular via the dependence on the third derivative of $\tilde{A}$.\\

To linearize about the solution $\tilde{A}(\Hx,\Ht)=(A_0+\e A_1)e^{i (\kappa \Hx-\omega \Ht)}$ and $\tilde{B}=B_0+\e B_1$ of \eqref{eq:tildevarsystem}, we ``preprocess'' the perturbations as $\tilde{A}\to (A_0+\e A_1+(1+\e)U(\Hx,\Ht))e^{i(\kappa\Hx-\omega\Ht)}$ and $\tilde{B}\to B_0+\e B_1+(1+\e)V(\Hx,\Ht)$. A computation similar to the one in \ref{subsec:linearizedcGL} produces the linearized system
\ba
-i\omega U+U_{\Ht}&=aA_{xx}+2i\kappa a U_{\Hx}-\kappa^2aU+
(b+ V_1\cdot B_0)U+\g A_0^2U+\g A_0^2(U+\overline{U})
+ A_0V_1\cdot V+\cO(\e),
\\
V_{\Ht}&=
D V_{\Hx\Hx} + \eps^{-1}(1+\eps)( (-i\Pi_0 S_k(0,0)+d_*+\delta) V + (A_0(U + \bar{U})+\eps A_1(U+\bar{U})) V_0)_{\Hx}+\\
&\quad+W_0(A_0^2 + A_0(U+\bar {U}) )_{\Hx\Hx}+\d_{\Hx}\Re\Big[\Big(-i\kappa A_0(U+\bar{U})+A_0\bar{U}_{\Hx}\Big)W_1 \Big]   +\cO(\e),
\ea
where $a=-\frac{1}{2}\tl_{kk}(k_*,0)$ and $b=\tl_\mu(k_*,0)$. While we will not show the all the details here, we will see that under the identifications $A_0\leftrightarrow \a$ and $A_1\leftrightarrow\frac{\d\a}{\d\e}$ and assuming $U,V$ are exponentials, we will successfully match the reduced equation in Appendix \ref{app:example} for the example and Section \ref{sec:StabilityLSE} in the general case. As we will see, this ``ghost-free'' system matches the spectrum predicted by Lyapunov-Schmidt after dividing the eigenvalues by $\e^2$ on an interval of $|\s|\lesssim\e^2$. The main downside of the model for the $\tilde{A},\tilde{B}$ variables is that it could potentially be a delicate process to recover the leading order amplitudes $A$ and $B$.\\

Another alternative model, which is only guaranteed to be valid out to $\cO(\e^3)$, is to assume that $-i\Pi_0 S_k(0,0)\Pi_0+(d_*+\delta)\Pi_0$ has full rank and then writing
\be\label{eq:epscubedclosure}
B(\Hx,\Ht)=B_0-\frac{1}{4}|A(\Hx,\Ht)|^2\Big(-i\Pi_0 S_k(0,0)\Pi_0+(d_*+\delta)\Pi_0 \Big)^{-1}\Pi_0\cQ(1,-1)(r,\bar{r}),
\ee
for some function $B_0(\Ht)\in\Pi_0\RR^n$ to be determined later. Plugging this into the complex Ginzburg-Landau equation gives an equation for $A$ of the form
\be\label{eq:epscubedcgl}
A_{\Ht}=-\frac{1}{2}\tl_{kk}(k_*,0)A_{\Hx\Hx}+\tl_\mu(k_*,0)A+\g|A|^2A+AV_1\cdot B_0+\delta_A|A|^2A,
\ee
where $\delta_A:=-\frac{1}{4}V_1\cdot\Big(-i\Pi_0 S_k(0,0)\Pi_0+(d_*+\delta)\Pi_0 \Big)^{-1}\Pi_0\cQ(1,-1)(r,\bar{r})\in\CC$ is a known constant. Effectively, what this has done is change the coefficients of the original Ginzburg-Landau equation. In analogy with Darcy's Law, we will call this model the Darcy model. It is unlikely this greatly simplified model can determine all of the stability criteria. That said, as we will see later, this model does carry important stability information. In the incompatible $SO(2)$-case, $-i\Pi_0 S_k(0,0)\Pi_0+(d_*+\delta)\Pi_0$ being invertible is open and dense in all matrices and hence it is a very mild assumption. More generally, this approach could partially extend to the case where $\Pi_0\cQ(1,-1)(r,\bar{r})$ is in the range of $-i\Pi_0 S_k(0,0)\Pi_0+(d_*+\delta)\Pi_0$. That said, if as in the compatible $SO(2)$-case when $-i\Pi_0 S_k(0,0)\Pi_0+(d_*+\delta)\Pi_0$ has a kernel, then the projection of $B_0$ onto that kernel is non-constant in $\Hx$ and then one needs information from $\cO(\e^4)$ to determine the relevant piece of $B_0$.\\

For yet another alternative model, as mentioned in Section \ref{sec:MSE_Example} there is the model obtained by allowing $B$ to depend on an intermediary time scale $\Tt=\e t$ and choosing $\delta$ so that $A$ is independent of $\Tt$. Written out fully, we have
\ba
A_{\Ht}&=-\frac{1}{2}\tl_{kk}(k_*,0)A_{\Hx\Hx}+\tl_\mu(k_*,0)A+\g|A|^2A+AV_1\cdot B,\\
B_{\Tt}&=\Big(-i\Pi_0 S_k(0,0)\Pi_0+(d_*+\delta)\Pi_0 \Big)B_{\Hx}+\frac{1}{4}\d_{\Hx}|A|^2\Pi_0\cQ(1,-1)(r,\bar{r}).
\ea
Like the previous model, this is also only valid to $\cO(\e^3)$. As the equation for $B$ in this model is a driven transport equation by Hypothesis 8 of \eqref{hyp:Lin}, we will call this model the hyperbolic model.\\

\subsection{Well-posedness of the amplitude system}\label{sub:wellposed}
In the classical Turing bifurcation, one must show that the amplitude equation is well-posed on time intervals of length $\Ht\lesssim 1$ in order to show the validity of the approximation. This has been done for Ginzburg-Landau in \cite{DGL} and for a special decoupled case of the truncated model in \cite{HSZ}. Here, the question of whether or not the singular equations are solvable on $\cO(1)$ time remains. We begin by studying the singular-parabolic model we called the truncated model
\ba\label{eq:truncatedsystem}
A_T&=a A_{XX}+b A+ c|A|^2A + d AB,\\
B_T &=  e_B B_{XX} + \eps^{-1}( f B + h|A|^2)_X + \d_X \Re(gA\overline{A_X}),
\ea
where $A\in \CC$, $B\in \R$ and $a,b,c,d,g$ are complex numbers with the appropriate signs on their real parts, 
i.e. $\Re(a),\Re(b)>0$ and $\Re(c)<0$, and $e_B>0$, $f,h$ are real. Note that we've changed notation for clarity.\\

By standard Picard iteration techniques, it is clear that there exists a unique solution to \eqref{eq:truncatedsystem} on a time interval of size $\e$. In this subsection, we will extend that time interval to size $\cO(1)$. First, we will establish this for so-called mild solutions. For the following theorem, we will only allow periodic boundary conditions or work on all of $\RR$. This is to avoid boundary effects from the singular advection term in \eqref{eq:truncatedsystem}.
\begin{theorem}\label{thm:truncatedwellposed}
	Suppose the initial data $(A_0,B_0)$ of \eqref{eq:truncatedsystem} is in $H^1\times H^1_{\RR}$, where $H^1_{\RR}$ denotes the subspace of real-valued functions. Then there exists constants $\rho>0$ and $T_0>0$, \textbf{independent of $\bm{\e}$}, so that for all $(A_0,B_0)$ with $||A_0||_{H^1}+||B_0||_{H^1}\leq\rho$, there exists a unique mild solution to
	\ba
	A_{T}&=a A_{XX}+b A+ c|A|^2A + d AB,\\
	B_{T} &=\eps^{-1}( f B + h|A|^2)_{X}+ e_B B_{XX} + g (|A|^2)_{XX},\\
	(A(X,0),B(X,0))&=(A_0(X),B_0(X)),
	\ea
	with $A\in C([0,T_0];H^1)$ and $B\in C([0,T_0];H^1_{\RR})$.
\end{theorem}
\begin{proof}
	Let $G(X,T)$ denote the fundamental solution to the linear equation
	\be\label{eq:fundamentalU}
	U_{X}=aU_{XX}+bU,
	\ee
	and $H(X,T)$ denote the fundamental solution to
	\be\label{eq:fundamentalV}
	V_{X}=e_B V_{XX}.
	\ee
	We note that the fundamental solution $\tilde{H}_\e(X,T)\in S'$ to the problem
	\be
	\tilde{V}_{T}=\e^{-1}f\tilde{V}_{X}+e_B\tilde{V}_{XX},
	\ee
	admits the ``factorization''
	\be\label{eq:keyfactor}
	\tilde{H}_\e(X,T)=\delta_{-\frac{fT}{\e}}(X)\star H(X,T),
	\ee
	for $\delta_c(X)$ the Dirac mass centered at $c$ and $\star$ denotes convolution. That is, the solution operator $S_B(T)$ admits the representation
	\be
	S_B(T)f(X)=(\tilde{H}_\e(\cdot,T)\star f)(X)=(H\star f)\Big(X+\frac{fT}{\e}\Big).
	\ee
	We remark that this is the first place where we use the domain being $\RR$ or the circle, as proving this factorization involves the Fourier transform.\\
	
	Define a mild solution of \eqref{eq:truncatedsystem} to be a fixed point in an appropriate space of the map
	\be\label{eq:mildsoln}
	\cT\bp A(X,T)\\ B(X,T)\ep=
	\bp G(\cdot,T)\star A_0(X)+\int_0^{T}G(\cdot,T-S)\star(c|A|^2A+dAB)(S)dS\\
	\tilde{H}_{\e}(\cdot,T)\star B_0(X)+\int_0^{T}\tilde{H}_{\e}(\cdot,T-S)\star(\e^{-1}h|A|^2_{X}+g|A|^2_{XX} )(S)dS\ep.
	\ee
	The usual approach would be to show that $\cT$ is a contraction for sufficiently small initial data on a time interval $[0,T_0]$ independently of $\e$. We will not do this directly, and instead find a separate map $\tilde{\cT}$ which has the same fixed points, on a time scale independent of $\e$, but with better regularity properties. Let $\cT_j(A,B)$ denote the components of the right hand side of \eqref{eq:mildsoln} and define a new map $\tilde{\cT}(A,B)$ as
	\be\label{eq:tildecontract}
	\tilde{\cT}(A,B):=\bp \cT_1(A,B)\\ \cT_2(\cT_1(A,B),B)\ep.
	\ee
	Note that a fixed point $(A,B)$ of $\tilde{\cT}$ satisfies $A=\cT_1(A,B)$ and hence the second component satisfies $B=\cT_2(\cT_1(A,B),B)=\cT_2(A,B)$, proving that $\cT$ and $\tilde{\cT}$ have the same fixed points. The second step in this process is to show that $\tilde{\cT}$ maps $C([0,T_0];H^1)\times C([0,T_0];H_{\RR}^1)$ back into itself. We let $\cX:=C([0,T_0];H^1)$ and $\cX_{\RR}:=C([0,T_0];H^1_{\RR})$ be equipped with the natural norm.
	\begin{obs}\label{obs:cT1}
		$\cT_1(A,B)$ can be equivalently defined as the unique strong solution to the \textit{linear} equation
		\be
		\begin{cases}
			U_{T}=aU_{XX}+bU+c|A|^2A+dAB,\\
			U(X,0)=A_0(X).
		\end{cases}
		\ee
		Similarly, $\cT_2(A,B)$ can be defined as the unique strong solution to the (singular) linear equation
		\be\label{eq:singularparabolic}
		\begin{cases}
			V_{T}=\e^{-1}fV_{X}+e_BV_{XX}+\e^{-1}h|A|^2_{X}+g|A|^2_{XX},\\
			V(X,0)=B_0(X).
		\end{cases}
		\ee
		In either case, $U$ and $V$ are either periodic or defined on the whole real line.
	\end{obs}
	We recall the standard parabolic regularity estimate \cite[Thm 5, pg384]{Ev} for
	\be\label{eq:goodparabolicV}
	\begin{cases}
		V_T=e_BV_{XX}+g,\\
		V(X,0)=f,
	\end{cases}
	\ee
	when $f\in H^1$ and $g\in C([0,T_0];L^2)$ of
	\be\label{eq:parabolicregularity}
	\ess\sup_{0\leq T\leq T_0}||V(T)||_{H^1}+||V||_{L^2([0,T_0];H^2)}+||V_T||_{L^2([0,T_0];L^2)}\leq C(||f||_{H^1}+||g||_{L^2([0,T_0];L^2)}),
	\ee
	where $C$ only depends on $e_B$ and $T_0$. Correspondingly, one has the bounds
	\be\label{eq:singularvariant}
	\ess\sup_{0\leq T\leq T_0}||V||_{H^1}+||V||_{L^2([0,T_0];H^2)}\leq C(||f||_{H^1}+||g||_{L^2([0,T_0];L^2)}),
	\ee
	with the same constant $C$ as in \eqref{eq:parabolicregularity} for $V$ satisfying \eqref{eq:singularparabolic} with forcing term $g$ and initial data $f$. This follows by making the coordinate change $X\to X-fT/\e$, which absorbs the $\e^{-1}f\d_XV$ into $\d_TV$. We remark that this is the second place where we use the domain being $\RR$ or the circle with periodic boundary conditions. Note that the estimate on $\d_{T}\tilde{H}_\e(\cdot,T)\star f$ is no longer independent of $\e$, which is only an issue if one tries to upgrade mild solutions into strong solutions. Moreover, it also standard that the map $f\to \tilde{H}_\e(\cdot,T)\star f$ is a continuous function of $T$. A similar result to \eqref{eq:parabolicregularity} also holds for the equation
	\be
	\begin{cases}
		U_T=aU_{XX}+bU+g,\\
		U(X,0)=0.
	\end{cases}
	\ee From this, we conclude that the linear part of $\tilde{\cT}$ satisfies
	\be
	||(G(\cdot,T)\star A_0,\tilde{H}_\e(\cdot,T)\star B_0)||_{\cX\times\cX_{\RR}}\leq C(a,b,e_B,T_0)||(A_0,B_0)||_{H^1\times H^1_{\RR}}.
	\ee
	For the equation for $A$, we recall that $H^1$ is an algebra under pointwise multiplication, and so we have that
	\be
	(A,B)\to c|A|^2A+dAB,
	\ee
	is a continuous map from $\cX\times\cX_{\RR}\to \cX$. So applying \eqref{eq:parabolicregularity}, we see that
	\ba
	&\sup_{0\leq T\leq T_0}||\int_0^{T} G(\cdot,T-S)\star (c|A|^2A+dAB)(S)dS||_{H^1}\\
	&\quad \leq C(a,b,c,d,T_0)\sqrt{T_0}(||\ |A|^2A\ ||_{\cX} +||AB||_{\cX})\\
	&\quad \leq C(a,b,c,d,T_0)\sqrt{T_0}( ||A||_{\cX}^3+||A||_{\cX}||B||_{\cX_{\RR}}),
	\ea
	where we've also applied the trivial bound $||f||_{L^2([0,T_0];H^s)}\leq \sqrt{T_0}||f||_{L^\infty([0,T_0];H^s)}$.
	Our bounds on $\cT_1(A,B)$ read
	\ba
	&||\cT_1(A,B)||_{\cX}+||\cT_1(A,B)||_{L^2([0,T_0];H^2)}+||\d_{T}\cT_1(A,B)||_{L^2([0,T_0];L^2)}\\
	&\quad \leq C(a,b,c,d,T_0)\Big(||A_0||_{H^1}+\sqrt{T_0}\Big(||A||_{\cX}^3+||A||_{\cX}||B||_{\cX_{\RR}} \Big) \Big).
	\ea
	To show $\cT_2(\cT_1(A,B),B)$ maps $X$ into itself, it suffices to show by the observation and standard parabolic regularity results that $\d_X|\cT_1(A,B)|^2,\d_{XX}|\cT_1(A,B)|^2\in L^2([0,T_0];L^2)$. Fix a $T$ so that $\cT_1(A,B)(T)\in H^2$, and observe that by the product rule we have
	\ba
	\d_{XX}|\cT_1(A,B)|^2&=(\d_{XX}\cT_1(A,B))\overline{\cT_1(A,B)}+\cT_1(A,B)(\d_{XX}\overline{\cT_1(A,B)})\\
	&\quad +2(\d_{X}\cT_1(A,B))(\d_{X}\overline{\cT_1(A,B)}).
	\ea
	We can then bound the $L^2$ norm $\d_{XX}|\cT_1(A,B)|^2$ by applying the H\"older and Sobolev inequalities as
	\ba
	&	||\d_{XX}|\cT_1(A,B)|^2(T)||_{L^2}\\
	&\quad \leq2\Big(||\cT_1(A,B)(T)||_{L^\infty}||\d_{XX}\cT_1(A,B)(T)||_{L^2}+||\d_{X}\cT_1(A,B)(T)||_{L^\infty}||\d_{X}\cT(A,B)(T)||_{L^2} \Big)\\
	&\quad \leq C\Big(||\cT_1(A,B)(T)||_{H^1}||\d_{XX}\cT_1(A,B)(T)||_{L^2}+||\d_{X}\cT_1(A,B)(T)||_{H^1}||\d_{X}\cT(A,B)(T)||_{L^2} \Big).
	\ea
	We can further bound this by applying the trivial inequalities $||\d_{X}^jf||_{H^k}\leq ||f||_{H^{j+k}}$ and $||f(T)||_{H^1}\leq ||f||_{\cX}$ to get
	\be\label{eq:secondderivboundcT}
	||\d_{XX}|\cT_1(A,B)|^2(T)||_{L^2}\leq C\Big( ||\cT_1(A,B)||_{\cX}||\cT_1(A,B)(T)||_{H^2}\Big),
	\ee
	as $\cT_1(A,B)$ is in $L^2([0,T_0];H^2)$, the desired result follows. Showing $\d_{X}|\cT_1(A,B)|^2\in L^2([0,T_0];L^2)$ is entirely similar and will be omitted. In particular, we note that by the parabolic regularity estimates above 
	\be
	||\d_{XX}|\cT_1(A,B)|^2 ||_{L^2([0,T_0];L^2)}\leq C(a,b,c,d,T_0)\Big(||A_0||_{H^1}+\sqrt{T_0}\Big(||A||_{\cX}^3+||A||_{\cX}||B||_{\cX_{\RR}} \Big) \Big)^2.
	\ee
	\\
	
	So far, we've shown that $\tilde{\cT}:\cX\times \cX_{\RR}\to \cX\times \cX_{\RR}$. The next step is to show that it has a Lipschitz bound independent of $\e$. We begin with $\cT_1(A,B)$ and apply the observation again to see that $\cT_1(A,B)-\cT_1(A',B')$ is the unique solution to the equation
	\be
	\begin{cases}
		W_{T}=aW_{XX}+bW+c\Big(|A|^2A-|A'|^2A' \Big)+d\Big(AB-A'B'\Big),\\ 
		W(X,0)=0.
	\end{cases}	
	\ee
	Applying parabolic regularity estimates again, we get that
	\be
	||W||_{\cX}\leq C(a,b,c,d,T_0)\sqrt{T_0}\Big(|| |A|^2A-|A'|^2A' ||_{\cX}+|| AB-A'B'||_{\cX} \Big).
	\ee
	Fix $\rho>0$, then for $||A||_{\cX}+||B||_{\cX_{\RR}}, ||A'||_{\cX}+||B'||_{\cX_{\RR}}\leq\rho$, the above is bounded by
	\be\label{eq:LipschitzBound}
	||W||_{\cX}\leq C(a,b,c,d,T_0)\max\{\rho,\rho^2 \}\sqrt{T_0}(||A-A'||_{\cX}+||B-B'||_{\cX_{\RR}}).
	\ee
	This shows that $\cT_1(A,B)$ is a locally Lipschitz function of $(A,B)$ with Lipschitz bounds independent of $\e$. We now turn to the more difficult task of showing that $\cT_2(\cT_1(A,B),B)$ is locally Lipschitz with Lipschitz bounds independent of $\e$. We will do this with the following sleight of hand.
	\begin{obs}\label{obs:chainrule}
		If $F\Big(X+\frac{f(T-S)}{\e},S \Big)$ is smooth enough, then
		\be
		\frac{f}{\e}F_{X}=-\frac{d}{dS}F+\frac{\d}{\d S}F,
		\ee
		where the second partial only acts on the second input of $F$.
	\end{obs}
	The reason for swapping $\cT$ with $\tilde{\cT}$ is now apparent, as $A\in \cX$ is not regular enough to apply this trick; but $\cT_1(A,B)$ is smooth enough for this to work. Indeed, we have the following bound for all $T$ with $\d_T\cT_1(A,B)\in L^2$
	\ba
	||\frac{\d}{\d T}|\cT_1(A,B)|^2(T)||_{L^2}&\leq 2||\cT_1(A,B)(T)\d_T\overline{\cT_1(A,B)}(T)||_{L^2}\\
	&\quad\leq C||\cT_1(A,B)||_{\cX}||\d_T\cT_1(A,B)(T)||_{L^2}.
	\ea
	Hence, we have $\d_T|\cT_1(A,B)|^2\in L^2([0,T_0];L^2)$. Define $\tilde{\cT}_2(A,B)=\cT_2(\cT_1(A,B),B)$ and consider $\tilde{\cT}_2(A,B)-\tilde{\cT}_2(A',B')$. Expanding this out, we get
	\ba
	\tilde{\cT}_2(A,B)-\tilde{\cT}_2(A',B')(T)&=\int_0^Th\tilde{H}_{\e}(\cdot,T-S)\star\e^{-1}\Big(|\cT_1(A,B)|^2-\cT1(A',B')|^2\Big)_{X}dS\\
	&\quad+\int_0^Tg\tilde{H}_{\e}(\cdot,T-S)\star \Big(|\cT_1(A,B)|^2-|\cT_1(A',B')|^2\Big)_{XX}(S)dS=:\\
	&\quad =:I_1(T)+I_2(T).
	\ea
	Showing that $||I_2||_{\cX_{\RR}}\leq C(a,b,c,d,e_B,f,g,T_0)\max\{\rho,\rho^2 \}\sqrt{T_0}\Big( ||A-A'||_{\cX}+||B-B'||_{\cX_{\RR}}\Big)$ is similar to the argument for $\cT_1(A,B)$ and will be omitted. We finally come to $I_1(T)$. For each $T$, let $F(X,S;T)$ be defined as
	\be
	F(X,S;T):=H(\cdot,T-S)\star\Big(|\cT_1(A,B)|^2-|\cT_1(A',B')|^2\Big)(S).
	\ee
	We will postpone showing $F$ is continuous into $H^1$ as a function of $S$ for the time being, see Lemma \ref{lem:Fcont} for details. Rewriting $I_1(T)$, we get by the factorization \eqref{eq:keyfactor} and the observation \eqref{obs:chainrule}
	\be
	I_1(T)=\int_0^T\frac{h}{\e}\d_XF(X+\frac{f(T-S)}{\e},S;T)dS=-\frac{h}{f}\int_0^T\Big(\frac{d}{dS}-\frac{\d}{\d S}\Big)F(X+\frac{f(T-S)}{\e},S;T)dS.
	\ee
	By the fundamental theorem of calculus, this reduces to
	\be
	I_1(T)=-\frac{h}{f}\Big(F(X,T;T)-F(X+\frac{fT}{\e},0;T) \Big)+\frac{h}{f}\int_0^T\frac{\d}{\d S}F(X+\frac{f(T-S)}{\e},S;T)dS.
	\ee
	Since $F$ is in $\cX_{\RR}$, the first two terms are harmless. Applying the product rule to the integrand, we get
	\ba
	\d_S F(X,S;T)&=H_{S}(\cdot,T-S)\star\Big(|\cT_1(A,B)|^2-|\cT_1(A',B')|^2\Big)(S)\\
	&\quad+H(\cdot,T-S)\star \d_S\Big(|\cT_1(A,B)|^2-|\cT_1(A',B')|^2\Big)(S).
	\ea
	For the heat kernel, we have that $H_S=-H_T=-e_BH_{XX}$, so the above can be rewritten as
	\be
	\frac{\d}{\d S}F(X,S;T)=-e_BH_{XX}(\cdot,T-S)\star \Theta(S)+H(\cdot,T-S)\Theta_S(S),
	\ee
	where $\Theta$ is defined to be
	\be
	\Theta:=|\cT_1(A,B)|^2-|\cT_1(A',B')|^2.
	\ee
	Since the convolutions are with respect to $X$, we can move the $\d_{XX}$ onto $\Theta$ to get
	\be
	\frac{\d}{\d S}F(X,S;T)=H(\cdot,T-S)\star(\Theta_S-e_B\Theta_{XX}).
	\ee
	As argued earlier, $\Theta_S-e_B\Theta_{XX}\in L^2([0,T_0];L^2)$ with bounds independent of $\e$. The final observation to make is that
	\be
	\int_0^T\frac{\d}{\d S}F(X+\frac{f(T-S)}{\e},S;T)dS=\int_0^T\delta_{\frac{fT}{\e}}\star H(\cdot,T-S)\star(\Theta_S-e_B\Theta_{XX})dS,
	\ee
	is the unique solution to
	\be
	\begin{cases}
		W_T=\e^{-1}fW_X+e_BW_{XX}+(\Theta_T-e_B\Theta_{XX}),\\
		W(X,0)=0.
	\end{cases}
	\ee
	From here, the same arguments gives a Lipschitz bound of the form \eqref{eq:LipschitzBound}, with constants independent of $\e$. Choosing $\rho,T_0$ small enough makes $\tilde{\cT}$ a contraction on the ball $B_\rho(0)$ in $\cX\times\cX_{\RR}$ and Picard iteration completes the proof.
\end{proof}
\begin{lemma}\label{lem:Fcont}
	Suppose $f\in C([0,T];H^1)$ and $H(x,t)$ is the heat kernel. Then the function $F(S;T):=H(T-S)\star f(S)$ is in $C([0,T];H^1)$.
\end{lemma}
\begin{proof}
	Suppose $0\leq S_0<T$, then for there is some radius $r>0$ on which
	\be
	\sup_{|S-S_0|\leq r}H(0,T-S)<\infty.
	\ee
	Moreover, since $H(x,T-S)\to H(x,T-S_0)$ pointwise, by dominated convergence $\int H(x,T-S)dx\to \int H(x,T-S_0)dx$. However, $H\geq0$ so this further implies that $H(x,T-S)\to H(x,T-S_0)$ in $L^1(dx)$. So, we have that
	\ba
	H(T-S)\star f(S)-H(T-S_0)\star f(S_0)&=H(T-S)\star f(S)-H(T-S_0)\star f(S)\\
	&\quad+H(T-S_0)\star f(S)-H(T-S_0)\star f(S_0).
	\ea
	Taking $H^1$-norms and applying the triangle inequality, we get
	\ba
	||H(T-S)\star f(S)-H(T-S_0)\star f(S_0)||_{H^1}&\leq ||(H(T-S)-H(T-S_0)\star f(S)||_{H^1}\\
	&\quad+||H(T-S_0)\star (f(S)-f(S_0))||_{H^1}.
	\ea
	A straightforward consequence of Young's convolution inequality is that
	\be
	||f\star g||_{H^1}\sim ||f\star g||_{L^2}+||(\d_xf)\star g||_{L^2}\leq C||f||_{H^1}||g||_{L^1}.
	\ee
	Hence, we have that
	\ba
	||H(T-S)\star f(S)-H(T-S_0)\star f(S_0)||_{H^1}&\leq C||H(T-S)-H(T-S_0)||_{L^1}||f(S)||_{H^1}\\
	&\quad+C||H(T-S_0)||_{L^1}||f(S)-f(S_0)||_{H^1},
	\ea
	which goes to zero by the above argument for $H$ and by continuity for $f$. For the endpoint, we note that
	\be
	f(T)-H(T-S)\star f(S)=f(T)-H(T-S)\star f(T)+H(T-S)\star f(T)-H(T-S)\star f(S).
	\ee
	By the triangle and the convolution inequality above, we get
	\ba
	||f(T)-H(T-S)\star f(S)||_{H^1}&\leq ||f(T)-H(T-S)\star f(T)||_{H^1}\\
	&\quad+C||H(T-S)||_{L^1}||f(T)-f(S)||_{H^1},
	\ea
	which goes to zero since $H(T-S)\star f(T)\to f(T)$ in $H^1$ as $S\to0$ because $H$ is an approximate identity.
\end{proof}
\begin{remarks}
	Our assumption that the $B$ equation had $g|A|^2_{XX}$ in Theorem \ref{thm:truncatedwellposed} was for simplicity; the proof easily extends to the case that $g|A|^2_{XX}$ in the equation for $B$ is replaced with the more typical $\d_X \Re(gA\bar{A}_X)$.\\
	
	A more substantial remark is that this technique was designed for the hyperbolic-parabolic model given by
	\ba\label{eq:hypparmodel}
	A_T&=aA_{XX}+bA+c|A|^2A+dAB,\\
	B_T&=\e^{-1}(fB+h|A|^2)_X.
	\ea
	For this model, the trick in Observation \ref{obs:chainrule} works with significantly less difficulty. The only change is that one expects $B\in C([0,T_0];L^2_{\RR})$ as opposed to $B\in C([0,T_0];H^1_{\RR})$. This model is obtainable from the multiscale expansion by allowing $B$ to depend on an intermediary time scale $\e^{-1}T$ but not allowing $A$ to depend on that time scale. \\
	
	The final remark is that we did not see a way to generalize this argument to other boundary conditions of interest, such as Dirichlet or Neumann.
\end{remarks}
Our well-posedness result is local. For complex Ginzburg-Landau, it is known that on the circle with $C^2$ initial data there is a unique global strong solution \cite[Thm 4.1]{DGL}. We will not address the question of whether or not \eqref{eq:truncatedsystem} is globally well-posed. Uniform energy estimates are also possible, which we will demonstrate for the simpler model \eqref{eq:hypparmodel}. For simplicity, we set $b=d=f=h=1$ and scale $a$ so that $\Re(a)=1$. We work in the time scale $t=\e^{-1}T$, hence our goal becomes getting a energy estimate valid out to $t\sim \e^{-1}$.
\begin{theorem}\label{thm:energyestimate}
	Consider the rescaled version of \eqref{eq:hypparmodel} given by
	\ba
	A_t&= \eps( A + c |A|^2A + (1+i\a) A_{XX} + AB),\\
	B_t&= (B + |A|^2)_X.
	\ea
	Then one has the $H^1/L^2$ energy estimate
	\ba\label{combest}
	\partial_t \frac12(\|A\|_{H^1}^2 + \|\tilde B\|^2)&\leq -(\eps/4)\|A_x\|_{H^1}^2
	+C_2 \eps \Big(
	(1+\|A\|_{H^1}^6 + \|\tilde B\|^2)\|A\|_{H^1}^2\Big),
	\ea
	where $\tilde{B}:=B+|A|^2$.
\end{theorem}
\begin{proof}
	Let $\langle\cdot,\cdot\rangle$ denote the $L^2$ inner product, $||\cdot||$ the $L^2$ norm, $||\cdot||_{\infty}$ the $L^\infty$ norm and $||\cdot||_{H^1}$ the $H^1$ norm.\\
	The estimate on the $L^2$ norm $A$ reads
	\ba\label{eq:Aest1}
	\partial_t \frac12 \|A\|^2&= \Re \langle A, A_t\rangle  \\
	&=
	\eps \Re \langle A, A + \gamma |A|^2A + (1+ic) A_{xx} + AB \rangle\\
	&=
	\eps \big( \|A\|^2 + \Re \gamma \| |A|^2\|^2 -\eps \|A_x\|^2 +\langle |A|^2,B\rangle\big)\\
	&\leq \eps\|A\|^2(1+2\|A\|_\infty^2) - \eps \|A_x\|^2 + \eps \|B\|^2. 
	\ea
	Similarly, one can the estimate $||A_x||$ as
	\ba\label{eq:Aest2}
	\partial_t \frac12 \|A_x\|^2&= \Re \langle A_x, (A_x)_t\rangle  \\
	&= \eps \Re \langle A_x, A_x + \gamma (|A|^2A)_x + (1+ic) (A_x)_{xx} + (AB)_x \rangle\\
	&\leq \eps\|A_x\|^2(1+6\|A\|_\infty^2) - \frac12 \eps \|A_{xx}\|^2 +\frac12 \eps \|A\|_\infty^2 \|B\|^2. 
	\ea
	where we have used 
	$$
	|\langle A_x, (AB)_x\rangle|= - |\langle A_{xx}, AB\rangle|\leq \frac12\|A_{xx}\|^2
	+\frac12\|B\|^2\|A\|_\infty^2
	$$
	in the final inequality.
	
	Summing, and rearranging using the Sobelev embedding $\|A\|_\infty \leq \|A\|_{H^1}$, we obtain, finally
	the {\it $H^1$ energy estimate}
	\be\label{eq:finalA}
	\partial_t \frac12 \|A\|_{H^1}^2\leq
	-\frac12\eps \|A_x\|_{H^1}^2 + C\eps (1+\|A\|_{H^1}^2 + \|B\|^2)\|A\|_{H^1}^2.
	\ee
	
	We now turn our attention to the equation for $B$. For the $B$-estimates, we change to the moving coordinate $\tau= t+x$, in which \eqref{eq:hypparmodel} becomes
	\ba\label{eq:mfull}
	A_\tau&= -A_x+ \eps( A + \gamma |A|^2A + (1+ic) A_{xx} + AB),\\
	B_\tau&=  (|A|^2)_x= \Re (AA_x).
	\ea
	Substituting the first equation into the second gives then
	$$
	(B+ |A|^2)_\tau =\eps \Re A ( A + \gamma |A|^2A + (1+ic) A_{xx} + AB),
	$$
	or, defining $\tilde B:=B+|A|^2$,
	\be\label{eq:tBeq}
	\tilde B_\tau =\eps \Re A ( A + \gamma |A|^2A + (1+ic) A_{xx} + A(\tilde B-|A|^2)).
	\ee
	
	Thus, for some $C_1>0$ and any $C>0$, we have using translation invariance of Sobolev norms
	\ba\label{eq:Best1}
	\partial_t \frac12 \|\tilde B\|^2&=
	\partial_\tau \frac12 \|\tilde B\|^2\\
	&=
	\eps \langle \tilde B, \Re A ( A + \gamma |A|^2A + (1+ic) A_{xx} + A(\tilde B-|A|^2)) \rangle\\
	&\leq C\eps \|B\|^2 + \eps (C_1/C)\Big((1+ \|A\|^6_{\infty})\|A\|^2+ \|A\|_\infty \|\tilde B\|^2
	+  \|A_x\|_{H^1}^2\|A\|_\infty^2 \Big)\\
	&\leq C\eps \|B\|^2 + \eps (C_1/C)\Big((1+ \|A\|_{H^1}^6)\|A\|^2+ \|A\|_{H^1}\|\tilde B\|^2
	+  \|A_x\|_{H^1}^2\|A\|_{H^1}^2 \Big).
	\ea
	
	For any $M>0$,
	summing \eqref{eq:finalA} and \eqref{eq:Best1} with $C>0$ chosen sufficiently large, we have so long as
	$\|A\|_{H^1}\leq M$ the estimate
	\ba\label{eq:Bcombest}
	\partial_t \frac12(\|A\|_{H^1}^2 + \|\tilde B\|^2)&\leq -(\eps/4)\|A_x\|_{H^1}^2
	+C_2 \eps \Big(
	(1+\|A\|_{H^1}^6 + \|\tilde B\|^2)\|A\|_{H^1}^2\Big),
	\ea
	whence, setting $\mathcal{E}(t):= \frac12(\|A\|_{H^1}^2 + \|\tilde B\|^2)$,
	$$
	d\mathcal{E}/dt\leq C_2 \eps (1+ \mathcal{E}^2)\mathcal{E}.
	$$
	so long as $\mathcal{E}\leq M$. Changing to the fast time scale $T:=t/\eps$, we have
	$$
	d\mathcal{E}/dT\leq C_2 (1+ \mathcal{E}^2)\mathcal{E},
	$$
	giving bound $\mathcal{E}\leq M$ up to time $T'\sim 1$ so long as $\mathcal{E}_{t=0}\leq M/2$
	by standard ODE estimates.  But, this gives 
	$\mathcal{E}\leq M$ up to fast time $t'\sim 1/\eps$, as was our goal. 
\end{proof}
We note that both proofs rely on the same observation, which is
\be
\frac{d}{dT}A\big(X-\frac{fT}{\e},T\big)\approx -\frac{f}{\e}A_X\big(X-\frac{fT}{\e},T\big).
\ee
Informally stated, along characteristics of the $B$-equation, one can exchange the singular $\d_X$ term with a total time derivative which integrates to a harmless boundary term.\\

We end this subsection with the observation that we can extend the well-posedness to $H^s$ perturbations of a fixed background periodic state. To do this, let $(A_{per},B_{per})$ be periodic solutions of \eqref{eq:truncatedsystem} with $(A_{per},B_{per})\in W^{m,\infty}(\TT\times[0,\infty))$ for some $m$ sufficiently large. Then let we seek solutions of \eqref{eq:truncatedsystem} of the form $(A,B)=(A_{per}+\cA,B_{per}+\cB)$ with $(\cA,\cB)$ $H^s$ on the line. Feeding this Ansatz into \eqref{eq:truncatedsystem}, we can rearrange to obtain a system of the form
\begin{equation}
	\begin{aligned}
		\cA_T&=a\cA_{XX}+b(X,T)\cA+\cF(X,T;\cA,\cB),\\
		\cB_T&=\e^{-1}(\cB+\cH(X,T;\cA,\cB) )_X+e_B\cB_{XX}+\cG(X,T;\cA,\cB),
	\end{aligned}
\end{equation}
where $\cF$, $\cG$, and $\cH$ are known smooth functions of $(X,T)$ and polynomial functions of $(\cA,\cB)$ and $b(X,T)\in W^{m,\infty}(\TT\times[0,\infty))$. Now, we note Observation \ref{obs:cT1} extends to this situation here, which lets us apply the standard parabolic estimates to $\mathcal{T}_1(\cA,\cB)$ to conclude $\mathcal{T}_1(\cA,\cB)$ is a locally Lipschitz function of $(\cA,\cB)$. For the $\cB$ equation, we observe that we can apply the same trick in Observation \ref{obs:chainrule} here as well, and arguing as in Theorem \ref{thm:truncatedwellposed} completes the proof. We then end up with a restriction of the size of the perturbation and time of existence depending only on the model parameters and $(A_{per},B_{per})$.

\subsection{More on the Darcy model}\label{sub:darcy}
There is a sense in which the Darcy model is the ``correct'' model because it is what is naturally produced by the multiscale expansion. However, there are some technical complications in solving for $B_0(\Ht)$. One option to determine $B_0(\Ht)$ is to match the approximate solution with an exact solution. For a more systematic approach, one may try to look at the coefficient of $\e^4e^{0}$ reproduced below
\be\label{eq:Darcy}
(\vec{B}_0-F^{-1}\vec{h}|A|^2)_{\Ht}= (F\vec{\mathcal{B}}+(\overline{A}\mathcal{A}+A\overline{\mathcal{A}})\vec{h} )_{\Hx}+\d_{\Hx}(DB_{\Hx}+\Re(A\overline{A}_{\Hx}\vec{g}) ),
\ee
for known matrices $F$ and $D$ and known vectors $\vec{g},\ \vec{h}$. At this point we run into a problem because this equation has two undetermined unknowns $\vec{B_0}(\Ht)$ and $\vec{\mathcal{B}}$, with $\mathcal{A}$ determined by a suitable Ginzburg-Landau equation as in the tilde model. Thinking of \eqref{eq:Darcy} as an elliptic equation for $\vec{\mathcal{B}}$, we see that \eqref{eq:Darcy} is of the form
\be\label{eq:dxB=g}
\d_{\Hx} \vec{\mathcal{B}}=g,
\ee	
for some suitable function $g(\Hx)$ encompassing the remaining terms of \eqref{eq:Darcy}. If $g(\Hx)$ is periodic, then there is a periodic solution $\mathcal{B}$ if and only if $\hat{g}(0)=0$ as can readily be seen by writing $g(\Hx)$ in terms of a Fourier series. From this, if we want $A$, $\vec{B}$, $\mathcal{A}$ and $\vec{\mathcal{B}}$ to all be periodic in $\hat{x}$, then we deduce an ODE for $\vec{B_0}$ as a solvability condition in \eqref{eq:Darcy} for $\mathcal{B}$. Specifically, we obtain
\be\label{eq:B0ODE}
\frac{d}{d\Ht} \vec{B}_0(\Ht)=\frac{1}{|\TT|}\int_{\TT} F^{-1}\vec{h}(|A(\Hx,\Ht)|^2)_{\Ht}d\Hx, 
\ee
where $|\TT|$ denotes the measure of the torus. One can use the Ginzburg-Landau equation for $A$ to express the right hand side of \eqref{eq:B0ODE} in terms of $A$, $A_{\Hx}$, $A_{\Hx\Hx}$ and $\vec{B}_0$ explicitly, but we will not do this here. We also note that \eqref{eq:B0ODE} is reassuringly equivalent to the statement
\be
\int \vec{B}(\Hx,\Ht)d\Hx=const,
\ee
as $\vec{B}$ is supposed to be conserved. We note this process can be continued indefinitely, allowing for an asymptotic expansion valid to all orders.\\

On the real line, however, it is much less obvious how to make sense of the Darcy reduction procedure we've used here. The key issue is that we lose access to the simple solvability condition for \eqref{eq:dxB=g} used to solve for $\vec{B}_0(\Ht)$, and there does not seem to be a readily available substitute. A related question is to what extent this reduction procedure can be carried out on bounded intervals with other boundary conditions such as Dirichlet.\\

We will discuss two cases on the line, where there is a workaround. The first is the situation in which $A$ and its derivatives are in $L^2(\RR)$ and further assume that $A$, $\mathcal{A}$, $\mathcal{B}$ and their derivatives with respect to $\Hx$ all go to zero as $|\Hx|\to\infty$. We integrate \eqref{eq:Darcy} from $-R$ to $R$ to get
\begin{equation}
	\d_{\Ht} 2R B_0-\int_{-R}^RF^{-1}\vec{h}(|A|^2)_{\Ht}d\Hx=(F\vec{\mathcal{B}}+(\overline{A}\mathcal{A}+A\overline{\mathcal{A}})\vec{h} )+(DB_{\Hx}+\Re(A\overline{A}_{\Hx}\vec{g}) )|_{-R}^R.
\end{equation}
This holds for all $0<R<\infty$. Solving for $(B_0)_{\Ht}$ and sending $R\to\infty$ yields
\begin{equation}
	\frac{d}{d \Ht}B_0(\Ht)=0,
\end{equation}
which further forces $B_0(0)=0$ if $B$ is to be in $L^2(\RR)$.\\

A similar idea works with $L^2$ perturbations of some fixed background periodic state $(A_{per},B_{per})$, which upon a similar averaging-type argument implies that $B=B_{per}$. We note that this approach does not extend nicely to other background states, as we chose $[-R,R]$ as our reference interval here but in either case any interval can be chosen.

\section{Multiscale expansion: linear stability}\label{sec:MSEstability}
In this section we show how to reduce the question of stability of solutions of the amplitude system \eqref{eq:O2amplitude} or the singular systems of Section \ref{sec:MSE_Example} of the form
\be\label{eq:periodictravelingwaves}
	(A,B)(\Hx,\Ht)=(A_0 e^{i(\kappa\Hx-\Omega\Ht)},B_0),
\ee
where $A_0>0$, $\kappa,\Omega,B_0\in\RR$ can be reduced to a suitable constant coefficient system. We then discuss some aspects of well-posedness theory of the amplitude system on long time intervals and how to relate the predictions of the various singular models to each other.\\

For the rest of this section, we replace $(\Hx,\Ht)$ with $(X,T)$.

\subsection{The linearized amplitude system}\label{subsec:linearizedcGL}
Consider the modified complex Ginzburg-Landau equation with one conservation law:
\ba
A_T&=a A_{XX}+b A+ c|A|^2A + d AB,\\
B_T &=  e_B B_{XX} + \eps^{-1}( f B + h|A|^2)_X + \d_X \Re(gA\overline{A_X}),
\ea
$A\in \CC$, $B\in \R$,
where $a,b,c,d,g$ are complex numbers with the appropriate signs on their real parts, 
i.e. $\Re(a),\Re(b)>0$ and $\Re(c)<0$, and $e_B>0$, $f,h$ are real.\\

Assume that $(A,B)$ has the form
\begin{equation}\label{eq:genericAnsatz}
	(A,B)(X,T)=(\a e^{i(\kappa X-\omega T)}, \b),
\end{equation}
where, without loss of generality, $\kappa ,\omega\in\RR$ are yet to be determined constants, $\a$ is a positive
real constant, and $\b$ a real constant.

Plugging \eqref{eq:genericAnsatz} into \eqref{eq:conamp}, we obtain the nonlinear
dispersion relation 
\begin{equation}\label{eq:qOmegaEqn}
	-i\omega=-a\kappa ^2+(b+d\b)+c\a^2,
\end{equation}
giving in particular the condition on $\beta$:
\be\label{betacond}
\hbox{\rm $\Re d \beta > -\Re b$ or $\Re(b+d\beta)>0$.}
\ee
Solving real and imaginary parts separately in \eqref{eq:qOmegaEqn}, we find that
$$
\omega=\Im(a)\kappa^2-\Im(b+d\b)-\Im(c)\a^2,
\quad
0=-\Re(a)\kappa^2+\Re(b+d\b)+\Re(c)\a^2.
$$
The second equation is solvable on the range of existence
\begin{equation}\label{eq:CGLExistence}
	\kappa^2\leq \kappa^2_E:=\frac{\Re( b+d\b)}{\Re a}
\end{equation}
yielding 
\be\label{eq:cglomega}
	\omega=\Im(a)\kappa^2-\Im(b+d\b))-\Im(c)\a^2,
\ee and
\be\label{eq:cglamplitude}
	\a^2=\frac{-\Re (b+d\b)+\Re a \kappa^2}{\Re c},
\ee
as functions of $\kappa$ and $\b$. We emphasize that the singularity in the $B$ equation has completely disappeared, as \emph{any} Ansatz of the form \eqref{eq:genericAnsatz} satisfies the $B$ equation.\\

We now perturb the solution $A(X,T)$ constructed above by
\begin{equation}
	u(X,T)=\left(\a+\tilde A(X,T) \right)e^{i(\kappa X-\omega T)},
	\qquad
	v(X,T)= \b + \tilde B,
\end{equation}
factoring out periodic behavior to obtain $\tilde A$ as a perturbation of a constant solution $\alpha$.\\

Plugging this Ansatz into \eqref{eq:conamp} gives
\ba\label{prestep}
-i\omega u+\tilde A_Te^{i(\kappa X-\omega T)}&=-\kappa ^2u+2i\kappa \tilde A_Xe^{i(\kappa X-\omega T)}+bu+c(|\a|^2\a+2|\a|^2\tilde A\\
&\quad +\a^2\overline{\tilde A}+\cO(|\tilde A|^2))e^{i(\kappa X-\omega T)} + duv,\\
\tilde B_T&=
e_B \tilde B_{X X} + \eps^{-1}( f \tilde B + h(\a^2+ \alpha(\tilde A + \bar{\tilde A})) )_{X} \\
&\quad + 
\d_X \Re( g(\a+\tilde{A})\overline{-i\kappa(\a+\tilde{A})+\tilde{A}_X})\\
&\quad=e_B \tilde{B}_{XX}+\e^{-1}(f\tilde{B}+h\a (\tilde{A}+\overline{\tilde{A}}) )_X\\
&\quad+\d_X \Re(g(\a\overline{\tilde{A}}_X-i\kappa \a^2-i\kappa\a(\tilde{A}+\overline{\tilde{A}})+\cO(|\tilde{A}|^2))) .
\ea

We can simplify \eqref{prestep} using the fact that $(\a e^{i(\kappa X-\omega T)},\b)$ is a solution, 
factoring out the exponential term $e^{i(\kappa x-\omega t)}$, and dropping the $\cO(|(\tilde A,\tilde B)|^2)$ 
terms, to obtain a constant-coefficient system
\ba
-i\omega \tilde A+\tilde A_T&=a\tilde A_{XX}+2i\kappa a\tilde A_X-\kappa^2a\tilde A+
(b+ d\b)\tilde A+c\a^2\tilde A+c\a^2(\tilde A+\overline{\tilde A})
+ d \a \tilde B,
\\
\tilde B_T&=
e_B \tilde B_{X X} + \eps^{-1}( f \tilde B + h(\alpha(\tilde A + \bar{\tilde A})) )_{X} 
+ \d_X \Re(g(\a\overline{\tilde{A}}_X-i\kappa \a^2-i\kappa\a(\tilde{A}+\overline{\tilde{A}}))
\ea
in the modified unknown $(\tilde A,\tilde B)$.
\begin{remark}
	Here we've used a secondary manifestation of translation-invariance from ``upstairs'', namely we're using the invariance $A\to Ae^{i\chi}$ and $B\to B$ for $\chi\in\RR$ which corresponds to translation in $\xi$. The more obvious way that translation invariance is reflected in the amplitude system is that is constant coefficient, which is how translation in $\Hx$ appears in the amplitude system.
\end{remark}
Applying \eqref{eq:qOmegaEqn} gives, finally, the linearized (cGL) equation
\ba\label{eq:LinearizedGenericCGL}
\tilde A_T&=a\tilde A_{XX}+2i\kappa a\tilde A_x+c\a^2(\tilde A+\overline{\tilde A})+ d\alpha \tilde B,\\
\tilde B_T&= 
e_B \tilde B_{X X} + \eps^{-1}( f \tilde B + h\alpha(\tilde A + \bar{\tilde A}) )_{X} 
+ \d_X \Re(g(\a\overline{\tilde{A}}_X-i\kappa\a(\tilde{A}+\overline{\tilde{A}}))
\ea
again, in the coordinates with background periodic behavior factored out.\\

We now write $\tilde A=u+iv$ where $u$ and $v$ are two real-valued functions, and $\tilde B=w$.
This gives the system of linear equations
\ba\label{eq:LinearizedSystem}
\bp u_T \\ v_T\\ w_T \ep&=
\bp \Re a & -\Im a & 0 \\ \Im a & \Re a& 0\\ 2\Re(g)\a  & 2\Im(g)\a &e_B \ep 
\bp u_{XX}  \\ v_{XX}\\ w_{XX} \ep
+\bp -2\kappa\Im a & -2\kappa\Re a &0 \\ 2\kappa\Re a & -2\kappa\Im a &0\\ 2 \eps^{-1}h\a+2\a \kappa\Im(g) & 0  & \eps^{-1} f\ep 
\bp u_X \\ v_X\\ w_X \ep  \\
&\quad 
+\bp 2\a^2\Re c & 0& \Re d \a \\ 2\a^2\Im c & 0& \Im d \a\\ 0 & 0 & 0 \ep
\bp u \\ v\\ w \ep.
\ea
Assume that
\begin{equation}
	\bp u \\ v\ep=\bp u_0\\ v_0\ep e^{i\sigma X+\l T}
\end{equation}
for $\sigma \in\RR$ small and $\l\in\CC$ to be determined. Plugging this 
into \eqref{eq:LinearizedGenericCGL} gives
\ba\label{eq:hopefulreducedeqn}
\l\bp u_0\\v_0\\w_0 \ep &=
\Big(
-\sigma^2 \bp \Re a & -\Im a & 0 \\ \Im a & \Re a& 0\\ 2\Re(g)\a  & 2\Im(g)\a &e_B \ep \\
&\quad
+i\sigma \bp -2\kappa\Im a & -2\kappa\Re a &0 \\ 2\kappa\Re a & -2\kappa\Im a &0\\ 2 \eps^{-1}h\a+2\a \kappa\Im(g) & 0  & \eps^{-1} f\ep 
+\bp 2\a^2\Re c & 0& \Re d \a \\ 2\a^2\Im c & 0& \Im d \a\\ 0 & 0 & 0 \ep \Big)
\bp u_0\\v_0\\w_0 \ep,
\ea
i.e., that $\l$ is an eigenvalue of the matrix $M(\eps,\sigma)$ on the right-hand side.\\

We emphasize that, although the existence of periodic traveling waves is non-singular, their stability is singular for the generic $SO(2)$ but not $O(2)$-invariant system.\\ 

While we showed that linear stability of solutions of the form \eqref{eq:genericAnsatz} is determined by exponentials for the truncated model \eqref{eq:truncatedsystem} directly, the same conclusion holds for much more generic models.
\begin{theorem}\label{thm:exponentialsymmetry}
	Consider a quasilinear local system of the form
	\ba
	A_t&=F(A,B),\\
	B_t&=\d_x G(A,B),
	\ea
	for $A\in\CC$ and $B\in\RR^n$, $F$ and $G$ induced by maps $\hat{F}:\CC^a\times(\RR^n)^b\to\CC$, $\hat{G}:\CC^k\times(\RR^n)^l\to\RR^n$ respectively via $F(A,B)=\hat{F}(A,\d_x A,...,\d_x^a A,B,\d_x B,...,\d_x^b B)$ and a similar convention for $G$. Suppose that $F$ and $G$ have the following properties
	\begin{enumerate}
		\item For each real $\chi$, we have $F(e^{i\chi}A,B)=e^{i\chi}F(A,B)$ and $G(e^{i\chi}A,B)=G(A,B)$, equivalently $\hat{F}(e^{i\chi}z_1,...,e^{i\chi}z_a,B_1,...,B_b)=e^{i\chi}\hat{F}(z_1,...,z_a,B_1,...,B_b)$ and a similar equality for $\hat{G}$.\\
		\item $F$ and $G$ are smooth and translation invariant.\\
		\item There exists a constant $A_0>0$ so that $F(A_0,0)=0$ and $\Re\d_A F(A_0,0)\not=0$.
	\end{enumerate}
	Then there exists a neighborhood of $0\in\RR^{n+1}$ and a family of exponential solutions of the form $(A,B)(x,t)=(A_0(\kappa,B_0)e^{i(\kappa x-\omega(\kappa,B_0) t) },B_0)$ with $A_0>0$ and $\omega\in\RR$ smoothly depending on $(\kappa,B_0)$. In addition, perturbing the exponential solutions as $A(x,t)\to(A_0+u(x,t))e^{i\kappa x-\omega t}$ and $B(x,t)\to B_0+v(x,t)$ and linearizing gives a constant coefficient system.
\end{theorem}
\begin{proof}
	Our strategy to show existence is to use the implicit function theorem. The first step is to show that for the $B$ equation that any pair of functions of the form $(A,B)=(A_0e^{i(\kappa x-\omega t)},B_0 )$ is automatically a solution. To show this, first observe that $e^{-i\omega t}$ is independent of $x$ and hence can be safely ignored by the property $G(e^{i\chi}A,B)=G(A,B)$. For the rest of the dependence on $A$, we expand out the arguments of $G$ as
	\be
	G(A,B_0)=G(A,\d_x A,...,\d_x^k A,B_0)=\hat{G}(A_0e^{i\kappa x},i\kappa A_0e^{i\kappa x},...,(i\kappa)^kA_0e^{i\kappa x},B_0,0...,0 ).
	\ee
	Applying the invariance again implies that $G(A,B)$ is constant for such pairs of functions completing the first step. A similar computation gives the first equation as
	\be
	-i\omega A_0=\hat{F}(A_0,i\kappa A_0,...,(i\kappa)^aA_0,B_0,0,...,0).
	\ee
	Splitting this equation into real and imaginary parts and first applying the implicit function theorem to the real part to solve for $A_0$ and then to the imaginary part to get $\omega$ completes the proof of existence of exponential solutions.\\
	
	Turning to the stability question, we linearize the system in the usual way to get
	\ba
	U_t&=\sum_{i=0}^a\hat{F}_{z_i}(A,B)\d_x^iU+\sum_{j=0}^b\hat{F}_{B_j}(A,B)\d_x^jV,\\
	V_t&=\d_x\Big(\sum_{i=0}^k\hat{G}_{z_i}(A,B)\d_x^iU+\sum_{j=0}^l\hat{G}_{B_j}(A,B)\d_x^jV\Big).
	\ea
	Differentiating the invariances with respect to $z_i$ and applying the chain rule gives
	\ba
	e^{i\chi}\d_{z_i}\hat{F}(e^{i\chi}z_1,...,e^{i\chi}z_a,B_1,...,B_b)=e^{i\chi}\d_{z_i}\hat{F}(z_1,...,z_a,B_1,...,B_b),\\
	e^{i\chi}\d_{z_i}\hat{G}(e^{i\chi}z_1,...,e^{i\chi}z_k,B_1,...,B_l)=\d_{z_i}\hat{G}(z_1,...,z_k,B_1,...,B_l).
	\ea
	So in the linearized system, we can write
	\ba
	U_t&=\sum_{i=0}^a\hat{F}_{z_i}(A_0,...,(i\kappa)^aA_0,B_0,...0)\d_x^iU+e^{i(\kappa x-\omega t)}\sum_{j=0}^be^{i(\kappa x-\omega t)}\hat{F}_{B_j}(A_0,...,(i\kappa)^aA_0,B_0,...0)\d_x^jV,\\
	V_t&=\d_x\Big(\sum_{i=0}^ke^{-i(\kappa x-\omega t)}\hat{G}_{z_i}(A_0,...,(i\kappa)^kA_0,B_0,...0)\d_x^iU+\sum_{j=0}^l\hat{G}_{B_j}(A_0,...,(i\kappa)^kA_0,B_0,...0)\d_x^jV\Big).
	\ea
	Plugging in $U(x,t)=e^{i(\kappa x-\omega t)}u(x,t)$, relabeling $V(x,t)=v(x,t)$ and using the Leibniz rules gives us a constant coefficient system in $(u,v)$, whose stability is then determined by exponentials.
\end{proof}

\subsection{Computation of the dispersion relations}\label{sub:dispersion}
In this section, we will drop the hats on $\Hs$ and $\hat{\l}$ for ease of notation.\\

Before we compute the dispersion relations for the truncated model, we recall an extension of Proposition \ref{prop:spectralid} from \cite{K}.
\begin{proposition}[Theorem 5.11]\label{prop:Kato511}
	Let $T(x):X\to X$ be an operator admitting the asymptotic expansion $T(x)=T+xT^{(1)}+x^2T^{(2)}+o(x^2)$ for $x\to 0$. Let $\l$ be a semisimple eigenvalue of $T$ with eigenprojection $P$. Then the total eigenprojection $P(x)$ for the $\l$-group eigenvalues of $T(x)$ has the form
	\be
	P(x)=P+xP^{(1)}+x^2P^{(2)}+o(x^2),
	\ee
	where $P^{(1)}$ is given by
	\be
	P^{(1)}=-PT^{(1)}S-ST^{(1)}P,
	\ee
	where $S$ is the reduced resolvent $S:=[(I-P)(T-\l)(I-P) ]^{-1}$. If $\l_j^{(1)}$ is an eigenvalue of $PT^{(1)}P$ in $PX$ with eigenprojection $P_j^{(1)}$, then $T(x)$ has exactly $m_j^{(1)}:=\rank(P_j^{(1)})$ repeated eigenvalues of the form $\l+x\l_j^{(1)}+o(x)$, the $\l+x\l_j^{(1)}$-group, and has eigenprojection $P_j^{(1)}(x)=P_j^{(1)}+xP_j^{(11)}+o(x)$. If in addition, $\l_j^{(1)}$ is also a semisimple eigenvalue of $PT^{(1)}P$, then the $m_j^{(1)}$ eigenvalues of the $\l+x\l_j^{(1)}$-group have the form
	\be
	\mu_{jk}(x)=\l+x\l_j^{(1)}+x^2\mu_{jk}^{(2)}+o(x^2) \quad,k=1,...,m_j^{(1)},
	\ee
	where the $\mu_{jk}^{(2)}$ are the repeated eigenvalues of $P_j^{(1)}T^{(2)}P_j^{(1)}-P_j^{(1)}T^{(1)}ST^{(1)}P_j^{(1)}$ in the subspace $P_j^{(1)}X$.
\end{proposition}
There are also formulas for the higher order coefficients in the expansions for $P$ and $P_j^{(1)}$, but we will not need them.\\

Consider a model truncated singular system of the form, for example the one in Section \ref{sec:MSE_Example}.
\ba
A_t&=aA_{xx}+bA+c|A|^2A+dAB,\\
B_t&=\e^{-1}(f B_x+h |A|^2_x)+e_B B_{xx}+\d_x\Re(gA\bar{A}_x),
\ea
where $a,b,c,d,g\in\CC$, $f,h,e_B\in\RR$ with appropriate signs on the real parts of $a,b,c$, and $e_B$. Following the proof of Theorem \ref{thm:exponentialsymmetry}, we perturb $(A,B)=(A_0e^{i(\kappa x-\omega t)},B_0)$ as $A\to (A_0+u+iv)e^{i(\kappa x-\omega t)}$ and $B\to B_0+w$, with $u,v,w$ real valued. This gives the system
\ba
u_t&=\Re(a)u_{xx}-\Im(a)v_{xx}-2\kappa\big(\Im(a)u_x+\Re(a)v_x\big)+2A_0^2\Re(c)u+A_0\Re(d)w,\\
v_t&=\Im(a)u_{xx}+\Re(a)v_{xx}+2\kappa\big(\Re(a)u_x-\Im(a)v_x\big)+2A_0^2\Im(c)v+A_0\Im(d)w,\\
w_t&=\e^{-1}(f w_x+2h A_0 u_x )+e_B w_{xx}+2A_0\d_x\big(u_x\Re(g)+v_x\Im(g)+\kappa u\Im(g) \big),
\ea
where we've used
\be
-i\omega U=-a\kappa^2U+bU+cA_0^2U+dB_0U,
\ee
for $U:=u+iv$. To describe stability, we write $(u,v,w)=(u_0,v_0,w_0)e^{i\s x-\l t}$ to get the following eigenvalue problem
\ba\label{eq:truncatedmatrix}
-\l \bp u_0 \\ v_0 \\ w_0\ep&=\Bigg(-\s^2\bp [[a]] & \bp0\\0\ep \\
\bp 2A_0\Re(g) & 2A_0\Im(g)\ep & e_B\ep+i\s\bp [[2i\kappa a]] & \bp 0 \\ 0 \ep\\
\bp 2A_0h \e^{-1}+2A_0\kappa \Im(g) & 0\ep & \e^{-1}f\ep+\\
&\quad+\bp \bp 2A_0^2\Re(c) & 0\\
2A_0^2\Im(c) & 0\ep &\bp A_0\Re(d)\\ A_0\Im(d)\ep\\
\bp 0 & 0 \ep & 0\ep  \Bigg)\bp u_0 \\ v_0 \\ w_0\ep.
\ea
Define $M(\e,\s)$ to be the matrix on the left hand side of \eqref{eq:truncatedmatrix}.\\

At $\s=0$, we have the (generalized) eigenvectors of $M(\e,0)$ given by
\be\label{eq:truncatedeigenvectors}
R_t:=\bp 0 \\ 1 \\ 0\ep \quad\quad R_c:=\bp -\frac{\Re(d)}{2A_0\Re(c)} \\ 0 \\ 1\ep \quad\quad R_s:=\bp 1 \\ \frac{\Im(c)}{\Re(c)} \\ 0\ep.
\ee
In addition, we have corresponding left (generalized) eigenvectors
\be\label{eq:truncatedleft}
L_t:=\bp -\frac{\Im(c)}{\Re(c)} & 1 & -\frac{\Re(d)\Im(c)}{2A_0\Re(c)^2}\ep  \quad\quad L_c:=\bp 0 & 0 & 1 \ep \quad\quad L_s:=\bp 1 & 0 & \frac{\Re(d)}{2A_0\Re(c)}\ep,
\ee
where $(L_s,R_s)$ are the left/right eigenvector pair for the unique stable eigenvalue and $(L_t,R_t)$, $(L_c,R_c)$ are associated to the zero eigenvalue. We note that $L_t$ and $R_c$ are generalized left and right eigenvectors respectively, provided that $\Re(c)\Im(d)-\Im(c)\Re(d)\not=0$. We denote the eigenvalues of \eqref{eq:truncatedmatrix} as $\l_{Trun}^{Tran}(\e,\s)$, $\l_{Trun}^{Con}(\e,\s)$ and $\l_{Trun}^{Stab}(\e,\s)$ where the first letter of the subscript matches the associated (generalized) eigenvectors. E.g. $\l_{Trun}^{Tran}$ has (generalized) eigenvectors $R_t,L_t$ and $\l_{Trun}^{Stab}$ has eigenvectors $R_s,L_s$.\\

\begin{remark}\label{rem:bronski}
	The subscript $t$ comes from expressing the eigenvector $\frac{\d u}{\d \xi}$ in terms of the basis $(\Re a,\Im a,b)$, hence the subscript $t$ is appropriate as this eigenvector comes from translation invariance. The nontrivial coefficient of $R_c$ is given by $\frac{\d A_0}{\d B_0}$, as can be checked from \eqref{eq:cglamplitude}. This gives a justification to the subscript $c$. Moreover, it is also related to why the Jordan block is typically present as it is a common occurrence in conservation laws, see \cite{JZ1,JZ2,BJZ2} and sources therein for more discussion.
\end{remark}

For the Darcy model, where when $f\not=0$ we can solve $B=B_0-\frac{h}{f}|A|^2$ for some $B_0$, we get the linearization
\be
\l\bp u_0\\v_0\ep =\left(-\sigma ^2\bp \Re a & -\Im a \\ \Im a & \Re a\ep+i\sigma \bp -2\kappa\Im a & -2\kappa\Re a \\ 2\kappa\Re a & -2\kappa\Im a\ep+\bp 2\tilde{A}_0^2\Re \tilde{c} & 0 \\ 2\tilde{A}_0^2\Im \tilde{c} & 0\ep\right)\bp u_0\\v_0\ep,
\ee
where $\tilde{c}:=c-\frac{dh}{f}$ and $\tilde{A}_0$ is the corresponding amplitude. Similar to what we did for the truncated model, we compute the eigenvectors to be
\be
R_t:=\bp 0 \\ 1 \ep \quad\quad R_s:=\bp 1 \\ \frac{\Im(\tilde{c})}{\Re(\tilde{c})}\ep.
\ee
Correspondingly, the left eigenvectors are given by
\be
L_t:=\bp -\frac{\Im(\tilde{c})}{\Re(\tilde{c})} & 1 \ep \quad\quad L_s:=\bp 1 & 0 \ep.
\ee
Asides from the distinction of $c$ and $\tilde{c}$, the eigenvectors for the Darcy model are the projections of the corresponding eigenvectors for the truncated model. Note that the stable eigenvalue for the two models differ by $A_0^2\Re(c)-\tilde{A}_0^2\Re(\tilde{c})$, which is generically nonzero.\\

Continuing our exploration of the eigenvalues of these two models near $\s=0$, we recall from \cite{WZ2} the expansion of the neutral eigenmode for the Darcy model
\be\label{eq:DarcySecondCoeff}/
C_{2,Darcy}(\e)=\frac{(2\kappa^2\Im \tilde{c} ^2\Re a^2+\a^2\Im a\Im \tilde{c}(\Re \tilde{c})^2+\Re a(\Re \tilde{c})^2(2\kappa^2\Re a+\tilde{A}_0^2\Re \tilde{c}))}{(\tilde{A}_0^2(\Re \tilde{c})^3)},
\ee
where we write the neutral eigenvalue $\l_{Darcy}(\e,\s)=C_{1,Darcy}(\e)\s+C_{2,Darcy}(\e)\s^2+\cO(\s^3)$. We will call the stability criteria obtained from $C_{2,Darcy}(\e)<0$ the ``fine'' Darcy condition and we call the stability criteria $\Re(\tilde{c})<0$ the ``coarse'' Darcy criterion. The motivation for the name coarse comes the observation that the condition on $\Re(\tilde{c})$ determines whether the bifurcation is supercritical or subcritical, which is a very coarse notion of stability since all periodic solutions are unstable in the subcritical case. Note that such an expansion is valid by the implicit function theorem because the neutral eigenvalue is simple for complex Ginzburg-Landau. Turning to the truncated model, we apply Proposition \ref{prop:Kato511} in the special case where $M_0$ has 0 as a semisimple eigenvalue. The first step is to compute $PM_\s(\e,0)P$ for the truncated model and $P$ the projection onto the 0 eigenspace. To minimize the number of fractions in the following computation, define two parameters
\be
p:=-\frac{\Re(d)}{2A_0\Re(c)}\quad\quad q:=-\frac{\Im(c)}{\Re(c)},
\ee
so that
\be
R_t= \bp 0 \\ 1 \\ 0 \ep \quad\quad R_c=\bp p \\ 0 \\ 1\ep \quad\quad R_s=\bp 1 \\ -q\\ 0 \ep,
\ee
and
\be
L_t=\bp q & 1 & -pq \ep \quad\quad L_c=\bp 0 & 0 & 1\ep \quad\quad L_s=\bp 1  & 0 & -p\ep.
\ee
We begin this process by writing out $M_\s(\e,0)R_t$ and $M_\s(\e,0)R_c$ to get
\ba
M_\s(\e,0)R_t&=i\left(\kappa \bp -2\Im(a) & -2\Re(a) & 0\\
2\Re(a) & -2\Im(a) & 0\\
2A_0 \Im(g) & 0 & 0\ep+\e^{-1} \bp 0 & 0 & 0\\ 0 & 0 & 0\\ 2A_0h & 0 & f\ep\right)\bp 0 \\ 1 \\ 0\ep\\
&\quad=-2i\kappa \bp \Re(a) \\ \Im(a) \\0 \ep,\\
M_\s(\e,0)R_c&=i\left(\kappa \bp -2\Im(a) & -2\Re(a) & 0\\
2\Re(a) & -2\Im(a) & 0\\
2A_0 \Im(g) & 0 & 0\ep+\e^{-1} \bp 0 & 0 & 0\\ 0 & 0 & 0\\ 2A_0h & 0 & f\ep\right)\bp p \\ 0 \\ 1\ep\\
&\quad=i\bp -2\kappa\Im(a)p\\2\kappa\Re(a)p\\2\kappa A_0\Im(g)p+\e^{-1}(2A_0h p+f)\ep.
\ea
Next, we apply $\{L_t,L_c\}$ to find the coefficients of $PM_\s(\e,0)P$ in the $\{R_c,R_t\}$ basis. Starting with $M_\s(\e,0)R_t$ we have
\ba
L_tM_\s(\e,0)R_t&=-2i\kappa\bp q & 1 & -pq \ep\bp \Re(a) \\ \Im(a) \\0 \ep=-2i\kappa( \Re(a)q+\Im(a)),  \\
L_cM_\s(\e,0)R_t&=-2i\kappa\bp 0 & 0 & 1\ep \bp \Re(a) \\ \Im(a) \\0 \ep=0.
\ea
For $M_\s(\e,0)R_c$ we have
\ba
L_tM_\s(\e,0)R_c&=i\bp q & 1 & -pq \ep\bp -2\kappa\Im(a)p\\2\kappa\Re(a)p\\2\kappa A_0\Im(g)p+\e^{-1}(2A_0h p+f)\ep\\
&\quad=2\kappa i(-\Im(a)pq+\Re(a)p-p^2qA_0\Im(g))-pqi\e^{-1}(2A_0h b+f),\\
L_cM_\s(\e,0)R_c&=i\bp 0 & 0 & 1\ep\bp -2\kappa\Im(a)p\\2\kappa\Re(a)p\\2\kappa A_0\Im(g)p+\e^{-1}(2A_0h p+f)\ep=i2A_0\kappa\Im(g)p+i\e^{-1}(2A_0h b+f).
\ea
So in the $\{R_t,R_c \}$ basis, we have the representation of $PM_\s(\e,0)P$ as
\be
PM_\s(\e,0)P=i\bp -2\kappa(\Re(a)q+\Im(a) ) & -pq\e^{-1}(2A_0h p+f )+2\kappa(-\Im(a)pq+\Re(a)p-p^2qA_0\Im(g))\\
0 & \e^{-1}(2A_0h p+f)+2A_0\kappa\Im(g)p\ep.
\ee
Let $F(\e)$ be defined as
\be
F(\e):=\frac{-pq\e^{-1}(2A_0h p+f )+2\kappa(-\Im(a)pq+\Re(a)p-p^2qA_0\Im(g))}{\e^{-1}(2A_0h p+f)+2A_0\kappa\Im(f)p+2\kappa(\Re(a)q+\Im(a))}=-pq+\cO(\e).
\ee
provided that $2A_0h p +f\not=0$, which is an open condition on the model parameters. This function $F(\e)$ has the important property that the following is a right eigenvector of $PM_\s(\e,0)P$
\be
\tilde{R}_c=R_c+F(\e)R_t,
\ee
and
\be
\tilde{L}_t=L_t-F(\e)L_c,
\ee
is a left eigenvector, as we note that $R_t$ and $L_c$ remain eigenvectors of $PM_\s(\e,0)P$.
\begin{remark}\label{rmk:FLSvsTrun}
	One can also view $F(\e)$ as a function of $\tilde{\e}$, where $\tilde{\e}$ is defined to be
	\be
	\tilde{\e}:=\frac{\e}{(2A_0h p+f )}.
	\ee
	In terms of $\tilde{\e}$, $F(\e)$ admits the representation
	\be
	F(\tilde{\e})=\frac{-pq\tilde{\e}^{-1}+c_1(\kappa)}{\tilde{\e}^{-1}+c_2(\kappa)}=-pq+(c_1(\kappa)+pqc_2(\kappa) )\tilde{\e}+\cO(\tilde{\e}^2).
	\ee
	For the ``true'' model coming from Lyapunov-Schmidt, one replaces $f$ with $f(\e)=f+\e f'$ and $h$ with $h(\e)=h+\e h'$ for some known $f'$ and $h'$. Correspondingly, one can adjust the definition of $\tilde{\e}$ as above by replacing $f$ and $h$ with $f(\e)$ and $h(\e)$ respectively. The series of $F$ with respect to $\tilde{\e}$ is unchanged, and of equal importance a routine calculation shows that
	\be
	\frac{x}{a+bx}=\frac{x}{a}+\cO(|x|^2).
	\ee
	The importance of this is that then $F_{Trun}$ and $F_{LS}$, with subscripts identifying which choice of $\tilde{\e}$ was made, have the same first derivative; which appears in the coefficient of the second derivative of the translational eigenvalue.
\end{remark}
As we can see, the diagonal entries of $PM_\s(\e,0)P$ are generically distinct for all $\e<\e_0$ for $\e_0\ll1 $ depending on the model parameters and that the eigenvalues are pure imaginary. We will show in Theorem \ref{thm:truncatedanalytic} that this is enough to show that $M(\e,\s)$ has analytic spectrum for all such truncated models, not just this special case.\\

Turning to the second derivatives, we first compute $M_\s(\e,0)R_s$ and $L_sM_\s(\e,0)$. We find that
\ba
L_sM_\s(\e,0)&=i\bp 1  & 0 & -p\ep\left(\kappa \bp -2\Im(a) & -2\Re(a) & 0\\
2\Re(a) & -2\Im(a) & 0\\
2A_0 \Im(g) & 0 & 0\ep+\e^{-1} \bp 0 & 0 & 0\\ 0 & 0 & 0\\ 2A_0h & 0 & f\ep\right)\\
&\quad=i\bp -2\kappa\Im(a)-2A_0\kappa p\Im(g)-2\e^{-1}A_0h p & -2\kappa\Re(a) & -\e^{-1}f p\ep, \\
M_\s(\e,0)R_s&=i\left(\kappa \bp -2\Im(a) & -2\Re(a) & 0\\
2\Re(a) & -2\Im(a) & 0\\
2A_0 \Im(g) & 0 & 0\ep+\e^{-1} \bp 0 & 0 & 0\\ 0 & 0 & 0\\ 2A_0h & 0 & f\ep\right)\bp 1 \\ -q\\ 0 \ep\\
&\quad=i\bp -2\kappa\Im(a)+2\kappa\Re(a)q\\ 2\kappa\Re(a)+2\kappa\Im(a)q\\ 2A_0\kappa\Im(g)+2\e^{-1}A_0h \ep.
\ea
Let $S=1/(2A_0^2\Re(c))R_sL_s$ be the reduced resolvent. Formally applying Proposition \ref{prop:Kato511} again we have asymptotic expansions for the two neutral eigenvalues given by
\ba
\l_t(\e,\s)&=i(-2\kappa(\Re(a)q+\Im(a)))\s+\mu_t(\e)\s^2+\cO(\s^3),\\
\l_c(\e,\s)&=i(\e^{-1}(2A_0h b+f)+2A_0\kappa\Im(g)p)\s+\mu_c(\e)\s^2+\cO(\s^3),
\ea
with
\ba
\mu_t(\e)&=-\tilde{L}_t\bp \Re(a) & -\Im(a) & 0\\ \Im(a) & \Re(a) & 0 \\ 2A_0\Re(g) & 2A_0\Im(g) & e_B\ep R_t-\tilde{L}_t(M_\s(\e,0)SM_\s(\e,0))R_t,\\
\mu_c(\e)&=-L_c\bp \Re(a) & -\Im(a) & 0\\ \Im(a) & \Re(a) & 0 \\ 2A_0\Re(g) & 2A_0\Im(g) & e_B\ep\tilde{R}_c-L_c(M_\s(\e,0)SM_\s(\e,0))\tilde{R}_c.
\ea
Starting with $\mu_c(\e)$, we compute
\ba
L_cM_\s(\e,0)R_s&=i\bp 0 & 0 & 1\ep\bp -2\kappa\Im(a)+2\kappa\Re(a)q\\ 2\kappa\Re(a)+2\kappa\Im(a)q\\ 2A_0\kappa\Im(g)+2\e^{-1}A_0h \ep=i(2A_0\kappa\Im(g)+2\e^{-1}A_0h),\\
L_sM_\s(\e,0)\tilde{R}_c&=i\bp -2\kappa\Im(a)-2A_0\kappa p\Im(g)-2\e^{-1}A_0h p & -2\kappa\Re(a) & -\e^{-1}f p\ep\Bigg(\bp p \\ 0 \\ 1\ep\\
&\quad+F(\e)\bp 0 \\ 1 \\ 0 \ep \Bigg)=-i\e^{-1}p(2A_0h p+f)\\
&\quad-i\Big(p(2\kappa\Im(a)+2A_0\kappa p\Im(g))+F(\e)\kappa\Re(a)\Big).
\ea
So to leading order, we can write $\mu_c(\e)$ as
\be
\mu_c(\e)=-(-i)(i)\e^{-2}\frac{1}{2A_0\Re(c)}(2A_0h)(p((2A_0h p+f)))+\cO(\e^{-1})=-\e^{-2}\frac{h p}{\Re(c)}(2A_0h p+f)+\cO(\e^{-1}).
\ee
Recalling the definition of $p$, we have that $\mu_c(\e)$ is given by
\be
\mu_c(\e)=\e^{-2}\frac{\Re(d)h}{2A_0\Re(c)^2}\Big(-\frac{h\Re(d)}{\Re(c)}+f \Big)+\cO(\e^{-1})=\e^{-2}\frac{\Re(d)hf}{2A_0\Re(c)^3}\Big(\Re(c)-\frac{h\Re(d)}{f} \Big)+\cO(\e^{-1}),
\ee
where we've factored out a copy of $f/\Re(c)$ from the expression in the parentheses. That is, in terms of $\tilde{c}=c-dh/f$ from the Darcy reduction, $\mu_c(\e)$ is given by
\be
\mu_c(\e)=\e^{-2}\frac{\Re(d)fh}{2A_0\Re(c)^3}\Re(\tilde{c})+\cO(\e^{-1}).
\ee
This implies a few things. The first is that the neutral curve associated to the conserved quantity is either always stable or always unstable for all $|\kappa|<\kappa_E$ on the open set of parameters where $\Re(d)fh\not=0$, moreover, it also implies that the Darcy reduction carries some important stability information in the coarse notion of stability coming from super vs. subcritical as on each octant of $\Re(d)fh\not=0$, the sign of $\mu_c(\e)$ is given by $\pm \sgn(\Re(\tilde{c}))$ where $\pm$ depends on the octant.
\begin{lemma}
	Suppose \eqref{eq:truncatedmatrix} is such that $M(\e,0)$ has 0 as a semisimple eigenvalue. If the Darcy reduction is subcritical, i.e. $\Re\tilde{c}>0$, then the corresponding periodic solution to the truncated system is unstable.
\end{lemma}
\begin{proof}
	If $\Re\tilde{c}>0$, then necessarily we have $\frac{\Re(d)h}{f}<0$, or equivalently $\Re(d)hf<0$. As we are assuming $\Re(c)<0$, we see that the sign of $\mu_c(\e)$ is positive for $\e$ sufficiently small.
\end{proof}
We emphasize that the reverse statement is false, i.e. the periodic wave for the truncated system can be unstable while the Darcy model is simultaneously supercritical.\\

For $\mu_t(\e)$, we have that
\be
-\tilde{L}_t\bp \Re(a) & -\Im(a) & 0\\ \Im(a) & \Re(a) & 0 \\ 2A_0\Re(g) & 2A_0\Im(g) & e_B\ep R_t=-\bp q & 1 & -pq-F(\e) \ep \bp -\Im(a) \\ \Re(a) \\ 2A_0\Im(g)\ep=-\Re(a)+\Im(a)q+\cO(\e),
\ee
as $F(\e)=-pq+\cO(\e)$. Continuing along, we have
\ba
\tilde{L}_tM_\s(\e,0)R_s&=i(L_t-F(\e)L_c)\bp -2\kappa\Im(a)+2\kappa\Re(a)q\\ 2\kappa\Re(a)+2\kappa\Im(a)q\\ 2A_0\kappa\Im(g)+2\e^{-1}A_0h \ep\\
&\quad=i\Big(2\kappa\Re(a)q^2+2\kappa\Re(a)+(-pq-F(\e))(2A_0\kappa\Im(g)+2\e^{-1}A_0h) \Big),\\
L_sM_\s(\e,0)R_t&=i\bp -2\kappa\Im(a)-2A_0\kappa p\Im(g)-2\e^{-1}A_0h p & -2\kappa\Re(a) & -\e^{-1}f p\ep\bp 0 \\ 1 \\0\ep\\
&\quad=-2i\kappa\Re(a).
\ea
Hence, to leading order one has that
\be
\mu_t(\e)=-\Re(a)+\Im(a)q+\frac{2\kappa\Re(a)}{2A_0^2\Re(c)}(2\kappa\Re(a)q^2+2\kappa\Re(a)-2F'(0)A_0h)+\cO(\e).
\ee
For small $\e$, this is symmetric in $\kappa$ as we'll show that $F'(0)$, suitably interpreted, is proportional to $\kappa$. From the expansion of $F$ in Remark \ref{rmk:FLSvsTrun}, we can express $F'(0)$ as
\ba
F'(0)&=\frac{1}{(2A_0h p+f )}\big(2\kappa(-\Im(a)pq+\Re(a)p-p^2qA_0\Im(g))+pq(2A_0\kappa\Im(g)p+2\kappa(\Re(a)q+\Im(a))) \big)\\
&\quad= \frac{2\kappa}{2A_0h p+f}\big(-\Im(a)pq+\Re(a)p-p^2qA_0\Im(g)+p^2qA_0\Im(g)+\Re(a)pq^2+\Im(a)pq\Big)\\
&\quad= \frac{2\kappa p(1+q^2)\Re(a)}{2A_0h p+f}. 
\ea
At this point, we note that we can now show, modulo smoothness of the dispersion relations, that the real parts of the dispersion parts for the hyperbolic, truncated, and tilde models match for $|\s|$ sufficiently small.
\begin{theorem}\label{thm:singularcglspecagree}
	Let $P$ be the projection onto the 0-eigenspace of $M(\e,0)$ and that 0 is a semisimple eigenvalue. Further suppose that $P(M_{\Hs}(\e,0)+\e^{-1}\cS^{cGL}_{\Hs}(\e,0))P$ has all simple eigenvalues, where $\cS^{cGL}(\e,0)$ is the matrix
	\be
	\cS^{cGL}=\bp 0  & 0 & 0\\ 0 & 0 & 0 \\ 2\a h & 0 & f \ep,
	\ee
	and $\cS^{LS}$ is the corresponding matrix coming from the Lyapunov-Schmidt reduction. Assume that the corresponding families of eigenvalues and eigenprojections are all smooth functions of $\e,\Hs,\kappa,\vec{\b}$ and that they admit the following expansions uniformly in $\e,\Hs$
	\be
	\hat{\L}^J(\e,\s)=\hat{\L}_j^J(\e,\Hs)\Hs+\hat{\mu}_{jk}^J(\e,\Hs)\Hs^2+h.o.t. \ .
	\ee
	Then, $|\Re \hat{\mu}_{jk}^{cGL}(\e,\Hs)-\Re \hat{\mu}_{jk}^{LS}(\e,\Hs)|=\cO(\e,\Hs)$, hence the truncated model accurately predicts the necessary stability condition coming from Taylor expanding the dispersion relations.
\end{theorem}
\begin{proof}
	To finish the proof, we use the observation in \ref{rmk:FLSvsTrun} and Taylor expansion of $\tilde{L}_t$ as the tilde model has the same structure as the truncated model; with the difference being $\tilde{f}=f+\e f'$ and $\tilde{h}=h+\e h'$ in the coefficients for the singular terms as in Remark \ref{rmk:FLSvsTrun}. Taylor expanding $\mu_c(\e)$ with respect proves the claim for that eigenvalue and the observation in Remark \ref{rmk:FLSvsTrun} proves the claim for $\mu_t(\e)$ as the two choices of $F$ agree to second order in $\e$. However, one can also see that the imaginary parts of the dispersion relations do not agree between the two models as the imaginary part of $\mu_c$ does depend on $f$ and $h$.\\
\end{proof}
We make the crucial observation that the Ginzburg-Landau model does not predict the correct dispersion relation, as it easy to check that $|\L_j^{cGL}(\e,\s)-\L_j^{LS}(\e,\s)|=\cO(\e)$ for each $j$ from the leading order coefficient in the asymptotic expansion provided in Proposition \ref{prop:Kato511}.\\
\begin{remark}
	The reason that the hyperbolic model also correctly predicts the necessary condition stability is because $\mu_t$ and $\mu_c$ both turn out the independent of $e_B$ and $g$.
\end{remark}
We end this section with the final remark that even though $d/c\in\RR$ is not generic, it does occur naturally in some problems. We recall that \cite{HSZ} obtain an amplitude system where the coupling constant $d=0$. We see when $d=0$ that \eqref{eq:truncatedmatrix} is lower block triangular, and thus the spectrum splits into the top left 2$\times$2 block and the bottom right 1$\times$1 or 2$\times$2 block. The upper left block is the familiar matrix describing spectral stability for complex Ginzburg-Landau. The bottom right block is unconditionally stable due to the Turing hypotheses \ref{hyp:Lin}. Hence, what we see that the Eckhaus condition for complex Ginzburg-Landau is the only nontrivial stability criterion in the problems \cite{HSZ} consider. In particular, this shows that the coupling between $A$ and $B$ in the Ginzburg-Landau equation for $A$ has a profound effect on the stability criteria.

\subsection{Analyticity of the dispersion relations}
We continue this section with a proof that the dispersion relations are generically analytic.
\begin{theorem}\label{thm:truncatedanalytic}
	Generically, there exists an $\e_0>0$ and a $\s_0(\e_0)>0$ so that for each fixed $0<\e<\e_0$, the spectrum of \eqref{eq:truncatedmatrix} is real analytic with respect to $\s$ on $|\s|<\s_0\e$.
\end{theorem}
\begin{proof}
	For the moment, we fix $\e>0$. For the stable eigenvalue, this is an immediate consequence of the implicit function theorem as it is a simple eigenvalue and hence $\frac{\d}{\d\l}\det(M(\e,0)-\l Id)|_{\l=\l_{Stab}^{Trun}}\not=0$. A further consequence of the implicit function theorem is the existence of an analytic projection onto the stable subspace and we denote this projection by $Q(\s)$. Note that $Q(\s)$ admits the expansion $Q(\s)=R_sL_s+\s Q'(0)+\cO(\s^2)$ for $L_s,R_s$ as in \eqref{eq:truncatedleft} and \eqref{eq:truncatedeigenvectors} respectively. Correspondingly, we then have that $I-Q(0)=P$ for $P$ defined to be the projection onto the generalized 0-eigenspace of $M(\e,0)$. In addition, we have the normalization condition $Q(0)Q'(0)Q(0)=0$. We have the spectrum of $M(\e,\s)$ admits the decomposition $\{\l_{Stab}^{Trun}(\e,\s) \}\cup Spec\Big((I-Q(\s))M(\e,\s)(I-Q(\s)) \Big)$ and that $(I-Q(\s))M(\e,\s)(I-Q(\s))$ has coefficients which are analytic in $\s$. Define a new matrix $\tilde{M}(\e,\s)$ by 
	\be
	\tilde{M}(\e,\s):=(I-Q(\s))M(\e,\s)(I-Q(\s)).
	\ee
	Note that $\tilde{M}$ is real analytic with respect to $\s$ and has a Jordan block at $\s=0$. From Section 2.4.2 of \cite{K}, we have the existence of an analytic matrix function $U(\s)$ which has the following four properties
	\begin{enumerate}
		\item $U(\s)$ has the same domain in $\s$ as $Q(\s)$.
		\item $U(\s)$ satisfies the differential equation $U'(\s)=[Q'(\s),Q(\s)]U(\s)$ with initial condition $U(0)=I$, where $[A,B]$ denotes the commutator of $A$ and $B$.
		\item $U(\s)^{-1}$ exists for all $\s$ in the domain of $U(\s)$ and also analytically depends on $\s$.
		\item For all $\s$, $U(\s)Q(0)U(\s)^{-1}=Q(\s)$ and $U(\s)(I-Q(0))U(\s)^{-1}=I-Q(\s)$.
	\end{enumerate}
	We now conjugate $\tilde{M}$ by $U(\s)^{-1}$ to get
	\be\label{eq:tildeMconjugated}
	U(\s)^{-1}\tilde{M}(\e,\s)U(\s)=U(\s)^{-1}(I-Q(\s))U(\s)U(\s)^{-1}M(\e,\s)U(\s)U(\s)^{-1}(I-Q(\s))U(\s).
	\ee
	By the fourth property of $U(\s)$, \eqref{eq:tildeMconjugated} can be simplified to
	\be
	\hat{M}(\e,\s):=U(\s)^{-1}\tilde{M}(\e,\s)U(\s)=P\Big(U(\s)^{-1}M(\e,\s)U(\s)\Big)P.
	\ee
	The upshot is that we have replaced the $\s$ dependent projection $I-Q(\s)$ with a fixed rank 2 projection, namely $P=I-Q(0)$, which allows us to reduce the question of analyticity of the spectrum to a question about 2x2 matrices. As $\hat{M}$ is a conjugate of $\tilde{M}$, it has the same spectrum and so it will suffice to show that $\hat{M}$ has real analytic spectrum.\\
	
	Observe that $\hat{M}(\e,0)$ is given by
	\be
	\hat{M}(\e,0)=PU(0)^{-1}M(\e,0)U(0)P=PM(\e,0)P.
	\ee
	By the second property of $U(\s)$. Moving to the first derivative with respect to $\s$, we have that
	\be
	\hat{M}_\s(\e,0)=P(U(0)^{-1})'M(\e,0)P+PM_\s(\e,0)P+PM(\e,0)U'(0)P.
	\ee
	As a consequence of the product rule, we have that for $V(\s):=U(\s)^{-1}$
	\be
	V'(0)=-U'(0).
	\ee
	Moreover, the normalization on $Q(\s)$ in conjunction with the second property of $U(\s)$ ensures that
	\be
	-PV'(0)P=PU'(0)P=P[Q'(0),Q(0)]U(0)P=0.
	\ee
	Hence, we conclude
	\be
	\hat{M}_\s(\e,0)=PM_\s(\e,0)P.
	\ee
	Working in the $\{R_t,R_c\}$ basis, we see that we are in the following situation
	\be
	\hat{M}(\e,\s)=\bp 0 & \Xi\\
	0 & 0 \ep+i\s\bp C & s_1(\e)\\
	0 & s_2(\e)\ep+h.o.t. \,
	\ee
	where $\Xi$ is some constant whose value is unimportant, save for the observation that it is generically nonzero, and $C$ and $s_2(\e)$ are known values whose values generically differ on an open set of the form $0<\e<\e_0$. In both cases, they are generic in the sense that the relevant conditions hold on an open set of model parameters. For $\s\not=0$, we define the following invertible matrix
	\be\label{eq:Tsigma}
	T(\s):=\bp \s & 0\\
	0 & 1 \ep.
	\ee
	Conjugating $\hat{M}$ by $T(\s)$ gives
	\be\label{eq:balancing}
	\cM(\e,\s):=T(\s)\hat{M}(\e,\s)T(\s)^{-1}=\bp 0 & \s \Xi\\
	0 & 0\ep+\s\bp C & \s s_1(\e)\\
	0 & s_2(\e)\ep+h.o.t. \ .
	\ee
	Note that the above matrix is analytic in $\s$ on $0<|\s|<\s_0$ and admits an analytic extension to $\s=0$. Importantly, we have that
	\be
	\cM(\e,0)=0,
	\ee
	which is semisimple and as we will soon see has spectrum that splits to first order, and hence by \cite{K} has real analytic spectrum. We note that the balancing pulls down the coefficient of $\s^2$ given by
	\be\label{eq:secondorderanalyticproof1}
	s_3(\e):=L_cV'(0)\hat{M}_\s(\e,0)R_t+L_c\hat{M}_{\s\s}(\e,0)R_t+L_c\hat{M}_\s(\e,0)U'(0)R_t.
	\ee
	We recall the expression for $Q'(0)$ from Theorem 5.11 of \cite{K} as
	\be
	Q'(0)=-Q(0)M_\s(\e,0)S-SM_\s(\e,0)Q(0),
	\ee
	for $S$ the reduced resolvent for the stable eigenvalue. Plugging this expression into \eqref{eq:secondorderanalyticproof1} and using $V'(0)=-U'(0)=[Q'(0),Q(0)]$, we have that
	\be
	s_3(\e)=L_c\hat{M}_\s(\e,0)Q(0)\hat{M}_\s(\e,0)R_t+L_c\hat{M}_\s(\e,0)Q(0)\hat{M}_\s(\e,0)R_t+\cO(1)=\cO(\e^{-1}),
	\ee
	as $SR_t$ is proportional to $R_t$ and $\hat{M}_\e(\e,0)R_t=\cO(1)$. Hence $\cM_\s(\e,0)$ is of the form
	\be
	\cM_\s(\e,0)=\bp \cO(1) & \cO(1)\\ \cO(\e^{-1}) & \cO(\e^{-1})\ep=\e^{-1} \bp \cO(\e) & \cO(\e) \\ \cO(1) & \cO(1) \ep.
	\ee
	That is, $\cM_\s(\e,0)$ can be thought of as a small perturbation of a lower triangular matrix with distinct eigenvalues, and hence $\cM_\s(\e,0)$ has distinct eigenvalues.\\
	
	This concludes the proof of analyticity with respect to a fixed $\e$ because $\cM$ and $\hat{M}$ have the same spectrum and their shared spectrum coincides with $spec(M(\e,\s))\backslash\{\l_{Stab}^{Trun}(\e,\s)\}$.\\
	
	Turning to the question of the radius of convergence, we first prove that the radius of convergence for $\l_{Stab}^{Trun}$ is (generically) of size $\e$. First note that $M(\e,\s)$ admits a natural holomorphic extension to complex $\s$ in a neighborhood of $\s=0$, and so we have that each of the three eigenvalues admit natural holomorphic extensions to complex $\s$ by real analyticity and the identity theorem. Moreover, we can extend this neighborhood to $\CC\backslash D_\e$ where $D_\e$ is defined to be the set of $\s\in\CC$ such that the discriminant of the characteristic polynomial of $M(\e,\s)$ vanishes, as that discriminant vanishes precisely when the matrix $M(\e,\s)$ has repeated eigenvalues. In our case, the discriminant of the characteristic polynomial of $M(\e,\s)$ is a degree twelve polynomial in $\s$ of the form
	\be
	\cP(\e,\s):=\text{Disc}(\e,\s)=\sum_{j=2}^{12}c_j(\e)\s^j,
	\ee
	as there is a double root at $\s=0$. The $c_j(\e)$ are known polynomial functions of $\e^{-1}$, moreover, each $c_j(\e)$ is of at most degree 4. One can then define a new polynomial $\cQ(\s)$ by
	\ba
	\cQ(\s):=\lim_{\e\to0}\e^4\cP(\e,\s)&=-4f^4\Im(a)\s^8+16i f^4\kappa\Im(a)\Re(a)\s^7\\
	&\quad-4(-4 f^4\kappa^2\Re(a)^2-2 f^4\Im(a)\Im(c)A_0^2+2f^3h\Im(a)\Im(d)A_0^2)\s^6\\
	&\quad-4(4if^4\kappa\Re(a)\Im(c)A_0^2-4if^3h\Re(a)\Im(d)A_0^2)\s^5\\
	&\quad-4(-f^4\Re(c)^2A_0^4+2f^3h\Re(c)\Re(d)A_0^4-f^2h^2\Re(d)^2A_0^4)\s^4.
	\ea
	From this, we can draw three conclusions about the roots of $\cP(\e,\s)$. The first is that four of them go to infinity as $\e\to0$ since $\cQ$ is degree eight and $\cP$ is degree twelve and the roots of $\cP$, counting multiplicity, converge to the roots of $\cQ$ as $\e\to0$. The second conclusion is that $\cP$ has two roots going to zero as $\e\to0$. The third conclusion is that $\cP$ has four roots remaining distance $\cO(1)$ from 0, and these four roots of $\cQ$ can be computed explicitly but we will not need their expansions. Already we can see that the radius of analyticity should depend on $\e$ as generically one expects branching to occur when the roots of the characteristic polynomial collide. To get this rate at which the radius goes to 0 as $\e\to0$, we can instead look at
	\ba
	\tilde{\cQ}(\s):=\lim_{\e\to0}\cP(\e,\e\s)&=(4f^4\Re(c)^2A_0^4-8f^3\b\Re(c)\Re(d)A_0^4+4f^2h^2\Re(d)^2A_0^4 )\s^4+\\
	&\quad+\Big(16if^3\Re(c)^3A_0^6-64if^2\b\Re(c)^2\Re(d)A_0^6\\
	&\quad+80ifh^2\Re(c)\Re(d)^2A_0^6-32ih^3\Re(d)^3A_0^6\Big)\s^3\\
	&\quad+(-16f^2\Re(c)^4A_0^8+32fh \Re(c)^3\Re(d)A_0^8-16h^2\Re(c)^2\Re(d)^2A_0^8)\s^2.
	\ea
	Letting $q_1,q_2$ denote the two generically nonzero roots of $\tilde{\cQ}(\s)$ we get that the two small roots of $\cP(\e,\s)$ are of the form $\e q_j+o(\e)$ by continuity of roots. The two roots of the discriminant near zero would generically indicate that there are singular points in the spectrum, suggesting that the radius of convergence is comparable to $\e$.
\end{proof}
To show Theorem \ref{thm:truncatedanalytic} for more than one conservation law, one assumes that $-i\Pi_0 S_k(0,0)\Pi_0$ is diagonal with all distinct eigenvalues, and then chooses this basis of $\ker(M_0)$.
\be
R_t=\bp 0 \\ 1 \\ \vec{0}\ep \ R_{c_i}=\bp -p_i \\ 0 \\ e_i \ep,
\ee
here $e_i$ is the $i$-th element of the standard basis on $\RR^N$. Similarly, one chooses the corresponding basis of the left eigenspace to be
\be
L_t=\bp q & 1 & \vec{q} \ep \ L_{c_i}=\bp 0 & 0 & e_i\ep,
\ee
where $\vec{q}$ is chosen so that the left and right eigenvectors are suitably normalized. Then one can show that $PM_\s(\e,0)P$ is upper triangular in this basis with generically distinct eigenvalues. As a slight remark, all but one of these eigenvalues will be of size $\e^{-1}$. The techniques used to show Theorem \ref{thm:singularcglspecagree} for one conservation will also extend in a similar way.\\

An alternative argument that leads to the radius of analyticity most likely comparable to $\e$ is that the formula $n$-th derivative of $\l$ in \cite{K} includes what is morally an $n$-th power of $M_\s(\e,\s)$, which implies that one should expect $\d_\s^n\l(\e,0)\sim\e^{-n}$. That said, it technically possible that a miracle occurs and the roots of the discriminant of size $\e$ do not cause problems; so we will not claim that the radius of convergence is actually comparable to $\e$. \\

	For the $O(2)$-case, the question of regularity for the neutral eigenvalues is more subtle. Proposition \ref{prop:Kato511} of \cite{K} implies that the neutral eigenvalues are $C^1$ functions and that the derivatives are given by the eigenvalues of $PM_\s(0)P$. However, $PM_\s(0)P$ turns out to be a nontrivial Jordan block of eigenvalue 0 unless $\kappa fd=0$, which corresponds to an underlying constant state or one of the two equations is decoupled from the other. In spite of this, we still have smoothness of the dispersion relations in the $O(2)$ case. This was shown for the specific model arising from Swift-Hohenberg in \cite{S} using the Cardano formula.
\begin{theorem}\label{thm:O2smooth}
	The eigenvalues of the matrix function $M(\s)$ defined by
	\be
	M(\s):=\bp 2A_0^2c-a\s^2 & -2i\kappa a \s & A_0d\\
	2i\kappa a\s & -a\s^2 & 0\\
	-2A_0g\s^2			& 0 & -e_B\s^2\ep,
	\ee
	are generically analytic functions of $\s$.
\end{theorem}
\begin{proof}
	As before, there is a unique smooth eigenvalue $\l_{Stab}(\s)$ with $\l_{Stab}(0)<0$. We regard $M(\s)$ as a block matrix 
	\be
	M(\s)=\bp M_{11}(\s) & M_{12}(\s)\\
	M_{21}(\s) & M_{22}(\s)\ep,
	\ee
	with $M_{11}(\s)=2A_0^2c-a\s^2$, $M_{12}(\s)=\bp -2i\kappa a\s & A_0d\ep$, $M_{21}(\s)=\bp2i\kappa a \s & -2A_0g\s^2\ep^T$, and $M_{22}(\s)$ the diagonal matrix whose entries along the diagonal are $\bp -a\s^2 & -e_B\s^2\ep$. For the small eigenvalues, we note that $2A_0^2c-a\s^2-\l$ is nonzero, and so we for small $\s,\l$ we can compute the characteristic polynomial as
	\ba
	\det(M(\s)-\l)&=(2A_0^2c-a\s^2-\l)\det\Bigg(\bp -a\s^2-\l & 0\\ 0 & -e_B\s^2-\l\ep-\\
	&\quad-\frac{1}{2A_0^2c-a\s^2-\l}\bp 2i\kappa a\s\\ -2A_0g\s^2\ep \bp -2i\kappa a \s & A_0d\ep \Bigg).
	\ea
	The relevant part is the determinant of 
	\be\label{eq:O2reduceddet}
	\det\Bigg(\bp -a\s^2-\l & 0\\ 0 & -e_B\s^2-\l\ep-\frac{1}{2A_0^2c-a\s^2-\l}\bp 2i\kappa a\s\\ -2A_0g\s^2\ep \bp -2i\kappa a \s & A_0d\ep\Bigg).
	\ee
	We see that \eqref{eq:O2reduceddet} is the characteristic polynomial of 
	\be\label{eq:O2reducedmatrix}
	\hat{M}(\s,\l):=-\s^2\bp a & 0 \\ 0 & e_B\ep-\frac{1}{2A_0^2c-a\s^2-\l}\bp 2i\kappa a\s \\ -2A_0g\s^2\ep \bp -2i\kappa a\s & A_0d\ep.
	\ee
	We conjugate \eqref{eq:O2reducedmatrix} to get
	\ba\label{eq:O2reducedconj}
	\tilde{M}(\s,\l)&:=\bp 1 & 0 \\ 0 & \s^{-1} \ep\hat{M}(\s,\l)\bp 1 & 0 \\ 0 & \s\ep=\\
	&\quad=-\s^2\bp a & 0 \\ 0 & e_B\ep-\frac{\s^2}{2A_0^2c-a\s^2-\l}\bp 2i\kappa a \\ -2A_0g\ep \bp -2i\kappa a & A_0d\ep.
	\ea
	Taylor expanding with respect to $\l$ and $\s$, we find that \eqref{eq:O2reducedconj} is of the form
	\be\label{eq:O2reducedconj2}
	\tilde{M}(\s,\l)=-\s^2(D+\cR_1)+\cO(|\s|^2(|\s^2|+|\l|)),
	\ee
	where $D$ is diagonal and $\cR_1$ is rank one. We note that $\tilde{M}(\s,0)$ satisfies $\tilde{M}(\s,0)=\cO(\s^2)$ and so it has $C^2$ spectrum by Proposition \ref{prop:Kato511}, and one expects analyticity in general as $D+\cR_1$ generically has a full set of distinct eigenvalues. Appealing to the Weierstrass preparation theorem, we can choose an analytic function $q(\s,\l)$ so that $q(0,0)=1$ and one has
	\be
	q(\s,\l)\det(\tilde{M}(\s,\l)-\l Id)=\l^2+c_1(\s)\l+c_0(\s),
	\ee
	where the $c_j(\s)$ are analytic functions of $\s$. On the other hand, from \eqref{eq:O2reducedconj2} we see that $\det(\tilde{M}(\s,\l)-\l Id)$ is of the form $\det(\tilde{M}(\s,0)-\l Id)+h.o.t.$. Writing $\det(\tilde{M}(\s,0)-\l Id)=\l^2+\tilde{c}_1(\s)\l+\tilde{c}_0(\s)$, one then has that $c_j(\s)=\tilde{c}_j(\s)+h.o.t.$ .
\end{proof}
Provided one can diagonalize $\Pi_0 S_{kk}(0,0)\Pi_0$, the argument in Theorem \ref{thm:O2smooth} can be applied to any number of conservation laws.\\

\section{Lyapunov-Schmidt reduction: Existence and coperiodic stability}\label{sec:ExistenceLS}
In this section, with our ultimate goal being to connect with the multiscale expansion of Section \ref{sec:MSE}, we will use the Lyapunov-Schmidt reduction procedure to find traveling wave solutions to the following equation
\begin{equation}\label{eq:mastereqn}
\frac{\d u}{\d t}=L(k,\mu)u+\cN(u,k,\mu),
\end{equation}
where $L(k,\mu)$ has symbol $S(k,\mu)$ satisfying Hypotheses \ref{hyp:Lin} and $\cN$ is a quasilinear nonlinearity satisfying Hypotheses \ref{hyp:Nonlin}.\\

We assume that $u=u(\xi_0,t)$ with $\xi_0=kx$ is 2$\pi$ periodic with respect to $\xi_0$ and that $L(k,\mu)$ acts only on the $\xi_0$ variable, and a similar assumption for $\cN$. We write the solution $u$ to \eqref{eq:mastereqn} as $U(\xi_0-kdt)=u(\xi_0,t)$ for some real constant $d$ and some $U\in H^m_{per}([0,2\pi];\RR^n)$. Then by translation invariance, \eqref{eq:mastereqn} becomes
\begin{equation}\label{eq:mastereqn2}
L(k,\mu)U+dk\d_\xi U+\cN(U,k,\mu)=0.
\end{equation}
Define the following projection $P$ by
\begin{equation}\label{eq:projectiondef}
PU(\xi):=\Pi_0\hat{U}(0)+\Pi_1\hat{U}(1)e^{i\xi}+c.c.
\end{equation}
Split $U=PU+(I-P)U$ and write $V:=PU$ and $W=(I-P)U$. Solving \eqref{eq:mastereqn2} is equivalent to solving $(I-P)$\eqref{eq:mastereqn2}=0 and $P$\eqref{eq:mastereqn2}=0 simultaneously.\\


For the equation $(I-P)$\eqref{eq:mastereqn2}=0, we apply proposition \ref{prop:Tbounded} and the implicit function theorem to conclude there exists a smooth $\Phi(V;k,\mu,d)$ satisfying
\begin{equation}
(I-P)\left[L(k,\mu)+kd\d_\xi \right]\Phi(V;k,\mu,d)=-(I-P)L(k,\mu)V-(I-P)\cN(V+\Phi(V;k,\mu,d),k,\mu).
\end{equation}
From which we conclude that
\ba
&||P_2\Phi(V;k,\mu,d)||_{H^m}=\cO(||V||_{H^m}^2), \\
&||(I-P_2)\Phi(V;k,\mu,d)||_{H^m}=\cO(||V||_{H^m}),
\ea
where $P_2U$ is defined by
\begin{equation}
P_2U(\xi):=\sum_{|\eta|\geq 2}\hat{U}(\eta)e^{i\eta\xi}.
\end{equation}

This brings us to $P$\eqref{eq:mastereqn2}. We split $P$\eqref{eq:mastereqn2} into the system, with the equations corresponding to Fourier modes 0 and 1 respectively.
\begin{align}\label{eq:PSystem}
\Pi_0S(0,\mu)\left(\hat{V}(0)+\hat{W}(0) \right)+\Pi_0\widehat{\cN(V+W,k,\mu)}(0)=0,\\
\Pi_1\left(S(k,\mu)+idk \right)\left(\hat{V}(1)+\hat{W}(1) \right)+\Pi_1\widehat{\cN(V+W,k,\mu)}(1)=0.
\end{align}
By hypotheses 2 and 4 of \ref{hyp:Lin} and \ref{hyp:Nonlin} respectively, we see that the first equation in \eqref{eq:PSystem} vanishes identically. \cite{S} concluded the first equation in \eqref{eq:PSystem} vanished for their model Swift-Hohenberg equation by observing that it is of the form $k^2\d_\xi^2F(u,k,\mu)=0$. If one normalizes $\Pi_0$ to be the projection onto the first $r$ coordinates, then we can write the first $r$ equations of \eqref{eq:mastereqn2} as $k\d_\xi F(u,k,\mu,d)=0$. Hence our argument is the natural generalization of the argument in \cite{S}.\\

\begin{remark}
	It is somewhat surprising that the mode 0 equation vanishes identically in the Lyapunov-Schmidt reduction, whereas there were some complications in mode 0 for the multiscale expansion. This is because the ansatz \eqref{eq:ansatz} can also be used to construct pulse-like solutions, among others; and the the compatibility conditions arise in the low frequency behavior of the pulse.
\end{remark}
Turning towards the second equation of \eqref{eq:PSystem}, we choose $v_1,...,v_N\in\RR^n$ spanning $\ker(S(0,\mu))$ and write
\begin{equation}\label{eq:VExpansion}
V(\xi)=\Upsilon_{\e\a}(\xi)+\sum \e^2\b_iv_i,
\end{equation}
where $\b_i\in\RR$, and
\begin{equation}\label{eq:UpsilonDef}
\Upsilon_z(\xi):=\frac{1}{2}z e^{i\xi}r+c.c.,
\end{equation}
where $z\in\CC$ and $r$ spans $\ker(S(k_*,0))$. We assume that $\a=\cO(1)$ and that $\b_i=\cO(1)$ for each $i$. We further impose the scalings $\mu\sim\e^2$ and $k-k_*=\kappa$ with $\kappa=\cO(\e)$. We write the second equation of \eqref{eq:PSystem} as a scalar equation using the definition of $\Pi_1=r\ell$ as
\begin{equation}\label{eq:Mode1Eqn1}
\frac{1}{2}\e\a(\ell S(k,\mu)r+idk)+\ell S(k,\mu)\hat{W}(1)+\ell\widehat{\cN(V+W,k,\mu)}(1)=0.
\end{equation}
We see that there are essentially three terms and so we begin the simplification of \eqref{eq:Mode1Eqn1} with the first term
\begin{equation}\label{eq:Mode1Part1No1}
\ell S(k,\mu)r+idk=\ell \left[S(k_*,0)+\kappa S_k(k_*,0)+\frac{1}{2}\kappa^2 S_{kk}(k_*,0)+\mu S_\mu(k_*,0)+\cO(\e^3) \right]r+idk,
\end{equation}
by Taylor's theorem. By the identities
\begin{equation}\label{eq:SpecPertOrder1}
\tl_k(k_*,0)=\ell S_k(k_*,0)r\quad \tl_\mu(k_*,0)=\ell S_\mu(k_*,0)r,
\end{equation}
we can reduce \eqref{eq:Mode1Part1No1} to
\begin{equation}\label{eq:Mode1Part1No2}
\ell S(k,\mu)r+idk=\left(\tl(k_*,0)+\kappa\tl_k(k_*,0)+\mu\tl_\mu(k_*,0)+idk \right)+\frac{1}{2}\kappa^2\ell S_{kk}(k_*,0)r.
\end{equation}
This is as far as we can reduce \eqref{eq:Mode1Part1No1}, so we move on to the second term
\begin{equation}\label{eq:Mode1Part2}
\ell S(k,\mu)\hat{W}(1).
\end{equation}
In order to evaluate this expression, we need the Taylor expansion of $\hat{W}(1)$. For future convenience, we compute all the modes of $W$ up to order $\e^2$ that we will need in the following proposition.
\begin{proposition}\label{prop:WExpansion}
	$W=\Phi(V;k,\mu,d)$ admits the following Taylor series expansion at modes 0, 1, and 2 respectively.
	\ba\label{eq:WExpansionb2}
	\hat{W}(0)&=-\frac{1}{4}\e^2|\a|^2N_0\Re\cQ(k,-k,\mu)(r,\bar{r})+\cO(\e^3),\\
	\hat{W}(1)&=-\frac{1}{2}\e\kappa\a N_1(I-\Pi_1)S_k(k_*,0)\Pi_1r+\cO(\e^3),\\
	\hat{W}(2)&=-\frac{1}{8}(\e\a)^2\left(S(2k,\mu)+2idk \right)^{-1}\cQ(k,k;\mu)(r,r)+\cO(\e^3),
	\ea
	where $\cQ(k\eta_1,k\eta_2;\mu)$ denotes the Fourier multiplier of $D_u^2\cN(0,k,\mu)$.
\end{proposition}
\begin{proof}
	We recall the equation $W$ satisfies
	\begin{equation}
	(I-P)\left[L(k,\mu)+kd\d_\xi \right](I-P)W=-(I-P)L(k,\mu)V-(I-P)\cN(V+W,k,\mu).
	\end{equation}
	Taylor expanding the nonlinearity, and recalling $U=V+W=\cO(\e)$ and $P_2W=\cO(\e^2)$ we get
	\begin{equation}
	\begin{split}
	\cN(U,k,\mu)&=\frac{1}{2}D_u^2\cN(0,k,\mu)(U,U)+\cO(\e^3)\\
	&\quad=\frac{1}{2}\left(D_u^2\cN(0,k,\mu)(V,V)+2D_u^2\cN(0,k,\mu)(V,W)+Du^2\cN(0,k,\mu)(W,W) \right)+\cO(\e^3).
	\end{split}
	\end{equation}
	By Proposition \ref{prop:multilin}, there exists a function $\cQ:\ZZ^2\to \text{Sym}^2(\CC^n)$ where $\text{Sym}^2(\CC^n)$ denotes the vector space of symmetric bilinear forms on $\CC^n$ and satisfies
	\begin{equation}
	D_u^2\cN(0,k,\mu)(u,v)(\xi)=\sum_{\eta_1,\eta_2\in\ZZ}\cQ(k\eta,k\eta_2;\mu)(\hat{u}(\eta_1),\hat{v}(\eta_2))e^{i(\eta_1+\eta_2)\xi}.
	\end{equation}
	We also observe that since $D_u^2\cN(0,k,\mu)$ is a symmetric bilinear form, $\cQ(k\eta_1,k\eta_2;\mu)=\cQ(k\eta_2,k\eta_2;\mu)$.\\
	
	So, we get the following system for modes 0, 1, and 2.
	\begin{align}
	\hat{W}(0)&=-\frac{1}{2}N_0(I-\Pi_0)\sum_{\eta\in\ZZ}\cQ(k\eta,-k\eta;\mu)(\hat{U}(\eta),\hat{U}(-\eta)),\\
	\hat{W}(1)&=-\frac{1}{2}N_1(I-\Pi_1)(S(k,\mu)+idk)(\e\a r)-\frac{1}{2}N_1(I-\Pi_1)\sum_{\eta\in\ZZ}\cQ(k(\eta+1),-k\eta)(\hat{U}(\eta+1),\hat{U}(-\eta)),\\
	\hat{W}(2)&=-\frac{1}{2}\left(S(2k,\mu)+2idk \right)^{-1}\sum_{\eta\in\ZZ}\cQ(k(\eta+1),k(-\eta+1))(\hat{U}(\eta+1),\hat{U}(-\eta+1)),
	\end{align}
	where
	\begin{align}
	N_0=N_0(\mu)&:=\left((I-\Pi_0)S(0,\mu)(I-\Pi_0) \right)^{-1},\\
	N_1=N_1(\mu,k,d)&:=\left((I-\Pi_1)(S(k,\mu)+idk)(I-\Pi_1) \right)^{-1}.
	\end{align}
	Observe that
	\begin{equation}
	\begin{split}
	(I-\Pi_1)(S(k,\mu)+idk)\Pi_1&=(I-\Pi_1)\left[S(k_*,0)+S_k(k_*,0)(\kappa)+\cO(\e^2)+idk \right]\Pi_1
	\\&\quad	=\kappa(I-\Pi_1)S_k(k_*,0)\Pi_1+\cO(\e^2).
	\end{split}
	\end{equation}
	and so we conclude that $|\hat{W}(1)|=\cO(\e^2)$. Moreover, the nonlinearity in mode 1 doesn't contribute since the only $\cO(\e)$ terms have frequency $\pm1$ and so we get
	\begin{equation}
	\hat{W}(1)=-\frac{1}{2}\e\kappa\a N_1(I-\Pi_1)S_k(k_*,0)\Pi_1r+\cO(\e^3).
	\end{equation}
	Since $\hat{W}(1)=\cO(\e^2)$, we turn to mode 2 now where we find
	\begin{equation}
	\hat{W}(2)=-\frac{1}{8}\e^2\a^2\left(S(2k,\mu)+2idk \right)^{-2}\cQ(k,k;\mu)(r,r)+\cO(\e^3).
	\end{equation}
	Finally, we look at mode 0, where we have that
	\begin{equation}
	\hat{W}(0)=-\frac{1}{4}\e^2|\a|^2N_0\Re\cQ(k,-k,\mu)(r,\bar{r})+\cO(\e^3).
	\end{equation}
\end{proof}
\begin{remark}
	Note that $2\e^{-2}\hat{W}(1)=(I-\Pi_1)\Psi_1+h.o.t.$ and similarly, $2\e^{-2}\hat{W}(2)=\Psi_2+h.o.t.$ as expected. Moreover, $\e^{-2}\hat{W}(0)=\Psi_0+h.o.t.$ holds provided one identifies $\a$ with $A$.
\end{remark}
Now that we have this proposition in hand, we return to our goal of computing $\ell S(k,\mu)\hat{W}(1)$. Taylor expanding $S(k,\mu)$ and plugging in the result from Proposition \ref{prop:WExpansion}, we get
\begin{equation}
\begin{split}
\ell S(k,\mu)\hat{W}(1)&=\ell\left(S(k_*,0)+\kappa S_k(k_*,0)+\cO(\e^2) \right)\\
&\quad*\left(-\frac{1}{2}\e\kappa\a N_1(I-\Pi_1)S_k(k_*,0)\Pi_1 r+\cO(\e^3)\right),
\end{split}
\end{equation}
where in this instance $*$ refers to multiplication of scalars. Since $\ell S(k_*,0)=0$, we get the following expansion
\begin{equation}
\ell S(k,\mu)\hat{W}(1)=-\frac{1}{2}\a\e^3\tilde{\kappa}^2\ell S_k(k_*,0)(I-\Pi_1)N_1(I-\Pi_1)S_k(k_*,0)r+\cO(\e^4).
\end{equation}

Finally, we come to the contribution from the nonlinearity in \eqref{eq:Mode1Eqn1}. First, we Taylor expand the nonlinearity as
\begin{equation}
\begin{split}
\cN(U,k,\mu)&=\frac{1}{2}D_u^2\cN(0,k_*,0)(U,U)+\frac{1}{6}D_u^3\cN(0,k_*,0)(U,U,U)\\
&\quad+\frac{1}{2}\kappa\d_kD_u^2\cN(0,k_*,0)(U,U)+\frac{1}{2}\mu\d_\mu D_u^2\cN(0,k_*,0)(U,U)+\cO(\e^4).
\end{split}
\end{equation}
Note that $\mu\sim\e^2$ and $U=\cO(\e)$ implies that the $\d_\mu$ term is absorbed by the error. Moreover, each of the above is a multilinear multiplier operator by Proposition \ref{prop:multilin}, so we can compute the mode 1 term of $\cN(V+W,k,\mu)$ as
\begin{equation}\label{eq:nonlinTaylor}
\begin{split}
\widehat{\cN(V+W,k,\mu)}(1)&=\cQ(2k_*,-k_*;0)\left(\hat{W}(2),\frac{1}{2}\e\bar{\a}\bar{r}\right)+\cQ(0,k_*;0)\Big(\hat{W}(0)+\e^2\sum_i\b_i v_i,\frac{1}{2}\e\a r\Big)\\
&\quad+\frac{1}{2}\e^3\cC(k_*,k_*,-k_*;0)\left(\frac{1}{2}\a r,\frac{1}{2}\a r,\frac{1}{2}\bar{\a}\bar{r} \right)+\cO(\e^4).
\end{split}
\end{equation}
Note that the $\d_k D_u^2\cN(0,k_*,0)$ does not appear, this is because the $\cO(\e^3)$ term is $\d_k D_u^2\cN(0,k_*,0)(V,V)$; which has Fourier support $\{0,\pm 2\}$. Plugging in the results from Proposition \ref{prop:WExpansion} into \eqref{eq:nonlinTaylor} produces
\begin{equation}\label{eq:nonlinTaylor2}
\begin{split}
\widehat{\cN(V+W,k,\mu)}(1)&=-\e^3\frac{1}{16}|\a|^2\a \cQ(2k_*,-k_*;0)\left(\left(S(2k_*,0)+2id_*k_* \right)^{-1}\cQ(k_*,k_*;0)(r,r),\bar{r}\right)\\
&\quad-\frac{1}{8}\e^3|\a|^2\a \cQ(0,k_*;0)\left(N_0\Re\cQ(k_*,-k_*;0)(r,\bar{r}),r\right)\\
&\quad+\frac{1}{2}\e^3\a \cQ(0,k_*;0)(r,\sum\b_iv_i) +\frac{1}{16}\e^3|\a|^2\a\cC(k_*,k_*,-k_*;0)\left(r,r,\bar{r} \right)+\cO(\e^4).
\end{split}
\end{equation}
Comparing \eqref{eq:nonlinTaylor2} with \eqref{eq:gammadef}, we see that the reduced equation is given by
\begin{equation}\label{eq:reducedeqn1}
\begin{split}
&\frac{1}{2}\e\a(\ell S(k,\mu)r+idk)+\ell S(k,\mu)\hat{W}(1)+\ell\widehat{\cN(V+W,k,\mu)}(1)\\
&\quad=\frac{1}{2}\e\a\left(\left(\tl(k_*,0)+\kappa\tl_k(k_*,0)+\mu\tl_\mu(k_*,0)+idk \right)+\frac{1}{2}\kappa^2\ell S_{kk}(k_*,0)r\right)\\
&\quad-\frac{1}{2}\a\e^3\tilde{\kappa}^2\ell S_k(k_*,0)(I-\Pi_1)N_1(I-\Pi_1)S_k(k_*,0)r+\g \e^3|\a|^2\a\\
&\quad+\frac{1}{2}\e^3\a \cQ(0,k_*;0)(r,\sum\b_iv_i)+\cO(\e^4).
\end{split}
\end{equation}
An application of \eqref{eq:spectralid} gives our second form for the reduced equation
\begin{equation}\label{eq:reducedeqn2}
\begin{split}
&\frac{1}{2}\e\a(\ell S(k,\mu)r+idk)+\ell S(k,\mu)\hat{W}(1)+\ell\widehat{\cN(V+W,k,\mu)}(1)\\
&\quad=\frac{1}{2}\e\a\left(\kappa\tl_k(k_*,0)+\mu\tl_\mu(k_*,0)+idk-\frac{1}{2}\kappa^2\tl_{kk}(k_*,0)+\g\e^2|\a|^2+\e^2\ell\cQ(k_*,0;0)(r,\sum \b_iv_i)  \right)+\cO(\e^4).
\end{split}
\end{equation}
We would like to divide by the leading $\a$, and so we need the following result
\begin{lemma}
	For $W=\Phi(V,k,\mu,d)=\Phi(\a,\b_1,...,\b_N,k,\mu,d)$, we have that $\Phi(0,\b_1,...,\b_N,k,\mu,d)$ is constant in $\xi$.
\end{lemma}
\begin{proof}
	This follows from translation-invariance and uniqueness of $\Phi$ as when $\a=0$ one has that $\tau_y V(\xi)=V(\xi)$ for all $y$ and hence $\tau_y \Phi=\Phi$.
\end{proof}
Applying the lemma to divide out $\a$, setting $\vec{\b}:=\sum\b_iv_i$ and comparing with Theorem \ref{thm:O2multilin} we get the reduced equation.
\begin{equation}\label{eq:reducedeqn3}
\tl(k_*,0)+\kappa\tl_k(k_*,0)+\mu\tl_\mu(k_*,0)+idk-\frac{1}{2}\kappa^2\tl_{kk}(k_*,0)+\g\e^2|\a|^2+\e^2V_1\cdot\vec{\b}+\cO(\e^3)=0.
\end{equation}

Taking the real part of \eqref{eq:reducedeqn3}, we get by the Turing hypotheses
\begin{equation}
	\mu\Re\tl_\mu(k_*,0)+\frac{1}{2}\kappa^2\Re\tl_{kk}(k_*,0)+\Re\e^2\g|\a|^2+\e^2\Re V_1\cdot\vec{\b}+\cO(\e^3)=0.
\end{equation}
Rearranging, we get an equation for $|\a|^2$ as follows
\begin{equation}
	|\a|^2=-\frac{\mu\Re\tl_\mu(k_*,0)+\frac{1}{2}\kappa^2\Re\tl_{kk}(k_*,0)+\Re\e^2 V_1\cdot\vec{\b}}{\Re\e^2\g}+\cO(\e).
\end{equation}
Introducing the scalings $\mu=\e^2\tilde{\mu}$ and $\kappa=\e\tilde{\kappa}$ and dividing by $\e^2$, we find the equation
\begin{equation}
	|\a|^2=-\frac{\tilde{\mu}\Re\tl_{\mu}(k_*,0)+\frac{1}{2}\kappa^2\Re\tl_{kk}(k_*,0)+\Re V_1\cdot\vec{\b}} {\Re\g}+\cO(\e).
\end{equation}
Note that this equation has a unique positive solution for $\a$ if it has a solution at all, and that the corresponding solution $\a$ matches the prediction of the modified (cGL) model up to $\cO(\e)$.\\

The imaginary part of \eqref{eq:reducedeqn3} is given by
\begin{equation}
	\kappa\Im\tl_k(k_*,0)+\mu\Im\tl_\mu(k_*,0)+idk-\frac{1}{2}\tilde{\kappa}^2\Im\tl_{kk}(k_*,0)+\Im\g\e^2|\a|^2+\e^2\Im V_1\cdot\vec{\b}+\cO(\e^3)=0.
\end{equation}
As $k_*\not=0$, we can apply the implicit function theorem to solve for $d$ as a function of $(\e,\kappa,\vec{\b})$. Overall, what we've shown is the following theorem.
\begin{theorem}\label{thm:LSReduction}
	There exists a small open subset $\cU\subset\RR^{1+N}$ containing 0 so that for all $(\tilde{\kappa},\vec{\b})\in\cU$ and all $\e$ sufficiently small, there exists a unique solution $(\tilde{u}_{\e,\tilde{\kappa},\vec{\b}},d)$ to \eqref{eq:mastereqn2} depending smoothly on $\e$, $\vec{\b}$ and $\tilde{\kappa}$ such that $\a>0$ at $(\tilde{\kappa},\vec{\b})=0$. Moreover, $\tilde{u}_{\e,\tilde{\kappa},\vec{\b}}$ admits the expansion
	\begin{equation}
		\tilde{u}_{\e,\tilde{\kappa},\vec{\b}}=\frac{1}{2}\e\sqrt{-\frac{\tilde{\mu}\Re\tl_{\mu}(k_*,0)+\frac{1}{2}\kappa^2\Re\tl_{kk}(k_*,0)+\Re V_1\cdot\vec{\b}} {\Re\g}}e^{i\xi}r+c.c.+\e^2\vec{\b}+\cO(\e^2).
	\end{equation}
\end{theorem}
As in \cite{WZ1}, this result can be easily extended to the case of translation invariant pseudodifferential operators $L(k,\mu)$ and general translation invariant nonlinearities $\sN:H^s_{per}\to L^2_{per}$ satisfying the appropriate modifications of \eqref{hyp:Lin} and \eqref{hyp:Nonlin} respectively.
\begin{remark}
	It was remarked in the introduction of \cite{WZ1}, that the existence result when $S(0,\mu)$ had incomplete rank could be obtained from the existence result when $S(0,\mu)$ had full rank. The approach described there does not allow for an easy computation of the resulting solution however. The Lyapunov-Schmidt procedure used here allows one to compute the Taylor expansion of the solution in an efficient manner.
\end{remark}

In order to justify the singular model, we show how it is reflected in the reduced equation for linearized stability. This is because the singular model arises from the small non-zero frequency behavior as evidenced by the coefficients involving $S_k(0,0)$ and $S_{kk}(0,0)$. There is another avenue to explore this behavior, and that is to look at the existence of pulse-like solutions instead as pulses can be Fourier supported in a neighborhood of 0, but we will not pursue this topic here.\\

For the moment, we will restrict attention to differential operators $L(k,\mu)$ and local quasilinear nonlinearities $\cN(U,\mu)$. To reduce some of the notational clutter, we adopt the shift in notation $k=k_*+\e\kappa$. We will follow the procedure in \cite{WZ2} and define the Bloch-type operator
\begin{equation}\label{eq:Blochdef}
	B(\e,\kappa,\s,\l,\vec{\b}):=L(k,\mu;\s)+dk\d_\xi+i\s C(\e,\kappa)+D_U\cN(\tilde{u}_{\e,\kappa,\vec{\b}};k,\mu,\s )-\l,
\end{equation}
where $\tilde{u}_{\e,\kappa,\b}$ is the solution constructed in Theorem \ref{thm:LSReduction} and $L(k,\mu;\s)$ is the operator defined by
\begin{equation}
	L(k,\mu;\s):=\sum_{j=0}^mk^j(\d_\xi+i\s)^j\cL_j(\mu),
\end{equation}
$C(\e,\kappa)$ is the constant
\begin{equation}\label{eq:Cepskappa}
	C(\e,\kappa):=k\left(d_*+\frac{k_*d_\e(0,\kappa)}{\kappa}\right),
\end{equation}
and $D_U\cN(\bar{u};k,\mu,\s)$ is defined for $\bar{u}\in H^m_{per}(\RR;\RR^n)$ by
\begin{equation}\label{eq:locDUcN}
	D_U\cN(\bar{u};k,\mu,\s)v:=\sum_{j=0}^mF_j(\bar{u};k,\mu)k^j(\d_\xi+i\s)^jv,
\end{equation}
where $F_j$ are known smooth matrix-valued functions of $\bar{u}$ and its derivatives arising from the Fr\'echet derivative of $\cN$ with respect to $U$.
\begin{remark}
	Observe that $D_U\cN(\bar{u};k,\mu)$ has Schwartz kernel $\cK(\xi,\nu;k,\bar{u},\mu)$ given by
	\begin{equation}
		\cK(\xi,\nu;k,\bar{u},\mu)=\sum_{j=0}^mF_j(\bar{u};k,\mu)(i\nu)^j.
	\end{equation}
	Hence the operator $D_U\cN(\bar{u};k,\mu,\s)$ can equivalently be defined in terms of the Schwartz kernel as
	\begin{equation}
		D_U\cN(\bar{u};k,\mu,\s)v(\xi)=\sum_{\eta\in\ZZ}\cK(\xi,k(\eta+\s);k,\bar{u},\mu)\hat{v}(\eta)e^{i\eta\xi}.
	\end{equation}
	This alternative viewpoint will be useful when working with nonlocal nonlinearities.
\end{remark}
We let our prospective eigenfunction $W(\xi)$ be $2\pi$-periodic in $\xi$ and satisfy the equation $B(\e,\kappa,\s,\l,\vec{\b})W=0$. We then expand $W$ as
\begin{equation}
W(\xi)=\Upsilon_a(\xi)+\e\vec{b}+\cV,
\end{equation}
where $P\cV=0$, $a\in\CC$ and $\vec{b}\in\ker(S(0,\mu))$.\\

Solving $B(\e,\kappa,\l,\s,\vec{\b})W=0$ is equivalent to solving $PB(\e,\kappa,\l,\s,\vec{\b})W=(I-P)B(\e,\kappa,\s,\l,\vec{\b})W=0$ simultaneously, and we first consider $(I-P)B(\e,\kappa,\s,\l,\vec{\b})W=0$.
\begin{proposition}\label{prop:CVExpansion}
	Consider the equation
	\begin{equation}\label{eq:IPBeqn}
	(I-P)B(\e,\kappa,\s,\l,\vec{\b})W=0.
	\end{equation}
	Then there is a unique smooth function $\cV=\cV(a,\vec{b};\e,\kappa,\l,\s,\vec{\b})$, defined for $\e,\l$ small, $|\s|<1/2$ and all $a\in\CC$, $\vec{\b}\in\ker(S(0,\mu))$. Moreover $\cV$ is linear in $a,\vec{b}$ and satisfies
	\begin{enumerate}
		\item $\cV(a,\vec{b};0,\kappa,\l,0,\vec{\b})\equiv0$,
		\item $\cV_\e(a,\vec{b};0,\kappa,\l,0,\vec{\b})=-T_\l(I-P)B_\e(0,\kappa,\l,0,\vec{\b})\Upsilon_a$,
		\item $\cV_\s(a,\vec{b};0,\kappa,\l,0,\vec{\b})=-T_\l(I-P)B_\s(0,\kappa,\l,0,\vec{\b})\Upsilon_a$,
	\end{enumerate}
	where we define $T_\l$ to be the operator
	\begin{equation}\label{eq:Tlambdadef}
	T_\l:=((I-P)(B(0,\kappa,\l,0,\vec{\b})(I-P) )^{-1}.
	\end{equation}
\end{proposition}
\begin{proof}
	Note \eqref{eq:IPBeqn} is equivalent to
	\begin{equation}\label{eq:IPBeqn2}
	(I-P)B(\e,\kappa,\s,\l,\vec{\b})(I-P)\cV=-(I-P)B(\e,\kappa,\s,\l,\vec{\b})(\Upsilon_a+\e\vec{b}).
	\end{equation}
	As $B(0,\kappa,\l,0,\vec{\b})=L(k_*,0;0)+d_*k_*\d_\xi-\l$, we know that $(I-P)B(0,\kappa,0,0,\vec{\b})(I-P)$ is invertible with bounded inverse because $P$ is the projection onto the kernel of this operator. Hence by continuity $(I-P)B(\e,\kappa,\l,\s,\vec{\b})(I-P)$ remains invertible for $\e,\ \l,\ \s$ small. Hence we have existence by the implicit function theorem. (1) immediately follows from evaluating \eqref{eq:IPBeqn2} at $\e=\s=0$. To prove (2), we differentiate \eqref{eq:IPBeqn2} with respect to $\e$ to get
	\begin{equation}
	(I-P)B_\e(\e,\kappa,\s,\l,\vec{\b})\cV+(I-P)B(\e,\kappa,\s,\l,\vec{\b})\cV_\e=-(I-P)B_\e(\e,\kappa,\s,\l,\vec{\b})(\Upsilon_a+\e\vec{b})-(I-P)B(\e,\kappa,\s,\l,\vec{\b})\vec{b}.
	\end{equation}
	Evaluating at $\e=\s=0$ and noting that $\vec{b}\in\ker(B(0,\kappa,0,0,\vec{\b}))$ we find that
	\begin{equation}
	(I-P)B(0,\kappa,0,\l,\vec{\b})\cV_\e=-(I-P)B_\e(0,\kappa,0,\l,\vec{\b})\Upsilon_a.
	\end{equation}
	By (1). Applying $T_\l$ finishes the proof of (2). (3) is entirely similar to (2).
\end{proof}
We now turn to solving $PBW=0$. Note that $PBW=0$ is a linear equation in $(\Re a,\Im a,\vec{b})$ by Proposition \ref{prop:CVExpansion} and that $PBW$ can be identified with a vector in $\RR^{2+N}$, hence $PBW=0$ is of the form $M(\e,\kappa,\s,\l,\vec{\b})(\Re a,\Im a,\vec{b})=0$ for some matrix $M$. In order to get the reduced equation, equivalently compute the matrix $M$, we Taylor expand $PB(\e,\kappa,\s,\l,\vec{\b})$ to second order in $(\e,\s)$ and linearize $\cV$ about $(\e,\s)=0$. Doing so gives
\ba\label{eq:PBTaylor}
&P\left[B(0,\kappa,\l,0,\vec{\b})+\e B_\e(0,\kappa,\l,0,\vec{\b})+\frac{1}{2}\e^2 B_{\e\e}(0,\kappa,\l,0,\vec{\b})\right. \\
&\quad\left.+\s B_\s(0,\kappa,\l,0,\vec{\b})+\frac{1}{2}\s^2B_{\s\s}(0,\kappa,\l,0,\vec{\b})+\e\s B_{\e\s}(0,\kappa,\l,0,\vec{\b})+\cO(\e^3,\e^2\s,\e\s^2,\s^3)\right]\cdot\\
&\quad(\Upsilon_a+\e\vec{b}+\e \cV_\e(a,\vec{b};0,\kappa,\l,0,\vec{\b})+\s\cV(a,\vec{b};0,\kappa,\l,0,\vec{\b})+\cO(\e^2,\e\s,\s^2))=0.
\ea
For now, we will be content with computing the ``co-periodic'' terms, namely
\ba\label{eq:PBTaylorCoperiod}
	P\left(B+\e B_\e+\frac{1}{2}\e^2 B_{\e\e}+\cO(\e^3)\right)(\Upsilon_a+\e\vec{b}+\e \cV_\e+\cO(\e^2)).
\ea
\begin{proposition}\label{prop:CV-Phi}
	Let $a=a_1+ia_2$ for $a_i\in\RR$. Fix a basis $v_1,...,v_N$ of $\ker(S(0,\mu))$ and write $\vec{b}=\sum b_iv_i$ for $b_i\in\RR$. One has the identities
	\begin{equation}
		\e\d_{a_i}\cV(a,\vec{b};\e,\kappa,0,0,\vec{\b})=\d_{\a'_i}\Phi(\a,\vec{\b};k,\mu,d), \quad i=1,2,
	\end{equation}
	and
	\begin{equation}
		\e\d_{b_i}\cV(a,\vec{b};\e,\kappa,0,0,\vec{\b})=\d_{\b'_i}\Phi(\a,\vec{\b};k,\mu,d), \quad i=1,2,...,N,
	\end{equation}
	Where $\a$ is as in Theorem \ref{thm:LSReduction}, $\vec{\b}=\sum\b_i v_i$, and $\Phi=(I-P)\tilde{u}_{\e,\kappa,\vec{\b}}$ is regarded as a map $\Phi=\Phi(\a',\vec{\b}';k,\mu,d)$.
\end{proposition}
Morally, this proposition says that $\cV$ is a sort of linearization of $(I-P)\tilde{u}_{\e,\kappa,\vec{\b}}$.
\begin{proof}
	By construction $\Phi$ and $\cV(a,\vec{b};\e,\kappa,0,0,\vec{\b})$ satisfy the equations
	\begin{equation}
		(L(k,\mu)+dk\d_\xi)(\e\Upsilon_\a+\e^2\vec{\b}+\Phi)+\cN(\e\Upsilon_\a+\e^2\vec{\b}+\Phi;k,\mu)=0,
	\end{equation}
	and
	\begin{equation}
		B(\e,\kappa,0,0,\vec{\b})(\Upsilon_a+\e\vec{b}+\cV)=(L(k,\mu;0)+dk\d_\xi+D_U\cN(\tilde{u}_{\e,\kappa,\b};k,\mu,0 ))(\Upsilon_a+\e\vec{b}+\cV)=0.
	\end{equation}
	Applying $(I-P)$ to both equations and differentiating with respect to $\a_1'$ and evaluating at $(\a',\vec{\b}')=(\a,\vec{\b})$ in the first and $a_1$ in the second yields
	\ba
		&(I-P)[L(k,\mu)+dk\d_\xi+D_U\cN(\e\Upsilon_\a+\e^2\vec{\b}+\Phi;k,\mu)](\e \Upsilon_1+\d_{\a_1'}\Phi)=0,\\
		&(I-P)[L(k,\mu)+dk\d_\xi+D_U\cN(\e\Upsilon_\a+\e^2\vec{\b}+\Phi;k,\mu)](\Upsilon_1+\d_{a_1}\cV)=0.
	\ea
	Multiplying the second equation by $\e$ and solving for $\e\d_{a_1}\cV$ and $\d_{\a_1}\Phi$, we see that
	\begin{equation}
		(I-P)[L(k,\mu)+dk\d_\xi+D_U\cN(\tilde{u}_{\e,\kappa,\vec{\b}};k,\mu)]\e\d_{a_1}\cV=(I-P)[L(k,\mu)+dk\d_\xi+D_U\cN(\tilde{u}_{\e,\kappa,\vec{\b}};k,\mu)]\d_{\a_1'}\Phi.
	\end{equation}
	But this operator is invertible on the subspace $(I-P)H^m_{per}([0,2\pi];\RR^n)$, so $\e\d_{a_1}\cV=\d_{\a_1'}\Phi(\a,\vec{\b};k,\mu,d)$ must hold. The other identities are obtained in an analogous manner.
\end{proof}
The most important corollary of this result is the reduced equation for co-periodic stability.
\begin{theorem}\label{thm:coperiodicstability}
		$PB(\e,\kappa,\l,0,\vec{\b})W$ in reduced form is given by
		\begin{equation}
			-\l\bp a_1\\ a_2\\b_1\\\vdots\\b_N\ep+\bp M_{11} & M_{12}\\
														M_{21} & M_{22}\ep \bp a_1\\ a_2\\b_1\\\vdots\\b_N\ep+\cO(\e^2|\l|)=0,
		\end{equation}
		where $M_{11}(\e,\kappa,\vec{\b})$ is the $2\times 2$ matrix
		\begin{equation}
			M_{11}=\bp 2\e^2\a^2\Re\g+\cO(\e^3) & 0\\
					2\e^2\a^2\Im\g+\cO(\e^3) & 0\ep,
		\end{equation}
		$M_{12}(\e,\kappa,\vec{\b})$ is the $2\times N$ matrix with entries in the $i$-th column
		\begin{equation}
			(M_{12}(\e,\kappa,\vec{\b}))^i=\bp \e^2\a\Re V_1\cdot v_i+\cO(\e^3)\\
			\e^2\a\Im V_1\cdot v_i+\cO(\e^3)\ep,
		\end{equation}
		$M_{21}$ is the $N\times 2$ matrix with all entries identically 0, and $M_{22}$ is the $N\times N$ matrix will all entries 0.
\end{theorem}
\begin{proof}
	We look at the other equation for the existence problem, which we recall is
	\begin{equation}
		P(L(k,\mu)+dk\d_\xi)(\e\Upsilon_\a+\e^2\vec{\b}+\Phi)+P\cN(\e\Upsilon_\a+\e^2\vec{\b}+\Phi)=0.
	\end{equation}
	Differentiating with respect to $\a_1$ gives
	\begin{equation}
		P(L(k,\mu)+dk\d_\xi+D_U\cN(\e\Upsilon_\a+\e^2\vec{\b}+\Phi))(\e\Upsilon_1+\d_{\a_1}\Phi )=0.
	\end{equation}
	By Proposition \ref{prop:CV-Phi}, this coincides with
	\begin{equation}
		\e PB(\e,\kappa,0,0,\vec{\b})(\Upsilon_1+\d_{a_1}\cV(a,\vec{b};\e,\kappa,0,0,\vec{\b}))=0.
	\end{equation}
	A similar argument will work for the other variables. This implies that the coefficients of the reduced equation for $PB(\e,\kappa,0,0,\vec{\b})W$ can be obtained by differentiating the reduced equation for the existence problem, and then divide by $\e$ in mode 1 and $\e^2$ in mode 0. The discrepancy in the powers is entirely due to the fact that $a$ comes in at order 1 but $\vec{b}$ comes in at order $\e$. We adopt the convention that the first two equations are the real and imaginary parts respectively of the mode 1 equation in $PBW=0$ and the remaining $N$ equations come from mode 0 in $PBW=0$. To get the reduced equation for $\l\not=0$, we Taylor expand with respect to $\l$ to find
	\ba
		&PB(\e,\kappa,\l,0,\vec{\b})(\Upsilon_a+\e\vec{b}+\cV((a,\vec{b};\e,\kappa,\l,0,\vec{\b}))\\
		&\quad=-\l(\Upsilon_a+\e\vec{b})+ PB(\e,\kappa,0,0,\vec{\b})(\Upsilon_a+\e\vec{b}+\cV((a,\vec{b};\e,\kappa,0,0,\vec{\b}))\\
		&\quad+PB(\e,\kappa,0,0,\vec{\b})(I-P)\cV_\l((a,\vec{b};\e,\kappa,0,0,\vec{\b}))+\cO(|\l|^2).
	\ea
	Computations similar to those performed in Proposition \ref{prop:WExpansion} shows that $PB(\e,\kappa,0,0,\vec{\b})(I-P)=\cO(\e)$, moreover by (1) in Proposition \ref{prop:CVExpansion} one has that $\cV_\l=\cO(\e)$ as well.
\end{proof}
\begin{remark}
	Like the existence result for Lyapunov-Schmidt, locality plays no role in this argument; because the strategy here is to essentially show that the co-periodic reduced equation is the linearization of the existence theory's reduced equation.
\end{remark}
\section{Lyapunov-Schmidt reduction: General stability}\label{sec:StabilityLSE} 
As in the example, we proceed by first reducing $PB(\Upsilon_a+\e\vec{b})$ and then reducing $PB\cV$. In particular, we can focus entirely on terms that have at least one $\s$ derivative because Theorem \ref{thm:coperiodicstability} gives us the all terms with only $\e$-derivatives. We have two main goals in this section, the first is to obtain the reduced equation for stability in the $O(2)$-invariant case and the second is to show that the singular behavior observed in the asymptotic expansion is indeed present in the full system. To facilitate the second goal, we will carry the dependence on $d$ and $C$ even though they are identically zero in the $O(2)$-invariant case. We will attempt to do as much as possible in the $SO(2)$-case, regardless of compatibility, only towards the end of each result in this section will we enforce the $O(2)$-invariance. \\

We adopt the decomposition $\cN(U,k,\mu)=M(k\d_\xi)\tilde{\cN}(U,k,\mu)$ for a quasilinear $\tilde{\cN}$ satisfying $\tilde{\cN}(0,k,\mu)=0$. In this decomposition, we have
\begin{equation}
	D_U\cN(\bar{u};k,\mu)=M(k\d_\xi)D_U\tilde{\cN}(\bar{u};k,\mu).
\end{equation}
We reduce $PB(\Upsilon_a+\e\vec{b})$ in the following sequence of lemmas.
\begin{lemma}\label{lem:generalBS}
	$PB_\s(0,\kappa,\l,0,\vec{b})(\Upsilon_a+\e\vec{b})$ vanishes.
\end{lemma}
\begin{proof}
		Note that $\d_\s D_U\cN(\tilde{u}_{\e,\kappa,\vec{\b}};k,\mu,\s)|_{\e=0}=0$ so $PB_\s(\Upsilon_a+\e\vec{b})$ is given by
		\begin{equation}
			PB_\s(0,\kappa,\l,0,\vec{b})(\Upsilon_a+\e\vec{b})=P(L_\s(k_*,0;0)+iC(0,\kappa))(\Upsilon_a+\e\vec{b}).
		\end{equation}
		Note that $C(0,\kappa)$ was engineered so that $\ell (L_\s(k_*,0;0)+iC(0,\kappa))\Upsilon_a=0$. $PB_\s(\e\vec{b})$ vanishes because $S_k(0,0)=0$ and $C(0,\kappa)=0$ by $O(2)$-invariance.
\end{proof}
\begin{remark}
	If one was working in the $SO(2)$-invariant case, the reduced form of $PB_\s(\Upsilon_a+\e\vec{b})$ would be given by $(k_*\Pi_0 S_k(0,0)+iC(0,\kappa))(\e\vec{b})$. This is (up to a factor of $ik_*$) the coefficient of $B_{\Hx}$ in the singular equation \eqref{eq:eps3mode0}.
\end{remark}
\begin{lemma}\label{lem:generalBSS}
	In reduced form $PB_{\s\s}(0,\kappa,\l,0,\vec{b})(\Upsilon_a+\e\vec{b})$ is given by
	\begin{equation}\label{eq:generalBSS}
		\begin{bmatrix}
			k_*^2[[\ell S_{kk}(k_*,0)r ]] & 0_{2,N}\\
			0_{N,2} & k_*^2\Pi_0 S_{kk}(0,0)\Pi_0
		\end{bmatrix}
		\begin{pmatrix}
			a_1\\
			a_2\\
			\vec{b}
		\end{pmatrix},
	\end{equation}
	where $0_{n,m}$ denotes the $n\times m$ matrix where all entries are 0.
\end{lemma}
\begin{proof}
	This follows from the observation that $\d_\s^2D_U\cN(\tilde{u}_{\e,\kappa,\vec{\b}};k,\mu,\s)|_{\e=0}=0$ and that the only place where $\s$ can appear to second order or higher is in the symbol $S(k(\eta+\s),\mu)$.
\end{proof}
Notice that Lemma \ref{lem:generalBSS} is insensitive to whether or not the original system was $O(2)$-invariant or not.
\begin{lemma}\label{lem:generalBES}
	In reduced form $PB_{\e\s}(0,\kappa,\l,0,\vec{\b})(\Upsilon_a+\e\vec{b})$ is given by
		\begin{equation}\label{eq:generalBES}
			\begin{bmatrix}
			-ik_*\kappa[[i\ell S_{kk}(k_*,0)r]] & \cO(\e) \\
			0_{N,2} & 0_{N,N}
			\end{bmatrix}
			\begin{pmatrix}
			a_1\\
			a_2\\
			\vec{b}
			\end{pmatrix}.
		\end{equation}
\end{lemma}
\begin{proof}
	The terms involving the symbol follow by the chain rule as $\d_\e \d_\s S(k(\eta+\s),\mu)=\kappa S_k(k(\eta+\s),\mu)+\kappa kS_{kk}(k(\eta+\s),\mu)\eta+k \mu_\e S_{k,\mu}(k(\eta+\s),\mu)$, which upon evaluating at $\e=\s=0$ reduces to $\kappa S_k(k_*\eta,0)+k_*\kappa S_{kk}(k_*\eta,0)\eta$. By $O(2)$-invariance, $S_k(0,0)=0$ and the other term from the symbol is $k_*\kappa \Pi_0 S_{kk}(k_*\eta,0)\eta\Pi_0$ which always vanishes in mode 0. For mode 1, we recall from \cite{WZ2} that $C_\e(0,\kappa)$ was engineered to cancel $\kappa \ell S_k(k_*,0)r+i(k_*d_\e+\kappa d_*)$. By $O(2)$-invariance $C_\e(0,\kappa)=0$ giving the desired linear term in mode 0. In mode 1, we have $\frac{1}{2}[[\ell S_{kk}(k_*,0)r]]ae^{i\xi}r-c.c.$, which upon application of the elementary fact $-i(iz+\bar{iz})=z-\bar{z}$ gives us the desired linear contribution.\\
	
	This leaves us with the task of reducing the terms coming from $D_U\cN(\tilde{u}_{\e,\kappa,\b};k,\mu,\s)$. To do this, we note the following key observation
	\begin{obs}\label{obs:DUTaylor}
		$D_U\cN(\tilde{u}_{\e,\kappa,{\b}};k,\mu,\s)v$ admits the expansion
		\ba
			D_U\cN(\tilde{u}_{\e,\kappa,\vec{\b}};k,\mu,\s)v&=e^{-i\s\xi}D_U\cN(\tilde{u}_{\e,\kappa,\vec{\b}};k,\mu)(e^{i\s\xi}v )= e^{-i\s\xi}D_U^2\cN(0;k,\mu)(\tilde{u}_{\e,\kappa,\vec{\b}},e^{i\s\xi}v)\\
			&\quad+\frac{1}{2}e^{-i\s\xi}D_U^3\cN(0;k,\mu)(\tilde{u}_{\e,\kappa,\vec{\b}},\tilde{u}_{\e,\kappa,\vec{\b}},e^{i\s\xi}v)+\cO(\e^3).
		\ea
	\end{obs}
	In particular, for this order in $\e$ the only relevant term is $e^{-i\s\xi}D_U^2\cN(0;k,\mu)(\tilde{u}_{\e,\kappa,\vec{\b} },e^{i\s\xi}v)$. At the level of multipliers, this given by
	\begin{equation}
		e^{-i\s\xi}D_U^2\cN(0;k,\mu)(\tilde{u}_{\e,\kappa,\vec{\b} },e^{i\s\xi}v)=\sum_{\eta_1,\eta_2\in\ZZ}\cQ(k\eta_1,k(\eta_2+\s);\mu)(\widehat{\tilde{u}_{\e,\kappa,\vec{\b}}}(\eta_1),\hat{v}(\eta_2))e^{i(\eta_1+\eta_2)\xi}.
	\end{equation}
	
	For mode 1, the coefficient of $\vec{b}$ is $\cO(\e)$ and hence an acceptable error term in the reduced equation as the terms in this lemma come into the reduced equation with a factor of $\e\s$ coming from $B$. To get the term in mode 0, we use the splitting $\cN(U,k,\mu)=M(k\d_\xi)\tilde{\cN}(U,k,\mu)$. In terms of multipliers, what we have is that
	\begin{equation}
		\cQ(k\eta_1,k\eta_2;\mu)=M(ik(\eta_1+\eta_2))\tilde{\cQ}(k\eta_1,k\eta_2;\mu),
	\end{equation}
	where $\cQ$ is the multiplier of $D_U^2\cN(0;k,\mu)$ and $\tilde{\cQ}$ is the multiplier of $D_U^2\tilde{\cN}(0;k,\mu)$. Taking the $\s$-derivative, we find that
	\begin{equation}\label{eq:cQdsigma}
		\frac{\d}{\d\s} ik(\eta_1+\eta_2+\s)\tilde{\cQ}(k\eta_1,k(\eta_2+\s);\mu)=ik\tilde{\cQ}(k\eta_1,k(\eta_2+\s);\mu)+ik^2(\eta_1+\eta_2+\s)\tilde{\cQ}_\nu(k\eta_1,k(\eta_2+\s);\mu),
	\end{equation}
	where $\tilde{\cQ}_\nu$ is as defined in \eqref{eq:defcQnu}. However, since we're only interesting in $\eta_1+\eta_2=0$ and we're evaluating at $\s=0$ we only need to focus on the first term in \eqref{eq:cQdsigma}. Hence the Fourier mean of $\d_\e\d_\s \Pi_0D_U\cN(\tilde{u}_{\e,\kappa,\vec{\b}};k,\mu,\s)|_{\e=\s=0}(\Upsilon_a)$ is given by
	\ba
		&\text{Mean}(\d_\e\d_\s \Pi_0D_U\cN(\tilde{u}_{\e,\kappa,\vec{\b}};k,\mu,\s)|_{\e=\s=0}(\Upsilon_a))\\
		&\quad=\frac{1}{4}ik_*\a\Pi_0\Big[\tilde{\cQ}(k_*,-k_*;0)(r,\bar{a}\bar{r} )+\tilde{\cQ}(-k_*,k_*;0)(\bar{r},ar) \Big]\\
		&\quad=\frac{1}{4}ik_*\a(a+\bar{a})\Pi_0\tilde{\cQ}(k_*,-k_*;0)(r,\bar{r}).
	\ea
	But in the $O(2)$-case, $\Pi_0\tilde{\cQ}(k_*,-k_*;0)(r,\bar{r})=0$ by Theorem \ref{thm:O2compat}. Strictly speaking, Theorem \ref{thm:O2compat} only guarantees that $i\Pi_0S_k(0,0)(I-\Pi_0)N_0(I-\Pi_0)\tilde{\cQ}(k_*,-k_*;0)(r,\bar{r})+\Pi_0\tilde{\cQ}(k_*,-k_*;0)(r,\bar{r})=0$, but in the $O(2)$-case $S_k(0,0)=0$ and so $\Pi_0\cQ(k_*,-k_*;0)(r,\bar{r})=0$ is equivalent to the full quantity vanishing.
\end{proof}
\begin{remark}
	In the general $SO(2)$-case, we recover the coefficient (up to a multiple of $\kappa$) of $\cB_{\Hx}$ in \eqref{eq:eps4mode0}. Moreover, we also recover part of the linearized form of the coefficient $|A|_{\Hx}$ \eqref{eq:eps3mode0}. In particular, what we've recovered is the term that doesn't come from $\Psi_0$.
\end{remark}
As in the example model, to complete the reduction of $PB(\Upsilon_a+\e\vec{b})$ we need the Fourier means of $PB_{\e\e\s}\Upsilon_a$ and $PB_{\e\s\s}\Upsilon_a$.
\begin{proposition}\label{prop:generalEESESS}
	The Fourier mean of $PB_{\e\e\s}(0,\kappa,\l,0,\vec{\b})\Upsilon_a$ in reduced form is given by
	\begin{equation}\label{eq:generalEES}
		-\frac{1}{2}k_*\kappa\a(a+\bar{a})\Pi_0\tilde{\cQ}(k_*,-k_*;0)(iN_1S_k(k_*,0)r,\bar{r}),
	\end{equation}
	and in reduced form, the Fourier mean of $PB_{\e\s\s}(0,\kappa,\l,0,\vec{\b})\Upsilon_a$ is given by
	\begin{equation}\label{eq:generalESS}
		-\frac{1}{2}k_*^2\a(a+\bar{a})\Pi_0\tilde{\cQ}_\nu(k_*,-k_*;0)(r,\bar{r}).
	\end{equation}
\end{proposition}
\begin{proof}
	It suffices to understand the terms coming from $D_U\cN(\tilde{u}_{\e,\kappa,\vec{\b}};k,\mu,\s)\Upsilon_a$ because the terms coming from $L$ and $C$ are Fourier multiplier operators and hence preserve Fourier support. We start with $B_{\e\e\s}$. From the proof of Lemma \ref{lem:generalBES}, it is clear that our one $\s$ derivative has to be spent on the factor of $ik(\eta_1+\eta_2+\s)$ in front of $\tilde{\cQ}$. Similarly, at least one of the $\e$ derivatives has to hit $\tilde{u}_{\e,\kappa,\vec{\b}}$ as otherwise we are pairing $\Upsilon_a$ with the zero function when we evaluate at $\e=0$. So there are three terms that can appear which are
	\begin{enumerate}
		\item $\frac{1}{4}i\kappa \Pi_0\Big( \tilde{\cQ}(k_*,-k_*;0)(\a r,\bar{ar})+\tilde{\cQ}(-k_*,k_*;0)(\bar{\a r},ar)\Big)$ coming from the second $\d_\e$ hitting the factor of $k$ out front.\\
		\item \ba
			\frac{1}{4}\kappa\Pi_0\Big(\tilde{\cQ}_\nu(k_*,-k_*;0)(\a r,\bar{ar})-\tilde{\cQ}_\nu(-k_*,k_*;0)(\bar{\a r},ar)\\
			-\tilde{\cQ}_\nu(k_*,-k_*)(\a r,\bar{ar})+\tilde{\cQ}_\nu(-k_*,k_*;0)(\bar{\a r},ar) \Big)=0,
		\ea
		coming from the second $\d_\e$ hitting the $k$ inside of $D_U^2\tilde{\cN}(0;k,\mu)$.
		\item $\frac{1}{2}ik_*\Pi_0\Big(\tilde{\cQ}(k_*,-k_*;0)\Big(\widehat{\frac{\d^2\tilde{u}_{0,\kappa,\vec{\b}}}{\d\e^2}}(1),\bar{ar} \Big)+\tilde{\cQ}(-k_*,k_*;0)\Big(\widehat{\frac{\d^2\tilde{u}_{0,\kappa,\vec{\b}}}{\d\e^2}}(-1),ar \Big)\Big)$ coming from both $\e$-derivatives hitting $\tilde{u}_{\e,\kappa,\vec{\b}}$.
	\end{enumerate}
	By $O(2)$-invariance, the first term is 0 as $\Pi_0\tilde{\cQ}(k_*,-k_*;0)(r,\bar{r})=0$. Hence we only need to consider the third term. From Proposition \ref{prop:WExpansion}, we recall the expansion for $(I-\Pi_1)\widehat{\frac{\d^2\tilde{u}_{0,\kappa,\vec{\b}}}{\d\e^2}}(1)$ as
	\begin{equation*}
	(I-\Pi_1)\widehat{\frac{\d^2\tilde{u}_{0,\kappa,\vec{\b}}}{\d\e^2}}(1)=-\kappa\a N_1(I-\Pi_1)S_k(k_*,0)\Pi_1r.
	\end{equation*}
	Plugging this into the third expression, we get
	\ba
	&\frac{1}{2}ik_*\Pi_0\Big(\tilde{\cQ}(k_*,-k_*;0)\Big(\widehat{\frac{\d^2\tilde{u}_{0,\kappa,\vec{\b}}}{\d\e^2}}(1),\bar{ar} \Big)+\tilde{\cQ}(-k_*,k_*;0)\Big(\widehat{\frac{\d^2\tilde{u}_{0,\kappa,\vec{\b}}}{\d\e^2}}(-1),ar \Big)\Big)\\
	&\quad=-\frac{1}{2}k_*\kappa\a(a+\bar{a})\Pi_0\tilde{\cQ}(k_*,-k_*;0)(iN_1(I-\Pi_1)S_k(k_*,0)\Pi_1r,r),
	\ea
	where we've used $\Pi_0\tilde{\cQ}(k_*,-k_*;0)(iN_1S_k(k_*,0)r,r )\in\Pi_0\RR^n$ by $O(2)$-invariance, c.f. the discussion preceding Corollary \ref{cor:cQnureal} for details. This term will cancel with the corresponding term in $PB_{\e\s}\cV_\e$ as we will soon see in Theorem \ref{thm:generalPBV}.\\
	
	For $PB_{\e\s\s}\Upsilon_a$, the only relevant term is
	\begin{equation*}
		-\frac{1}{2}k_*^2\Pi_0\Big(\tilde{\cQ}_\nu(k_*,-k_*;0)(\a r,\bar{a}r)+\tilde{\cQ}(-k_*,k_*;0)(\a\bar{r},ar) \Big)=-\frac{1}{2}k_*^2\a(a+\bar{a})\Pi_0\tilde{\cQ}_\nu(k_*,-k_*;0)(r,\bar{r}).
	\end{equation*}
	Note that the $i$ vanished because $\tilde{\cQ}_\s=\tilde{\cQ}_{\eta_2}=-i\tilde{\cQ}_\nu$ by definition.
\end{proof}	
At this point, at least in the $O(2)$-case, we've shown that there are no singular terms or ghosts coming from $\cA$ and $\cB$ in the reduced equation for stability; and that they do arise in the general $SO(2)$-invariant case. This is because (1) in the list of terms in $PB_{\e\e\s}\Upsilon_a$ is the ghost of $\cA$. Hence in some sense, it's mission accomplished. That said, we will proceed to reduce $PB\cV$ in order to properly match the coefficients of the stability equation to the modified (cGL) system.
\begin{theorem}\label{thm:generalPBV}
	In reduced form $PB(\e,\kappa,\l,\s,\vec{\b})\cV(a,\vec{b};\e,\kappa,\l,\s,\vec{\b})$ is given by
	\ba
		&\e\s \bp 2ik_*\kappa[[i\ell S_k(k_*,0)N_1S_k(k_*,0)r ]]+\cO(|\l|) & \cO(|\l|)\\
		0_{N,2} & 0_{N,N} \ep \bp a_1\\ a_2\\ \vec{b}\ep\\
		&\quad+\s^2\bp -k_*^2[[\ell S_k(k_*,0)N_1S_k(k_*,0)r]]+\cO(|\l|) & \cO(|\l|)\\
		0_{N,2} & 0_{N,N}\ep \bp a_1\\ a_2\\ \vec{b} \ep\\
		&\quad+\s^2 \bp \cO(\e) & \cO(\e)\\
		\bp -\frac{1}{4}k_*^2\a\Pi_0S_{kk}(0,0)N_0\tilde{\cQ}(k_*,-k_*;0)(r,\bar{r})+\cO(|\l|) & \cO(|\l|) \ep & \cO(|\l|)\ep\bp a_1\\a_2\\ \vec{b}\ep\\
		&\quad+\s^2\bp\cO(\e) & \cO(\e)\\
		 \bp -\frac{1}{2}\a k_*^2\Pi_0\tilde{\cQ}(k_*,-k_*;0)(r,\overline{iN_1S_k(k_*,0)r} )+\cO(|\l|) & \cO(|\l|)\ep &\cO(|\l|)\ep\bp a_1\\a _2\\ \vec{b}\ep\\ 
		 &\quad+\e\s\bp \cO(\e) & \cO(\e)\\
		\bp \frac{1}{2}k_*^2\a\Pi_0\tilde{\cQ}(k_*,-k_*;0)(r,\overline{iN_1S_k(k_*,0)r})+\cO(|\l|) &\cO(|\l|)\ep &\cO(|\l|)\ep \bp a_1 \\ a_2\\ \vec{b}\ep+h.o.t.\ ,
	\ea
	where $h.o.t.$ refers to terms that are at least $\cO(\e^3)$ under the scalings $\s\sim\e$ and $\l\sim\e^2$.
\end{theorem}
\begin{proof}
	We work term-by-term in the Taylor series expansion, ignoring the pure $\e$-terms because those are handled by Theorem \ref{thm:coperiodicstability}. Note that because $PB(0,\kappa,\l,0,0,\vec{\b})=-\l P$, we can ignore all terms of the form $PB(0,\kappa,\l,0,\vec{\b})\cV_j$ where $j$ is a dummy subscript referring to some collection of $\e,\s$ derivatives of $\cV$.\\
	First, we look at $PB_\e(0,\kappa,\l,0,\vec{\b})\cV_\s(a,\vec{b};0,\kappa,\l,0,\vec{\b})$. By Proposition \ref{prop:CVExpansion}, we have that $PB_\e\cV_\s$ is given by
	\ba
		PB_\e(0,\kappa,\l,0,\vec{\b})\cV_\s(a,\vec{b};0,\kappa,\l,0,\vec{\b})&=-PB_\e(0,\kappa,\l,0,\vec{\b})T_\l(I-P)B_\s(0,\kappa,\l,0,\vec{\b})\Upsilon_a,
	\ea
	where we recall from \eqref{eq:Tlambdadef} that $T_\l$ is defined to be
	\begin{equation*}
		T_\l:=((I-P)(B(0,\kappa,\l,0,\vec{\b})(I-P) )^{-1}.
	\end{equation*}
	Hence by Taylor's theorem, we have that
	\ba
		PB_\e(0,\kappa,\l,0,\vec{\b})\cV_\s(a,\vec{b};0,\kappa,\l,0,\vec{\b})=-PB_\e(0,\kappa,\l,0,\vec{\b})T_0B_\s(0,\kappa,\l,0,\vec{\b})\Upsilon_a+\cO(|\l|).
	\ea
	We recall from Lemma \ref{lem:generalBS} that $B_\s(0,\kappa,\l,0,\vec{\b})$ is given by
	\begin{equation}\label{eq:generalB_sigma}
		B_\s(0,\kappa,\l,0,\vec{\b})=L_\s(k_*,0;0)+iC(0,\kappa,\vec{\b})=k_*L_k(k_*,0;0)+iC(0,\kappa,\vec{\b}).
	\end{equation}
	It is a straightforward computation from the approximate formula for $D_U\cN(\tilde{u}_{\e,\kappa,\vec{\b}});k,\mu,\s)$ in Observation \ref{obs:DUTaylor} and the chain rule to show that
	\begin{equation}\label{eq:generalB_epsilon}
		B_\e(0,\kappa,\l,0,\vec{\b})=\kappa L_k(k_*,0;0)D_\xi+(\kappa d_*+k_*d_\e(0,\kappa,\vec{\b}))\d_\xi+M(k_*\d_\xi)D_U^2\tilde{\cN}(0;k_*,0,0)(\Upsilon_\a,\cdot),
	\end{equation}
	where $D_\xi$ is the Fourier multiplier operator whose symbol is $\eta$. We split $PB_\e\cV_\s$ into two terms, the ``linear-linear'' and the ``nonlinear-linear'' term, and begin by reducing the linear-linear term
	\ba
		LL:=-P\Big(\kappa L_k(k_*,0;0)D_\xi+(\kappa d_*+k_*d_\e(0,\kappa,\vec{\b}))\d_\xi)T_0(k_*L_k(k_*,0;0)+iC(0,\kappa,\vec{\b})\Big)\Upsilon_a. 
	\ea
	Notice that the $C$-term drops out independently of $C=0$ because $T_0$ acts on $(I-P)H^m_{per}([0,2\pi];\RR^n)$ and $\Upsilon_a\in PH^m_{per}([0,2\pi];\RR^n)$. Our second observation is that since $\d_\xi$ commutes with all Fourier multiplier operators, $(\kappa d_*+k_*d_\e(0,\kappa,\vec{\b}))\d_\xi$ also doesn't contribute as $P$ will commute with the $\d_\xi$ and annihilate $T_0$. Crucially, these two terms vanish independently of $d=C=0$. Hence we can reduce $LL$ to
	\ba
		LL=-\kappa k_*PL_k(k_*,0;0)D_\xi T_0L_k(k_*,0;0)\Upsilon_a.
	\ea
	Writing this out in terms of the Fourier symbols of $L_k(k_*,0;0)$ and $T_0$, we get that
	\ba
		LL=-\frac{1}{2}k_*\kappa(\ell S_k(k_*,0)N_1S_k(k_*,0))ae^{i\xi}r-c.c.
	\ea
	So $LL$ in reduced form is given by $ik_*\kappa [[i\ell S_k(k_*,0)N_1S_k(k_*,0)r ]]\bp a_1\\ a_2\ep$. To complete this term, we look at the nonlinear-linear term, which we define to be
	\ba
		NL:=PM(k_*\d_\xi)D_U^2\tilde{\cN}(0;k_*,0,0)(\Upsilon_\a,k_*T_0L_k(k_*,0;0)\Upsilon_a ).
	\ea
	Since $T_0L_k(k_*,0;0)$ is a Fourier multiplier operator, $T_0L_k(k_*,0;0)\Upsilon_a$ is Fourier supported in $\{\pm 1\}$. Hence we only need to compute the Fourier mean of $NL$, but $\widehat{NL}(0)=0$ since there is a factor of $\Pi_0M(0)=0$ on the left.\\
	
	The next we reduce is $PB_\s(0,\kappa,\l,0,\vec{\b})\cV_\e(a,\vec{b};0,\kappa,\l,0,\vec{\b})$. As in the previous term, $PB_\e\cV_\s$, we recall the expansion for $\cV_\e$ from Proposition \ref{prop:CVExpansion} nd Taylor expand $T_\l$ to get
	\ba
		PB_\s(0,\kappa,\l,0,\vec{\b})\cV_\e(a,\vec{b};0,\kappa,\l,0,\vec{\b})=-PB_\s(0,\kappa,\l,0,\vec{\b})T_0B_\e(0,\kappa,\l,0)\Upsilon_a+\cO(|\l|).
	\ea
	As before, we have a linear-linear term and a linear-nonlinear term. Starting with the linear-linear term, we have
	\ba
		LL:=-P(k_* L_k(k_*,0;0)+iC(0,\kappa,\vec{\b}) )T_0(\kappa L_k(k_*,0;0)D_\xi+(\kappa d_*+k_*d_\e(0,\kappa,\vec{\b}))\d_\xi)\Upsilon_a.
	\ea
	Essentially the same argument as before gives this in reduced form as
	\ba
		LL=ik_*\kappa[[i\ell S_k(k_*,0)N_1S_k(k_*,0)r ]]\bp a_1\\a_2\ep.
	\ea
	Turning to the linear-nonlinear term, we have
	\ba
		LN:=-P\Big(k_*L_k(k_*,0;0)+iC(0,\kappa,\vec{\b})\Big)T_0M(k_*\d_\xi)D_U^2\tilde{\cN}(0;k_*,0,0)(\Upsilon_\a,\Upsilon_a).
	\ea
	This, in general, the linearization of $-iS_k(k_*,0)\Psi_{0,\Hx}$. We have that $\widehat{LN}(0)=0$ in the $O(2)$ case because $S_k(0,0)=0$. Observe that in showing $NL$ and $LN$ were mean-free, we didn't need to Taylor expand $T_\l$ and so there is no error term in mode 0.\\
	
	To complete the discussion of terms where $\cV$ and $\cB$ each have exactly one derivative, we look at $PB_\s(0,\kappa,\l,0,\vec{\b})\cV_\s(a,\vec{b};0,\kappa,\l,0,\vec{\b})$. As before, we use the expansion for $\cV$ in Proposition \ref{prop:CVExpansion} and the expression for $B_\s$ in \eqref{eq:generalB_sigma} to get
	\ba
		PB_\s(0,\kappa,\l,0,\vec{\b})\cV_\s(a,\vec{b};0,\kappa,\l,0,\vec{\b})=-P\Big(k_*L_k(k_*,0;0)+iC(0,\kappa,\vec{\b})\Big)T_\l\Big(k_*L_k(k_*,0;0)+iC(0,\kappa,\vec{\b})\Big)\Upsilon_a.
	\ea
	Because operator in the above is a Fourier multiplier operator, $PB_\s\cV_\s$ is necessarily mean free, and slight modification used to reduce the $LL$'s above gives the reduced form $PB_\s\cV_\s$ as
	\ba
		-k_*^2[[\ell S_k(k_*,0)N_1S_k(k_*,0)r]]\bp a_1\\ a_2\ep+\cO(|\l|).
	\ea
	Now, we turn to the terms which have one derivative on $\cV$ and one on $B$ or one derivative on $B$ and two on $\cV$. These terms are of cubic order, so we only need their mean values because we divide by $\e$ in mode 0. For this, we will need the second derivatives of $B$ and $\cV$. We record the needed second derivatives on $\cV$ as
	\ba\label{eq:cVsecondderiv}
		\cV_{\e\e}(a,\vec{b};0,\kappa,\l,0,\vec{\b})&=-T_\l B_{\e\e}(0,\kappa,\l,0,\vec{\b})\Upsilon_a-2T_\l B_\e(0,\kappa,\l,0,\vec{\b})(\vec{b}\\
		&\quad+\cV_\e(a,\vec{b};0,\kappa,\l,0,\vec{\b})),\\
		\cV_{\e\s}(a,\vec{b};0,\kappa,\l,0,\vec{\b})&=-T_\l B_{\e\s}(0,\kappa,\l,0,\vec{\b})\Upsilon_a-T_\l B_\s(0,\kappa,\l,0,\vec{\b})\vec{b}\\
		&\quad-T_\l B_\s(0,\kappa,\l,0,\vec{\b})\cV_\e(a,\vec{b};0,\kappa,\l,0,\vec{\b})\\
		&\quad-T_\l B_\e(0,\kappa,\l,0,\vec{\b})\cV_\s(a,\vec{b};0,\kappa,\l,0,\vec{\b}),\\
		\cV_{\s\s}(a,\vec{b};0,\kappa,\l,0,\vec{\b})&=-T_\l B_{\s\s}(0,\kappa,\l,0,\vec{\b})\Upsilon_a-2T_\l B_\s(0,\kappa,\l,0,\vec{\b})\cV_\s(a,\vec{b};0,\kappa,\l,0,\vec{\b}).
	\ea
	First, we will analyze the terms which have two derivatives on $\cV$. In the $O(2)$-invariant case, we have the crucial observation that the symbol of $PB_\s(0,\kappa,\l,0\vec{\b})(I-P)$ vanishes in mode 0 so whatever $PB_\s\cV_j$ happens to be, for $j\in\{\e\e,\e\s,\s\s \}$, we know its Fourier mean is 0. So what we are left to compute are the Fourier means of $PB_\e\cV_{\e\s}$ and $PB_\e\cV_{\s\s}$ as $PB_\e\cV_{\e\e}$ is a pure $\e$-term and is thus accounted for by Theorem \ref{thm:coperiodicstability}. As $\Pi_0M(0)=0$, it follows that nonlinearity in $B_\e$ doesn't contribute because the nonlinearity is given by $M(k_*\d_\xi)D_U^2\tilde{\cN}(0;k_*,0,0)(\Upsilon_\a,\cdot)$. Hence the only way for $PB_\e\cV_j$ to contribute is through the linear part, but the linear part $L_k(k_*,0;0)D_\xi+(k_*d_\e(0,\kappa,\vec{\b})+\kappa d_*)\d_\xi$ is a Fourier multiplier whose symbol vanishes at 0 and so no term involving second derivatives of $\cV$ contributes to the mean. Notice that the Fourier mean of $PB_\e\cV_j$ vanished independently of compatibility, however, we did use $O(2)$-invariance to get that $PB_\s\cV_j$ terms vanished in mean 0. Later on, we will sketch the matching between $PB_\s\cV_j$ to $-iS_k(0,0)\Psi_{4,\Hx}$ in the compatible $SO(2)$ case, but for now we will move on to the last type of term.\\
	
	The last few terms are $PB_{\e\e}\cV_\s$, $PB_{\e\s}\cV_\s$, $PB_{\e\s}\cV_\e$, $PB_{\s\s}\cV_\e$ and $PB_{\s\s}\cV_{\s}$. For the $B_{\s\s}\cV_j$ terms, we have the following
	\ba
		PB_{\s\s}(0,\kappa,\l,0,\vec{\b})\cV_\s(a,\vec{b};0,\kappa,\l,0,\vec{\b})=-P(k_*^2L_{kk}(k_*,0;0))T_\l(k_*L_k(k_*,0;0)+iC(0,\kappa,\vec{\b}))\Upsilon_a.
	\ea
	Since each of three operators above is a Fourier multiplier operator, it follows that $PB_{\s\s}\cV_\s$ is Fourier supported in $\{\pm 1\}$ and thus mean free.
	\ba
		&PB_{\s\s}(0,\kappa,\l,0,\vec{\b})\cV_\e(a,\vec{b};0,\kappa,\l,0,\vec{\b})=-P(k_*^2L_{kk}(k_*,0;0))T_\l\cdot \\
		&\quad\cdot\Big(\kappa L_k(k_*,0;0)D_\xi+(\kappa d_*+k_*d_\e(0,\kappa,\vec{\b})) +M(k_*\d_\xi)D_U^2\tilde{\cN}(0;k_*,0,0)(\Upsilon_\a,\cdot)\Big)\Upsilon_a.
	\ea
	Similar reasoning to $PB_{\s\s}\cV_\s$ shows that the only term in $PB_{\s\s}\cV_\e$ that can have a nonzero Fourier mean is
	\ba
	-P(k_*^2L_{kk}(k_*,0;0))T_\l M(k_*\d_\xi)D_U^2\tilde{\cN}(0;k_*,0,0)(\Upsilon_\a,\Upsilon_a).
	\ea
	Expanding out the Fourier mean of this quantity and using Taylor's theorem gives
	\ba
	\text{Mean}(PB_{\s\s}(0,\kappa,\l,0,\vec{\b})\cV_\e(a,\vec{b};0,\kappa,\l,0,\vec{\b}))=-\frac{1}{4}k_*^2\a(a+\bar{a}) \Pi_0 S_{kk}(0,0)N_0(I-\Pi_0)\tilde{\cQ}(k_*,-k_*;0)(r,\bar{r})+\cO(|\l|).
	\ea
	For $B_{\e\e}$ terms, we note that because $\Pi_0 M(0)=0$ and the observation that no $\e$-derivative of $M(k\d_\xi)$ eliminates the $\d_\xi$, no term arising from the nonlinearity in $B_{\e\e}$ can contribute to the Fourier mean. However, this means that $PB_{\e\e}\cV_\s$ has zero Fourier mean because $\cV_\s$ is mean-free and the only terms that can contribute from $B_{\e\e}$ are Fourier multiplier operators. To finish the proof, we need to handle the $B_{\e\s}$ terms. First, we expand out $B_{\e\s}$ to
	\ba
		B_{\e\s}(0,\kappa,\l,0)&=\kappa L_k(k_*,0;0)+\kappa k_*L_{kk}(k_*,0;0)D_\xi+iC_\e(0,\kappa,\vec{\b})\\
		&\quad+ik_*\Pi_0D_U^2\tilde{\cN}(0;k_*,0,0)(\Upsilon_\a,\cdot)+M(k_*\d_\xi)D_U^2\tilde{\cN}_\s(0;k_*,0,0)(\Upsilon_\a,\cdot).
	\ea
	We note that $L_{kk}(k_*,0;0)D_\xi$ cannot contribute to the mean because the symbol of $D_\xi$ vanishes at frequency 0. Similarly, since $\Pi_0 M(0)=0$ the only terms that can contribute are $\kappa L_k(k_*,0;0)+ik_*\Pi_0 D_U^2\tilde{\cN}(0;k_*,0,0)(\Upsilon_\a,\cdot)$. For $PB_{\e\s}\cV_\s$, we have that $\cV_\s$ is Fourier supported in $\{\pm 1\}$ and so we only need to find the Fourier mean of
	\ba
		-ik_*\Pi_0 D_U^2\tilde{\cN}(0;k_*,0,0)(\Upsilon_\a,T_\l B_\s(0,\kappa,\l,0,\vec{\b})\Upsilon_a).
	\ea
	Expanding out the Fourier mean, we get
	\ba
		&\text{Mean}(-ik_*\Pi_0 D_U^2\tilde{\cN}(0;k_*,0,0)(\Upsilon_\a,T_\l B_\s(0,\kappa,\l,0,\vec{\b})\Upsilon_a ))\\
		&\quad=-\frac{1}{4}i\a k_*^2\Pi_0\Big( \bar{a}\tilde{\cQ}(k_*,-k_*;0)(r,\bar{N_1}S_k(-k_*,0)\bar{r} )+a\tilde{\cQ}(k_*,-k_*;0)(\bar{r},N_1S_k(k*,0)r)\Big)+\cO(|\l|).
	\ea
	To match the corresponding term in Theorem \ref{thm:O2multilin}, we move the factor of $i$ inside of $\cQ$ and note that $iS_k(-k_*,0)=\overline{iS_k(k_*,0)}$ by the chain rule. So the $\tilde{\cQ}$ terms are conjugates of each other, but they're also real by the argument preceding Corollary \ref{cor:cQnureal} and thus they are equal. Hence $\text{Mean}(PB_{\e\s}\cV_\s)$ is given by
	\ba
		\text{Mean}(PB_{\e\s}(0,\kappa,\l,0,\vec{\b})\cV_\s(a,\vec{b};0,\kappa,\l,0,\vec{\b}))=-\frac{1}{4}\a k_*^2(a+\bar{a})\Pi_0\tilde{\cQ}(k_*,-k_*;0)(r,\overline{iN_1S_k(k_*,0)r} )+\cO(|\l|).
	\ea
	We come to the final step, computing the Fourier mean of $PB_{\e\s}(0,\kappa,\l,0,\vec{\b})\cV_\e(a,\vec{b};0,\kappa,\l,0,\vec{\b})$. Expanding $PB_{\e\s}\cV_\e$ and applying the observations from the reduction of $PB_{\e\s}\cV_\s$, we get
	\ba
		&PB_{\e\s}(0,\kappa,\l,0,\vec{\b})\cV_\e(a,\vec{b};0,\kappa,\l,0,\vec{\b})=-P\Big[\kappa L_k(k_*,0;0)+ik_*\Pi_0 D_U^2\tilde{\cN}(0;k_*,0,0)(\Upsilon_\a,\cdot)\Big]\cdot\\
		&\quad \cdot T_\l\Big[\kappa L_k(k_*,0;0)D_\xi\Upsilon_a+M(k_*\d_\xi)D_U^2\tilde{\cN}(0;k_*,0,0)(\Upsilon_\a,\Upsilon_a)\Big].
	\ea
	This has a linear-nonlinear term and a nonlinear-linear term. For the linear-nonlinear term, we have
	\ba
		LN:=-\kappa PL_k(k_*,0;0)T_\l M(k_*\d_\xi)D_U^2\tilde{\cN}(0;k_*,0,0)(\Upsilon_\a,\Upsilon_a).
	\ea
	We have that $\hat{LN}(0)=0$ because $S_k(0,0)=0$ in the $O(2)$-invariant case. More generally, this ``pairs'' with (1) in the list of terms in $PB_{\e\e\s}\Upsilon_a$ to recover the ghost of $\cA$ in the singular $B$ equation. Turning to the nonlinear-linear term, we have
	\ba
		NL:=-ik_*P\Pi_0D_U^2\tilde{\cN}(0;k_*,0,0)(\Upsilon_\a,T_\l L_k(k_*,0;0)D_\xi\Upsilon_a),
	\ea
	and
	\ba
		\widehat{NL}(0)&=-\frac{1}{4}ik_*\a\Pi_0\Big(-\bar{a}\tilde{\cQ}(k_*,-k_*;0)(r,\bar{N_1}S_k(-k_*,0)\bar{r})\\
		&\quad+a\tilde{\cQ}(-k_*,k_*;0)(\bar{r}, N_1S_k(k_*,0)r)\Big)+\cO(|\l|).
	\ea
	Since $-S_k(-k_*,0)=\overline{S_k(k_*,0)}$ by the chain rule, the above is given by
	\ba
		\widehat{NL}(0)&=-\frac{1}{4}ik_*\a(a+\bar{a})\tilde{\cQ}(k_*,-k_*;0)(r,\overline{N_1S_k(k_*,0)r})+\cO(|\l|)\\
		&\quad=\frac{1}{4}k_*\a(a+\bar{a})\tilde{\cQ}(k_*,-k_*;0)(r,\overline{iN_1S_k(k_*,0)r})+\cO(|\l|).
	\ea
\end{proof}
Combining Lemmas \ref{lem:generalBS}, \ref{lem:generalBES}, \ref{lem:generalBSS} with Proposition \ref{prop:generalEESESS} and Theorems \ref{thm:coperiodicstability} and \ref{thm:generalPBV}, we get the key result
\begin{theorem}\label{thm:O2reduced}
		Suppose $L(\mu)$ and $\cN(U,\mu)$ are is in Theorem \ref{thm:LSReduction} and further assume that they are $O(2)$-invariant. Then the reduced equation for linear stability is given by
		\ba\label{eq:O2Reduced}
			-\l \bp a_1\\a_2\\ \vec{b}\ep+M(\e,\s,\kappa,\vec{\b})\bp a_1\\a_2\\ \vec{b}\ep+\cE(\e,\s,\kappa,\vec{\b},\l)\bp a_1\\a_2\\ \vec{b}\ep=0,
		\ea
		where $\cE=\cO(\e^3)$ under the scalings $\s\sim\e$ and $\l\sim\e^2$, $M(\e,\s,\kappa,\vec{\b})$ is a block matrix of the form
		\ba
			M(\e,\s,\kappa,\vec{\b})=\bp M_{11}(\e,\s,\kappa,\vec{\b}) & M_{12}(\e,\s,\kappa,\vec{\b})\\
			M_{21}(\e,\s,\kappa,\vec{\b}) & M_{22}(\e,\s,\kappa,\vec{\b}) \ep,
		\ea
		with $M_{11}$ a $2\times 2$ matrix, $M_{12}$ a $2\times N$ matrix, $M_{21}$ an $N\times 2$ matrix and $M_{22}$ an $N\times N$ matrix given by
		\ba
			M_{11}(\e,\s,\kappa,\vec{\b})=\e^2\bp 2\a\Re\g & 0\\
											2\a\Im\g & 0 \ep -i\kappa k_*\e\s[[i\tl_{kk}(k_*,0)]]+\frac{1}{2}k_*^2\s^2[[\tl_{kk}(k_*,0) ]],
		\ea
		\ba
			M_{12}(\e,\s,\kappa,\vec{\b})^i=\e^2\bp \a\Re V_1\cdot v_i\\
													\a\Im V_1\cdot v_i\ep,
		\ea
		where $M_{12}^i$ denotes the $i$-th column of $M_{12}$,
		\ba
			M_{21}(\e,\s,\kappa,\vec{\b})=\s^2 \bp -2\a k_*^2W_0 & 0\ep,
		\ea
		where $W_0$ is as in Theorem \ref{thm:O2multilin}, and finally $M_{22}$ is given by
		\ba
			M_{22}(\e,\s,\kappa,\vec{\b})=-\frac{1}{2}k_*^2\s^2\Pi_0 S_{kk}(0,0)\Pi_0.
		\ea
\end{theorem}
To extend this to the non-local case, we first observe that $D_U\sN(U;k,\mu)$ can act on exponentials whose frequencies are rational numbers, not just integers. What underlies this observation, is that if $u$ is a $\frac{2\pi}{k}$-periodic function then for any integer $q\geq1$ $u$ is also a $\frac{2\pi}{k/q}$-periodic function as well. As $D_U\sN(U;k,\mu)$ is (up to conjugation by the coordinate change $I_ku=u(kx)$) the Fr\'echet derivative of $\sN$ restricted to the subspace of $\frac{2\pi}{k}$-periodic functions, it thus admits an extension to the $\frac{2\pi}{k/q}$-periodic functions for all integers $q\geq 1$. So to extend to the nonlocal case, one needs to assume that the Schwartz kernel of $D_U\sN(U;k,\mu)$ and the multipliers $\sQ$ and $\sC$ admit (necessarily unique) smooth extensions to all frequencies. Once one has these extensions, then the argument in this section goes through mutatis mutandis.\\

Similar to the $O(2)$-case, the reduced equation for the compatible $SO(2)$-case can be computed as
\begin{theorem}\label{thm:SO2reduced}
		If $L(k,\mu)$ and $\cN$ are as in Theorem \ref{thm:LSReduction} and further assume that they are compatible in the sense that the conclusion of Theorem \ref{thm:O2compat} holds. Then the reduced equation for stability is given by
		\ba\label{eq:SO2Reduced}
			-\l\bp a_1\\ a_2\\ \vec{b}\ep+M(\e,\s,\kappa,\vec{\b})\bp a_1\\a_2\\ \vec{b}\ep+\cE(\e,\s,\kappa,\vec{\b},\l)\bp a_1\\a_2\\ \vec{b}\ep=0,
		\ea
		where $\cE=\cO(\e^3)$ under the scalings $\s\sim\e$ and $\l\sim\e^2$ and $M(\e,\s,\kappa,\vec{\b})$ is a block matrix of the form
		\ba
		M(\e,\s,\kappa,\vec{\b})=\bp M_{11}(\e,\s,\kappa,\vec{\b}) & M_{12}(\e,\s,\kappa,\vec{\b})\\
		M_{21}(\e,\s,\kappa,\vec{\b}) & M_{22}(\e,\s,\kappa,\vec{\b}) \ep,
		\ea
		with $M_{11}$ a $2\times 2$ matrix, $M_{12}$ a $2\times N$ matrix, $M_{21}$ an $N\times 2$ matrix and $M_{22}$ an $N\times N$ matrix given by
		\ba
		M_{11}(\e,\s,\kappa,\vec{\b})=\e^2\bp 2\a\Re\g & 0\\
		2\a\Im\g & 0 \ep -i\kappa k_*\e\s[[i\tl_{kk}(k_*,0)]]+\frac{1}{2}k_*^2\s^2[[\tl_{kk}(k_*,0) ]],
		\ea
		\ba
		M_{12}(\e,\s,\kappa,\vec{\b})^i=\e^2\bp \a\Re V_1\cdot v_i\\
		\a\Im V_1\cdot v_i\ep,
		\ea
		where $M_{12}^i$ denotes the $i$-th column of $M_{12}$,
\end{theorem}
\begin{proof}
	A slight modification of Lemmas \ref{lem:generalBS}, \ref{lem:generalBSS} and \ref{lem:generalBES} gives $P(B_\s(0,\kappa,\l,0,\vec{\b})+\frac{1}{2}B_{\s\s}(0,\kappa,\l,0,\vec{\b})+B_{\e\s}(0,\kappa,\l,0,\vec{\b}))(\Upsilon_a+\e\vec{b})$ in reduced form as
	\ba\label{eq:SO2BSBSSBES}
		& \frac{1}{2}\s^2\begin{bmatrix}
			k_*^2[[\ell S_{kk}(k_*,0)r ]] & 0_{2,N}\\
			0_{N,2} & k_*^2\Pi_0 S_{kk}(0,0)\Pi_0
		\end{bmatrix}
		\begin{pmatrix}
			a_1\\
			a_2\\
			\vec{b}
		\end{pmatrix}\\
	&\quad+\e\s\begin{bmatrix}
		-ik_*\kappa[[i\ell S_{kk}(k_*,0)r]] & \cO(\e) \\
		\e^{-1}\bp \frac{1}{2}ik_*\a\Pi_0\tilde{\cQ}(k_*,-k_*;0)(r,\bar{r}) & 0_{N,1} \ep & 0_{N,N}
	\end{bmatrix}
	\begin{pmatrix}
	a_1\\
	a_2\\
	\vec{b}
	\end{pmatrix}.
	\ea
	The relevant modification to Proposition \ref{prop:generalEESESS} is given by
	\ba\label{eq:SO2EESESS}
		&\frac{1}{4}i\e\s\kappa\a(a+\bar{a}) \Pi_0\tilde{\cQ}(k_*,-k_*;0)(r,\bar{r})-\frac{1}{4}\e\s k_*\kappa\a(a\Pi_0\tilde{\cQ}(k_*,-k_*;0)(iN_1S_k(k_*,0)r,\bar{r})+c.c.)\\
		&\quad-\frac{1}{4}\s^2 k_*^2\a(a+\bar{a})\Pi_0\tilde{\cQ}_\nu(k_*,-k_*;0)(r,\bar{r})+\frac{1}{4}\e\s i k_*\frac{\d\a}{\d\e}(0,\kappa)(a+\bar{a}) \Pi_0\cQ(k_*,-k_*;0)(r,\bar{r}),
	\ea
	where the last term is coming from looking at $\Pi_1\frac{\d^2\tilde{u}_{0,\kappa,\vec{\b}}}{\d\e^2}$ in $\frac{1}{2}ik_*\Pi_0\Big(\tilde{\cQ}(k_*,-k_*;0)\Big(\widehat{\frac{\d^2\tilde{u}_{0,\kappa,\vec{\b}}}{\d\e^2}}(1),\bar{ar} \Big)+\tilde{\cQ}(-k_*,k_*;0)\Big(\widehat{\frac{\d^2\tilde{u}_{0,\kappa,\vec{\b}}}{\d\e^2}}(-1),ar \Big)\Big)$. For the analog of Theorem \ref{thm:generalPBV}, we note that the analysis in mode 1 is unchanged as the compatibility conditions are only relevant in mode 0, which recovers the claimed form of the mode 1 equation in the theorem statement. So now, we will only consider the terms which we showed vanished by using $O(2)$-invariance in the proof. These terms are given by
	\be\label{eq:LNepssigma}
	-P\Big(k_*L_k(k_*,0;0)+iC(0,\kappa,\vec{\b})\Big)T_0M(k_*\d_\xi)D_U^2\tilde{\cN}(0;k_*,0,0)(\Upsilon_\a,\Upsilon_a),
	\ee
	which was the linear-nonlinear term of $PB_\s\cV_\e$, the Fourier means of $PB_\s\cV_j$ for $j\in\{\e\e,\e\s,\s\s \}$, and 
	\be\label{eq:SO2BESCVE}
		-\kappa PL_k(k_*,0;0)T_\l M(k_*\d_\xi)D_U^2\tilde{\cN}(0;k_*,0,0)(\Upsilon_\a,\Upsilon_a),
	\ee
	which was part of the linear-nonlinear term of $PB_{\e\s}\cV_\e$. Looking at \eqref{eq:LNepssigma}, we can compute the Fourier mean as
	\ba
		&\text{Mean}(-P\Big(k_*L_k(k_*,0;0)+iC(0,\kappa,\vec{\b})\Big)T_0M(k_*\d_\xi)D_U^2\tilde{\cN}(0;k_*,0,0)(\Upsilon_\a,\Upsilon_a))\\
		&\quad=-\frac{1}{4}k_*\Pi_0S_k(0,0)N_0\Big(\tilde{\cQ}(k_*,-k_*;0)(\a r,\bar{a r})+\tilde{\cQ}(-k_*,k_*;0)(\a\bar{r},ar)\Big)\\
		&\quad=-\frac{1}{4}\a k_*(a+\bar{a})\Pi_0 S_k(0,0)N_0\tilde{\cQ}(k_*,-k_*;0)(r,\bar{r}).
	\ea
	In reduced form, this is given by
	\be
		\s\bp -\frac{1}{2}\a k_*\Pi_0 S_k(0,0)N_0\tilde{\cQ}(k_*,-k_*;0)(r,\bar{r}) & 0_{N,1} \ep\bp a_1 \\ a_2 \ep.
	\ee
	By the compatibility condition, this cancels with the corresponding $\cO(\s)$-term in \eqref{eq:SO2BSBSSBES}. For \eqref{eq:SO2BESCVE}, we note that the Fourier mean is given by
	\be
		-\frac{1}{4}\kappa\a(a+\bar{a})\Pi_0S_k(0,0)N_0\tilde{\cQ}(k_*,-k_*;0)(r,\bar{r})+\cO(|\l|).
	\ee
	This term cancels with the first term of \eqref{eq:SO2EESESS} by the compatibility condition. We are left with three final terms, $PB_\s\cV_{\e\e}$, $PB_\s\cV_{\e\s}$ and $PB_\s\cV_{\s\s}$. Starting with $PB_\s\cV_{\s\s}$, we apply the expansion in \eqref{eq:cVsecondderiv} to get
	\ba
	&PB_\s(0,\kappa,\l,0,\vec{\b})\cV_{\s\s}(a,\vec{b};0,\kappa,\l,0,\vec{\b})\\
	&\quad=-PB_\s(0,\kappa,\l,0,\vec{\b})\Big(T_\l B_{\s\s}(0,\kappa,\l,0,\vec{\b})\Upsilon_a+2T_\l B_\s(0,\kappa,\l,0,\vec{\b})\cV_\s(a,\vec{b};0,\kappa,\l,0,\vec{\b}) \Big).
	\ea
	As $B_\s$ and $B_{\s\s}$ are Fourier multipler operators, and $\cV_\s=-T_\l B_\s\Upsilon_a$, it follows that $PB_\s\cV_{\s\s}$ is always mean-free independently of compatibility.\\
	
	For $PB_\s\cV_{\e\e}$, we have
	\ba
	&PB_\s(0,\kappa,\l,0,\vec{\b})\cV_{\e\e}(a,\vec{b};0,\kappa,\l,0,\vec{\b})\\
	&\quad=-PB_\s(0,\kappa,\l,0,\vec{\b})T_\l\Big(B_{\e\e}(0,\kappa,\l,0,\vec{\b})\Upsilon_a-2B_\e(0,\kappa,\l,0,\vec{\b})(\vec{b}+\cV_\e(a,\vec{b};0,\kappa,\l,0,\vec{\b})\Big).
	\ea
	As we're interested in the Fourier mean of this quantity, there are three terms of interest, which we will denote by $LN_1$, $LL$, and $LN_2$ coming from $PB_\s T_\l B_{\e\e}\Upsilon_a$, $PB_\s T_\l B_\e\vec{b}$ and $PB_\s T_\l\cV_\e$ respectively. We begin by computing $\d_\e^2 D_U\cN(\tilde{u}_{\e,\kappa,\vec{\b}};k,\mu,\s)=\d_\e^2M(k\d_\xi)D_U\tilde{\cN}(\tilde{u}_{\e,\kappa,\vec{\b}};k,\mu,\s)$
	\ba
		&\d_\e^2M(k\d_\xi)D_U\tilde{\cN}(\tilde{u}_{\e,\kappa,\vec{\b}};k,\mu,\s)|_{\e=\s=0}=2\kappa\d_\xi\Pi_0D_U^2\tilde{\cN}(0;k_*,0,0)(\Upsilon_\a,\cdot)+M(k_*\d_\xi)D_U^2\tilde{\cN}(0;k_*,0,0)(\frac{\d^2\tilde{u}_{0,\kappa,\vec{\b}}}{\d\e^2},\cdot)\\
		&\quad+M(k_*\d_\xi)D_U^3\tilde{\cN}(0;k_*,0,0)(\Upsilon_\a,\Upsilon_\a,\cdot)+\kappa M(k_*\d_\xi)\d_k D_U^2\tilde{\cN}(0;k_*,0,0)(\Upsilon_\a,\cdot).
	\ea
	As we're interested in the Fourier mean, the first term drops out leaving us with the other three to work with. The third term drops out as well because $D_U^3\cN(0;k_*,0,0)(\Upsilon_\a,\Upsilon_\a,\Upsilon_a)$ is Fourier-supported in $\{\pm1,\pm3 \}$. This leaves us with two quadratic terms. Expanding out the relevant part of $LN_1$, we get
	\ba
		\widehat{LN_1}(0)&=\text{Mean}(-k_*PL_k(k_*,0;0)T_\l\d_\e^2M(k\d_\xi)D_U\tilde{\cN}(\tilde{u}_{\e,\kappa,\vec{\b}};k,\mu,\s)|_{\e=\s=0}\Upsilon_a)\\
		&\quad=-k_*\Pi_0 S_k(0,0)N_0\Big(\frac{1}{2}\tilde{\cQ}(k_*,-k_*;0)(\widehat{\frac{\d^2\tilde{u}_{0,\kappa,\vec{\b}}}{\d\e^2} }(1),\bar{a}\bar{r} )+c.c.\\
		&\quad+\frac{1}{4}\kappa\d_k\tilde{\cQ}(k_*,-k_*;0)(\a r,\bar{a}\bar{r})+\frac{1}{4}\kappa\d_k\cQ(-k_*,k_*;0)(\a \bar{r},ar) \Big)+\cO(|\l|).
	\ea
	By compatibility, the term $-\frac{1}{2}k_*\frac{\d\a}{\d\e}(0,\kappa)(a+\bar{a})\Pi_0S_k(0,0)N_0\tilde{\cQ}(k_*,-k_*;0)(r,\bar{r})$ cancels with the corresponding term in \eqref{eq:SO2EESESS}. For what's left, we plug in the formula for $(I-\Pi_1)\frac{\d^2\tilde{u}}{\d\e^2}$ from \eqref{eq:WExpansionb2} and simplifying we get
	\ba
		\widehat{LN_1}(0)&=-\frac{1}{4}k_*\kappa\a(a+\bar{a})\Pi_0 S_k(0,0)N_0\tilde{\cQ}(k_*,-k_*;0)(-N_1(I-\Pi_1)S_k(k_*,0)\Pi_1r,\bar{r}).
	\ea
	As by symmetry one has
	\be
		\frac{\d}{\d k}(I-\Pi_0)\cQ(k\eta_1,k\eta_2;0)=-i(\eta_1+\eta_2)(I-\Pi_0)\cQ_\nu(k\eta_1,k\eta_2).
	\ee
	For $LL$, we have that
	\be
		\widehat{LL}(0)=-\text{Mean}(PB_\s T_\l B_\e\vec{b} )=0.
	\ee
	As the nonlinear term will map $\vec{b}$ into Fourier modes $\pm1$ and each term in the linear part of $B_\e$ has a factor of $\d_\xi$ which annihilates $\vec{b}$. We turn to the final piece of $PB_\s\cV_{\e\e}$, which is $LN_2=-PB_\s T_\l\cV_\e$. Expanding out the mean of $LN_2$ and applying Proposition \ref{prop:CVExpansion}, we get
	\ba
		\widehat{LN_2}(0)=-\text{Mean}(PB_\s T_\l B_\e\cV_\e)=\text{Mean}(PB_\s T_\l B_\e T_\l B_\e\Upsilon_a  ).
	\ea
	The first step in simplifying $\widehat{LN_2}(0)$ is to compute the Fourier mean of $B_\e T_\l B_\e\Upsilon_a$. To that end, we write out $B_\e T_\l B_\e\Upsilon_a$ as
	\begin{equation*}
		B_\e T_\l B_\e\Upsilon_a=\tilde{LN}+\tilde{NL}+other,
	\end{equation*}
	where $\tilde{LN}$ is the linear-nonlinear term and $\tilde{NL}$ is the nonlinear-linear term. Expanded out, we have that $\tilde{LN}$ and $\tilde{NL}$ are given by
	\ba
		\tilde{LN}&=\Big(\kappa L_k(k_*,0;0)D_\xi+(\kappa d_*+k_*d_\e(0,\kappa,\vec{\b}))\d_\xi\Big)T_\l M(k_*\d_\xi)D_U^2\tilde{\cN}(0;k_*,0,0)(\Upsilon_\a,\Upsilon_a),\\
		\tilde{NL}&=M(k_*\d_\xi)D_U^2\tilde{\cN}(0;k_*,0,0)(\Upsilon_\a,T_\l(\kappa L_k(k_*,0;0)D_\xi\Upsilon_a+(\kappa d_*+k_*d_\e(0,\kappa,\vec{\b}))\d_\xi\Upsilon_a ).
	\ea
	Strictly speaking there are linear-linear and nonlinear-nonlinear terms as well, but they do not contribute to the Fourier mean as they have Fourier support $\{\pm1\}$ and $\{\pm1,\pm3\}$ respectively. As $B_\s$ and $T_\l$ are Fourier multiplier operators, it follows that $B_\s T_\l\tilde{LN}$ is mean-free because $\tilde{LN}$ itself is mean-free as both $L_k(k_*,0;0)D_\xi$ and $\d_\xi$ annihilate constants. For $\tilde{NL}$, what we find in mode 0 is
	\be
		\text{Mean}(\tilde{NL})=\frac{1}{4}k_*\a(I-\Pi_0)\Big(\bar{a}\tilde{\cQ}(k_*,-k_*;0)(r,\overline{N_1S_k(k_*,0)r})+a\tilde{\cQ}(-k_*,k_*;0)(\bar{r},N_1S_k(k_*,0)r)\Big).
	\ee
	After applying $PB_\s T_\l$, we get
	\be
		\widehat{LN_2}(0)=\frac{1}{4}k_*^2\a\Pi_0 S_k(0,0)N_0\Big(\bar{a}\tilde{\cQ}(k_*,-k_*;0)(r,\overline{N_1S_k(k_*,0)r})+a\tilde{\cQ}(k_*,-k_*;0)(\bar{r},N_1S_k(k_*,0)r)\Big).
	\ee
	To finish the proof, we need to reduce the Fourier mean of $PB_\s\cV_{\e\s}$. We first expand $PB_{\s}\cV_{\e\s}$ as
	\be
		PB_\s\cV_{\e\s}=-PB_\s T_\l\Big(B_{\e\s}\Upsilon_a+B_\s\vec{b}+B_\s\cV_\e+B_\e\cV_\s\Big).
	\ee
	For the first term, we only need to consider the term of $B_{\e\s}$ coming from the nonlinearity as we're interested in the mean value here. This is given by
	\ba
		\text{Mean}(-PB_\s T_\l B_{\e\s}\Upsilon_a )&=\text{Mean}(-PB_\s T_\l\Big(ik_*\Pi_0D_U^2\tilde{\cN}(0;k_*,0,0)(\Upsilon_\a,\Upsilon_a)\\
		&\quad+M(k_*\d_\xi)D_U^2\tilde{\cN}_\s(0;k_*,0,0)(\Upsilon_\a,\Upsilon_a) \Big)).
	\ea
	Expanding this out gives
	\ba
		\text{Mean}(-PB_\s T_\l B_{\e\s}\Upsilon_a )&=-\frac{1}{4}ik_*^2\a(a+\bar{a})\Pi_0 S_k(0,0)N_0\tilde{\cQ}(k_*,-k_*;0)(r,\bar{r})\\
		&\quad+i\frac{1}{4}k_*\a\Pi_0S_k(0,0)N_0\Big(\bar{a}\tilde{\cQ}_\nu(k_*,-k_*;0)(r,\bar{r})+a\tilde{\cQ}_\nu(-k_*,k_*;0)(\bar{r},r) \Big).
	\ea
	By the chain rule, $\cQ_\nu(k_*,-k_*;0)=-\cQ_\nu(-k_*,k_*;0)$, and so the second term can be written as $i\frac{1}{4}k_*\a(a-\bar{a})\Pi_0S_k(0,0)N_0\tilde{\cQ}_\nu(r,\bar{r})$. Turning to $PB_\s T_\l B_\s\vec{b}$, we get the mean value as
	\be
		PB_\s T_\l B_\s\vec{b}=k_*^2\Pi_0 S_k(0,0)N_0S_k(0,0)\Pi_0\vec{b}.
	\ee
	Next, we look at $PB_\s T_\l B_\s\cV_\e$. We recall that the mean value of $B_\s\cV_\e$ is given by 
	\be
		\widehat{B_\s\cV_\e}=-\frac{1}{4}\a k_*(a+\bar{a})S_k(0,0)N_0\tilde{\cQ}(k_*,-k_*;0)(r,\bar{r})+\cO(|\l|).
	\ee
	Applying $PB_\s T_\l$ to this quantity gives
	\be
		\text{Mean}(PB_\s T_\l B_\s\cV_\e)=\frac{1}{4}\a k_*^2(a+\bar{a})\Pi_0 S_k(0,0)N_0 S_k(0,0) N_0\tilde{\cQ}(k_*,-k_*;0)(r,\bar{r})+\cO(|\l|).
	\ee
	To finish this computation, we look at $PB_\s T_\l B_\e\cV_\s$. As in the proof of Theorem \ref{thm:generalPBV}, we only need to consider the nonlinear-linear term, which we recall is given by
	\be
		NL:=M(k_*\d_\xi)D_U^2\tilde{\cN}(0;k_*,0,0)(\Upsilon_\a,k_*T_0L_k(k_*,0;0)\Upsilon_a ).
	\ee
	Applying $PB_\s T_\l$ to this quantity, we get
	\ba
		\text{Mean}(PB_\s T_\l B_\e\cV_\s)&=\frac{1}{4}\a k_*^2\Pi_0 S_k(0,0)N_0\Big(\bar{a}\tilde{\cQ}(k_*,-k_*;0)(r,\overline{N_1}S_k(-k_*,0)\bar{r} )+a\tilde{\cQ}(-k_*,k_*;0)(\bar{r},N_1S_k(k_*,0)r) \Big)\\
		&\quad=\frac{1}{4}\a (a-\bar{a})k_*^2\Pi_0 S_k(0,0)N_0\tilde{\cQ}(k_*,-k_*;0)(r,\overline{N_1S_k(k_*,0)r}).
	\ea
	Putting all of this information together, in the compatible $SO(2)$-case, the mean value of $PB(\Upsilon_a+\e\vec{b}+\cV)$ is given by
	\ba
		&-\l\vec{b}+\frac{1}{2}k_*^2\s^2\Pi_0 S_{kk}(0,0)\Pi_0\vec{b}-k_*^2\s^2\Pi_0 S_k(0,0)N_0S_k(0,0)\Pi_0\vec{b}\\
		&\quad--\frac{1}{4}k_*^2\s^2\a(a+\bar{a}) \Pi_0 S_{kk}(0,0)N_0(I-\Pi_0)\tilde{\cQ}(k_*,-k_*;0)(r,\bar{r})\\
		&\quad-\frac{1}{4}\e\s k_*\kappa\a(a\Pi_0\tilde{\cQ}(k_*,-k_*;0)(iN_1S_k(k_*,0)r,\bar{r})+c.c.)-\frac{1}{4}\s^2 k_*^2\a(a+\bar{a})\Pi_0\tilde{\cQ}_\nu(k_*,-k_*;0)(r,\bar{r})\\
		&\quad-\frac{1}{4}\e\s k_*\kappa\a(a+\bar{a})\Pi_0 S_k(0,0)N_0\tilde{\cQ}(k_*,-k_*;0)(-N_1(I-\Pi_1)S_k(k_*,0)\Pi_1r,\bar{r})\\
		&\quad+\frac{1}{4}k_*^2\e\s\a\Pi_0 S_k(0,0)N_0\Big(\bar{a}\tilde{\cQ}(k_*,-k_*;0)(r,\overline{N_1S_k(k_*,0)r})+a\tilde{\cQ}(k_*,-k_*;0)(\bar{r},N_1S_k(k_*,0)r)\Big)\\
		&\quad-\frac{1}{4}ik_*^2\s^2\a(a+\bar{a})\Pi_0 S_k(0,0)N_0\tilde{\cQ}(k_*,-k_*;0)(r,\bar{r})+i\frac{1}{4}k_*\s^2\a(a-\bar{a})\Pi_0S_k(0,0)N_0\tilde{\cQ}_\nu(r,\bar{r})\\
		&\quad+\frac{1}{4}\a k_*^2\s^2(a+\bar{a})\Pi_0 S_k(0,0)N_0 S_k(0,0) N_0\tilde{\cQ}(k_*,-k_*;0)(r,\bar{r})\\
		&\quad+\frac{1}{4}\a (a-\bar{a})k_*^2\s^2\Pi_0 S_k(0,0)N_0\tilde{\cQ}(k_*,-k_*;0)(r,\overline{N_1S_k(k_*,0)r})+h.o.t.\ .
	\ea
\end{proof}
In general following the same argument, one can find a reduced equation for all $SO(2)$ systems, compatible or not.
\begin{theorem}\label{thm:generalSO2reduced}
	
	Suppose $L(k,\mu)$ has Fourier symbol satisfying the Turing hypotheses in Hypothesis \ref{hyp:Lin} and $\cN(u;k,\mu)$ is a local quasilinear nonlinearity satisfying Hypothesis \ref{hyp:Nonlin}. Then the reduced equation takes the form
	\be
		\Big[-\l+M(\e,\s,\kappa,\vec{\b})+\cS(\e,\s,\kappa,\vec{\b})+\cE(\e,\s,\kappa,\vec{\b},\l) \Big]\bp a_1\\a_2\\ \vec{b}\ep=0,
	\ee
	where $M$ is the same matrix as in Theorem \ref{thm:SO2reduced}, $\cE$ is $\cO(\e^3)$ under the scalings $\s\sim\e$ and $\l\sim\e^2$, and $\cS\sim\s$ admits the expansion $\cS=\s\cS_0+\e\s\cS_1$ which arises from the linearization of the singular terms and the terms involving $\cA$, when $\cA$ is identified with $\frac{\d\a}{\d\e}$, and $\cB$ respectively.\\
\end{theorem}
We will keep track of when the singular terms have an effect on the remainder of the proof of stability in the following subsection. In either Theorem \ref{thm:O2reduced} or Theorem \ref{thm:SO2reduced}, we show the matching of the matrices by directly comparing the coefficients.\\

Let us finish this section by sketching the relation of the Lyapunov-Schmidt reduction procedure for stability to the corresponding multi-scale expansion procedure. Summarizing the argument in Section 4 of \cite{SZJV}, we pre-process the multiscale expansion and post-process the Lyapunov-Schmidt reduction by the following prescriptions. One begins by assuming the expansion of the Ansatz 
\be 
\tilde{A}=(\a(0,\kappa,\vec{\b})+\e\frac{\d\a}{\d\e}(0,\kappa,\vec{\b})+h.o.t.+a )e^{i(\kappa\Hx-\Omega\Ht)}, \quad \tilde{B}=\b+b,
\ee
and then linearizes in $a$ and $b$. The post-processing on the Lyapunov-Schmidt side is done by writing
\be
\l(\e,\kappa,\s,\b)=\e^2\hat{\l}(\e,\kappa,\Hs,\b), \quad \s_{cGL}=k\Hs_{LS},
\ee
where we are using the true value of $k$ and not $k_*$ as in \cite{SZJV}. This is done to account for the $\cB$ term as well as the $\kappa\a$ term in $\cS_2$.

\section{From the reduced equation to stability}\label{subsec:ReducedToStability}
In this section we will show that the Lyapunov-Schmidt reduction procedure in the previous section recovers the linear stability prediction of the (mcGL)-system in the compatible $SO(2)$-invariant case. The first step is to reduce to $\s$ small, so that we can safely use arguments based on Taylor's theorem. To reduce to $\s$ small, we adapt the argument in \cite{WZ2} to the case where there are conservation laws.
\begin{proposition}\label{prop:smallsigma}
	To show stability or instability for all $|\s|\leq\frac{1}{2}$, it suffices to show the corresponding property for $|\s|\ll 1$.
\end{proposition}
\begin{proof}
	 Note that $L(k,\mu)$ is the linearization about $u=0$, and by the Turing hypotheses it has stable spectrum outside a small open set centered around $\{(0,0) \}\cup\{\pm (k_*,0) \}$. Hence if $|\s|$ is large enough, it follows that $L(k,\mu;\s)$; or equivalently $L(k,\mu;\s)+d(\e,\kappa,\vec{\b})\d_\xi$, has stable spectrum.\\
	 
	 By the ellipticity of $L(k,\mu)$, we have for $\l\in\rho(L(k,\mu;\s)$ that $(\l-L(k_,\mu;\s))^{-1}:L^2_{per}([0,2\pi];\RR^n)\rightarrow H^s_{per}([0,2\pi];\RR^n)$ is a bounded operator. In addition, for all $\e,\kappa,\vec{\b}$ such that $\tilde{u}_{\e,\kappa.\vec{\b}}$ exists we have that $D\cN(u_{\e,\kappa,\vec{\b}};k,\mu,\s):H^s_{per}([0,2\pi];\RR^n)\rightarrow L^2_{per}([0,2\pi];\RR^n)$ is a bounded operator, with bounds independent of $\s$ as $\s$ runs through the compact set $|\s|\leq \frac{1}{2}$. We denote $L(k,\mu;\s)+d(\e,\kappa,\vec{\b})\d_\xi+D\cN(\tilde{u}_{\e,\kappa,\vec{\b}};k,\mu,\s)$ by $\cL(\e)$ and the corresponding resolvent by $R(\e,\l$), and rewrite $\l-\cL(\e)$ as $\l-\cL(0)-\Delta\cL(\e)=(\l-\cL(0))(Id-R(0,\l)\Delta\cL(\e))$. Since $\Delta\cL(\e)$ is bounded from $H^s_{per}([0,2\pi];\RR^n)\rightarrow L^2_{per}([0,2\pi];\RR^n)$ with norm $\cO(\e)$ and $R(0,\l):L^2_{per}([0,2\pi];\RR^n)\rightarrow H^s_{per}([0,2\pi];\RR^n)$ is bounded, we find for $\e$ sufficiently small that $(Id-R(0,\l)\Delta\cL(\e)):H^s_{per}([0,2\pi];\RR^n)\rightarrow H^s_{per}([0,2\pi];\RR^n)$ is invertible by expanding in a Neumann series. In particular, $R(\e,\l)=R(0,\l)+\cO(\e)$. A similar calculation gives continuity of the resolvent about other $\e_0$. From this we conclude spectral continuity of $B(\e,\kappa,\l,\s,\vec{\b})$. 
	 By the spectral continuity argument above, we see that $L(k,\mu;\s)+d(\e,\kappa,\vec{\b})\d_\xi+D\cN(\tilde{u}_{\e,\kappa,\vec{\b}};k,\mu,\s)$ has stable spectrum for $|\s|>=\s_0>0$ for $\e$ sufficiently small.
\end{proof}
Note that the above proof makes no reference to compatibility, and hence applies to the reduced equation obtained in Theorem \ref{thm:generalSO2reduced}. We now establish coperiodic stability.
\begin{theorem}\label{thm:coperiodicspectrum}
	Suppose $L$, $\cN$ are as in Theorem \ref{thm:SO2reduced}. At $\s=0$, 0 is an eigenvalue of algebraic multiplicity $N+1$ and has stable spectrum if and only if $\Re\g<0$.
\end{theorem}
\begin{proof}
	By Theorem \ref{thm:coperiodicstability}, we have that the reduced equation is upper block triangular, with blocks on the diagonal
	\be
		B_{11}(\e,0,\kappa,\vec{\b})=\bp 2\e^2\a^2\Re\g+\cO(\e^3) & 0\\
		2\e^2\a^2\Im\g+\cO(\e^3) & 0\ep+\cO(\e^2|\l|),
	\ee
	and the $N\times N$ lower block in the diagonal being identically 0. As $B_{11}$ has determinant $\cO(|\l|)$ and the lower block has characteristic polynomial $\l^N$, the desired statement about the multiplicity of 0 follows. For the claim about stability, $M_{11}$ has an eigenvector $\bp 1 & 0\ep^T+h.o.t.$ of eigenvalue $2\e^2\a^2\Re\g+h.o.t.$, and extending $\bp 1 & 0\ep^T+h.o.t.$ to $\bp 1 & 0 & 0_N\ep^T+h.o.t.$ proves the second claim.
	
\end{proof}
As before, this also applies to all $SO(2)$-invariant systems, even if they are incompatible as the singular term is $\cO(\s)$. We make the following important observation, in the $SO(2)$-case, compatible or not, there will generically be one generalized eigenvector in the 0-eigenspace. This is because $M_{12}$ in the $SO(2)$-case will generically have a nonzero second row and hence one of the columns will generically produce a full rank minor by pairing with the leading order column of $M_{11}$, but since $M$ can have rank at most 2 this means we will only have a single generalized eigenvector. Hence while we can conclude that there are $N+1$ continuous eigenvalues $\L_1(\e,\kappa,\s,\vec{\b})$,...,$\L_{N+1}(\e,\kappa,\s,\vec{\b})$ of $M$ of size $\e^2$ each of which vanishes at $\s=0$ as shown in \cite{K}, we cannot ensure that they are always smooth in the compatible $SO(2)$-case. Note that the claim about $\L_i=\cO(\e^2)$ does make use of compatibility in an important way, as it is proven by factoring an $\e^2$ out of $-\L(\e,\kappa,\s,\vec{\b})\bp a_1 & a_2 & \vec{b}\ep^T+M(\e,\s,\kappa,\vec{\b})\bp a_1 & a_2 & \vec{b}\ep^T=0$ and applying continuity of eigenvalues to 0. We start by showing that solutions to reduced equation exist and that the solutions can be matched with the linearized disperison relations obtained from the Ginzburg-Landau model.
\begin{theorem}\label{thm:specagree}
	Let $\L_{stab}$ denote the smooth eigenvalue of $M(\e,\s)$ which satisfies $\L_{stab}(\e,\kappa,0,\vec{\b})<0$ and $\L_1,...,\L_{N+1}$ denote the $N+1$ continuous eigenvalues of $M(\e,\s)$ satisfying $\L_j(\e,\kappa,0,\vec{\b})=0$. Then on the region where $\L_j$ is bounded by $\cO(\e^2,\e\s,\s^2)$, there exist precisely $N+2$ continuous functions (counting multiplicity) $\l_j$ satisfying the reduced equation in Theorem \ref{thm:SO2reduced} for some nonzero vector $(a_1,a_2,\vec{b})$ and for each $j\in\{stab,1,2,...,N \}$ we have that
	\be 
	\l_j(\e,\kappa,\s,\vec{\b})=\L_j(\e,\kappa,\s,\vec{\b})+e(\e,\kappa,\s,\vec{\b}),
	\ee where $e$ continuously depends on the parameters. Moreover, we have the bound
	\be
		e(\e,\kappa,\s,\vec{\b})=o(\e^2),
	\ee on the set of $\s$ satisfying $|\s|\leq\Hs_0\e$ for any fixed $\Hs_0$. Under the additional assumption that the $\l_j$ and $\L_j$ are $C^k$ for $k$ large enough, then the error $e$ can generically be bounded by $\cO(\e^4,\e^3\s,\e^2\s^2,\e\s^3,\s^4)$. Furthermore, if $\l(\e,\kappa,0,\vec{\b})=0$, then the error $e$ vanishes at $\s=0$ and in the smooth case can be reduced to $\cO(\e^3\s,\e^2\s^2,\e\s^3,\s^4)$.
\end{theorem}
\begin{remark}
	The key issue regarding the smoothness issue is that in the compatible $SO(2)$ case, it is not clear when the eigenvalues $\L_j$, $j=1,...,N+1$, split at first order due to the generic presence of the Jordan block at $\s=0$. This must be determined on a case-by-case basis depending on model parameters. For $N$ is large, then this could be done via rigorous numerics such as a straightforward application of interval arithmetic for instance.
\end{remark}
\begin{proof}
	This is a straightforward modification of the proof of Theorem 4.15 in \cite{WZ2}.\\
	
	Observe that \eqref{eq:SO2Reduced} is equivalent to
	\be
	((I-\tilde{\cE}(\e,\kappa,\s,\vec{\b}))\l+\cO(\e^2|\l|^2,\e\s|\l|^2,\s^2|\l|^2))\bp a_1 \\ a_2 \\ \vec{b} \ep=M(\e,\kappa,\s,\vec{\b})\bp a_1 \\ a_2 \\ \vec{b} \ep,
	\ee
	for the nonzero vector in the hypothesis of the Theorem, and where $\tilde{\cE}$ is the leading order term in the Taylor expansion of $\cE$ with respect to $\l$. Since $\tilde{\cE}(\e,\kappa,\s,\vec{\b})=\cO(\e^2,\e\s,\s^2)$, it follows that $I-\tilde{\cE}(\e,\kappa,\s,\vec{\b})$ is invertible for small $\e,\s$. So, we see that $\l$ is an eigenvalue of $(I-\tilde{\cE}(\e,\kappa,\s,\vec{\b}))^{-1}M(\e,\kappa,\s,\vec{\b})$ up to quadratic errors in $\l$. We can expand $(I-\tilde{\cE})^{-1}$ into its Neumann series to find
	\be\label{eq:NeumannSeries}
	(\l+\cO(\e^2|\l|^2,\e\s|\l|^2,\s^2|\l|^2))\bp a_1 \\ a_2 \\ \vec{b} \ep=\left(\sum_{n=0}^\infty\tilde{\cE}(\e,\kappa,\s,\vec{\b})^nM(\e,\kappa,\s,\vec{\b})  \right)\bp a_1 \\ a_2 \\ \vec{b} \ep.
	\ee
	But we also have that both $\tilde{\cE}(\e,\kappa,\s,\vec{\b}),M(\e,\kappa,\s,\vec{\b})$ are $\cO(\e^2,\e\s,\s^2)$, so 
	$$
	\tilde{\cE}(\e,\kappa,\s,\vec{\b})M(\e,\kappa,\s,\vec{\b})=\cO(\e^4,\e^3\s,\e^2\s^2,\e\s^3,\s^4).
	$$
	A priori we know that $\l=\cO(\e^2,\e\s,\s^2)$, so in particular the quadratic errors in \eqref{eq:NeumannSeries} are at least as small as the error bound given by $\tilde{\cE}(\e,\kappa,\s,\vec{\b})M(\e,\kappa,\s,\vec{\b})$. Let us now write
	\ba
		\s&=\e\Hs,\quad \l=\e^2\hat{\l}, \quad \L=\e^2\hat{\L}, \\
		m(\e,\s)&=\e^2\hat{m}(\e,\Hs), \quad M(\e,\s)=\e^2\hat{M}(\e,\Hs), \quad \tilde{E}(\e,\s)=\e^2\hat{E}(\e,\Hs),\\
	\ea
	so that $\hat{\l}$ satisfies
	\be
		((I-\e^2\hat{E} )\hat{\l}+\e^2\cO((1+\hat{\s})^2|\hat{\l}|)^2 )\bp a_1\\ a_2\\ \vec{b}\ep =\hat{M}(\e,\kappa,\s,\vec{\b})\bp a_1\\ a_2\\ \vec{b}\ep.
	\ee
	Hence, by Rouch\'e's theorem we have that $\hat{\l}=\hat{\L}_j+o(1)$ as $\e\to0$ locally uniformly in $\Hs$ for some $j\in\{stab,1,2,...,N \}$. Multiplying through by $\e^2$ gives the desired error estimate.\\
	
	In the situation where both the $\l_j$'s and the $\L_j$'s are smooth, then one uses the uniform expansions
	\begin{equation}
		\l_j(\e,\s)=c_0^j(\e)+c_1^j(\e)\s+c_2^j(\e)\s^2+h.o.t. \ ,
	\end{equation}
	and
	\begin{equation}
		\L_j(\e,\s)=C_0^j(\e)+C_1^j(\e)\s+C_2^j(\e)\s^2+h.o.t. \ ,
	\end{equation}
	near the origin. One then applies Kato style formulas for the $c_j^k$ and $C_j^k$ and matches powers of $\e$ to show the desired matching near $\s=0$. Away from $\s=0$, the generic behavior is that each eigenvalue is simple and hence smoothly depend on the coefficients of the matrix $M$ as shown in Chapter 2 Section 5.8 of \cite{K}. Hence one can Taylor expand with respect to the operator in order to get the desired error bound.
\end{proof}
We note that in the case where there are no conservation laws that both eigenvalues $\l_{stab}$ and $\l_1$ can be shown to be smooth via an implicit function theorem argument as in \cite{WZ2}. Here, because $\l=0$ has algebraic multiplicity $N+1\geq 2$, the implicit function theorem approach is not viable for the neutral eigenvalues as $\frac{\d}{\d\l}\det(M-\l I)|_{\l=0}=0$ forbids the use of the implicit function theorem.\\

Suppose for a moment, that we have a full set of solutions $\l_1,...,\l_{N+1}$ to the reduced equation and that both the $\l$'s and $\L$'s are all smooth. As in the case with no conservation laws, we know that the unordered lists $\cL_{LS}:=\{\l_1,...,\l_{N+1} \}$ and $\cL_{cGL}:=\{\L_1,...,\L_{N+1} \}$ both satisfy $\d_\s\cL_j(\e,0)=-\overline{\d_\s\cL_j(\e,0)}$, c.f. Lemma 4.13 of \cite{WZ2} for details. However, unlike the situation with no conservation laws, this no longer forces $\d_\s \cL_j(\e,0)\subset i\RR$ as exemplified by the unordered list $\{\pm\s \}$ where the list has the desired symmetry but each curve in the list can be arranged to have nonzero derivative with respect to $\s$. Moreover, the added symmetry from $O(2)$-invariance does not rule out lists like $\{\pm\s \}$ either. That said, if we assume that the lists $\cL_{LS}$ and $\cL_{cGL}$ can be made ``individually symmetric'' in the sense that one can arrange each $\l_j$ and $\L_j$ to be smooth and individually satisfy $\l_j(\e,\s)=\overline{\l_j(\e,-\s)}$, then we can conclude as before in \cite{WZ2} that the real parts are automatically $\cO(\s^2)$.\\

\begin{lemma}\label{lem:Re0}
	If the set of solutions of the reduced equation, $\cL_{LS}$, and the eigenvalues of $M$, $\cL_{cGL}$, can be made individually symmetric, then $\d_\s \cL_j(\e,0)\subset i\RR$. In addition, we also have that $\Re(\l_j(\e,\s)-\L_j(\e,\s) )=\cO(\e^2\s^2,\e\s^3,\s^4)$ if $\l_j$ and $\L_j$ are the matching eigenvalues in Theorem \ref{thm:specagree}.
\end{lemma}
\begin{proof}
	This follows from the symmetry $\l_j(\e,\s)=\overline{\l_j(\e,-\s)}$ and $\L_j(\e,\s)=\overline{\L_j(\e,-\s)}$ for each $j$.
\end{proof}
To get the existence of the $\l$'s, \cite{SZJV} and \cite{S} use a Weierstrass preparation argument to convert the reduced equation into an equation which is a degree 2, respectively degree 3, monic polynomial in $\l$ with coefficients which smoothly depend on $\e,\s,\kappa$. The $\l$ are then found by solving the new polynomial equation using the quadratic, respectively cubic, formula and then they are matched to the $\L$ by means of Taylor expansion. Here, to get existence of continuous solutions, we can use a variation on the Weierstrass preparation theorem, the Malgrange preparation theorem \cite{La}, to write 
\be\label{eq:WeierstrassPrep}
	q(\e,\s,\kappa,\vec{\b},\l)\det(m(\e,\s,\kappa,\vec{\b},\l))=\l^{N+2}+a_{N+1}(\e,\s,\kappa,\vec{\b})\l^{N+1}+...+a_0(\e,\s,\kappa,\vec{\b})
\ee with $|q(0,0,\kappa,\vec{\b},\l)|=1$. Moreover $q$ and each $a_i$ are smooth, and $q$ is analytic in $\l$. As the left hand side of \eqref{eq:WeierstrassPrep} is divisible by $\l^{N+1}$ by Theorem \ref{thm:coperiodicspectrum} at $\s=0$, so is the right hand side. By continuity of roots of polynomials, we have a continuous solution $\l_{stab}$ with $\l_{stab}(\e,\kappa,0,\vec{\b})<0$ and a continuous unordered list $\l_1,...,\l_{N+1}$, each of which vanishes at $\s=0$ and is of order $\e^2$ under the scaling $\s\sim\e$. In addition, $\l_{stab}$ is smooth by the implicit function theorem as it is a simple zero of $\det(m(\e,0,\kappa,\vec{\b},\l) )=0$. Alternatively, one can use Hurwitz's Theorem and the observation that $\det(m(\e,\s,\kappa,\vec{\b},\l))$ and $\det(-\l+M(\e,\s,\kappa,\vec{\b}))$ are two functions, holomorphic in $\l$, with the property that $\det(m(\e,\s,\kappa,\vec{\b},\l))\to\det(-\l+M(\e,\s,\kappa,\vec{\b}))$ uniformly on compact sets as $\e\to0$. With either technique, we have the existence of a full set of solutions to the reduced equation.\\

Overall, what we have is the expected result modulo the usual subtleties that arise when there are repeated eigenvalues and Jordan blocks. Stated more formally, we have the following
\begin{theorem}\label{thm:generalstability}
	Let $L$ and $\cN$ be as in Theorem \ref{thm:SO2reduced}. There exists a unique smooth curve $\l_{stab}(\e,\kappa,\s,\vec{\b})$ such that $\l_{stab}$ is a solution to \eqref{eq:SO2Reduced} with $\l_{stab}(\e,\kappa,0,\vec{\b})<0$. In addition, there exists an unordered list $\cL_{LS}$ of $N+1$ eigenvalues, $\l_1,..,\l_{N+1}$, which is continuous as a function of $\e,\s,\kappa,\vec{\b}$ and $\cL_{LS}(\e,0,\kappa,\vec{\b})=\{0,0,...,0\}$. Moreover, $\cL_{LS}=\cL_{cGL}+o(\e^2)$ on the region $|\s|\leq \Hs_0\e$ for any fixed $\Hs_0>0$ where $\cL_{cGL}$ is the unordered list $\L_1,...,\L_{N+1}$ of eigenvalues of $M$. If in addition, $\cL_{LS}$ and $\cL_{cGL}$ are both smooth with respect to $\s$, then one has the refined error estimate $\cL_{LS}(\e,\s)=\cL_{cGL}(\e,\s)+\cO(\e^3\s,\e^2\s^2,\e\s^3,\s^4)$. The further assumption that $\cL_{LS}$ and $\cL_{cGL}$ can be made individually symmetric implies the modified Ginzburg-Landau model correctly predicts stability.
\end{theorem}
\begin{proof}
	To explain how the matching of the stability predictions follows from smoothness, we split $\Hs$ into three regions for some $C$ to be determined later
	\begin{enumerate}
		\item $|\Hs|\leq \frac{1}{C}$,
		\item $\frac{1}{C}\leq |\Hs|\leq C$,
		\item $C\leq|\Hs|$.
	\end{enumerate}
	 In region one, we use the Taylor expansion $\Re\cL_{LS}=\e^2\cO(\Hs^2)$ with $\Re\cL_{LS}=\Re\cL_{cGL}+\cO(\e^2\s^2,\e\s^3,\s^4)$, hence $\Re\cL_{Ls}$ has the same collection of signs as $\Re\cL_{cGL}$. For region 2, each entry of $\Re\cL_{cGL}$ will have a uniformly definite sign and so by continuity $\Re\cL_{Ls}$ will have the same collection of signs. In region 3, the eigenvalues will be a perturbation of the block diagonal matrix
	 \be
	 	\e^2k_*^2\Hs^2\bp [[\frac{1}{2}\tl_{kk}(k_*,0)]] & 0_{2,N}\\
	 	0_{N,2} & D \ep,
	 \ee
	 where $D=\frac{1}{2}\Pi_0 S_{kk}(0,0)\Pi_0-\Pi_0 S_k(0,0)N_0S_k(0,0)\Pi_0$. This block diagonal matrix has stable spectrum by assumption and so by continuity, the lists $\cL_{LS}$ and $\cL_{cGL}$ will lie in $\{z:\Re z<-\L_0 \}$ for some $\L_0$. This is because in region 3, $|\s|\gg\e$ and $|\e,\s|\ll 1$ is equivalent to $|\Hs|\gg1$ but cubic terms are still negligible.
\end{proof}
We will leave determining what the modified Ginzburg-Landau model predicts for stability as a numerical problem to be solved on a case-by-case basis, as well as fringe cases where the spectral curves $\l$ are not smooth. The predictions of the modified Ginzburg-Landau model for $O(2)$ systems with a single conservation law were addressed in \cite{MC,S}.\\

The condition about being to ensure each item in the lists has the correct symmetry is in order to rule out situations where at least two $\l_j$ are such that $\Re(\d_\s \l_j(\e,0))\not=0$. This is a situation where the underlying traveling wave solutions is very unstable. Moreover, if at least two $\l_j$ are such that $\Re(\d_\s \l_j(\e,0))\not=0$, then the list $\cL_{LS}$ does not have this property because for that $\l_j$ one has
\be
	\l_j(\e,\s)-\overline{\l_j(\e,-\s)}=2\Re(\d_\s\l_j(\e,0))\s+h.o.t.\not=0.
\ee

Our final issue to address here is that of smoothness of the $\l$'s and $\L$'s. At this level of generality, smoothness of the $\l$'s or $\L$'s is tricky to get hold of. For the $\L$'s, the primary general condition that implies smoothness is 0 being a semisimple eigenvalue of $M$, which in the $SO(2)$-case is a closed condition and hence difficult to come by. We establish the requisite smoothness for $O(2)$ in Theorem \ref{thm:O2smooth} and for the singular $SO(2)$ case in Theorem \ref{thm:truncatedanalytic}. We expect that the $\l$'s are smooth when the $\L$'s are, however, we did not see a way to prove this.\\
	
While we can only ensure that the eigenvalues are continuous in full generality, they do satisfy a bound of the form $|\l(\e,\s,\kappa,\vec{\b})|\leq C\e^2|\Hs|^{\frac{1}{p}}$ for some integer $p$ by expanding the roots of the Weierstrass polynomial in a Puiseax series, c.f. \cite{K} and sources therein for a more detailed discussion.
\subsection{A look at the singular case and further questions}
In the preceding subsection, we showed that for compatible $SO(2)$-systems that Lyapunov-Schmidt and the modified Ginzburg-Landau model had the same predictions for the stability criteria. We also noted that much of the machinery in the preceding section applied equally well to all $SO(2)$-invariant systems. In this subsection, we will briefly mention that there are additional technical difficulties in extending the full machinery to the singular case. The main difficulty is two-fold. The first main issue is that the matching between the eigenvalues obtained by Lyapunov-Schmidt reduction and the predictions of the singular Ginzburg-Landau models is not as good as the matching in the $O(2)$ case. Specifically, we can match the eigenvalues obtained by Lyapunov-Schmidt reduction to the tilde model's dispersion relations. The reason it is the tilde model here that correctly matches is that the tilde model is the only one of the singular models we discussed that carries any information about $\cA$ or $\cB$, which we saw made an appearance in Theorem \ref{thm:exampleReducedPart2}.
\begin{theorem}
	Let $\tilde{\Lambda}_j(\e,\kappa,\sigma,\vec{\b})$, $j\in\{stab,1,2,..,N\}$ denote the set of eigenvalues of $M(\e,\s,\kappa,\vec{\b})+\cS(\e,\s,\kappa,\vec{\b})$. Then, there exists a full set of continuous solutions $\l_j$, $j\in\{stab,1,2,...,N \}$, to the reduced equation in Theorem \ref{thm:generalSO2reduced} such that for each $j$ we have 
	\begin{equation}
		\l_j(\e,\kappa,\s,\vec{\b})=\tilde{\L}_j(\e,\kappa,\sigma,\vec{\b})+e(\e,\kappa,\sigma,\vec{\b}),
	\end{equation}
	where $e(\e,\kappa,\s,\vec{\b})=o(\e^2)$ on the set of $\s$ with $|\s|\leq \Hs_0\e^2$ for any fixed $\Hs_0$. Moreover, for $\e>0$ sufficiently and generic model parameters the eigenvalues $\l_j$ are smooth in $\s$ and the error can be refined to
	\be
	\lambda_j(\e,\kappa,\s,\vec{\b})=\tilde{\Lambda}_j(\e,\kappa,\s,\vec{\b})+\cO(\e^4,\e^2\s,\e\s^2,\s^3).
	\ee
	Finally, if in addition $\l_j(\e,\kappa,0,\vec{\b})=0$, then the error estimate can be refined to $\cO(\e^2\s,\e\s^2,\s^3)$.
\end{theorem}
The proof of this a straightforward adaption of the proof of Theorem \ref{thm:specagree} and will be omitted. For the existence of continuous eigenvalues the first adaption is to set $\s=\check{\s}\e^2$, which is done to prevent the eigenvalues from blowing up under the scaling $\l=\e^2\hat{\l}$ in the Rouch\'e theorem argument. In the refined error estimate, the main adaption is that $\tilde{\cE}(M(\e,0,\kappa,\vec{\b})+\cS(\e,\s,\kappa,\vec{\b}))$ now dominates the Neumann error, which leads to the effective loss of one $\e$ in the error bounds stated in Theorem \ref{thm:specagree}.\\

The final point we wish to elaborate on is why the eigenvalues $\l_j$ are generically smooth in $\s$. We know that the stable eigenvalue is smooth by the implicit function theorem, so we are we left with understanding the regularity of the neutral eigenvalues. To show the neutral eigenvalues are smooth, we use the same strategy as in Theorem \ref{thm:truncatedanalytic}, \cite{JZ2} in the case where all equations are conservation laws and \cite{JNRZ} in the general case. We begin by noting that for each $\e>0$ and each $\s$ that the Bloch operator $B(\e,\s,\kappa,\vec{\b})$ is elliptic and acts on periodic functions and thus has discrete spectrum. In particular, by \cite{K} the total eigenprojection $P_{n}(\s)$ onto the 0 eigengroup is smooth in $\s$, which gives us an alternative reduced equation satisfied by the eigenvalues which vanish at $\s=0$. The key insight in \cite{JZ2}, \cite{JNRZ}, and our proof of Theorem \ref{thm:truncatedanalytic} is that by cleverly choosing the left eigenvectors to be the constant functions in the conserved directions and a particular set of right eigenvectors with $\frac{\d \tilde{u}_{\e,\kappa,\vec{\b}}}{\d\xi}$ as the base of the Jordan chain, then one can ensure the needed structure on the alternative reduced equation for the balancing transformation to work. It is crucial that one balances first and then afterwards uses the Neumann series argument in the matching. We note that the balancing transformation only changes the position of each coefficient in the reduced system, it does not change their values in a meaningful way. Hence the balancing transformation does not significantly alter the bounds on the error matrix $\tilde{\cE}$.\\

This issue regarding the matching is the less severe of the two issues as we have some wiggle room in the error estimates, as in the $O(2)$ case the error estimate in the matching is a higher order than the error obtained in the Taylor expansion of the dispersion relations.\\

The second main issue is that the radius of analyticity in $\Hs$ is expected to be of size $\e$. This leads to a problem in running the stability argument, as we need the radius of analyticity in $\Hs$ to be of size $1$ to run the stability argument of \cite{M2,SZJV,WZ2}. To briefly recap the argument, one splits $\Hs$ into the regions $|\Hs|\ll1$, $|\Hs|\sim1$ and $|\Hs|\gg1$. On the region $|\Hs|\ll 1$, we use analyticity and the other two are continuity type arguments, but here we can't run the analyticity argument for $|\Hs|\ll1$ because we don't a priori know what is happening in the region $\e\lesssim|\Hs|\ll1$. This means that whatever conditions obtained by the linearized dispersion relations are only a priori necessary, though one can ask the natural question
	\begin{question}
		Are the stability criteria obtained from Taylor expanding the eigenvalues of \eqref{eq:truncatedmatrix} sufficient as well?
	\end{question}

	With regards to the assumption that leads to diffusive stability in Theorem \ref{thm:generalstability} we see that one might expect it in the singular case when $\Pi_0S_k(0,0)\Pi_0$ is hyperbolic with all distinct eigenvalues which are well-separated for $\e\ll1$. For the $O(2)$ case, one might expect that the identity $\d_\s \cL_{LS}(\e,0)=0$ is inherited from the linearized model, which has $\d_\s \cL_{cGL}(\e,0)=0$. This leaves smoothness of the eigenvalues in the compatible $SO(2)$ case to be understood.
	\begin{question}
		In the compatible $SO(2)$ case, can one choose an ordering of the list $\cL_{cGL}$ so that each eigenvalue in the list is smooth and satisfies $\L_j(\e,\s)=\overline{\L_j(\e,-\s)}$?
	\end{question}
	To relate the Darcy model and the truncated model, one thing that could explain why the Darcy model captures some stability information is that the Darcy model arises as the limit as $\e\to0$. More, precisely
	\begin{question}
		For an initial data $(A_0,B_0)\in H^1\times H^1_{\RR}$ with $||A_0||_{H^1}+||B_0||_{H^1}\leq \rho$ for $\rho$ as in Theorem \ref{thm:truncatedwellposed} for each $\e>0$ let $(A^\e,B^\e)$ be the unique solution to the truncated system \eqref{eq:truncatedsystem} with initial data $(A_0,B_0)$ on the line. Then is there a suitable topology so that one has
		\be
			\lim_{\e\to0}(A^\e,B^\e)=\big(A^0,-\frac{h}{f}|A^0|^2\big),
		\ee
		where $A^0$ is the unique solution to
		\ba\label{eq:DarcyModel}
			A^0_T&=aA^0_{XX}+bA^0+\big(c-\frac{h}{f}\big)|A^0|^2A^0,\\
			A^0(X,0)&=A_0(X).
		\ea
	\end{question}
	One may also ask a weaker prepared version of this question where each solution of the Darcy model \eqref{eq:DarcyModel} is (in a suitable sense) the limit as $\e\to0$ of solutions of the truncated system \eqref{eq:truncatedsystem}. To explain why one might expect the unprepared version, we recall from the proof of Theorem \ref{thm:truncatedwellposed} that the fundamental solution $\tilde{H}_\e(X,T)$ of the $B$ equation of \eqref{eq:truncatedsystem} is given by
	\be
		\tilde{H}_\e(X,T)=\delta_{-\frac{fT}{\e}}\star H(X,T),
	\ee
	where $\star$ denotes convolution and $H(X,T)$ is a suitably scaled heat kernel. Roughly speaking, for a given compact set $K$ and some $L^2$ function $f$ supported on $K$, we find that for a given $T$
	\be
		||\tilde{H}_\e(\cdot,T)\star f||_{L^2(K)}\to0,
	\ee
	as $\e\to0$ because we can bound
	\be
		|H(\cdot,T)\star f(X)|\leq \frac{C}{\sqrt{T}}\int_{\RR}e^{-\frac{|X-Y|^2}{T}}|f(Y)|dY\sim \frac{e^{-\frac{\text{dist}(X,K)}{T}}}{\sqrt{T}}\int |f(Y)|dY,
	\ee
	for $X$ with $\text{dist}(X,K)\gg \text{diam}(K)$. Hence the majority of the $L^2$ norm of $H(\cdot,T)\star f$ is contained in some $T$-dependent compact neighborhood of $K$, which we denote $K_T$. As a consequence of this observation, for $\e>0$ majority of the $L^2$ norm of $\tilde{H}_\e(\cdot,T)\star f$ is contained in $\frac{fT}{\e}+K_T$, which becomes disjoint with $K$ as $\e\to0$. We also have the representation of $B$ from the proof of Theorem \ref{thm:truncatedwellposed}
	\be
		B(X,T)=\tilde{H}_\e(\cdot,T)\star B_0(X)-\frac{h}{f}|A(X,T)|^2+\frac{h}{f}\tilde{H}_\e(\cdot,T)\star |A_0|^2(X)+\cI_\e(T),
	\ee
	where $\cI_\e(T)$ is a known function of $A,\e,T$. One has that the terms depending on the initial data go to 0 as $\e\to0$ in $L^2([0,T_0];L^2_{loc}(\RR))$ or $L^2([0,T_0];H^1_{loc}(\RR))$ as appropriate. Heuristically, one expects that the integral term $\cI_\e(T)$ goes to 0 in a similar fashion because one expects that
	\be
		\int_0^T F\Big(X+\frac{f(T-S)}{\e},S \Big)dS\approx \int_{T-\cO(\e)}^T F\Big(X+\frac{f(T-S)}{\e},S \Big)dS,
	\ee
	provided that the forcing term $dAB$ in the complex Ginzburg-Landau doesn't cause $A$ to ``spread'' too quickly as to spoil our desired approximation.\\
	\begin{remark}\label{rem:spacetime}
		One can ask a related question on the torus, where techniques such as spacetime resonances seem more appropriate. We note that the truncated model \eqref{eq:truncatedsystem} is of the form studied by Schneider and de Rijk in \cite{dRS1,dRS2}, where they develop a spacetime resonance approach to study reaction-diffusion-advection systems and in particular study the long-time asymptotics of solutions of such systems. In these papers they work on the real line, so it would be interesting to see if the techniques can be extended to work on the torus.
	\end{remark}
	\begin{question}
		As we will see in the numerical appendix, for all of our test cases, the associated Darcy model being stable is a necessary but not sufficient condition for low frequency stability in the region determined by the Taylor expansion (i.e. $|\Hs|\lesssim\e$). Does this always happen?
	\end{question}
\appendix
\section{Derivation of the reduced equation for the example model}\label{app:example}
In this section our goal will be to compute the coefficients of $\s$, $\e\s$, and $\s^2$ in the linearized system's reduced equation for the example model reproduced below
\begin{equation*}
	\frac{\d U}{\d t}=\begin{bmatrix}
		2 & 1\\
		1 & 2
	\end{bmatrix}\d_x^2U+\begin{bmatrix}
		0 & 0 \\
		c_1+\mu & c_2
	\end{bmatrix}\d_xU+\begin{bmatrix}
		0 & 0\\
		0 & -1
	\end{bmatrix}U+\frac{1}{2}\bp \d_x (u^2)\\ v^2\ep+\frac{c}{6}\bp \d_x (u)^3\\v^3\ep,
\end{equation*}
where $U=(u,v)\in H^2_{per}(\RR;\RR^2)$. We adopt the scaling $\mu=\e^2$. For this model, the corresponding Bloch-type operator is given by
\ba\label{eq:exampleBloch}
B(\e,\kappa,\s,\l,\vec{\b})&=k^2\begin{bmatrix}
	2 & 1\\
	1 & 2
\end{bmatrix}(\d_\xi+i\s)^2+k\begin{bmatrix}
	0 & 0\\
	c_1+\e^2 & c_2
\end{bmatrix}(\d_\xi+i\s)+\begin{bmatrix}
	0 & 0\\
	0 & -1
\end{bmatrix}+dk\d_\xi\\
&\quad+i\s C(\e,\kappa)+\Big(\begin{bmatrix}
	k\d_\xi \tilde{u}^1_{\e,\kappa,\b} & 0\\
	0 & \tilde{u}^2_{\e,\kappa,\b}
\end{bmatrix}+\begin{bmatrix}
	\tilde{u}_{\e,\kappa,\b}^1 & 0\\
	0 & 0 \end{bmatrix}k(\d_\xi+i\s)\Big)\\
&\quad+CM(k(\d_\xi+i\s))\cC(\tilde{u}_{\e,\kappa,\vec{\b}},\tilde{u}_{\e,\kappa,\vec{\b}},\cdot )-\l,
\ea
where $\tilde{u}_{\e,\kappa,\b}=(\tilde{u}_{\e,\kappa,\b}^1,\tilde{u}_{\e,\kappa,\b}^2)$ is the solution constructed via the Lyapunov-Schmidt reduction in Theorem \ref{thm:LSReduction}, $C$ is a known constant, and $\cC(U,V,W)=(u_1v_1w_1,u_2v_2,w_2)$ is a trilinear form from $\CC^2\times\CC^2\times\CC^2\to\CC^2$. In what follows, the trilinear form plays essentially no role as for it to appear in the expansion of the mode 0 equation of $PBW=0$, one needs to take a $\s$-derivative at least 2 $\e$-derivatives on $B$ for it to have a chance at contributing. There are three ways for the trilinear term to contribute to mode 0, the the trilinear term in $B_{\s\e\e}$ applied to $\e\vec{b}$ or $\cV$ is an error term as $\e\vec{b}$ and $\cV$ are both $\cO(\e)$ under the scaling $\s\sim\e$. However, the trilinear term in $B_{\s\e\e}\Upsilon_a$ cannot contribute to mode 0 because it is of the form $\cC(\Upsilon_\a,\Upsilon_\a,\Upsilon_a)$ and that has Fourier support $\{\pm1,\pm3 \}$. Hence one needs at least three $\e$ derivatives on the trilinear term for $B_j\Upsilon_\a$ to contribute, but now it's an acceptable error term. Hence in what follows, we may safely ignore the trilinear term. We remark that the only reason we can ignore the trilinear terms contribution is because Theorem \ref{thm:coperiodicstability} handles the terms with only $\e$-derivatives, which is the only place the trilinear term does not give a negligible error term. In what follows we will pretend that the trilinear term is not there. \\

As a preliminary step, we record the necessary derivatives of $B(\e,\kappa,\s,\l,\vec{\b})$
\ba\label{eq:exampleBE}
&B_\e(0,\kappa,0,\l,\vec{\b})=2\kappa k_*\begin{bmatrix}
	2 & 1\\
	1 & 2
\end{bmatrix}\d_\xi^2+\kappa\begin{bmatrix}
	0 & 0\\
	c_1 & c_2\end{bmatrix}\d_\xi\\
&\quad+(k_*d_\e(0,\kappa)+\kappa d_* )\d_\xi+\frac{1}{2}\Big(\begin{bmatrix}
	k_*\d_\xi\Upsilon_\a^1 & 0\\
	0 & \Upsilon_\a^2\end{bmatrix}+\begin{bmatrix}
	\Upsilon_\a^1 & 0\\
	0 & 0\end{bmatrix}k_*\d_\xi\Big),
\ea
\ba\label{eq:exampleBS}
B_\s(0,\kappa,0,\l,\vec{\b})=2ik_*^2\begin{bmatrix}
	2 & 1\\
	1 & 2
\end{bmatrix}\d_\xi+ik_*\begin{bmatrix}
	0 & 0\\
	c_1 & c_2\end{bmatrix}+iC(0,\kappa),
\ea
\ba\label{eq:exampleBSS}
B_{\s\s}(0,\kappa,0,\l,\vec{\b})=-2k_*^2\begin{bmatrix}
	2 & 1\\
	1 & 2
\end{bmatrix},
\ea
\ba\label{eq:exampleBEE}
B_{\e\e}(0,\kappa,0,\l,\vec{\b})&=-2\kappa^2\begin{bmatrix} 2 & 1 \\ 1 & 2\end{bmatrix}\d_\xi^2+k_*\begin{bmatrix}0 & 0 \\ 2 & 0\end{bmatrix}\d_\xi+(k_*d_{\e\e}(0,\kappa,\vec{\b})+2\kappa d_\e(0,\kappa,\vec{\b}))\d_\xi \\
&\quad+2\begin{bmatrix}\kappa (\d_\xi\Upsilon_\a^1+\Upsilon_\a^1\d_\xi) & 0 \\ 0 & 0\end{bmatrix}+\begin{bmatrix} k_*(\d_\xi \frac{\d^2\tilde{u}_{0,\kappa,\vec{\b}}}{\d\e^2}^1+\frac{\d^2\tilde{u}_{0,\kappa,\vec{\b}}}{\d\e^2}^1\d_\xi) & 0 \\ 0 & \frac{\d^2\tilde{u}_{0,\kappa,\vec{\b}}}{\d\e^2}^2\end{bmatrix}\\
&\quad+trilinear,
\ea
\ba\label{eq:exampleBES}
B_{\e\s}(0,\kappa,0,\l,\vec{\b})=4i\kappa k_*\begin{bmatrix} 2 & 1\\
	1 & 2\end{bmatrix}\d_\xi+i\kappa\begin{bmatrix}
	0 & 0\\
	c_1 & c_2
\end{bmatrix}+iC_\e(0,\kappa)+ik_*\Pi_0\cQ(\Upsilon_\a,\cdot),
\ea
\ba\label{eq:exampleBESS}
B_{\e\s\s}(0,\kappa,0,\l,\vec{\b})=-4\kappa k_*\begin{bmatrix}
	2 & 1\\
	1 & 2
\end{bmatrix},
\ea
and
\ba\label{eq:exampleBEES}
B_{\e\e\s}(0,\kappa,0,\l,\vec{\b})=4i\kappa^2\begin{bmatrix}
	2 & 1\\
	1 & 2
\end{bmatrix}\d_\xi+i\begin{bmatrix}
	0 & 0 \\ 2 & 0
\end{bmatrix}+2i\kappa\Pi_0\cQ(\Upsilon_\a,\cdot)+ik_*\Pi_0\cQ\Big(\frac{\d^2\tilde{u}_{0,\kappa,\vec{\b}} }{\d\e^2},\cdot\Big).
\ea

Before we continue with the calculation, it is important to observe the presence of the higher derivatives $B_{\e\e\s}$, $B_{\e\s\s}$. These are purely for the mode 0 equation and arise entirely because $\vec{b}$ was already scaled to be $\cO(\e)$. Hence when we divide by $\e$ to make $\e B_\s\vec{b}=\cO(\s)$ for example, we also make $B_{\e\e\s}\Upsilon_a=\cO(\e\s)$, which is not an acceptable error term. We do not need to compute any term with only $\e$ derivatives as those have already been handled by Theorem \ref{thm:coperiodicstability}.\\

As a first step, we reduce $PB_j(\Upsilon_a+\e\vec{b})$ in the following sequence of lemmas where $j$ denotes one of $\s$, $\e\s$, $\s\s$, $\e\s\s$, $\e\e\s$.
\begin{lemma}\label{lem:exampleOS}
	In reduced form $PB_\s(\Upsilon_a+\e\vec{b})$ vanishes in mode 1 and is given by
	\begin{equation}
		i\e C(0,\kappa)\vec{b},
	\end{equation}
	in mode 0.
\end{lemma}
\begin{proof}
	We expand $PB_\s(\Upsilon_a+\e\vec{b})$ as
	\ba
	&PB_\s(0,\kappa,0,\l,\vec{\b})(\Upsilon_a+\e\vec{b})\\
	&\quad=\frac{1}{2}\Pi_1\Big(2ik_*^2\begin{bmatrix}
		2 & 1\\
		1 & 2
	\end{bmatrix}\d_\xi+ik_*\begin{bmatrix}
		0 & 0\\
		c_1 & c_2\end{bmatrix}+iC(0,\kappa)\Big)ae^{i\xi}r\\
	&\quad+\frac{1}{2}\overline{\Pi_1}\Big(2ik_*^2\begin{bmatrix}
		2 & 1\\
		1 & 2
	\end{bmatrix}\d_\xi+ik_*\begin{bmatrix}
		0 & 0\\
		c_1 & c_2\end{bmatrix}+iC(0,\kappa)\Big)\bar{a}e^{-i\xi}\bar{r}\\
	&\quad+\e\Pi_0\Big(2ik_*^2\begin{bmatrix}
		2 & 1\\
		1 & 2
	\end{bmatrix}\d_\xi+ik_*\begin{bmatrix}
		0 & 0\\
		c_1 & c_2\end{bmatrix}+iC(0,\kappa)\Big)\vec{b}.
	\ea
	The claim for mode 0 now follows because $\d_\xi\vec{b}=0$ and $\Pi_0 S_k(0,0)=0$.\\
	
	For mode 1, we note that
	\begin{equation}
		\Pi_1\Big(-2k_*^2\begin{bmatrix}
			2 & 1\\
			1 & 2
		\end{bmatrix}+ik_*\begin{bmatrix}
			0 & 0\\
			c_1 & c_2\end{bmatrix}\Big)e^{i\xi}r=\Pi_1k_*S_k(k_*,0)e^{i\xi }r.
	\end{equation}
	Hence, in mode 1 we have
	\begin{equation}
		\frac{1}{2}(k_*\tl_k(k_*,0)+iC(0,\kappa))ar.
	\end{equation}
	For frequency -1, we have
	\begin{equation}
		\overline{\Pi_1}\Big(2k_*^2\begin{bmatrix}
			2 & 1\\
			1 & 2
		\end{bmatrix}+ik_*\begin{bmatrix}
			0 & 0\\
			c_1 & c_2\end{bmatrix}\Big)\bar{r}=k_*\overline{\Pi_1}S_k(-k_*,0)\bar{r}=-k_*\overline{\tl_k(k_*,0)}\bar{r}=k_*\tl_k(k_*,0)\bar{r}.
	\end{equation}
	As $S(k,\mu)=\overline{S(-k,\mu)}$ implies $S_k(k,\mu)=-\overline{S_k(-k,\mu)}$ by the chain rule. To finish the proof of the lemma, note that $C(0,\kappa)$ was designed specifically so that $iC(0,\kappa)+k_*\tl_k(k_*,0)=0$.
\end{proof}
\begin{lemma}\label{lem:exampleOSS}
	In reduced form $PB_{\s\s}(\Upsilon_a+\e\vec{b})$ is given by
	\begin{equation}
		-k_*^2\Pi_1\begin{bmatrix}
			2 & 1\\
			1 & 2
		\end{bmatrix}ar,
	\end{equation}
	in mode 1 and in mode 0 is given by
	\begin{equation}
		-\e 2k_*^2\begin{bmatrix}
			2 & 1\\
			0 & 0
		\end{bmatrix}\vec{b}.
	\end{equation}
\end{lemma}
\begin{proof}
	Immediate from the formula for $B_{\s\s}$ given in \eqref{eq:exampleBSS}.
\end{proof}
\begin{lemma}\label{lem:exampleOES}
	In reduced form $PB_{\e\s}(\Upsilon_a+\e\vec{b})$ is given by
	\begin{equation}
		-i[[i\kappa k_*\ell\Big(\begin{bmatrix} 2 & 1\\
			1 & 2\end{bmatrix}\Big)r]]\bp a_1\\a_2\ep,
	\end{equation}
	in mode 1 and is in mode 0 is given by
	\begin{equation}
		i\e C_\e(0,\kappa)\vec{b}+\frac{1}{4}ik_*\a (a+\bar{a})\bp |r_1|^2\\0\ep.
	\end{equation}
\end{lemma}
\begin{proof}
	We expand $PB_{\e\s}(\Upsilon_a+\e\vec{b})$ using \eqref{eq:exampleBES} as
	\ba\label{eq:exampleOES1}
	&PB_{\e\s}(0,\kappa,0,\l,\vec{\b})(\Upsilon_a+\e\vec{b})=P\Big(4i\kappa k_*\begin{bmatrix} 2 & 1\\
		1 & 2\end{bmatrix}\d_\xi+i\kappa\begin{bmatrix}
		0 & 0\\
		c_1 & c_2
	\end{bmatrix}+iC_\e(0,\kappa)\Big)(\Upsilon_a+\e\vec{b})\\ 
	&\quad
	+P\Big(ik_*\Pi_0\cQ(\Upsilon_\a,\Upsilon_a+\e\vec{b})\Big).
	\ea
	We tackle mode 0 terms in \eqref{eq:exampleOES1} first, where we have
	\begin{equation}
		\e\Pi_0\Big(4i\kappa k_*\begin{bmatrix} 2 & 1\\
			1 & 2\end{bmatrix}\d_\xi+i\kappa\begin{bmatrix}
			0 & 0\\
			c_1 & c_2
		\end{bmatrix}+iC_\e(0,\kappa)\Big)\vec{b}+\frac{1}{4}ik_*\Pi_0\big(\cQ(\a r,\overline{ar})+\cQ(\overline{\a r},ar)\big).
	\end{equation}
	Simplifying somewhat, we get
	\begin{equation}
		i\e C_\e(0,\kappa)\vec{b}+\frac{1}{4}ik_*\a (a+\bar{a})\bp |r_1|^2\\0\ep.
	\end{equation}
	
	For mode 1, we can ignore the terms involving $\e\vec{b}$ because those terms are ultimately of order $\cO(\e^2\s)$ which is an acceptable error. With this in mind, we have
	\ba
	&\Pi_1\Big(-4\kappa k_*\begin{bmatrix} 2 & 1\\
		1 & 2\end{bmatrix}+i\kappa\begin{bmatrix}
		0 & 0\\
		c_1 & c_2
	\end{bmatrix}+iC_\e(0,\kappa)\Big)\frac{1}{2}ae^{i\xi}r\\
	&\quad+\overline{\Pi_1}\Big(4\kappa k_*\begin{bmatrix} 2 & 1\\
		1 & 2\end{bmatrix}+i\kappa\begin{bmatrix}
		0 & 0\\
		c_1 & c_2
	\end{bmatrix}+iC_\e(0,\kappa)\Big)\frac{1}{2}\bar{a}e^{-i\xi}\bar{r}.
	\ea
	By a similar argument to the one in Lemma \ref{lem:exampleOS}, this reduces to
	\begin{equation}
		-\kappa k_*\Pi_1\Big(\begin{bmatrix} 2 & 1\\
			1 & 2\end{bmatrix}\Big)ae^{i\xi}r+c.c.+(\kappa\tl_k(k_*,0)+iC_\e(0,\kappa))\Upsilon_a.
	\end{equation}
	Hence, the only term left is
	\begin{equation}
		-\kappa k_*\Pi_1\Big(\begin{bmatrix} 2 & 1\\
			1 & 2\end{bmatrix}\Big)ae^{i\xi}r-c.c.
	\end{equation}
	To complete the proof, we use the simple identity
	\begin{equation}
		z-\bar{z}=-i(iz+\overline{iz}).
	\end{equation}
\end{proof}
Before we continue, let us pause and observe that something odd has happened in our expansion of $PBW=0$. In particular, observe that the coefficients of $\vec{b}$ in Lemma \eqref{lem:exampleOS} and Lemma \eqref{lem:exampleOSS} are nearly identical. Something similar will happen with the $a$ terms in Lemma \eqref{lem:exampleOES} and the upcoming Proposition \ref{prop:exampleEES}. These extra terms in Lemma \eqref{lem:exampleOSS} come from the ghosts of the departed amplitudes $\cA$ and $\cB$ in \eqref{eq:exampleeps4mode0}, as we recall that the way we ``solved'' that equation was to add in \eqref{eq:exampleeps3mode0} as a singular term and declared anything involving $\cA$ and $\cB$ to be part of the mode 0 equation at $\cO(\e^5)$.\\

For the remaining terms, it will suffice for our purposes to find the Fourier mean value of $PB_{\e\e\s}\Upsilon_a$ and $PB_{\e\s\s}\Upsilon_a$.
\begin{proposition}\label{prop:exampleEES}
	$PB_{\e\e\s}\Upsilon_a$ has Fourier mean value
	\begin{equation}
		\frac{1}{2}i\kappa\a(a+\bar{a})\bp |r_1|^2\\ 0\ep+ik_*\kappa\a\Re(-iav)+\frac{1}{2}ik_*\a_1(a+\bar{a})\bp |r_1|^2\\0\ep,
	\end{equation}
	where $v$ is as in \eqref{eq:exampleeps4AB2}.\\
	
	$PB_{\e\s\s}\Upsilon_a$ has Fourier mean 0.
\end{proposition}
\begin{proof}
	We begin with $PB_{\e\s\s}\Upsilon_a$, which we expand using \eqref{eq:exampleBESS}
	\begin{equation}
		PB_{\e\s\s}(0,\kappa,0,\l,\vec{\b})\Upsilon_a=-4\kappa k_*P\begin{bmatrix}
			2 & 1\\
			1 & 2
		\end{bmatrix}\Upsilon_a.
	\end{equation}
	But this is Fourier-supported in $\{\pm1\}$ and hence has Fourier mean 0.\\
	
	For $PB_{\e\e\s}\Upsilon_a$, we expand using \eqref{eq:exampleBEES} to get
	\ba
		PB_{\e\e\s}(0,\kappa,0,\l,\vec{\b})\Upsilon_a&=P\Big(4i\kappa^2\begin{bmatrix}
			2 & 1\\
			1 & 2
		\end{bmatrix}\d_\xi\Big)\Upsilon_a+2iP\kappa\Pi_0\cQ(\Upsilon_\a,\Upsilon_a)\\
	&\quad+ik_*P\Pi_0\cQ\Big(\frac{\d^2\tilde{u}_{0,\kappa,\vec{\b}} }{\d\e^2},\Upsilon_a\Big).
	\ea
	As before, $\d_\xi\Upsilon_a$ has Fourier support $\{\pm1\}$ and hence can be safely ignored. For the $\cQ$ terms, it is easy to check that only the following terms contribute to the Fourier mean
	\begin{equation}\label{eq:exampleOEESmean}
		\frac{1}{2}i\kappa\Pi_0\big(\cQ(\a r ,\bar{a}\bar{r})+\cQ(\bar{\a}\bar{r},ar)\big)+ik_*\Pi_0\Big(\cQ(\widehat{\frac{\d^2\tilde{u}_{0,\kappa,\vec{\b}} }{\d\e^2}}(1),\bar{a}\bar{r}  )+\cQ(\widehat{\frac{\d^2\tilde{u}_{0,\kappa,\vec{\b}} }{\d\e^2}}(-1),ar)\Big).
	\end{equation}
	The first and a part of the third term in \eqref{eq:exampleOEESmean} is the promised version of the mean in Lemma \ref{lem:exampleOES}, so we can write the first term as
	\begin{equation}
		\frac{1}{2}i\kappa\a(a+\bar{a})\bp |r_1|^2\\ 0\ep.
	\end{equation}
	For the second term in \eqref{eq:exampleOEESmean}, we use the expansion for $(I-P)\tilde{u}_{\e,\kappa,\vec{\b}}$ in Proposition \ref{prop:WExpansion}. Doing so, we get
	\ba\label{eq:example:OEESmean2}
	& ik_*\Pi_0\Big(\cQ(\widehat{\frac{\d^2\tilde{u}_{0,\kappa,\vec{\b}} }{\d\e^2}}(1),\bar{a}\bar{r}  )+\cQ(\widehat{\frac{\d^2\tilde{u}_{0,\kappa,\vec{\b}} }{\d\e^2}}(-1),ar)\Big)\\
	&\quad =ik_*\Pi_0\Big(\a\bar{a}\cQ(-\kappa N_1(I-\Pi_1)S_k(k_*,0)\Pi_1r,\bar{r})+\a a\cQ(\overline{\kappa N_1(I-\Pi_1)S_k(k_*,0)\Pi_1r},r)\Big).
	\ea
	Recalling the corresponding term in \eqref{eq:exampleeps4AB2}, we see that
	\begin{equation}
		\cQ(\overline{\kappa N_1(I-\Pi_1)S_k(k_*,0)\Pi_1r},r)=i\kappa v,
	\end{equation}
	so that in \eqref{eq:example:OEESmean2}, we have
	\ba
		&\frac{1}{2}ik_*\Pi_0\Big(\cQ(\widehat{(I-\Pi_1)\frac{\d^2\tilde{u}_{0,\kappa,\vec{\b}} }{\d\e^2}}(1),\bar{a}\bar{r}  )+\cQ(\widehat{\overline{(I-\Pi_1)}\frac{\d^2\tilde{u}_{0,\kappa,\vec{\b}} }{\d\e^2}}(-1),ar)\Big)\\
		&\quad=\frac{1}{2}ik_*\kappa(\a\bar{a}i\bar{v}-i\a av)=ik_*\kappa\a\Re(iav).
	\ea
	We also have a contribution from $\Pi_1\frac{\d^2\tilde{u}_{0,\kappa,\vec{\b}} }{\d\e^2}(1)$, which we write as $\a_1r$ for $\a_1:=\frac{\d\a}{\d\e}|_{\e=0}$. The contribution from $\Pi_1\frac{\d^2\tilde{u}_{0,\kappa,\vec{\b}} }{\d\e^2}(1)$ is given by
	\be
	\frac{1}{2}ik_*\Pi_0\Big(\cQ(\Pi_1\widehat{\frac{\d^2\tilde{u}_{0,\kappa,\vec{\b}} }{\d\e^2}}(1),\bar{a}\bar{r})+\cQ(\widehat{\overline{\Pi_1}\frac{\d^2\tilde{u}_{0,\kappa,\vec{\b}} }{\d\e^2}}(-1),ar)\Big)=\frac{1}{2}ik_*\a_1(a+\bar{a})\Pi_0\cQ(r,\bar{r}).
	\ee
\end{proof}
\begin{remark}
	One notes that there is a factor of $\kappa$ multiplying the singular coefficient $\frac{1}{4}\a|r_1|^2$ instead of the previously typical factor of $k_*$. This quirk is present in other terms arising from linearizing about $\cA$, as for instance the symbol of $B_{\e\s\s}$ contributes a copy of $2\kappa k_* \ell S_{kk}(k_*,0)r$ into the reduced form of mode 1 at order $\e\s^2$.
\end{remark}
So far, we've shown the following result.
\begin{theorem}\label{thm:exampleReducedPart1}
	Write $\vec{b}=be_1$ for some $b\in\RR$ with $e_1=(1,0)$ and write $a=a_1+ia_2$ with $a_i\in\RR$. Then $PB(\e,\kappa,\l,\s,\vec{\b})(\Upsilon_a+\e\vec{b})$ in reduced form is given by
	\ba\label{eq:examplePBker1}
	&-\l\bp a_1\\ a_2\ep+\e^2\bp 2\a^2\Re\g+\cO(\e) & 0\\
	2\a^2\Im\g+\cO(\e) & 0\ep\bp a_1\\a_2\ep+\e^2\bp \a\Re V_1\cdot e_1+\cO(\e)\\
	\a\Im V_1\cdot e_1+\cO(\e)\ep b\\
	&\quad+\frac{1}{2}\s^2 k_*^2[[\ell S_{kk}(k_*,0)r]]\bp a_1\\a_2\ep-i\e\s[[i\kappa k_*\ell S_{kk}(k_*,0)r]]\bp a_1\\a_2\ep+h.o.t.\ ,
	\ea
	in mode 1 and in mode 0
	\ba\label{eq:examplePBker0}
	&-\l\vec{b}+i\s C(0,\kappa)\vec{b}+\frac{1}{4}i\s k_*\a (a+\bar{a})\bp |r_1|^2\\0\ep-2\s^2k_*^2\vec{b}+\\
	&\quad +i\e\s C_\e(0,\kappa)\vec{b}+i\e\s\Big(\frac{1}{4}i(\kappa\a+k_*\a_1)(a+\bar{a})\bp |r_1|^2\\ 0\ep+ik_*\kappa\a\Re(iav)   \Big)+h.o.t.\ ,
	\ea
	where $h.o.t$ refers to terms that are of order $\e^2|\l|$, $\e^3$, $\e\s^2$, $\e^2\s$ or $\s^3$ and for a complex number $z=x+iy$ we let $[[z]]$ be the matrix
	\begin{equation}
		[[z]]=\bp x & -y\\
		y & x\ep.
	\end{equation}
\end{theorem}
At this point, the singular behavior of \eqref{eq:examplesingularBeqn} has manifested as the coefficient of $\s$ in \eqref{eq:examplePBker0} because an important step in showing stability is to introduce the scalings $\l=\e^2\hat{\l}$, $\s=\e\Hs$ and divide the reduced equation for stability by $\e^2$ in order to hopefully get an expansion for $\l$ which is smooth at $\e=0$. That said, we have yet to fully recover the linearized singular system consisting of the linearizations of \eqref{eq:exampleCGL} and \eqref{eq:examplesingularBeqn}. Hence the singular system has not been fully justified yet.\\

To continue, we look at the term $PB\cV$ which write in reduced form in the following sequence of lemmas.
\begin{lemma}\label{lem:examplePB0cV}
	For any $j\in\{\e,\s,\e\s,\s\s,\e\s\s,\e\e\s,\s\s\s \}$, one has
	\ba
	PB(0,\kappa,\l,0,\vec{\b})\cV_j(a,\vec{b};0,\kappa,0,0,\vec{\b}))=0.
	\ea
\end{lemma}
\begin{proof}
	This immediately follows from the following key commutation relation and that $\cV=(I-P)\cV$
	\begin{equation}
		PB(0,\kappa,\l,0\vec{\b})=-\l P.
	\end{equation}
	
\end{proof}
Equivalently, the highest derivative of $\cV$ does not contribute to the coefficient of $\e^j\s^k$ in the reduced equation.
\begin{lemma}\label{lem:examplecVOES}
	In reduced form, $PB_\e\cV_\s$ is given by
	\ba
	\frac{1}{2}ik_*\kappa [[i\ell(S_k(k_*,0)(I-\Pi_1)N_1(I-\Pi_1)S_k(k_*,0))r]]\bp a_1\\a_2\ep+\cO(|\l|),
	\ea
	in mode 1 and has Fourier mean $\cO(|\l|)$.\\
	
	Similarly, $PB_\s\cV_\e$ is given by
	\ba
	\frac{1}{2}ik_*\kappa [[i\ell(S_k(k_*,0)(I-\Pi_1)N_1(I-\Pi_1)S_k(k_*,0))r]]\bp a_1\\a_2\ep+\cO(|\l|),
	\ea
	in mode 1 and has Fourier mean $\cO(|\l|)$.
\end{lemma}
\begin{proof}
	We recall the necessary identities for $\cV_\e$ and $\cV_\s$ from Proposition \ref{prop:CVExpansion}
	\begin{enumerate}
		\item $\cV_\e(a,\vec{b};0,\kappa,\l,0,\vec{\b})=-T_\l(I-P)B_\e(0,\kappa,\l,0,\vec{\b})\Upsilon_a$,
		\item $\cV_\s(a,\vec{b};0,\kappa,\l,0,\vec{\b})=-T_\l(I-P)B_\s(0,\kappa,\l,0,\vec{\b})\Upsilon_a$.
	\end{enumerate}
	We start with $PB_\s(0,\kappa,\l,0,\vec{\b})\cV_\e(a,\vec{b};0,\kappa,\l,0,\vec{b})$, which we expand using \eqref{eq:exampleBS} as
	\ba
	&PB_\s(0,\kappa,\l,0,\vec{\b})\cV_\e(a,\vec{b};0,\kappa,\l,0,\vec{b})\\&
	\quad=-P\Big(2ik_*^2\begin{bmatrix}
		2 & 1\\
		1 & 2
	\end{bmatrix}\d_\xi+ik_*\begin{bmatrix}
		0 & 0\\
		c_1 & c_2\end{bmatrix}\Big)(I-P)T_0(I-P)B_\e(0,\kappa,\l,0,\vec{\b})\Upsilon_a+\cO(|\l|).
	\ea
	From \eqref{eq:exampleBE}, we can expand $B_\e\Upsilon_a$ as
	\ba
	&B_\e(0,\kappa,\l,0,\vec{\b})\Upsilon_a=2\kappa k_*\begin{bmatrix}
		2 & 1\\
		1 & 2
	\end{bmatrix}\d_\xi^2\Upsilon_a+\kappa\begin{bmatrix}
		0 & 0\\
		c_1 & c_2\end{bmatrix}\d_\xi\Upsilon_a\\
	&\quad+(k_*d_\e(0,\kappa)+\kappa d_* )\d_\xi\Upsilon_a+\frac{1}{2}\Big(\begin{bmatrix}
		k_*\d_\xi\Upsilon_\a^1 & 0\\
		0 & \Upsilon_\a^2\end{bmatrix}\Upsilon_a+\begin{bmatrix}
		\Upsilon_\a^1 & 0\\
		0 & 0\end{bmatrix}k_*\d_\xi\Upsilon_a\Big).
	\ea
	Where $\Upsilon_\a=(\Upsilon_\a^1,\Upsilon_\a^2)$ are the components of the vector $\Upsilon_\a$. There are thus two terms in $PB_\s\cV_\e$, which we call the ``linear-linear'' term and the ``linear-nonlinear'' term. Starting with the linear-linear term, which we define to be the following expression
	\ba
	LL:=-P\Big(2ik_*^2\begin{bmatrix}
		2 & 1\\
		1 & 2
	\end{bmatrix}\d_\xi+ik_*\begin{bmatrix}
		0 & 0\\
		c_1 & c_2\end{bmatrix}\Big)(I-P)T_0(I-P)\Big( 2\kappa k_*\begin{bmatrix}
		2 & 1\\
		1 & 2
	\end{bmatrix}\d_\xi^2+\kappa\begin{bmatrix}
		0 & 0\\
		c_1 & c_2\end{bmatrix}\d_\xi\Big)\Upsilon_a,
	\ea
	where we've used the observation that $P\d_\xi=\d_\xi P$. We compute the Fourier modes of $LL$ as
	\ba
	\hat{LL}(1)=-\frac{1}{2}k_*\kappa \Pi_1(S_k(k_*,0)(I-\Pi_1)N_1(I-\Pi_1)S_k(k_*,0))ar,
	\ea
	and a similar computation shows that $\hat{LL}(-1)=-\overline{\hat{LL}(1)}$. We turn to the linear-nonlinear term of $PB_\s\cV_\e$, which we define to be
	\ba
	LN:=-\frac{1}{2}P\Big(2ik_*^2\begin{bmatrix}
		2 & 1\\
		1 & 2
	\end{bmatrix}\d_\xi+ik_*\begin{bmatrix}
		0 & 0\\
		c_1 & c_2\end{bmatrix}\Big)(I-P)T_0(I-P)\Big(\begin{bmatrix}
		k_*\d_\xi\Upsilon_\a^1 & 0\\
		0 & \Upsilon_\a^2\end{bmatrix}\Upsilon_a+\begin{bmatrix}
		\Upsilon_\a^1 & 0\\
		0 & 0\end{bmatrix}k_*\d_\xi\Upsilon_a\Big).
	\ea
	As $P$, $B_\s$, and $T_0$ are Fourier multiplier operators, it is readily apparent that $LN$ is Fourier supported in $\{0,\pm 2\}$. We have that $\hat{LN}(0)=0$ because
	\begin{equation}
		\Pi_0\widehat{B_\s(0,\kappa,\l,0,\vec{\b})}=ik_*\Pi_0\begin{bmatrix} 0 & 0\\
			c_1 & c_2\end{bmatrix}=0.
	\end{equation}
	We now turn to reducing $PB_\e\cV_\s$, which has two terms ``nonlinear-linear'' and ``linear-linear'' by the same sort of argument we used for $PB_\s\cV_\e$. Starting with the linear-linear term, we get
	\ba
	LL':=-P\Big(2\kappa k_*\begin{bmatrix}
		2 & 1\\
		1 & 2
	\end{bmatrix}\d_\xi^2+\kappa\begin{bmatrix}
		0 & 0\\
		c_1 & c_2\end{bmatrix}\d_\xi\Big)(I-P)T_0(I-P)\Big(2ik_*^2\begin{bmatrix}
		2 & 1\\
		1 & 2
	\end{bmatrix}\d_\xi+ik_*\begin{bmatrix}
		0 & 0\\
		c_1 & c_2\end{bmatrix}\Big)\Upsilon_a.
	\ea
	Using the same procedure used to reduce the linear-linear term of $PB_\s\cV_\e$, we see that $LL'$ has the same reduced form as the previous $LL$. To finish, we look at the nonlinear-linear term which is defined to be
	\ba
	NL:=-\frac{1}{2}P\Big(\begin{bmatrix}
		k_*\d_\xi\Upsilon_\a^1 & 0\\
		0 & \Upsilon_\a^2\end{bmatrix}+\begin{bmatrix}
		\Upsilon_\a^1 & 0\\
		0 & 0\end{bmatrix}k_*\d_\xi\Big)(I-P)T_0(I-P)\Big(2ik_*^2\begin{bmatrix}
		2 & 1\\
		1 & 2
	\end{bmatrix}\d_\xi+ik_*\begin{bmatrix}
		0 & 0\\
		c_1 & c_2\end{bmatrix}\Big)\Upsilon_a.
	\ea
	As with $LN$, $NL$ is Fourier supported in $\{0,\pm2 \}$. We claim that $\hat{NL}(0)=0$. We start by observing that $\hat{NL}(0)$ is given by
	\ba
	\hat{NL}(0)&=-\frac{1}{8}\Pi_0
	\Big[\Big(\begin{bmatrix}
		ik_*\a r_1 & 0\\
		0 & 0\end{bmatrix}+\begin{bmatrix}
		-ik_*\a r_1 & 0\\
		0 & 0\end{bmatrix}\Big)\overline{N_1}\Big(2k_*^2\begin{bmatrix}
		2 & 1\\
		1 & 2
	\end{bmatrix}+ik_*\begin{bmatrix}
		0 & 0\\
		c_1 & c_2\end{bmatrix}\Big)\bar{a}\bar{r}\\
	&\quad+
	\begin{bmatrix}
		-ik_*\a r_1 & 0\\
		0 & 0\end{bmatrix}+\begin{bmatrix}
		\a r_1 & 0\\
		0 & 0\end{bmatrix}ik_*\Big)N_1\Big(-2k_*^2\begin{bmatrix}
		2 & 1\\
		1 & 2
	\end{bmatrix}+ik_*\begin{bmatrix}
		0 & 0\\
		c_1 & c_2\end{bmatrix}\Big)ar\Big].
	\ea
	
\end{proof}
\begin{lemma}\label{lem:examplecVOSS}
	$PB_\s\cV_\s$ in reduced form has in mode 1
	\ba
	-\frac{1}{2}k_*^2[[\ell S_k(k_*,0)(I-\Pi_1)N_1(I-\Pi_1)S_k(k_*,0)\Pi_1r]]\bp a_1\\a_2\ep+\cO(|\l|),
	\ea
	and zero Fourier mean.
\end{lemma}
\begin{proof}
	By the expansion for $B_\s$ \eqref{eq:exampleBS} and the identity for $\cV_\s$ in Proposition \ref{prop:CVExpansion}, we see that
	\ba
	PB_\s(0,\kappa,\l,0,\vec{\b})\cV_\s(a,\vec{b};0,\kappa,\l,0,\vec{b})=-P\Big(2ik_*^2\begin{bmatrix}
		2 & 1\\
		1 & 2
	\end{bmatrix}\d_\xi+ik_*\begin{bmatrix}
		0 & 0\\
		c_1 & c_2\end{bmatrix}\Big)T_\l\Big(2ik_*^2\begin{bmatrix}
		2 & 1\\
		1 & 2
	\end{bmatrix}\d_\xi+ik_*\begin{bmatrix}
		0 & 0\\
		c_1 & c_2\end{bmatrix}\Big)\Upsilon_a.
	\ea
	The claim about the Fourier mean of $PB_\s\cV_\s$ now follows because $PB_\s\cV_\s$ is of the form $M\Upsilon_a$ where $M$ is a Fourier multiplier operator and thus preserves Fourier support. In terms of the symbol $PB_\s\cV_\s$ can be rewritten as
	\ba
	-\frac{1}{2}k_*^2\Pi_1 S_k(k_*,0)(I-\Pi_1)N_1(I-\Pi_1)S_k(k_*,0)\Pi_1 ae^{i\xi}r+c.c.+\cO(|\l|).
	\ea
\end{proof}
Combining Lemmas \ref{lem:examplecVOES}, \ref{lem:examplecVOSS} and Theorem \ref{thm:exampleReducedPart1} with the key spectral identity in Proposition \ref{prop:spectralid} completes the reconstruction of the linearized (cGL) equation in our singular model. To finish the reconstruction of the singular system, in particular \eqref{eq:examplesingularBeqn}, we need to find the Fourier mean values of the coefficients of $\e\s^2$, $\e^2\s$ and $\s^3$ in $PB\cV$. We start with $\s^3$, which has two terms $PB_\s\cV_{\s\s}$ and $PB_{\s\s}\cV_\s$.
\begin{lemma}\label{lem:examplecVOSSS}
	In reduced form, the Fourier means of $PB_{\s\s}\cV_\s$ and $PB_\s\cV_{\s\s}$ both vanish identically.
\end{lemma}
\begin{proof}
	$B_{\s}(0,\kappa,\l,0,\vec{\b})$ and $B_{\s\s}(0,\kappa,\l,0,\vec{\b})$ are both Fourier multiplier operators, and by the implicit function theorem
	\ba
	\cV_\s(a,\vec{b};0,\kappa,\l,0,\vec{\b})&=-T_\l B_\s(0,\kappa,\l,0,\vec{\b})\Upsilon_a,\\
	\cV_{\s\s}(a,\vec{b};0,\kappa,\l,0,\vec{\b})&=-2T_\l B_\s(0,\kappa,\l,0,\vec{\b})\cV_\s(a,\vec{b};0,\kappa,\l,0,\vec{\b})-T_\l B_{\s\s}(0,\kappa,\l,0,\vec{\b})\Upsilon_a.
	\ea
	Hence $\cV_\s$ and $\cV_{\s\s}$ are Fourier supported in $\{\pm1\}$ and are thus automatically Fourier mean free.
\end{proof}
Next, we look at $\e\s^2$, where there are 4 terms $PB_\e\cV_{\s\s}$, $PB_\s\cV_{\e\s}$, $PB_{\e\s}\cV_\s$ and $PB_{\s\s}\cV_\e$.
\begin{proposition}\label{prop:exampleB1cV2ESS}
	The Fourier mean of $PB_\e\cV_{\s\s}$ and $PB_\s\cV_{\e\s}$ both vanish.
\end{proposition}
\begin{proof}
	Starting with $PB_\e\cV_{\s\s}$, we recall the expansion of $\cV_{\s\s}$ from the proof of Lemma \ref{lem:examplecVOSSS} as
	\be
	\cV_{\s\s}(a,\vec{b};0,\kappa,\l,0,\vec{\b})=-2T_\l B_\s(0,\kappa,\l,0,\vec{\b})\cV_\s(a,\vec{b};0,\kappa,\l,0,\vec{\b})-T_\l B_{\s\s}(0,\kappa,\l,0,\vec{\b})\Upsilon_a.
	\ee
	Applying $B_\e$ to $\cV_{\s\s}$, we find
	\ba
	B_\e(0,\kappa,\l,0,\vec{\b})\cV_{\s\s}(a,\vec{b};0,\kappa,\l,0,\vec{\b})&=-2B_\e(0,\kappa,\l,0,\vec{\b})T_\l B_\s(0,\kappa,\l,0,\vec{\b})\cV_\s(a,\vec{b};0,\kappa,\l,0,\vec{\b})\\
	&\quad-B_\e(0,\kappa,\l,0,\vec{\b})T_\l B_{\s\s}(0,\kappa,\l,0,\vec{\b})\Upsilon_a\\
	&\quad=2B_\e(0,\kappa,\l,0\vec{\b})T_\l B_\s(0,\kappa,\l,0,\vec{\b})\big(T_\l B_\s(0,\kappa,\l,0,\vec{\b})\Upsilon_a\Big)\\
	&\quad-B_\e(0,\kappa,\l,0,\vec{\b})T_\l B_{\s\s}(0,\kappa,\l,0,\vec{\b})\Upsilon_a=(I)-(II).
	\ea
	Turning our attention to $(I)$, we have
	\ba
	(I)&=2B_\e(0,\kappa,\l,0\vec{\b})T_\l B_\s(0,\kappa,\l,0,\vec{\b}) T_\l B_\s(0,\kappa,\l,0,\vec{\b})\Upsilon_a\\
	&\quad2 \Big(2\kappa k_*\begin{bmatrix}
		2 & 1\\
		1 & 2
	\end{bmatrix}\d_\xi^2+\kappa\begin{bmatrix}
		0 & 0\\+
		c_1 & c_2\end{bmatrix}\d_\xi+(k_*d_\e(0,\kappa)+\kappa d_* )\d_\xi\Big)T_\l B_\s(0,\kappa,\l,0,\vec{\b}) T_\l B_\s(0,\kappa,\l,0,\vec{\b})\Upsilon_a\\
	&\quad +\Big(\begin{bmatrix}
		k_*\d_\xi\Upsilon_\a^1 & 0\\
		0 & \Upsilon_\a^2\end{bmatrix}+\begin{bmatrix}
		\Upsilon_\a^1 & 0\\
		0 & 0\end{bmatrix}k_*\d_\xi \Big)T_\l B_\s(0,\kappa,\l,0,\vec{\b}) T_\l B_\s(0,\kappa,\l,0\vec{\b}) \Upsilon_a=(IA)+(IB).
	\ea
	Now, $(IA)$ has Fourier mean zero since it is a constant differential operator, hence Fourier multiplier operator, applied to another Fourier multiplier operator $T_\l B_\s(0,\kappa,\l,0,\vec{\b}) T_\l B_\s(0,\kappa,\l,0,\vec{\b})$ that is then applied to a function Fourier supported in $\{\pm1\}$. Turning to $(IB)$, we have
	\ba
	(IB)=\Big(\begin{bmatrix}
		k_*\d_\xi\Upsilon_\a^1 & 0\\
		0 & \Upsilon_\a^2\end{bmatrix}+\begin{bmatrix}
		\Upsilon_\a^1 & 0\\
		0 & 0\end{bmatrix}k_*\d_\xi \Big)T_\l B_\s(0,\kappa,\l,0,\vec{\b}) T_\l B_\s(0,\kappa,\l,0\vec{\b}) \Upsilon_a.
	\ea
	To simplify $(IB)$, we first expand $B_\s(0,\kappa,\l,0,\vec{\b})\Upsilon_a$ as
	\ba
	B_\s(0,\kappa,\l,0,\vec{\b})\Upsilon_a&=2ik_*^2\begin{bmatrix}
		2 & 1\\
		1 & 2
	\end{bmatrix}\d_\xi\Upsilon_a+ik_*\begin{bmatrix}
		0 & 0\\
		c_1 & c_2\end{bmatrix}\Upsilon_a+iC(0,\kappa)\Upsilon_a\\
	&\quad=ik_*^2 \begin{bmatrix} 2 & 1 \\ 1 & 2\end{bmatrix}\big( iare^{i\xi}-i\bar{a}\bar{r}e^{-i\xi}\big)+\frac{ik_*}{2}\begin{bmatrix} 0 & 0 \\ c_1 & c_2\end{bmatrix}( ar e^{i\xi}+ \bar{a}\bar{r}e^{-i\xi})+iC(0,\kappa)\Upsilon_a.
	\ea
	Applying $T_\l$ annihilates $iC(0,\kappa)\Upsilon_a$ as the symbol of $T_\l$ at frequency 1 is given by $N_1(\l)$ defined to be $N_1(\l):=[(I-\Pi_1)(S(k_*,0)+id_*k_*-\l)(I-\Pi_1)]^{-1}$. From this, we can conclude that $T_\l T_\l B_\s(0,\kappa,\l,0,\vec{\b})\Upsilon_a$ is given by
	\ba
	T_\l B_\s(0,\kappa,\l,0,\vec{\b}) T_\l B_\s(0,\kappa,\l,0,\vec{\b})\Upsilon_a&=\frac{1}{2}a k_*N_1(\l)S_k(k_*,0)N_1(\l) S_k(k_*,0)re^{i\xi }\\
	&\quad+\frac{1}{2}\bar{a}k_*\overline{N_1(\overline{\l})} S_k(-k_*,0)\overline{N_1(\overline{\l})}S_k(-k_*,0)\bar{r}e^{-i\xi}.
	\ea
	From here, we may compute $\Pi_0(IB)$ as
	\ba
	\Pi_0(IB)&=\frac{1}{2}k_*^2\Pi_0\Big(\begin{bmatrix} \d_\xi\Upsilon_\a^1+\Upsilon_\a^1\d_\xi & 0\\ 0 & 0 \end{bmatrix}\Big)\Big(a N_1(\l)S_k(k_*,0)N_1(\l) S_k(k_*,0)e^{i\xi }\\
	&\quad+\bar{a}\overline{N_1(\overline{\l})} S_k(-k_*,0)\overline{N_1(\overline{\l})}S_k(-k_*,0)\bar{r}e^{-i\xi}  \Big)\\
	&\quad=\frac{1}{4}k_*^2\a\Big(\bar{a} ir_1\Pi_0\overline{N_1(\overline{\l})}S_k(-k_*,0)\overline{N_1(\overline{\l})}S_k(-k_*,0)\bar{r}-a i\bar{r_1}\Pi_0 N_1(\l)S_k(k_*,0)N_1(\l) S_k(k_*,0)r\\
	&\quad+ai\bar{r_1}\Pi_0N_1(\l)S_k(k_*,0)N_1(\l) S_k(k_*,0)r-\bar{a}ir_1\Pi_0\overline{N_1(\overline{\l})}S_k(-k_*,0)\overline{N_1(\overline{\l})}S_k(-k_*,0)\bar{r}  \Big)\\
	&\quad+f(a)e^{2i\xi}+g(a)e^{-2i\xi},
	\ea
	where $f$ and $g$ are known functions of $a$. In particular, one sees that Fourier mean of $\Pi_0(IB)$ is identically zero, from which we conclude that the Fourier mean of $(I)$ is identically zero. This leaves us with computing the Fourier mean of $(II)$, which we expand as
	\ba
	(II)&=2B_\e(0,\kappa,\l,0\vec{\b})T_\l B_{\s\s}(0,\kappa,\l,0,\vec{\b})\Upsilon_a\\
	&\quad=2 \Big(2\kappa k_*\begin{bmatrix}
		2 & 1\\
		1 & 2
	\end{bmatrix}\d_\xi^2+\kappa\begin{bmatrix}
		0 & 0\\+
		c_1 & c_2\end{bmatrix}\d_\xi+(k_*d_\e(0,\kappa)+\kappa d_* )\d_\xi\Big)T_\l B_{\s\s}(0,\kappa,\l,0,\vec{\b})\Upsilon_a\\
	&\quad +\Big(\begin{bmatrix}
		k_*\d_\xi\Upsilon_\a^1 & 0\\
		0 & \Upsilon_\a^2\end{bmatrix}+\begin{bmatrix}
		\Upsilon_\a^1 & 0\\
		0 & 0\end{bmatrix}k_*\d_\xi \Big)T_\l B_{\s\s}(0,\kappa,\l,0\vec{\b}) \Upsilon_a=(IIA)+(IIB).
	\ea
	One has that the Fourier mean of $(IIA)$ vanishes by a similar computation to the corresponding one for $(IA)$. For $(IIB)$, we look at $\Pi_0(IIB)$ as before and expand using \eqref{eq:exampleBSS} to get
	\ba
	\Pi_0(IIB)=-2k_*^3\Big(\begin{bmatrix}
		\d_\xi \Upsilon_\a^1+\Upsilon_\a^1\d_\xi & 0\\ 0 & 0
	\end{bmatrix} \Big)T_\l\begin{bmatrix}
		2 & 1\\
		1 & 2
	\end{bmatrix}\Upsilon_a.
	\ea
	From here, it a similar calculation to show that the Fourier mean of $\Pi_0(IIB)$ vanishes identically, which gives us the conclusion that the Fourier mean of $PB_\e\cV_{\s\s}$ vanishes identically.\\
	
	Turning to $PB_\s\cV_{\e\s}$, we note that the symbol of $PB_\s(0,\kappa,\l,0,\vec{\b})T_\l$ at frequency zero is given by
	\be
	\Pi_0\Big(ik_*\begin{bmatrix}
		0 & 0\\
		c_1 & c_2\end{bmatrix}+iC(0,\kappa) \Big)[(I-\Pi_0 )(S(0,0)-\l)(I-\Pi_0)]^{-1}=0,
	\ee
	as we recall
	\begin{equation*}
		\Pi_0=\begin{bmatrix}
			1 & 0\\
			0 & 0
		\end{bmatrix}.
	\end{equation*}
\end{proof}
To complete the reduction process of the coefficient of $\e\s^2$ in $PB\cV$, we prove the following proposition.
\begin{proposition}\label{prop:exampleB2cV1ESS}
	The Fourier means of $PB_{\e\s}\cV_\s$ and $PB_{\s\s}\cV_\e$ are in reduced form given by
	\be
	-\frac{1}{2}k_*^2\a \Re(a\bar{v})+\cO(|\l|),
	\ee
	and
	\be
	-\frac{1}{2}k_*^2\a(a+\bar{a})|r_2|^2+\cO(|\l|),
	\ee
	respectively.
\end{proposition}
\begin{proof}
	Starting with $PB_{\s\s}\cV_\e$, we see that the Fourier mean of $\cV_\e$ is given by
	\be
	\widehat{\cV_\e}(a,\vec{b};0,\kappa,\l,0,\vec{\b})(0)=-[(I-\Pi_0)(S(0,0)-\l)(I-\Pi_0) ]^{-1}\text{Mean}\Bigg(\Big(\begin{bmatrix}
		k_*\d_\xi\Upsilon_\a^1 & 0\\
		0 & \Upsilon_\a^2\end{bmatrix}+\begin{bmatrix}
		\Upsilon_\a^1 & 0\\
		0 & 0\end{bmatrix}k_*\d_\xi \Big)\Upsilon_a\Bigg).
	\ee
	Expanding this out, we get
	\be
	\widehat{\cV_\e}(a,\vec{b};0,\kappa,\l,0,\vec{\b})(0)=-\frac{1}{4}\a(a+\bar{a})[(I-\Pi_0)(S(0,0)-\l)(I-\Pi_0) ]^{-1}\bp 0 \\ |r_2|^2\ep.
	\ee
	Applying $PB_{\s\s}$ to the above expression and Taylor expanding with respect to $\l$ gives
	\be
	\text{Mean}\Big(PB_{\s\s}(0,\kappa,\l,0,\vec{\b})\cV_{\e}(a,\vec{b};0,\kappa,\l,0,\vec{\b})\Big)=-\frac{1}{2}k_*^2\a(a+\bar{a})|r_2|^2+\cO(|\l|).
	\ee
	To simplify the Fourier mean of $PB_{\e\s}\cV_\s$, it suffices to simplify the Fourier mean of
	\be
	-ik_*\Pi_0\cQ(\Upsilon_\a,T_\l B_\s(0,\kappa,\l,0,\vec{\b})\Upsilon_a).
	\ee
	In the notation of the proof of Proposition \ref{prop:exampleB1cV2ESS}, we see that the Fourier mean is given by
	\be
	\text{Mean}(PB_{\e\s}\cV_\s)=-\frac{1}{4}ik_*^2\a \Pi_0\Big(a\cQ(\bar{r},N_1(\l)S_k(k_*,0)r )+\bar{a}\cQ(r,\overline{N_1(\overline{\l})}S_k(-k_*,0)\bar{r})\Big).
	\ee	
	Taylor expanding with respect to $\l$ as before and using $S_k(-k_*,0)=-\overline{S_k(k_*,0)}$ gives
	\be
	\text{Mean}(PB_{\e\s}\cV_\s)=-\frac{1}{4}k_*^2\a \Pi_0\Big(a \cQ(\overline{r},i N_1S_k(k_*,0)r)+\bar{a}\cQ(r,\overline{iN_1S_k(k_*,0)r)})+\cO(|\l|).
	\ee
	Recalling the definition of $v$ from \eqref{eq:exampleeps4AB2} we get as our final reduced form
	\be
	\text{Mean}(PB_{\e\s}\cV_\s)=-\frac{1}{2}k_*^2\a \Re(a\bar{v})+\cO(|\l|).
	\ee
\end{proof}
To finish the reduction process, we look at the coefficient of $\e^2\s$.
\begin{proposition}\label{prop:examplePBcVEES}
	In reduced form, the Fourier means of $PB_{\e\s}\cV_\e$, $PB_{\e\e}\cV_\s$, $PB_\e\cV_{\e\s}$ and $PB_\s\cV_{\e\e}$ are
	\be
	\text{Mean}(PB_{\e\s}(0,\kappa,\l,0,\vec{\b})\cV_\e(a,\vec{b};0,\kappa,\l,0,\vec{\b}))=ik_*\a\kappa \Re(-ia\overline{v})+\cO(|\l|),
	\ee
	\be
	\text{Mean}(PB_{\e\e}(0,\kappa,\l,0,\vec{\b})\cV_\s(a,\vec{b};0,\kappa,\l,0,\vec{\b}))=0,
	\ee
	\be
	\text{Mean}(PB_\e(0,\kappa,\l,0,\vec{\b})\cV_{\e\s}(a,\vec{b};0,\kappa,\l,0,\vec{\b}))=0,
	\ee
	and
	\be
	\text{Mean}(PB_\s(0,\kappa,\l,0,\vec{\b})\cV_{\e\e}(a,\vec{b};0,\kappa,\l,0,\vec{\b}))=0,
	\ee
	respectively.
\end{proposition}
\begin{proof}
	Recalling the expansion for $B_{\e}(0,\kappa,\l,0,\vec{\b})$ from \eqref{eq:exampleBE} we find that
	\be
	\text{Mean}(PB_\e(0,\kappa,\l,\vec{\b})\cV_{\e\s}(a,\vec{b};0,\kappa,\l,0,\vec{\b}))=\text{Mean}\Big(\frac{1}{2}\Big(\begin{bmatrix}
		k_*\d_\xi\Upsilon_\a^1 & 0\\
		0 & 0\end{bmatrix}+\begin{bmatrix}
		\Upsilon_\a^1 & 0\\
		0 & 0\end{bmatrix}k_*\d_\xi\Big)\cV_{\e\s}(a,\vec{b};0,\kappa,\l,0,\vec{\b})\Big).
	\ee
	But this is 0, because it is of the form 
	\be
	\text{Mean}(\d_\xi (fg)),
	\ee
	by the product rule.\\
	
	From the expansion for $B_\s(0,\kappa,\l,0,\vec{\b})$ in \eqref{eq:exampleBS}, we see that $PB_\s(0,\kappa,\l,0,\vec{\b})$ is a Fourier multiplier operator with symbol at frequency zero given by
	\be
	\widehat{PB_\s(0,\kappa,\l,0,\vec{\b})}(0)=ik_*\Pi_0\begin{bmatrix}
		0 & 0\\
		c_1 & c_2\end{bmatrix}+iC(0,\kappa)\Pi_0=iC(0,\kappa)\Pi_0.
	\ee
	As $\cV_{\e\e}$ satisfies $P\cV_{\e\e}=0$, we see that the mean of $PB_\s\cV_{\e\e}$ is identically 0.\\
	
	For $PB_{\e\s}\cV_\e$ and $PB_{\e\e}\cV_\s$ we only need to consider the effects of the nonlinearity as
	\ba
	&\Pi_0\text{Mean}\Big(4i\kappa k_*\begin{bmatrix} 2 & 1\\
		1 & 2\end{bmatrix}\d_\xi+i\kappa\begin{bmatrix}
		0 & 0\\
		c_1 & c_2
	\end{bmatrix}+iC_\e(0,\kappa) \Big)\cV_\s(a,\vec{b};0,\kappa,\l,0,\vec{\b})=0,\\
	&\Pi_0\text{Mean}\Bigg(\Big(-2\kappa^2\begin{bmatrix} 2 & 1 \\ 1 & 2\end{bmatrix}\d_\xi^2+k_*\begin{bmatrix}0 & 0 \\ 2 & 0\end{bmatrix}\d_\xi\\
	&\quad+(k_*d_{\e\e}(0,\kappa,\vec{\b})+2\kappa d_\e(0,\kappa,\vec{\b}))\d_\xi\Big)\cV_\s(a,\vec{b};0,\kappa,\l,0,\vec{\b})\Bigg)=0,
	\ea
	by \eqref{eq:exampleBES} and \eqref{eq:exampleBEE} respectively. For the contribution of the nonlinearity in $PB_{\e\e}\cV_\s$, we see that the mean is also 0 which follows from writing $\Pi_0 D_U\cN(\tilde{u}_{\e,\kappa,\vec{\b}};k)$ in divergence form as $\d_\xi\Pi_0\tilde{\cN}(\tilde{u}_{\e,\kappa,\vec{\b}};k)$. Then note that no numbers of $\e$ derivatives can eliminate the $\d_\xi$ in front. This leaves computing the mean of $B_{\e\s}\cV_\e$ as the final step. We see that
	\ba
	&\text{Mean}(PB_{\e\s}(0,\kappa,\l,0,\vec{\b})\cV_\e(a,\vec{b};0,\kappa,\l,0,\vec{\b}))\\
	&\quad=\frac{1}{2}ik_*\a\Pi_0\Big(\begin{bmatrix} r_1 & 0 \\ 0 & r_2\end{bmatrix}\widehat{\cV_\e}(-1)+\begin{bmatrix} \overline{r_1} & 0 \\ 0 & \overline{r_2}\end{bmatrix}\widehat{\cV_\e}(1) \Big).
	\ea
	This can be expressed as
	\be
	\text{Mean}(PB_{\e\s}\cV_\e)=\frac{1}{2}ik_*\a\kappa \Pi_0\Big( -a\begin{bmatrix} \overline{r_1} & 0 \\ 0 & \overline{r_2} \end{bmatrix} (I-\Pi_1)N_1(I-\Pi_1) S_k(k_*,0)r+c.c. )+\cO(|\l|).
	\ee
	Recalling the definition of $v$ from \eqref{eq:exampleeps4AB2}, we see that this is given by
	\be
	\text{Mean}(PB_{\e\s}\cV_\e)=\frac{1}{2}ik_*\a\kappa (-ia\overline{v}+c.c. )+\cO(|\l|)=ik_*\a\kappa \Re(-ia\overline{v})+\cO(|\l|).
	\ee
\end{proof}
Combining Theorem \ref{thm:exampleReducedPart1} and Propositions \ref{prop:exampleB1cV2ESS}, \ref{prop:exampleB2cV1ESS} and \ref{prop:examplePBcVEES}, and Lemmas \ref{lem:examplePB0cV}, \ref{lem:examplecVOES}, \ref{lem:examplecVOSS} and \ref{lem:examplecVOSSS} gives the desired reduced equation.
\begin{theorem}\label{thm:exampleReducedPart2}
	In the same notation used in Theorem \ref{thm:exampleReducedPart1}, the $PB(\e,\kappa,\l,\s,\vec{\b})W(\xi)=0$ in reduced form is given by
	\be
	-\l\bp a_1 \\ a_2\\b\ep+ (-\s^2 M_2(\e,\s)+i\s(\cS_1+\e\cS_2+\e M_1(\e,\s))+\e^2 M_0(\e,\s))+\cO(\e^2|\l|,\e\s|\l|,\s^2|\l|)=0,
	\ee
	where the matrices are given by
	\be
	M_0(\e,\s)=\bp 2\a^2\Re(\g)+\cO(\e) & 0 & \a \Re(V_1)\cdot e_1+\cO(\e)\\ 
	2\a^2\Im(\g)+\cO(\e) & 0 & \a \Im(V_1)\cdot e_1+\cO(\e)\\
	0 & 0 & 0\ep,
	\ee
	\be
	M_1(\e,\s)=\kappa k_*\bp -\Im(\tl_{kk}(k_*,0)) & -\Re(\tl_{kk}(k_*,0) & 0\\
	\Re(\tl_{kk}(k_*,0)) & -\Im(\tl_{kk}(k_*,0)) & 0\\
	-2\a\Im(v)\cdot e_1 & 0 & 0\ep+\cO(\e,\s), 
	\ee
	\be
	\cS_1=\bp 0 & 0 & 0\\
	0 & 0 & 0\\
	\frac{1}{2}k_*\a|r_1|^2 & 0 & C(0,\kappa)\ep,
	\ee
	\be
	\cS_2=\bp 0 & 0 & 0\\
	0 & 0 & 0\\
	\frac{1}{2}(\kappa \a +k_*\a_1)|r_1|^2 & 0 & C_\e(0,\kappa)\ep,
	\ee
	\be
	M_2(\e,\s)=k_*^2\bp \frac{1}{2}\Re(\tl_{kk}(k_*,0)) & -\frac{1}{2}\Im(\tl_{kk}(k_*,0)) & 0\\
	\frac{1}{2}\Im(\tl_{kk}(k_*,0)) & \frac{1}{2}\Re(\tl_{kk}(k_*,0)) & 0\\
	-\a |r_2|^2-\a \Re(v)\cdot e_1 & -\a \Im(v)\cdot e_1 & -2 \ep+\cO(\e,\s).
	\ee
\end{theorem}
\section{Some numerical observations}
Consider the hyperbolic model,
\ba
A_T&=aA_{XX}+bA+c|A|^2A+dAB,\\
B_T&=\e^{-1}(fB+h|A|^2)_X,
\ea
where $a,b,c,d\in\CC$ and $f,h\in\RR$ with $\Re(a)>0$, $\Re(b)>0$, $\Re(c)<0$, $f\not=0$, and $\e>0$ a small parameter. We saw that after linearizing about solutions of the form $(A,B)(X,T)=(A_0e^{i(\kappa X-\Omega T)},B_0 )$ by $(A,B)(X,T)\to ((A_0+u+iv)e^{i(\kappa X-\Omega T)},B_0+w)$ and rearranging leads to a system of the form
\ba
\bp u_T \\ v_T\\ w_T \ep&=
\bp \Re a & -\Im a & 0 \\ \Im a & \Re a& 0\\ 0  & 0 & 0 \ep 
\bp u_{XX}  \\ v_{XX}\\ w_{XX} \ep
+\bp -2\kappa\Im a & -2\kappa\Re a &0 \\ 2\kappa\Re a & -2\kappa\Im a &0\\ 2 \eps^{-1}hA_0 & 0  & \eps^{-1} f\ep 
\bp u_X \\ v_X\\ w_X \ep  \\
&\quad 
+\bp 2A_0^2\Re c & 0& \Re d A_0 \\ 2A_0^2\Im c & 0& \Im d A_0\\ 0 & 0 & 0 \ep
\bp u \\ v\\ w \ep,
\ea
where $A_0>0$ is defined by
\begin{equation}
	A_0(\kappa,B_0)^2= \frac{-\Re (b+d B_0)+\Re a \kappa^2}{\Re c},
\end{equation}
for $\kappa$ with $\kappa^2\leq\kappa_E^2$ with $\kappa_E^2$ defined by
\begin{equation}
	\kappa_E^2:=\frac{\Re( b+dB_0)}{\Re a}.
\end{equation}
We will also have use of the associated Darcy reduction, obtained by setting $B=B_0-\frac{h}{f}|A|^2$ and simplifying the Ginzburg-Landau equation accordingly. We denote the updated values of $b$ and $c$ as $\tilde{b}$ and $\tilde{c}$ accordingly and observe they are given by
\ba
\tilde{b}&:=b+dB_0,\\
\tilde{c}&:=c-\frac{dh}{f}.
\ea

Assuming that $(u,v,w)$ is of the form
\be
\bp u \\ v \\ w\ep=\bp u_0 \\ v_0\\ w_0\ep e^{i\s X-\l T},
\ee
leads us to consider the eigenvalue problem of the matrix
\ba
M(\e,\s):=-\sigma^2 \bp \Re a & -\Im a & 0 \\ \Im a & \Re a& 0\\ 0  & 0 &0 \ep
+i\sigma \bp -2\kappa\Im a & -2\kappa\Re a &0 \\ 2\kappa\Re a & -2\kappa\Im a &0\\ 2 \eps^{-1}hA_0 & 0  & \eps^{-1} f\ep 
+\bp 2A_0^2\Re c & 0& \Re d A_0 \\ 2A_0^2\Im c & 0& \Im d A_0\\ 0 & 0 & 0 \ep.
\ea

\subsection{Degenerate Case}
Let us now begin numerically exploring the behavior of the eigenvalues of $M(\e,\s)$. We begin by choosing the model parameters
\begin{equation}\label{eq:param}
	\begin{gathered}
		a=1+i, \quad b=1,\\
		f=1, \quad h=2, \quad \e=10^{-2}, \quad B_0=0,	
	\end{gathered}
\end{equation}
while varying $c$ and $d$. In this section, we focus on $d$ and $c$ with $d/c\in\RR$ which implies that $M(\e,0)$ is semisimple as the third column is a real multiple of the first. We are then in a situation to apply our spectral expansions derived in Section \ref{sub:dispersion}. We begin with the conserved eigenvalue $\l_c$ whose real part admitted the expansion
\begin{equation}
	\Re\l_c(\e,\s)=(\e^{-2}\frac{\Re(d)fh}{2A_0\Re(c)^3}\Re(\tilde{c})+\cO(\e^{-1}))\s^2+\cO(\s^3).
\end{equation}
By flipping the sign on $d$ and leaving the other variables unchanged, we can either destabilize or stabilize the conserved mode as appropriate. We plot each set of dispersion relations on the interval $\s\in[-10\e,10\e]$. All figures in this section were made with the Python libraries numpy and matplotlib.pyplot. In Figures \ref{fig1} and \ref{fig2}, we plot the dispersion relations for the set of parameters in \eqref{eq:param}, where we see that the conserved mode is unstable.
\begin{figure}[h]
	\includegraphics[scale=0.5]{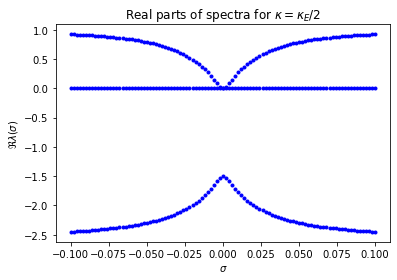}
	\caption{The real parts of the dispersion relations at $\kappa_E/2$ with $(c,d)=(-3+3i,1-1i)$.}
	\label{fig1}
\end{figure}
\begin{figure}[h]
	\includegraphics[scale=0.5]{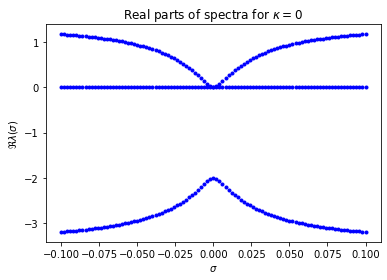}
	\caption{The real parts of the dispersion relations at $\kappa=0$ with $(c,d)=(-3+3i,1-1i)$.}
	\label{fig2}
\end{figure}
We see that this does not stabilize the translational mode as evidenced by Figures \ref{fig3} and \ref{fig4}.
\begin{figure}[h]
	\includegraphics[scale=0.5]{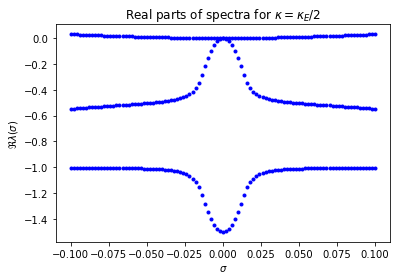}
	\caption{The real parts of the dispersion relations at $\kappa_E/2$ with $(c,d)=(-3+3i,-1+1i)$.}
	\label{fig3}
\end{figure}
\begin{figure}[h]
	\includegraphics[scale=0.5]{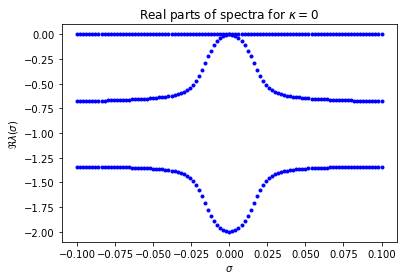}
	\caption{The real parts of the dispersion relations at $\kappa=0$ with $(c,d)=(-3+3i,-1+1i)$.}
	\label{fig4}
\end{figure}
To see numerically stable waves with $\kappa\not=0$, we change the $c$ and $d$ to
\be\label{eq:param2}
c=-3+2i, \quad d=\frac{1}{3}c.
\ee
Plotting the dispersion relations as before in Figures \ref{fig5} and \ref{fig6}, we see that in Figure \ref{fig5} that the constant state is stable and that the exponential state at $\kappa=0.5$ is unstable in Figure \ref{fig6}.
\begin{figure}[h]
	\includegraphics[scale=0.5]{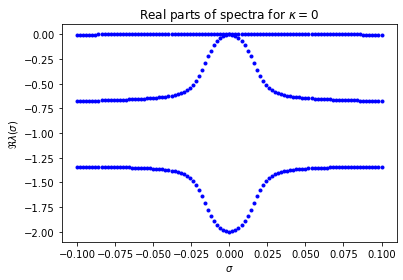}
	\caption{The real parts of the dispersion relations at $\kappa=0$ with $(c,d)=(-3+2i,-1+2/3i)$.}
	\label{fig5}
\end{figure}
\begin{figure}[h]
	\includegraphics[scale=0.5]{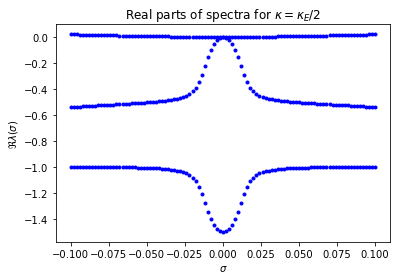}
	\caption{The real parts of the dispersion relations at $\kappa=0.5$ with $(c,d)=(-3+2i,-1+2/3i)$.}
	\label{fig6}
\end{figure}
We now plot the function $\cS(\kappa,\rho)$ defined by
\begin{equation}
	\cS(\kappa,\rho):=\sup_{|\s|\leq\rho}\max_{1\leq j\leq 3} \Re \l_j(\s),
\end{equation}
where $\l_j$ is an eigenvalue of the matrix $M(\e,\s)$ obtained by linearizing about the state $(A,B)=(A_0(\kappa)e^{i(\kappa X-\Omega T)},0)$ and the parameters \eqref{eq:param} and varying choices of $(c,d)$. Note that $\cS(\kappa,10\e)=0$ implies that, outside of edge cases, the background state is stable with respect to low-frequency perturbations. We will simultaneously plot the function $\cS_D(\kappa)$ defined by
\begin{equation}
	\cS_D(\kappa):=\sup_{|\s|\leq 1}\max_{j=1,2} \l_{j,D}(\sigma),
\end{equation}
where the $\l_{j,D}$ are the corresponding dispersion relations coming from the Darcy model. In Figures \ref{fig7}, \ref{fig8}, \ref{fig9} and \ref{fig10}, the blue curve is $\cS(\kappa,10\e)$ and the green curve is $\cS_D(\kappa)$. To avoid numerical instabilities, we restrict $\kappa$ to the range $|\kappa|\leq 0.9\kappa_E$.\\
\begin{figure}[h]
	\includegraphics[scale=0.5]{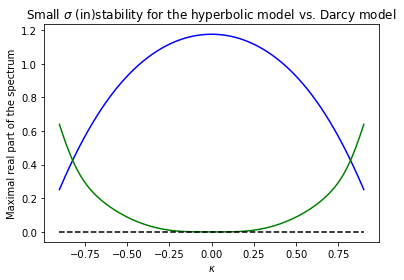}
	\caption{Low frequency and Darcy stability diagram for the set of parameters \eqref{eq:param} and $(c,d)=(-3+3i,1-i)$.}
	\label{fig7}
\end{figure}
\begin{figure}[h]
	\includegraphics[scale=0.5]{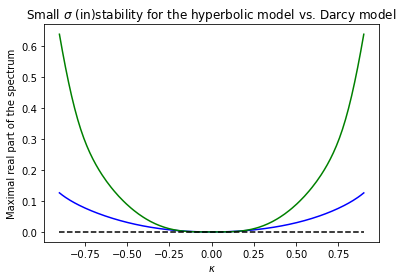}
	\caption{Low frequency and Darcy stability diagram for the set of parameters \eqref{eq:param} and $(c,d)=(-3+3i,-1+i)$.}
	\label{fig8}
\end{figure}
\begin{figure}[h]
	\includegraphics[scale=0.5]{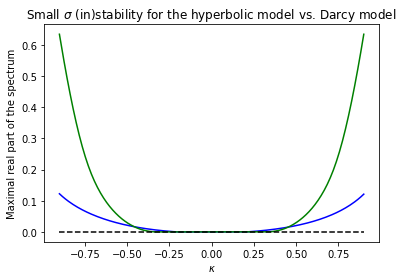}
	\caption{Low frequency and Darcy stability diagram for the set of parameters \eqref{eq:param} and $(c,d)=(-3+2i,-1+2/3i)$.}
	\label{fig9}
\end{figure}
\begin{figure}[h]
	\includegraphics[scale=0.5]{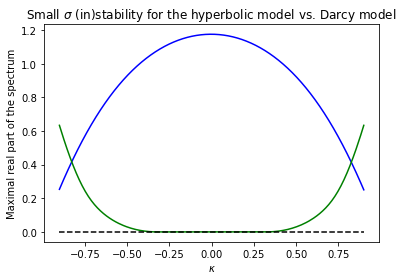}
	\caption{Low frequency and Darcy stability diagram for the set of parameters \eqref{eq:param} and $(c,d)=(-3+2i,1-2/3i)$.}
	\label{fig10}
\end{figure}
We can a make a few observations in Figures \ref{fig7}, \ref{fig8}, \ref{fig9} and \ref{fig10}. The first is that when the conserved mode is unstable, as evidenced by Figures \ref{fig7} and \ref{fig10}, then as expected \emph{every} background periodic wave is unstable. The second major observation is that the associated Darcy model having a stable wave is a necessary but not sufficient condition for low-frequency stability of the background wave at the same frequency. We note that Figure \ref{fig7} is associated to a set of parameters for which $\mu_t$ and $\mu_c$ go bad, Figure \ref{fig8} is associated to parameters for which $\mu_c$ is good but $\mu_t$ goes bad, Figure \ref{fig9} is associated to parameters for which $\mu_t$ and $\mu_c$ are both good on some interval and Figure \ref{fig10} is associated to a set of parameters with a good range on $\mu_t$ and $\mu_c$ is bad. We can organize this observation in a table, where all parameters but $(c,d)$ are as in \eqref{eq:param}
\begin{center}
	\begin{tabular}{c c c}
		& $\mu_c<0$ & $\mu_c>0$\\
		$\mu_t<0$ & $(c,d)=(-3+2i,-1+2/3i)$. & $(c,d)=(-3+2i,1-2/3i)$.\\
		$\mu_t>0$ & $(c,d)=(-3+3i,-1+i)$ & $(c,d)=(-3+3i,1-i)$.
	\end{tabular}
\end{center}

\subsection{Generic Case}
We now turn to numerically studying the behavior of the hyperbolic model when $c$ and $d$ are not real multiples of each other anymore. Let us start with the choice of $(c,d)$ given by
\be\label{eq:param3}
c=-3+2i, \quad d=1-1i,
\ee
and	the remaining parameters as in \eqref{eq:param}. For this set of parameters, we have as before that the conserved mode is unstable as evidenced by Figures \ref{fig11} and \ref{fig12}
\begin{figure}
	\includegraphics[scale=0.5]{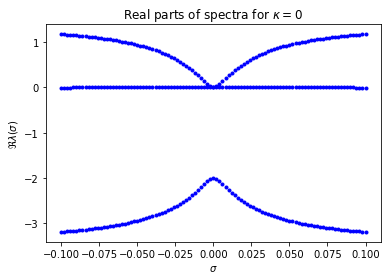}
	\caption{The real parts of the dispersion relations at $\kappa=0$ for an stable wave with $(c,d)=(-3+2i,1-i)$.}
	\label{fig11}
\end{figure}
\begin{figure}
	\includegraphics[scale=0.5]{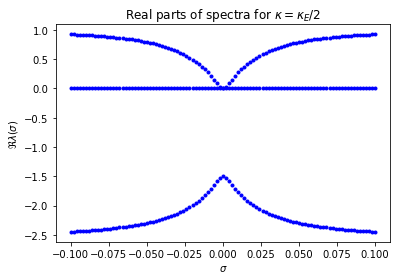}
	\caption{The real parts of the dispersion relations at $\kappa=\kappa_E/2$ for an stable wave with $(c,d)=(-3+2i,1-i)$.}
	\label{fig12}
\end{figure}
We also plot the stability diagram, where we also observe that the associated Darcy reduction has a larger stable band than the hyperbolic model.
\begin{figure}
	\includegraphics[scale=0.5]{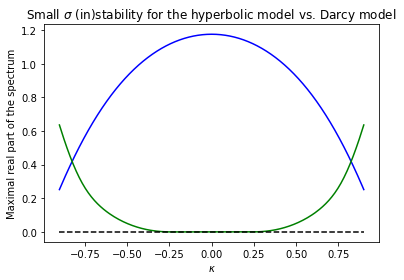}
	\caption{Low frequency and Darcy stability diagram for the set of parameters \eqref{eq:param} and $(c,d)=(-3+2i,1-i)$.}
	\label{fig13}
\end{figure}
What we want to emphasize here is that many of the same features of the degenerate case where $d/c\in\RR$ remain true here, namely that stability for the Darcy model is necessary but not sufficient and that we have two modes $\l_c$ and $\l_s$ which have large second derivatives at $\s=0$ and the third mode $\l_t$ has $O(1)$ second derivative at the origin. To get a stable band of waves, we flip the sign of $d$ in \eqref{eq:param3}. This leads to the sample spectral plots and stability diagram in Figures \ref{fig14}, \ref{fig15}, and \ref{fig16}.
\begin{figure}
	\includegraphics[scale=0.5]{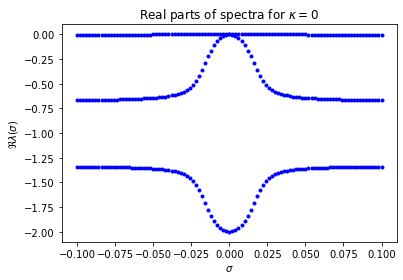}
	\caption{The real parts of the dispersion relations at $\kappa=0$ for an stable wave with $(c,d)=(-3+2i,-1+i)$.}
	\label{fig14}
\end{figure}
\begin{figure}
	\includegraphics[scale=0.5]{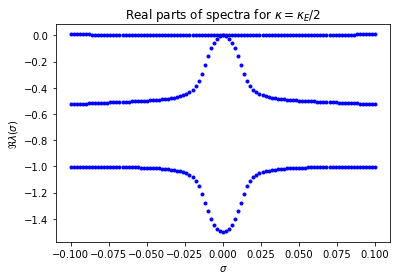}
	\caption{The real parts of the dispersion relations at $\kappa=\kappa_E/2$ for an stable wave with $(c,d)=(-3+2i,-1+i)$.}
	\label{fig15}
\end{figure}
We also plot the stability diagram, where we also observe that the associated Darcy reduction has a larger stable band than the hyperbolic model.
\begin{figure}
	\includegraphics[scale=0.5]{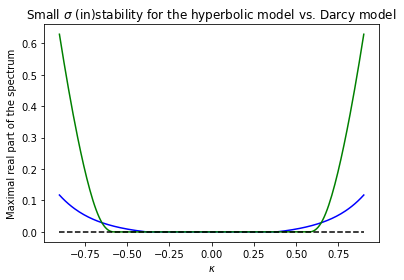}
	\caption{Low frequency and Darcy stability diagram for the set of parameters \eqref{eq:param} and $(c,d)=(-3+2i,-1+i)$.}
	\label{fig16}
\end{figure}
As we're in the generic case with $d/c\in\CC\backslash\RR$, we have the liberty to conjugate $d$ without simultaneously conjugating $c$. Doing this for the parameters in \eqref{eq:param3} leads to the sample spectra and stability diagram in Figures \ref{fig17}, \ref{fig18}, and \ref{fig19}.
\begin{figure}
	\includegraphics[scale=0.5]{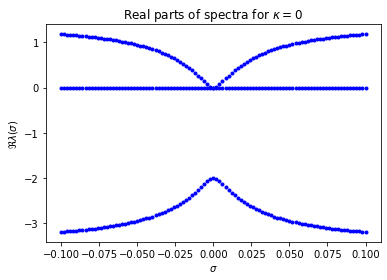}
	\caption{The real parts of the dispersion relations at $\kappa=0$ for an stable wave with $(c,d)=(-3+2i,1+i)$.}
	\label{fig17}
\end{figure}
\begin{figure}
	\includegraphics[scale=0.5]{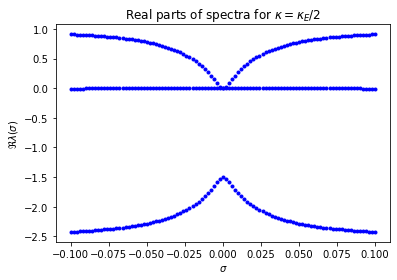}
	\caption{The real parts of the dispersion relations at $\kappa=\kappa_E/2$ for an stable wave with $(c,d)=(-3+2i,1+i)$.}
	\label{fig18}
\end{figure}
\begin{figure}
	\includegraphics[scale=0.5]{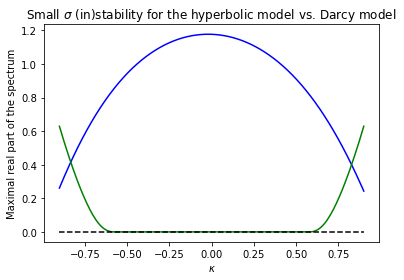}
	\caption{Low frequency and Darcy stability diagram for the set of parameters \eqref{eq:param} and $(c,d)=(-3+2i,1+i)$.}
	\label{fig19}
\end{figure}
To complete the quartet of $d,-d,\overline{d}$ and $-\overline{d}$ we plot the corresponding figures for $(-3+2i,-1-i)$ in Figures \ref{fig20}, \ref{fig21}, and \ref{fig22}. To explain the peculiarity of the plots in Figure \ref{fig22}, the reason is that the Benjamin-Feir-Newell criterion does not hold and hence the Darcy model has no stable solutions whatsoever. 
\begin{figure}
	\includegraphics[scale=0.5]{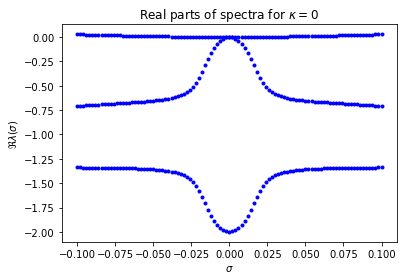}
	\caption{The real parts of the dispersion relations at $\kappa=0$ for an stable wave with $(c,d)=(-3+2i,1+i)$.}
	\label{fig20}
\end{figure}
\begin{figure}
	\includegraphics[scale=0.5]{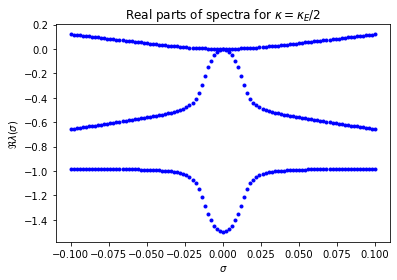}
	\caption{The real parts of the dispersion relations at $\kappa=\kappa_E/2$ for an stable wave with $(c,d)=(-3+2i,1+i)$.}
	\label{fig21}
\end{figure}
\begin{figure}
	\includegraphics[scale=0.5]{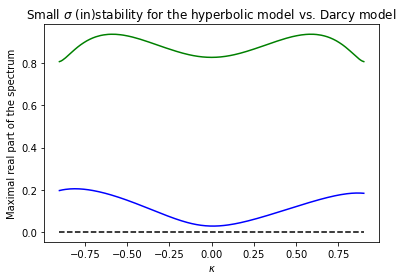}
	\caption{Low frequency and Darcy stability diagram for the set of parameters \eqref{eq:param} and $(c,d)=(-3+2i,1+i)$.}
	\label{fig22}
\end{figure}
What we see is that the stability of the conserved mode is, as expected, determined by the real part of $d$. We also see that the imaginary part of $d$ plays an important role in the width of the stable band for the Darcy model.

\end{document}